\documentclass{article}
\usepackage[a4paper, margin=3cm]{geometry}
\usepackage{authblk}
\usepackage{parskip}
\usepackage{amsthm,amsmath,amssymb}
\usepackage{mathtools}
\usepackage{tabularx}
\usepackage{booktabs}
\usepackage{enumitem}
\usepackage{etoolbox}
\usepackage{needspace}
\IfFileExists{dsfont.sty}{\usepackage{dsfont}\newcommand{\ind}{\mathds{1}}}{\newcommand{\ind}{\mathbf{1}}}
\usepackage{tikz}
\usepackage[framemethod=TikZ]{mdframed}
\usepackage{xurl}
\usepackage[colorlinks,linkcolor=blue,citecolor=blue,urlcolor=blue,hypertexnames=false]{hyperref}
\hypersetup{pdftitle={Weighted Power Fr\'echet Means in Metric Spaces with Curvature Bounded Above},pdfauthor={Christof Sch\"otz}}
\usepackage[capitalize,nameinlink,noabbrev]{cleveref}
\newcommand*{\mc}[1]{\mathcal{#1}}
\newcommand{\N}{\mathbb{N}}
\newcommand{\Q}{\mathbb{Q}}
\newcommand{\R}{\mathbb{R}}
\newcommand{\Rp}{\R_{\geq0}}
\newcommand{\Rpp}{\R_{>0}}
\newcommand{\lcb}{\left\lbrace} \newcommand{\rcb}{\right\rbrace} \newcommand{\cb}[1]{\lcb #1 \rcb} \newcommand{\lab}{\left[} \newcommand{\rab}{\right]} \newcommand{\ab}[1]{\lab #1 \rab} \newcommand{\abOf}[1]{\!\ab{#1}} \newcommand{\lb}{\left(} \newcommand{\rb}{\right)} \newcommand{\br}[1]{\lb #1 \rb} \newcommand{\brOf}[1]{\!\br{#1}} \newcommand{\abs}[1]{\left| #1 \right|} \newcommand{\absof}[1]{\vert#1\vert} \newcommand{\normof}[1]{\Vert#1\Vert}
\newcommand{\normOf}[1]{\left\Vert#1\right\Vert}
\newcommand{\cardof}[1]{\##1}
\newcommand{\innerProduct}[2]{\left\langle#1,#2\right\rangle}
\newcommand{\ip}[2]{\innerProduct{#1}{#2}}
\renewcommand{\Pr}{\mathbf{P}}
\newcommand{\PrOf}[1]{\Pr\brOf{#1}}
\newcommand{\E}{\mathbf{E}} \newcommand{\Eof}[1]{\E[#1]}
\newcommand{\EOf}[1]{\E\abOf{#1}}
\newcommand{\euclof}[1]{\Vert#1\Vert_2}
\newcommand{\euclOf}[1]{\left\Vert#1\right\Vert_2}
\newcommand{\mink}[2]{\langle #1,#2\rangle_{1,k}}
\newcommand{\pr}{^\prime}
\DeclareMathOperator{\landauO}{\mathcal{O}}
\newcommand{\dunif}[1]{\mathsf{Unif}\brOf{#1}}
\newcommand{\eqcm}{\,,} \newcommand{\eqfs}{\,.}
\newcommand{\dl}{\mathrm{d}} \renewcommand{\subset}{\subseteq}
\newcommand{\euler}{\mathrm{e}}
\DeclareMathOperator*{\argmin}{arg\,min}
\newcommand{\const}[1]{C_{\mathsf{#1}}}
\newcommand{\popdis}{\bar\mu}
\newcommand{\samdis}{\hat\mu}
\newcommand{\popfm}{\bar m}
\newcommand{\samfm}{\hat m}
\newcommand{\samndis}{\hat\mu_n}
\newcommand{\samnfm}{\hat m_n}
\newcommand{\ol}[2]{\overline{#1#2}}
\newcommand{\intset}[1]{\{1,\dots,#1\}}
\newcommand{\diam}{\mathsf{diam}}
\newcommand{\as}{\mathsf{a.s.}}
\newcommand{\id}{\mathrm{id}}
\newcommand{\Vf}{\mathbf{V_{\!F}}} \newcommand{\Vfof}[1]{\Vf[#1]}
\newcommand{\VfOf}[1]{\Vf\abOf{#1}}
\newcommand{\Vm}{\mathbf{V_{\!M}}} \newcommand{\Vmof}[1]{\Vm[#1]}
\newcommand{\VmOf}[1]{\Vm\abOf{#1}}
\newcommand{\Vmp}[1]{\mathbf{V}^{#1}_{\!M}} \newcommand{\VmpOf}[2]{\Vmp{#1}\abOf{#2}}
\newcommand{\Vfp}[1]{\mathbf{V}^{#1}_{\!F}}
\newcommand{\catof}[1]{$\mathsf{CAT}(#1)$}
\newcommand{\cato}{\catof{0}}
\newcommand{\catk}{\catof{\kappa}}
\newcommand{\Vmpof}[2]{\Vmp{#1}[#2]} \newcommand{\Vfpof}[2]{\Vfp{#1}[#2]}
\newcommand{\VmofP}[1]{\Vm(#1)} \newcommand{\VmOfP}[1]{\Vm\brOf{#1}}
\newcommand{\VfofP}[1]{\Vf(#1)}
\newcommand{\VmpofP}[2]{\Vmp{#1}(#2)} \newcommand{\VmpOfP}[2]{\Vmp{#1}\brOf{#2}}
\newcommand{\VfpofP}[2]{\Vfp{#1}(#2)} \newcommand{\VfpOfP}[2]{\Vfp{#1}\brOf{#2}}
\usepackage{xspace}
\newcommand{\QIa}{\eqref{eq:power:qi}\xspace}
\newcommand{\VIa}{\eqref{eq:power:vi}\xspace}
\newcommand{\CPa}{\eqref{eq:power:cp}\xspace}
\newcommand{\Lamof}[2]{\Lambda_\alpha(#1;#2)}
\newcommand{\LamOf}[2]{\Lambda_\alpha\brOf{#1;#2}}
\newcommand{\setC}{\mc C}
\newcommand{\seg}[2]{[#1,#2]}
\newcommand{\Dk}{D_\kappa}
\newcommand{\VIp}{\eqref{eq:2fmcatk:vi}\xspace}
\newcommand{\CPp}{\eqref{eq:2fmcatk:CP}\xspace}

\mdfsetup{splittopskip=\topskip, splitbottomskip=3\baselineskip}
\usepackage{thmtools}
\def\postBoxSkip{1.0ex}
\def\postBoxSkipCmd{\vskip\postBoxSkip}
\def\preBoxSkip{1.0ex}
\def\preBoxSkipCmd{\vskip\preBoxSkip}
\declaretheoremstyle[
	bodyfont=\normalfont,
	postfoothook={\postBoxSkipCmd},
	preheadhook={\preBoxSkipCmd},
	mdframed={
		backgroundcolor = black!2,
}]{ruledBoxStyle}
\declaretheoremstyle[
	bodyfont=\normalfont,
	postfoothook={\postBoxSkipCmd},
	preheadhook={\preBoxSkipCmd},
	mdframed={
		backgroundcolor=black!4,
		linecolor = black!4,
}]{boxStyle}
\declaretheoremstyle[
	bodyfont=\normalfont,
	postfoothook={\postBoxSkipCmd},
	preheadhook={\preBoxSkipCmd},
	mdframed={
		linecolor = white,
		tikzsetting = {
			draw = black,
			line width = 1pt,loosely dotted,
		},
	}
]{dashedStyle}
\declaretheorem[style=ruledBoxStyle,name=Definition, numberwithin=section]{definition}
\declaretheorem[style=ruledBoxStyle,name=Lemma,numberlike=definition]{lemma}
\declaretheorem[style=ruledBoxStyle,name=Proposition,numberlike=definition]{proposition}
\declaretheorem[style=ruledBoxStyle,name=Theorem,numberlike=definition]{theorem}
\declaretheorem[style=ruledBoxStyle,name=Corollary,numberlike=definition]{corollary}
\declaretheorem[style=boxStyle,name=Remark,numberlike=definition]{remark}
\declaretheorem[style=boxStyle,name=Notation,numberlike=definition]{notation}
\declaretheorem[style=dashedStyle,name=Example,numberlike=definition]{example}

\title{Weighted Power Fr\'echet Means in Metric Spaces with Curvature Bounded Above}
\date{}
\author[1,2]{Christof Sch\"otz\thanks{christof.schoetz@tum.de, ORCID: \href{https://orcid.org/0000-0003-3528-4544}{0000-0003-3528-4544}}}
\affil[1]{Munich Climate Center and Earth System Modelling Group, Department of Aerospace and Geodesy, TUM School of Engineering and Design, Technical University of Munich, Munich, Germany}
\affil[2]{Artificial Intelligence Group, Research Domain IV -- Complexity Science, Potsdam Institute for Climate Impact Research, Potsdam, Germany}

\begin{document}
\maketitle

\begin{abstract}
    We establish non-asymptotic risk bounds for power Fr\'echet means in geodesic metric spaces with curvature bounded above.
    The observations form weighted, possibly infinite sequences of independent random variables whose laws and means may differ.
    We treat three settings.
    For $2$-Fr\'echet means in Hadamard spaces, we obtain a sharp mean squared error bound that becomes an identity in Hilbert spaces.
    For $2$-Fr\'echet means in CAT($\kappa$) spaces with $\kappa>0$, we establish variance and Wasserstein contraction inequalities with optimal constants depending on the circumradius $\varrho$ of the closed convex data domain, and obtain mean squared error bounds throughout the maximal range $\varrho<\pi/(2\sqrt{\kappa})$.
    For $\alpha$-Fr\'echet means in Hadamard spaces, $1<\alpha<2$, we derive finite $L^\alpha$ risk bounds under weighted $\alpha$-moment conditions, allowing even an infinite $\alpha$-moment of the population mixture.
    The proofs combine variance, quadruple, and contraction inequalities with a leave-one-out stability technique.
    Applications give prior and posterior bounds for Dirichlet-process Fr\'echet means and finite-sample guarantees for local constant Fr\'echet regression.
    The regression results require no response-space entropy condition and replace density smoothness assumptions common in earlier work with transport smoothness; for $\alpha<2$, the bounds remain finite even when the responses have infinite variance.
\end{abstract}



\tableofcontents

\section{Introduction}\label{sec:intro}

Linearity is a defining property of expectation and underlies much of its probabilistic theory.
Results for Fr\'echet means show that substantial parts of this theory extend to metric spaces, from laws of large numbers \cite{Ziezold1977ExpectedFigures-Reidel,BhattacharyaPatrangenaru2003LargeSample-AOS,Sturm2003NPC-ContempMath} to recent variance-over-sample-size bounds for the mean squared error under independent and identically distributed (iid) observations \cite{LeGouicEtAl2023FastConvergence-JEMS,BrunelSerres2026FiniteSample-arxiv}.
These developments also support statistical methods such as nonparametric regression with metric responses \cite{PetersenMueller2019Regression-AOS}.
The classical Fr\'echet mean, defined through squared distance, retains a connection to linearity by coinciding with expectation in inner product spaces.
Replacing the exponent $2$ by $\alpha\in(1,2)$ moves beyond this connection: the resulting power Fr\'echet means are generally nonlinear even in Euclidean spaces, yet retain many useful properties of expectation and additionally gain robustness \cite{Schoetz2025VarianceInequalities-EJP,Schoetz2026Transformed-arxiv}.

This theory is particularly tractable under suitable upper curvature bounds, the geometric setting of this article.
Such bounds also hold in spaces used to represent modern data, including positive definite matrices and phylogenetic trees.
We extend the probabilistic theory of power Fr\'echet means beyond iid samples to weighted, possibly infinite sequences of independent observations with different laws.
Besides advancing the probabilistic theory of power Fr\'echet means, our results yield tail-robust nonparametric regression with metric responses and bounds for Dirichlet-process means in Bayesian statistics.

Let $(\mc Q,d)$ be a metric space.
For a distribution $\mu$ with a finite second moment, its \emph{Fr\'echet mean} \cite{Frechet1948RandomElements-AIHP} (or barycenter) minimizes $q\mapsto\int d(y,q)^2\,\dl\mu(y)$.
With only a finite first moment, one instead minimizes the renormalized objective $q\mapsto\int(d(y,q)^2-d(y,o)^2)\,\dl\mu(y)$ for a fixed reference point $o\in\mc Q$.
In \emph{Hadamard spaces}, that is, complete metric spaces of nonpositive curvature in the sense of Alexandrov \cite{BuragoEtAl2001MetricGeometry-AMS,Bacak2014ConvexAnalysis-DeGruyter}, this mean exists and is unique \cite[Proposition 4.3]{Sturm2003NPC-ContempMath}.
Examples include Hilbert spaces, Hadamard manifolds, complete $\R$-trees \cite{Evans2008RealTrees-Springer}, \cite[Example II.1.15(5)]{BridsonHaefliger1999NPC-Springer}, open books \cite[Lemma 1.3]{HotzEtAl2013Sticky-AAP} (see also \cite[Example 3.3]{Sturm2003NPC-ContempMath}), the space of phylogenetic trees with the Billera--Holmes--Vogtmann metric \cite[Lemma 4.1]{BilleraEtAl2001PhylogeneticTrees-AAM}, and positive definite matrices with the affine-invariant metric \cite[Propositions 5 and 10]{BhatiaHolbrook2006MatrixMeans-LAA}.
The quadruple inequality, the variance inequality, and $W_1$ contraction of the barycenter map \cite{Sturm2003NPC-ContempMath} are the main geometric tools for finite-sample bounds \cite{BrunelSerres2026FiniteSample-arxiv}.

In a complete \catk{} space with $\kappa>0$, the Fr\'echet mean need not be unique, but restricting the distribution to a sufficiently small ball restores uniqueness \cite[Theorem B]{Yokota2016Barycenter-JMSJ}.
Besides spheres, this setting includes sufficiently small closed geodesic balls in Riemannian manifolds of sectional curvature at most $\kappa$ \cite[Theorem II.1A.6]{BridsonHaefliger1999NPC-Springer}, such as rotation groups and Grassmann manifolds with their standard Riemannian metrics; for uniqueness of means in such balls, see also \cite[Theorem 2.1]{Afsari2011CenterOfMass-PAMS} and \cite{Kendall1990HarmonicMaps-PLMS}.
Finite-sample bounds again rely on a quadruple inequality, a variance inequality, and a contraction property.
The triangle inequality can replace the quadruple inequality in the argument, but the recent quadruple inequality of \cite{Schoetz2026Trapezoid-arxiv} removes the diameter of the ball from the bounds.
A variance inequality and an empirical version of the contraction property were established in \cite[Proposition 3 and Theorem 1]{BrunelSerres2026FiniteSample-arxiv}; we strengthen both in this article.

Replacing the exponent $2$ by $\alpha\in(1,2]$ gives the \emph{power Fr\'echet mean} of order $\alpha$ (also called the $\alpha$-Fr\'echet mean or $L^\alpha$ center of mass \cite[Definition 1.1]{Afsari2011CenterOfMass-PAMS}).
It minimizes $q\mapsto\int(d(y,q)^\alpha-d(y,o)^\alpha)\,\dl\mu(y)$ and exists uniquely in a Hadamard space under only a finite $(\alpha-1)$-moment.
Power Fr\'echet means interpolate between median-like behavior as $\alpha\searrow1$ and the classical mean at $\alpha=2$.
The corresponding quadruple, variance, and contraction inequalities have recently been established \cite{Schoetz2019ConvergenceRates-EJS,Schoetz2024Quadruple-MIA,Schoetz2025VarianceInequalities-EJP,Schoetz2026Transformed-arxiv}.
For $1<\alpha<2$, the variance and contraction inequalities have distribution-dependent moduli.
Controlling the empirical moduli while retaining the effect of unequal weights is the main difficulty in our power Fr\'echet mean error bound.

Our setting is as follows.
Let $Y_i\sim\mu_i$, $i\in\N$, be independent random variables in a complete \catk{} space, and let $w_i\geq0$ be deterministic weights with $\sum_iw_i=1$.
Fix $\alpha\in(1,2]$, with $\alpha=2$ when $\kappa>0$.
Write $\samfm$ for the $\alpha$-Fr\'echet mean of the empirical measure $\sum_iw_i\delta_{Y_i}$ and $\popfm$ for that of the population mixture $\sum_iw_i\mu_i$.
In the Hadamard setting, a finite $(\alpha-1)$-moment of this mixture ensures existence and uniqueness of both means, almost surely for the empirical mean.
For $\kappa>0$, we assume that all component laws are supported in a common closed convex set of circumradius less than $\pi/(2\sqrt\kappa)$, which again ensures existence and uniqueness.
The component laws need not be identical or have a common mean.

Our contributions are the following.
\begin{enumerate}[label=(\alph*)]
 \item\label{itm:intro:two}
 For the $2$-Fr\'echet mean in a Hadamard space, we prove
 \[
 \EOf{d(\popfm,\samfm)^2}\leq\sum_{i=1}^{\infty}w_i^2\VmOf{Y_i},
 \qquad \VmOf{Y_i}:=\tfrac12\EOf{d(Y_i,Y_i\pr)^2},
 \]
 where $Y_i\pr$ is an independent copy of $Y_i$ (\cref{thm:two:infty:independent}).
 The bound is an identity in Hilbert spaces.
 Compared with \cite[Theorems 3 and 4]{BrunelSerres2026FiniteSample-arxiv}, it allows arbitrary deterministic weights, countably many observations, and heterogeneous component means.

 \item\label{itm:intro:catk}
 In a complete \catk{} space with $\kappa>0$, let the data lie in a closed convex set $\setC$ of circumradius $\varrho<\pi/(2\sqrt\kappa)$ and diameter $\delta/\sqrt\kappa$, and put $\Delta:=2\varrho\sqrt\kappa<\pi$, so that $\delta\leq\Delta\leq2\delta$ (\cref{rem:pos:diameter}).
 We determine the optimal constants depending only on $\Delta$ throughout this range, and those depending only on $\delta$ when $\delta<\pi/2$, in the quadratic variance and $W_1$ contraction inequalities
 \begin{align*}
 d(m_2(\mu),q)^2
 &\leq A\int\br{d(y,q)^2-d(y,m_2(\mu))^2}\,\dl\mu(y),\\
 d(m_2(\mu),m_2(\nu))&\leq B\,W_1(\mu,\nu)
 \end{align*}
 for all probability measures $\mu,\nu$ on $\setC$ and every $q\in\setC$ (\cref{thm:pos:hemisphere:vi,thm:pos:hemisphere:cp} in the appendix).
 This doubles the circumradius range of the quadratic variance and empirical contraction bounds in \cite[Proposition 3 and Theorem 1]{BrunelSerres2026FiniteSample-arxiv}.
 The threshold $\pi/(2\sqrt\kappa)$ is optimal: at $\varrho=\pi/(2\sqrt\kappa)$, no uniform constant exists even within a single fixed space (\cref{prp:pos:diameterobstruction}\,\ref{itm:pos:obstruction:boundary} in the appendix).
 The variance proof combines the curved comparison of \cite[Proposition 23]{Yokota2016Barycenter-JMSJ} with a change of measure and Jensen's inequality for a jointly convex function.
 The optimal contraction constants follow by retaining the location of the mean in a local transport comparison, using the quadruple inequality of \cite[Theorem 1.4]{Schoetz2026Trapezoid-arxiv}.

 These geometric inequalities yield the weighted-sample bound
 \[
 \EOf{d(\popfm,\samfm)^2}
 \leq C_\Delta^4\sum_{i=1}^{\infty}w_i^2\VmOf{Y_i},
 \qquad C_\Delta:=\frac{\Delta}{\sin(\Delta)},
 \]
 and we also give sharper constants and a concentration inequality (\cref{thm:pos:infty:independent,cor:pos:concentration}).
 The factor $C_\Delta^4$ remains bounded as $\varrho\nearrow\pi/(4\sqrt\kappa)$, answering the divergence part of \cite[Open Question 2]{BrunelSerres2026FiniteSample-arxiv} for empirical barycenters.
 At the larger threshold $\varrho\nearrow\pi/(2\sqrt\kappa)$, a spherical-orthant construction shows that divergence is necessary even at fixed diameter: the required risk coefficient is of order $(\pi/2-\varrho\sqrt\kappa)^{-2}$ along that family, and our diameter-aware bound attains this order (\cref{prp:pos:diameterobstruction}\,\ref{itm:pos:obstruction:order} in the appendix).
 The same examples rule out finite diameter-only variance and risk bounds at any fixed $\pi/2\leq\delta<\pi$ (\cref{prp:pos:diameterobstruction}\,\ref{itm:pos:obstruction:vi} and \ref{itm:pos:obstruction:risk} in the appendix).
 Asymptotically in the sample size, related effects are known on spheres: the ratio of $n$ times the mean squared error to the Fr\'echet variance can be arbitrarily large (finite-sample smeariness) \cite{EltznerEtAl2021FiniteSmeariness-GSI,HundrieserEtAl2024FiniteSmeariness-EJS}, and for rotationally symmetric families with a fixed radial profile that charges the boundary of a cap of radius $R$, its large-sample limit tends to $(R\cot R)^{-2}\asymp(\pi/2-R)^{-2}$ as the dimension grows \cite[Theorem 6.3]{Pal2026FiniteSmeariness-arxiv}, the same order as along our orthant family.
 
 \item\label{itm:intro:main}
 For power Fr\'echet means of order $\alpha\in(1,2]$ in Hadamard spaces, our main statistical result is
 \begin{equation*}
 \br{\EOf{d(\popfm,\samfm)^\alpha}}^{1/\alpha}
 \leq c_\alpha\br{
 \sigma_{\alpha-1}^{\phi}\Sigma_{2\alpha-2}^{1/2}
 +\sigma_{\alpha-1}^{\phi/2}\Sigma_\alpha^{1/2}
 +\Sigma_\alpha^{1/\alpha}}
 \end{equation*}
 for a constant $c_\alpha>0$ depending only on $\alpha$, where $\phi=(2-\alpha)/(\alpha-1)$,
 \[
 \sigma_{\alpha-1}:=\sum_iw_i\EOf{d(Y_i,\popfm)^{\alpha-1}},
 \qquad
 \Sigma_s:=\sum_{i=1}^{\infty}w_i^{s/(\alpha-1)}\VmpOf{s}{Y_i},
 \]
 and $\Vmpof{s}{Y_i}:=\tfrac12\Eof{d(Y_i,Y_i\pr)^s}$ (\cref{thm:power:main}).
 The common-center moment $\sigma_{\alpha-1}$ supplies the distributional scale; the weighted pairwise moments $\Sigma_s$ measure only fluctuations within the component laws.
 Thus the bound vanishes for deterministic observations even when their locations differ.
 For $1<\alpha<2$, $\sigma_{\alpha-1}$ cannot in general be replaced by the pairwise moment $\Sigma_{\alpha-1}$ (\cref{lmm:power:moments}\,(ii) in the appendix).

 \item\label{itm:intro:regression}
 The weighted bounds give finite-sample bias--variance inequalities for local constant Fr\'echet regression with metric predictors.
 For quadratic means, these hold in Hadamard spaces and under the positive-curvature circumradius condition above (\cref{cor:two:localreg,cor:pos:localfrechet}).
 Under the kernel and local design conditions, they require a second-moment bound and H\"older continuity of the conditional laws in $W_1$, with no response-space entropy bound or smoothness of conditional densities.
 This weakens the assumptions of earlier local constant analyses and includes moving point masses and translated uniform noise, for which the density conditions can fail (\cref{rem:two:smoothness,ex:two:phase}).

 For $1<\alpha<2$ in Hadamard spaces, the corresponding bound for the conditional $\alpha$-Fr\'echet mean uses a uniform conditional $\alpha$-moment bound and transport continuity adapted to the contraction modulus (\cref{cor:power:localfrechet}).
 With transport smoothness of order $\beta\in(0,1]$ and a local design condition with dimension exponent $D>0$, bandwidth $h\asymp n^{-1/(2\beta+D)}$ gives $L^\alpha$ risk $\landauO(n^{-\beta/(2\beta+D)})$ as $n\to\infty$.
 Thus the classical nonparametric rate is available with fewer than two conditional moments.
 We apply the general result to responses with infinite-variance noise in Hilbert spaces and in hyperbolic space, as well as to distribution-valued responses in the $2$-Wasserstein space over $\R$ (\cref{cor:power:additive,ex:power:hyperbolic,ex:power:wasserstein}).

 \item\label{itm:intro:dirichlet}
 Conditioning on the stick-breaking weights gives prior and posterior bounds for Dirichlet-process Fr\'echet means.
 For quadratic means, these extend the classical Hilbert-space variance formula to Hadamard spaces and, with explicit curvature factors, to the positive-curvature setting above; they also give credible balls and bounds on the influence of the prior (\cref{ssec:dirichlet,ssec:pos:dirichlet}).
 For power Fr\'echet means in Hadamard spaces, a finite $\alpha$-moment of the base measure suffices for non-asymptotic prior and conditional posterior bounds and credible-radius bounds up to a constant depending only on $\alpha$ (\cref{cor:power:dirichlet}).
\end{enumerate}

\textbf{Related work.}
For the quadratic Fr\'echet mean, write $m$ for the population mean, $\hat m_n$ for the empirical mean with weights $1/n$, and $\sigma^2=\Eof{d(Y,m)^2}$.
For iid observations in a Hadamard space of finite diameter, Ahidar-Coutrix, Le Gouic, and Paris \cite{AhidarCoutrixEtAl2020Barycenters-PTRF} derive from their general $M$-estimation bound \cite[Theorem 2.1]{AhidarCoutrixEtAl2020Barycenters-PTRF} a squared-error bound of order $(D+t)/n$ with probability at least $1-2\euler^{-t}$ under a local doubling entropy condition with dimension parameter $D$.
Their results extend to general metric spaces under additional assumptions.
For iid observations in possibly unbounded Hadamard spaces, Sch\"otz \cite[Corollaries 3--4]{Schoetz2019ConvergenceRates-EJS} used the quadruple inequality and a local entropy bound with finite entropy integral to obtain polynomial concentration at the $n^{-1/2}$ scale under a second moment, and an $\landauO(n^{-1})$ squared-risk bound under a $(2+\epsilon)$-moment.
The power Fr\'echet mean extension there assumes an additional variance inequality, as discussed below.
These metric entropy assumptions exclude full balls in infinite-dimensional Hilbert spaces; a separate Hilbert-space treatment in \cite[Section 4.2]{Schoetz2019ConvergenceRates-EJS} uses a weighted auxiliary metric and additional summability conditions on the coordinate variances.
Le Gouic, Paris, Rigollet, and Stromme \cite[Theorem 1.3]{LeGouicEtAl2023FastConvergence-JEMS} prove the entropy-free bound $\Eof{d(\hat m_n,m)^2}\leq\sigma^2/n$ for iid observations with a finite second moment when curvature is nonpositive and also bounded below by a finite constant.
Their Theorem 1.1 gives a dimension-free $\landauO(n^{-1})$ squared-risk bound in spaces of nonnegative curvature under quantitative geodesic bi-extendibility, a curvature-lower-bound setting.
On spheres, which also have a positive curvature upper bound, this condition confines the support to a ball of radius $\pi/4$ about the mean \cite[Section 1.2]{LeGouicEtAl2023FastConvergence-JEMS}, a factor $2$ short of the uniqueness radius $\pi/2$ of \cite[Theorem 2.1]{Afsari2011CenterOfMass-PAMS}, as noted in \cite[Section 3.5]{HundrieserEtAl2026LowerBound-AISM}; high-probability bounds with law-dependent constants for larger supports are given in \cite[Theorem 3.5]{LeGouicEtAl2023FastConvergence-JEMS}.
\Cref{thm:pos:infty:independent} covers every circumradius below $\pi/(2\sqrt\kappa)$ with constants that do not depend on the law.

Leave-one-out stability arguments yield entropy-free bounds in general Hadamard spaces, so these bounds also hold in infinite-dimensional Hilbert spaces; such arguments are also used in proofs of exponential concentration inequalities.
For iid observations, Escande \cite[Corollary 17]{Escande2024Minimizers-JMLR} gives an $\landauO(n^{-1})$ squared-risk bound and exponential deviation bounds when the distance between two independent observations has a finite sub-exponential norm.
Brunel and Serres \cite[Corollaries 11--12 and Theorem 14]{BrunelSerres2024Concentration-ALT} show that iid data for which all centered $1$-Lipschitz functions are uniformly sub-Gaussian yield exponential concentration at the $n^{-1/2}$ scale; bounded data also yield Hoeffding- and Bernstein-type bounds, without entropy assumptions.
Their positive-curvature result \cite[Theorem 18]{BrunelSerres2024Concentration-ALT} requires a uniform doubling-type covering bound on all balls and support in a ball of radius less than $\pi/(4\sqrt\kappa)$.
For arbitrary Hadamard spaces, Brunel and Serres \cite[Theorems 3--4]{BrunelSerres2026FiniteSample-arxiv} prove $\Eof{d(\hat m_n,m)^2}\leq2\sigma^2/n$ under only a second moment, and $2n^{-2}\sum_i\sigma_i^2$ for independent observations with a common mean $m$, where $\sigma_i^2=\Eof{d(Y_i,m)^2}$.
The same theorems give curvature-dependent squared-risk bounds in complete $\mathsf{CAT}(\kappa)$ spaces when $\kappa>0$ and the common convex data domain lies in a ball of radius less than $\pi/(4\sqrt\kappa)$, again without entropy assumptions.
Their Theorems 7--8 add sub-Gaussian and Bernstein-type concentration for independent observations with a common mean under the corresponding tail or boundedness assumptions.
For independent observations that are not identically distributed, Kendall and Le \cite{KendallLe2011LimitTheorems-BJPS} prove a weak law of large numbers and a central limit theorem for empirical Fr\'echet means on Riemannian manifolds, and K\"ostenberger and Stark \cite[Theorems 1 and 2]{KoestenbergerStark2025SignalRecovery-JMVA} give an $L^2$ bound for the inductive mean in Hadamard spaces for independent observations with different means and show that every $L^2$ limit of the empirical Fr\'echet means is the limit of the inductive means; neither paper gives a finite-sample bound for the empirical Fr\'echet mean.
Our weighted bounds allow different component means and countably many observations; in positive curvature they also enlarge the admissible circumradius range.

\emph{Power Fr\'echet means.}
For $1<\alpha<2$, Sch\"otz \cite[Corollary 5]{Schoetz2019ConvergenceRates-EJS} obtained polynomial deviation bounds for the empirical $\alpha$-Fr\'echet mean of iid observations in a Hadamard space under a finite $(2\alpha-2)$-moment, an entropy condition, and a separately assumed variance inequality for the population objective.
That paper proves the power quadruple inequality but does not establish the required power Fr\'echet mean variance inequality.
Distribution-dependent variance inequalities and uniqueness for transformed means, including power Fr\'echet means, are obtained in \cite[Theorem 5.4 and Corollary 5.8]{Schoetz2025VarianceInequalities-EJP}.
For iid observations with uniform weights in any Hadamard space, the finite-sample bounds of \cite[Theorem 3.3 and Corollary 3.5]{Schoetz2026Transformed-arxiv} then imply
\[
 \br{\EOf{d(\hat m_{\alpha,n},m_\alpha)^\alpha}}^{1/\alpha}
 =\landauO(n^{-1/2})
\]
under only a finite $\alpha$-moment, with no entropy or bounded-support condition, including in infinite-dimensional Hilbert spaces.
Here $m_\alpha$ and $\hat m_{\alpha,n}$ are the population and empirical $\alpha$-Fr\'echet means; the leading term of the bound depends on the $(\alpha-1)$- and $(2\alpha-2)$-moments, while the $\alpha$-moment enters terms of smaller order in $n$.
Our power Fr\'echet mean result extends this setting to unequal weights and independent observations with different laws and different means.
The qualitative theory already includes strong laws of large numbers for power Fr\'echet means with general $\alpha$, including possibly nonunique mean sets \cite{Schoetz2022StrongLLN-Statistics,EvansJaffe2024MeanSets-Bernoulli}.
For iid samples and $\alpha\geq1$, Jaffe \cite[Corollary 4.1]{Jaffe2024InfiniteDimensions-arxiv} proves almost-sure one-sided Hausdorff convergence of empirical $\alpha$-Fr\'echet mean sets under a finite $(\alpha-1)$-moment in separable spaces admitting a suitable weak convergence, including Hadamard spaces, without a Heine--Borel condition (allowing for infinite dimension).

Robust estimators provide a complementary route to exponential concentration under weak moment assumptions.
For iid data in Hadamard spaces, Yun and Park \cite{YunPark2023Exponential-Bernoulli} obtain polynomial concentration for empirical means and exponential concentration for median-of-means tournament estimators under a second moment and entropy assumptions on the loss class.
For the quadratic loss, their Proposition 4 verifies these assumptions in separable Hilbert spaces of arbitrary dimension; the resulting bounds in their Theorems 2 and 4 include an additional $\log n$ factor.
Their power Fr\'echet mean results require both a corresponding loss-class entropy bound and an additional variance inequality.
In \cite[Theorem 3.3]{KimEtAl2025Robust-AISTATS}, Fr\'echet-median aggregation converts constant-confidence bounds for independent preliminary estimators into exponential concentration in $\mathsf{CAT}(\kappa)$ spaces, with additional localization conditions when $\kappa>0$.
Their exponential bounds concern robust aggregation procedures; our bounds apply directly to weighted empirical minimizers.

Earlier quadratic variance inequalities under positive curvature \cite{Karcher1977CenterOfMass-CPAM,Ohta2007Convexities-GD,Yokota2017PBarycenter-TJM,BrunelSerres2026FiniteSample-arxiv} rest on strong convexity and therefore require $\Delta<\pi/2$ or $\delta<\pi/2$, while the inequalities of \cite[Proposition 19]{Yokota2016Barycenter-JMSJ} and \cite[Theorem 1.1]{Gietl2026VarianceInequalities-AGMS}, which reach further, give non-quadratic growth, respectively a bound at a single base point with the constant on the variance.
Our inequality holds throughout $\Delta<\pi$, at every point of the convex data set, with optimal one-parameter constants; \cref{rem:pos:literature,rem:pos:BrunelSerres2026} give the detailed comparisons.

For local regression, Petersen and M\"uller \cite[Section 4 and Theorems 3--4]{PetersenMueller2019Regression-AOS} introduce local linear Fr\'echet regression and prove asymptotic $O_P$ rates for it under density smoothness and entropy assumptions; the local constant (Nadaraya--Watson) Fr\'echet estimator discussed there goes back to \cite{DavisEtAl2007ShapeRegression-ICCV,Hein2009RobustRegression-NIPS}.
In a fixed-design model, Sch\"otz \cite[Theorems 1--2]{Schoetz2022Regression-EJS} strengthened this type of $O_P$ guarantee to finite-sample $L^2$ risk bounds for local polynomial Fr\'echet regression, including the local constant case, again under density smoothness and entropy assumptions.
On bounded response spaces, the quadratic transport condition used here is weaker than the corresponding density smoothness conditions, and our risk bounds require no response-space entropy condition; \cref{rem:two:smoothness,ex:two:phase} give the comparison and examples.
The density conditions constrain the noise distribution as well as the regression function \cite[Remark 2]{Schoetz2022Regression-EJS}; they can fail even for smoothly moving atoms or translated uniform noise.
Tucker and Wu \cite[Section 6.1]{TuckerWu2025PartiallyGlobal-SS} avoid density smoothness by directly assuming H\"older continuity of the conditional Fr\'echet objective, a condition on the objective itself for which no further sufficient model condition is given there; a response-space entropy condition is still required.
The density-and-entropy framework also appears in \cite{QiuEtAl2024Semisupervised-arxiv,QiuEtAl2026Mondrian-TIT}; related local Fr\'echet results under density, entropy, and boundedness hypotheses include \cite{ChenMueller2022Uniform-AOS,ImEtAl2025Spherical-EJS,ImJeon2026Riemannian-arxiv}.
Li and Zhu \cite{LiZhu2026Regression-StatPapers} obtain finite-sample deviation bounds without a response-space entropy assumption for a median-of-means local linear estimator with Wasserstein-valued responses in classes admitting a common transport-coordinate representation.

\textbf{Outline.}
\Cref{sec:prelim} introduces curvature upper bounds, means, moments, and variances, and states the quadruple inequality.
The main sections treat $2$-Fr\'echet means in Hadamard spaces (\cref{sec:two}), $2$-Fr\'echet means under a positive curvature upper bound (\cref{sec:pos}), and power Fr\'echet means in Hadamard spaces (\cref{sec:main}).
Each gives a weighted-sample bound and applications to local constant regression and Dirichlet-process means.
The positive-curvature geometric inequalities are proved in \cref{app:pos:tools}; \cref{app:tools} collects the power Fr\'echet mean tools and variance comparisons.
\Cref{app:proofs} contains the statistical proofs and regression comparisons, and \cref{sec:powfun} the analytic tools.

\section{General Preliminaries}\label{sec:prelim}

\textbf{Notation.}
The order of the power Fr\'echet mean is $\alpha\in(1,2]$ and $\phi := \frac{2-\alpha}{\alpha-1}$; $s$ denotes a generic moment order.
Unspecified constants are written $c_\alpha$, $c_s$, and so on: such a symbol denotes an element of $\Rpp$ that depends only on the parameters in its index, and its value may change from occurrence to occurrence, including within a single chain of inequalities.
We use the convention $0\cdot\infty:=0$; in particular, summands with zero weight vanish.
For real sequences $(a_n)_{n\in\N}$ and $(b_n)_{n\in\N}$, we write $a_n=\landauO(b_n)$ as $n\to\infty$ if there are $C<\infty$ and $n_0\in\N$ with $\absof{a_n}\leq C\absof{b_n}$ for all $n\geq n_0$, and $a_n=\mathrm{o}(b_n)$ as $n\to\infty$ if $a_n/b_n\to0$.
For positive sequences, $a_n\asymp b_n$ means that $a_n/b_n$ is bounded above and away from zero, and $a_n\sim b_n$ means that $a_n/b_n\to1$ as $n\to\infty$.
The same notation is used for other limits, such as $t\to0$, which are then indicated in the text; without such an indication, the limit is $n\to\infty$ in the sample size $n$, with all other model parameters fixed unless stated otherwise.
The notation $Y\sim\mu$ means that the random variable $Y$ has law $\mu$.
For an estimator $\hat m$ of a point $m$ and $\alpha\geq1$, we call $\Eof{d(\hat m,m)^\alpha}^{1/\alpha}$ its \emph{$L^\alpha$ risk}.
The following symbols recur throughout and are defined where they first appear.
\begin{center}
	\begin{tabular}{ll}
		\toprule
		$\ol qp := d(q,p)$ & distance in $\mc Q$; \cref{ssec:prelim:curvature}\\
		$m_\alpha(\mu)$, $m_\alpha[Y]$ & mean of a law or random variable; \cref{ssec:prelim:fm}\\
		$\popdis = \sum_i w_i\mu_i$, $\samdis = \sum_i w_i\delta_{Y_i}$ & population and empirical mixture; \cref{ssec:two:main,ssec:main:bound}\\
		$\popfm$, $\samfm$ & their $\alpha$-Fr\'echet means; \cref{ssec:two:main,ssec:main:bound}\\
		$w_i$, $w_{\max}$, $\normof{w}_s$ & weights and $\ell^s$-norms; \cref{rem:weights,rem:power:weights}\\
		$\sigma_\alpha^s(\mu)$, $\sigma_\alpha^s[Y]$ & $s$-th moment about the $\alpha$-Fr\'echet mean; \eqref{eq:sigma:def}\\
		$\sigma_s$ & $s$-th moment of $\popdis$ about $\popfm$, see \eqref{eq:power:sigma}\\
		$\Vfpof{s}{Y}$, $\Vmpof{s}{Y}$ & Fr\'echet and metric $s$-variance, see \eqref{eq:variances}; $\Vf := \Vfp2$, $\Vm := \Vmp2$\\
		$\Sigma_s = \sum_i w_i^{s/(\alpha-1)}\Vmpof{s}{Y_i}$ & weighted $s$-th moment, see \eqref{eq:power:sigmaw}\\
		$W_s(\mu,\nu)$ & $s$-Wasserstein distance without root, see \eqref{eq:W1}\\
		$\widetilde W_{\alpha-1}(\mu,\nu)$ & location-aware transport cost, see \eqref{eq:power:Wtilde}\\
		$\Lamof{\mu}{t}$ & regularized inverse moment, see \eqref{eq:power:Lambda}\\
		$\setC$, $\varrho$ & convex data set and its circumradius, see \cref{ssec:pos:prelim}\\
		$\Delta$, $C_\Delta$ & radius--curvature parameter and its factor, see \eqref{eq:pos:constants}\\
		$\const{Kmi}$, $\const{Kma}$, $\const{dns}$ & kernel and design constants, see \eqref{eq:localfrechet:kernel}, \eqref{eq:localfrechet:design}\\
		$c_\alpha$, $c_s$ & constants in $\Rpp$; see the convention above\\
		\bottomrule
	\end{tabular}
\end{center}

\subsection{Geometry}\label{ssec:prelim:curvature}
Let $(\mc Q, d)$ be a metric space.
For $q,p \in \mc Q$, write $\ol qp:=d(q,p)$.
For $o\in\mc Q$ and $r>0$, $B(o,r)$ and $\overline B(o,r)$ denote the open and the closed ball of radius $r$ about $o$.
A continuous curve $\gamma\colon[a,b]\to\mc Q$, $a<b$, is a \emph{(constant-speed) geodesic} between $q:=\gamma(a)$ and $p:=\gamma(b)$ if $\ol{\gamma(s)}{\gamma(t)} = \frac{|s-t|}{b-a}\,\ol qp$ for all $s,t\in [a,b]$. Its speed is $\ol qp/(b-a)$, and it has \emph{unit speed} when this equals $1$. Constant curves are allowed. In particular, geodesics parametrized on $[0,1]$ have speed equal to the distance between their endpoints.
A metric space in which any two points are joined by a geodesic is called a \emph{geodesic metric space}. If all geodesics between two points $q,p\in\mc Q$ have the same image, we denote this image by $\seg qp$.

Now let $(\mc Q, d)$ be a geodesic metric space and $\kappa\in\R$.
Set 
\begin{equation*}
    \Dk := \begin{cases}
        \infty & \text{for } \kappa\leq0\eqcm\\
        \pi/\sqrt\kappa & \text{for } \kappa>0\eqfs
    \end{cases}
\end{equation*}
Denote by $(M_\kappa, \operatorname{dist}_{M_\kappa})$ the complete, simply connected model surface of constant curvature $\kappa$: the hyperbolic plane with its $(-1)$-curvature metric scaled by $(-\kappa)^{-1/2}$ for $\kappa<0$, the Euclidean plane for $\kappa=0$, and the sphere of radius $\kappa^{-1/2}$ with its intrinsic metric for $\kappa>0$; see \cite[Chapters I.2 and II.1]{BridsonHaefliger1999NPC-Springer}.
The space $(\mc Q, d)$ is \catk{} if any two points at distance less than $\Dk$ are joined by a unique geodesic and every geodesic triangle of perimeter less than $2\Dk$ is at least as ``thin'' as its comparison triangle in $M_\kappa$: for $q,p,y\in\mc Q$ with $\ol qp + \ol py + \ol yq < 2\Dk$, let $\bar q,\bar p,\bar y\in M_\kappa$ have the same pairwise distances, let $\gamma_{pq}\colon[0,\ol pq]\to\mc Q$ and $\gamma_{py}\colon[0,\ol py]\to\mc Q$ be the geodesics from $p$ to $q$ and from $p$ to $y$, and let $\bar\gamma_{\bar p\bar q}$ and $\bar\gamma_{\bar p\bar y}$ be the corresponding geodesics in $M_\kappa$; then
\begin{equation*}
    \ol{\gamma_{pq}(s)}{\gamma_{py}(t)} \leq \operatorname{dist}_{M_\kappa}\brOf{\bar\gamma_{\bar p\bar q}(s), \bar\gamma_{\bar p\bar y}(t)}
    \qquad\text{for all } s \in [0, \ol pq],\ t \in [0, \ol py]
    \eqfs
\end{equation*}
A complete \cato{} space is called a \emph{Hadamard space}.
Every \catof{\kappa} space is also \catof{\kappa\pr} if $\kappa\leq\kappa\pr$ \cite[Theorem II.1.12]{BridsonHaefliger1999NPC-Springer}; in particular, all Hadamard-space results below apply to complete spaces with a negative curvature upper bound as well.
See \cite{BridsonHaefliger1999NPC-Springer,BuragoEtAl2001MetricGeometry-AMS,Bacak2014ConvexAnalysis-DeGruyter} for textbooks on the general theory.

The following inequality is a central tool for obtaining tail-robust results for Fr\'echet means in \catk{} spaces.
\begin{proposition}[Quadruple inequality {\cite[Corollaries 5.1 and 5.5]{Schoetz2026Trapezoid-arxiv}}]\label{prp:prelim:qi}
 Let $(\mc Q,d)$ be a \catk{} space, $\kappa\geq0$, and $\alpha\in[1,2]$.
 For $y,z,q,p\in\mc Q$ with $\ol yz,\ol qp<\pi/\sqrt\kappa$ if $\kappa>0$, so that the segments $\seg yz$ and $\seg qp$ are unique \cite[Proposition II.1.4(1)]{BridsonHaefliger1999NPC-Springer}, put $R := \sqrt{\kappa} \,\diam(\seg yz\cup\seg qp)$, and assume $R < \pi$.
 Set $C_0 := 1$ and $C_x := x / \sin(x)$ for $0<x<\pi$.
 Then
 \begin{equation}\label{eq:prelim:qi}
    \abs{\ol yq^\alpha-\ol yp^\alpha-\ol zq^\alpha+\ol zp^\alpha}
    \leq 2^{2-\alpha}\alpha C_R\,\ol qp\,\ol yz\,
    \diam\cb{y,z,q,p}^{\alpha-2}.
 \end{equation}
 The right-hand side is defined as zero if either $\ol qp=0$ or $\ol yz=0$.
\end{proposition}
For $\kappa = 0$, we have $R=0$, so that $C_R=1$ and the assumption $R<\pi$ is void.
For $\alpha=2$ and $\kappa=0$, \eqref{eq:prelim:qi} is Reshetnyak's quadruple comparison \cite[Proposition 2.4]{Sturm2003NPC-ContempMath}, which in inner product spaces is equivalent to the Cauchy--Schwarz inequality.
For $\kappa=0$ and general $\alpha$, \cite[Theorem 3]{Schoetz2019ConvergenceRates-EJS} gives the weaker bound with $\ol yz^{\alpha-1}$ in place of $\ol yz\,\diam\{y,z,q,p\}^{\alpha-2}$ and the same constant.

\subsection{Fr\'echet Mean}\label{ssec:prelim:fm}

Let $(\mc Q,d)$ be a metric space.
Denote by $\mc P(\mc Q)$ the Borel probability measures on $\mc Q$ that are concentrated on a separable Borel subset of $\mc Q$.
For $s>0$, let $\mc P^s(\mc Q)$ consist of those $\mu\in\mc P(\mc Q)$ for which $\int\ol yq^s\,\dl\mu(y)<\infty$ for one, and hence every, $q\in\mc Q$; we say that such distributions have a \emph{finite $s$-th moment}.

\begin{remark}[Conditions on the distributions]\label{rem:separability}
  We assume a separable support in the definition of $\mc P(\mc Q)$ so that expressions like $d(Y,Y\pr)$ are measurable, where $Y, Y\pr$ are random variables in $\mc Q$ with laws in $\mc P(\mc Q)$.
\end{remark}

For $\alpha\in(1,\infty)$ and $\mu \in \mc P^{\alpha-1}(\mc Q)$, we have $\int \absof{\ol yq^\alpha - \ol yp^\alpha} \,\dl \mu(y) < \infty$ for all $q,p\in\mc Q$, which makes the following definition meaningful.
Fix an arbitrary reference point $o\in \mc Q$. For $\mu \in \mc P^{\alpha-1}(\mc Q)$, any minimizer of
\begin{equation}\label{eq:prelim:fmobjective}
    q\mapsto\int \ol yq^\alpha - \ol yo^\alpha \, \dl \mu(y)
\end{equation}
is called a \emph{power Fr\'echet mean of order $\alpha$} (short: $\alpha$-Fr\'echet mean) of $\mu$; the set of minimizers does not depend on the choice of $o$.
If the minimizer is unique, we denote it by $m_\alpha(\mu)$. For a random variable $Y$ with law $\mu$, we also write $m_\alpha[Y]:=m_\alpha(\mu)$.
In the main part of the paper, we only work in settings with unique $\alpha$-Fr\'echet means.
If $\mu\in\mc P^\alpha(\mc Q)$, the renormalization by $\ol yo^\alpha$ in \eqref{eq:prelim:fmobjective} is not needed, and $m_\alpha(\mu)$ minimizes $q\mapsto\int\ol yq^\alpha\,\dl\mu(y)$.
If $\mc Q$ is a Hilbert space, the $2$-Fr\'echet mean of $\mu\in\mc P^1(\mc Q)$ is its expectation.

\subsection{Moments}\label{ssec:prelim:moments}

Let $s\in\Rpp$ and $\mu\in\mc P(\mc Q)$. We call
\begin{equation*}
    \VfpOfP{s}{\mu} := \inf_{q\in\mc Q}\int \ol yq^{s}\,\dl\mu(y) \in [0,\infty]
\end{equation*}
the \emph{Fr\'echet $s$-variance} and
\begin{equation}\label{eq:variances}
    \VmpOfP{s}{\mu} := \frac12\iint \ol y{y\pr}^{s}\,\dl\mu(y)\,\dl\mu(y\pr) \in [0,\infty]
\end{equation}
the \emph{metric $s$-variance} of $\mu$.
We write $\Vf := \Vfp{2}$ for the \emph{Fr\'echet variance}, a term following \cite{DubeyMueller2019ANOVA-Biometrika}, and $\Vm := \Vmp{2}$ for the \emph{metric variance}, the diagonal of the metric covariance of \cite[(4)--(5)]{DubeyMueller2020TimeVarying-JRSSB}; see also \cite{SongMueller2026Dispersion-JASA}. The two types of $s$-variance differ by at most a constant factor:

\begin{lemma}[Comparison of the $s$-variances]\label{lmm:variances}
    Let $(\mc Q, d)$ be a metric space, $s\in\Rpp$, and $\mu\in\mc P(\mc Q)$.
    Then
    \begin{equation}\label{eq:variances:comparison}
        2^{\min(1-s,0)}\, \VmpOfP{s}{\mu}
        \leq
        \VfpOfP{s}{\mu}
        \leq
        2\, \VmpOfP{s}{\mu}
        \eqfs
    \end{equation}
\end{lemma}
Both constants are sharp (\cref{rem:app:variances:sharp} in the appendix).

For $\mu \in \mc P^{\alpha - 1}(\mc Q)$ with a unique $\alpha$-Fr\'echet mean and $s\in\Rpp$, we write
\begin{equation}\label{eq:sigma:def}
    \sigma_\alpha^s(\mu) := \int \ol y{m_\alpha(\mu)}^{s} \,\dl\mu(y)
\end{equation}
for the \emph{$s$-th moment of $\mu$ about its $\alpha$-Fr\'echet mean}.
For a fixed centering order $\alpha$, Lyapunov's inequality gives $\sigma_\alpha^r(\mu)^{1/r}\leq\sigma_\alpha^{s}(\mu)^{1/s}$ for $0<r\leq s$; if $s>1$, $\mu\in\mc P^s(\mc Q)$, and $m_s(\mu)$ exists uniquely, minimality also gives $\sigma_s^s(\mu)\leq\sigma_\alpha^s(\mu)$.
If $\mu\in\mc P^\alpha(\mc Q)$, then $\VfpofP{\alpha}{\mu} = \sigma_\alpha^\alpha(\mu)$.
For a $\mc Q$-valued random variable $Y$ with distribution $\mu$, we also write $\Vfpof{s}{Y} := \VfpofP{s}{\mu}$, $\Vmpof{s}{Y} := \VmpofP{s}{\mu}$, and $\sigma_\alpha^s[Y] := \sigma_\alpha^s(\mu)$.
Here and throughout, as for expectations, square brackets indicate that the argument is a random variable, which enters only through its distribution, whereas parentheses indicate that the argument is a distribution.

For $s\in(0,1]$ and $\mu,\nu\in\mc P^s(\mc Q)$, the \emph{$s$-Wasserstein distance} is
\begin{equation}\label{eq:W1}
    W_s(\mu,\nu):=\inf_{\pi\in\Pi(\mu,\nu)}\int\ol yz^s\,\dl\pi(y,z)
    \eqcm
\end{equation}
where $\Pi(\mu,\nu)$ denotes the set of couplings of $\mu$ and $\nu$, i.e., of probability measures on $\mc Q\times\mc Q$ with marginals $\mu$ and $\nu$.
We do not take the $s$-th root: only orders $s\leq1$ occur for $W_s$ in this paper, and for these $d^s$ is a metric on $\mc Q$, so that $W_s$ is the $1$-Wasserstein distance of $(\mc Q,d^s)$ and hence a metric on $\mc P^s(\mc Q)$, whereas the rooted version $W_s^{1/s}$ need not satisfy the triangle inequality for $s<1$.
The rooted quadratic Wasserstein distance, which appears only as the metric of a response space in \cref{ex:power:wasserstein}, is written $\overline W_2$.

\section{2-Fr\'echet Means in Hadamard Spaces}\label{sec:two}

In this section, we assume $(\mc Q,d)$ to be a Hadamard space ($\kappa=0$) and consider only the $2$-Fr\'echet mean ($\alpha=2$). After collecting the preliminaries for this setting (\cref{ssec:two:prelim}), we derive in \cref{ssec:two:main} an explicit non-asymptotic bound with an optimal constant (\cref{thm:two:infty:independent}). Then we apply it to local constant regression with metric predictors in \cref{ssec:localfrechet} and to Dirichlet-process Fr\'echet means in \cref{ssec:dirichlet}.

\subsection{Preliminaries}\label{ssec:two:prelim}

Let $(\mc Q,d)$ be a Hadamard space.
By \cref{prp:prelim:qi}, we have the \emph{quadruple inequality} \cite[Proposition 2.4]{Sturm2003NPC-ContempMath}
\begin{equation}\tag{QI}\label{eq:2fmcat0:qi}
    \abs{\ol yq^2-\ol yp^2-\ol zq^2+\ol zp^2}
    \leq 2 \, \ol qp\,\ol yz
    \qquad\text{for all } y,z,q,p\in\mc Q
    \eqfs
\end{equation}
Let $\mu \in \mc P^1(\mc Q)$ and $Y \sim \mu$.
The $2$-Fr\'echet mean of $\mu$,
\begin{equation}\label{eq:2fmcat0:fmdef}
    m := m_2(\mu) = m_2[Y] = \argmin_{q\in\mc Q} \EOf{\ol Yq^2 - \ol Yo^2}
    \eqcm
\end{equation}
exists and is unique \cite[Proposition~4.3]{Sturm2003NPC-ContempMath}.
It satisfies the \emph{variance inequality} \cite[Proposition~4.4]{Sturm2003NPC-ContempMath}: for all $q\in\mc Q$,
\begin{equation}\tag{VI}\label{eq:cat0:2fm:vi}
    \ol {m}{q}^2 \leq \EOf{\ol Yq^2 - \ol Ym^2}
    \eqfs
\end{equation}
Moreover, any $\mu,\nu \in \mc P^1(\mc Q)$ satisfy the \emph{contraction property} \cite[Theorem~6.3]{Sturm2003NPC-ContempMath}
\begin{equation}\tag{CP}\label{eq:cat0:2fm:CP}
    \ol {m_2(\mu)}{m_2(\nu)}\leq W_1(\mu, \nu)
    \eqfs
\end{equation}
For $\mu\in\mc P^2(\mc Q)$, the Fr\'echet and metric variances satisfy
\begin{equation}\label{eq:two:variances}
    \VfOf{Y}\leq\VmOf{Y}\leq2\,\VfOf{Y}
    \eqfs
\end{equation}
In Hilbert spaces, $\Vfof{Y}=\Vmof{Y}$. The inequalities and the Hilbert-space identity are proved in \cref{lmm:two:variances} in the appendix; the first inequality characterizes Hadamard spaces among complete metric spaces \cite[Theorem 4.9]{Sturm2003NPC-ContempMath}.

\subsection{The Weighted-Sample Bound}\label{ssec:two:main}

The following theorem, the main result of this section, gives a non-asymptotic $L^2$ risk bound for the Fr\'echet mean of a weighted infinite sequence of independent, not necessarily identically distributed random variables.
Its proof, given in \cref{app:two}, combines the quadruple inequality \eqref{eq:2fmcat0:qi}, the variance inequality \eqref{eq:cat0:2fm:vi}, and the $W_1$ contraction \eqref{eq:cat0:2fm:CP} with a leave-one-out argument (\emph{algorithmic stability}).
This argument is also used in \cite[Lemmas 6 and 7]{Escande2024Minimizers-JMLR}, \cite[proof of Theorem 3]{BrunelSerres2026FiniteSample-arxiv}, and \cite[Supplement S4]{Schoetz2026Transformed-arxiv}.

\begin{theorem}[Weighted-sample bound in Hadamard spaces]\label{thm:two:infty:independent}
    Let $(w_i)_{i\in\N} \subset [0,1]$ such that $\sum_{i=1}^\infty w_i = 1$.
    Let $(Y_i)_{i\in\N}$ be a sequence of independent random variables with distributions $\mu_1, \mu_2, \dots \in \mc P(\mc Q)$.
    Let $\popdis := \sum_{i=1}^\infty w_i \mu_i$ and $\samdis := \sum_{i=1}^\infty w_i \delta_{Y_i}$.
    Assume $\popdis \in \mc P^1(\mc Q)$.
    Set $\popfm := m_2(\popdis)$ and $\samfm := m_2(\samdis)$.
    Then
    \begin{equation}\label{eq:two:infty:independent}
        \EOf{\ol{\popfm}{\samfm}^2}
        \leq
        \sum_{i=1}^{\infty} w_i^2 \VmOf{Y_i}
        \eqfs
    \end{equation}
\end{theorem}

In Hilbert spaces, \eqref{eq:two:infty:independent} holds with equality, since there $\samfm = \sum_i w_i Y_i$, $\popfm = \sum_i w_i\, \Eof{Y_i}$, and $\Eof{\normof{\samfm - \popfm}^2} = \sum_i w_i^2\, \Vfof{Y_i} = \sum_i w_i^2\, \Vmof{Y_i}$.
Random weights $(w_i)_{i\in\N}$ that are independent of $(Y_i)_{i\in\N}$ are covered by conditioning on them; \cref{ssec:dirichlet} uses this.

The bound is nontrivial only if $\sum_iw_i^2\Vmof{Y_i}<\infty$, which requires $\mu_i\in\mc P^2(\mc Q)$ for every $i$ with $w_i>0$.
The mixture $\popdis$, however, need not have a finite second moment: both parts of the following example have $\popdis\in\mc P^1(\R)\setminus\mc P^2(\R)$ and $\sum_{i=1}^{\infty} w_i^2 \Vmof{Y_i} < \infty$.
\begin{example}\label{ex:two:infinitevariance}
 Set $\mc Q=\R$, $w_i=2^{-i}$, and $a_i=2^{i/2}$, $i\geq1$.
 \begin{enumerate}[label=(\alph*)]
  \item Let $\mu_i=\delta_{a_i}$. Then $\popfm=\sum_iw_ia_i=\sum_i2^{-i/2}<\infty$, whereas
  \[
   \sum_{i=1}^{\infty}w_i|a_i-\popfm|^2=\infty.
  \]
  Thus $\popdis\in\mc P^1(\R)\setminus\mc P^2(\R)$.
  Every component has zero metric variance and $\samdis=\popdis$ almost surely, so \eqref{eq:two:infty:independent} becomes $0\leq 0$.
  \item Let the $Y_i$ be independent with $\mu_i=\frac12(\delta_{-a_i}+\delta_{a_i})$.
  Then $\Eof{Y_i}=0$, $\Eof{|Y_i|}=2^{i/2}$, and $\Vmof{Y_i}=\Eof{|Y_i|^2}=2^i$.
  Consequently,
  \begin{align*}
   \sum_{i=1}^{\infty}w_i\EOf{|Y_i|}&=\sum_{i=1}^{\infty}2^{-i/2}<\infty,\\
   \sum_{i=1}^{\infty}w_i\EOf{|Y_i|^2}&=\sum_{i=1}^{\infty}1=\infty,\\
   \sum_{i=1}^{\infty}w_i^2\VmOf{Y_i}&=\sum_{i=1}^{\infty}2^{-i}=1.
  \end{align*}
  Thus, $\popdis\in\mc P^1(\R)\setminus\mc P^2(\R)$ and $\popfm=0$, and \eqref{eq:two:infty:independent} becomes $1\leq 1$.
 \end{enumerate}
\end{example}

\begin{remark}[Weight schemes]\label{rem:weights}
    We collect consequences of \cref{thm:two:infty:independent} for several weight schemes, retaining its independence and mixture-moment assumptions.
	\begin{enumerate}[label=(\roman*)]
        \item (Equal weights.)
        With $w_i = \frac1n\ind_{\{i\leq n\}}$, write $\samnfm$ for $\samfm$; identically distributed $Y_i$ give
        $\Eof{\ol{\popfm}{\samnfm}^2} \leq \frac1{n} \Vmof{Y_1}$, which is the penultimate step of the proof of \cite[Theorem 3]{BrunelSerres2026FiniteSample-arxiv}; their stated bound follows because $\Vmof{Y_1}\leq2\Vfof{Y_1}$.
        In general we obtain $\Eof{\ol{\popfm}{\samnfm}^2} \leq \frac1{n^2}\sum_{i=1}^{n} \Vmof{Y_i}$, which implies the $\kappa\leq0$ bound of \cite[Theorem 4]{BrunelSerres2026FiniteSample-arxiv}, where all $\mu_i$ are assumed to have the same Fr\'echet mean.
        Under the additional assumption of a curvature lower bound, Le Gouic et al.\ \cite[Theorem 1.3]{LeGouicEtAl2023FastConvergence-JEMS} give $\Eof{\ol{\popfm}{\samnfm}^2}\leq\Vfof{Y_1}/n$ for iid observations. Thus, by \eqref{eq:two:variances}, the Fr\'echet variance bound implied by \cref{thm:two:infty:independent} is off by at most a factor $2$.
        Whether that factor is needed, that is, whether $\Vf$ may replace $\Vm$ in every Hadamard space, is an open question.
		\item (Effective sample size.)
        Write $\normof{w}_2^2 := \sum_{i=1}^\infty w_i^2$ and call $n_{\mathrm{eff}} := \normof{w}_2^{-2}$ the \emph{effective sample size}.
		With $v := \sup\{\Vmof{Y_i}: i\in\N,\ w_i>0\}$, the bound \eqref{eq:two:infty:independent} implies $\Eof{\ol{\popfm}{\samfm}^2} \leq v\, n_{\mathrm{eff}}^{-1}$.
		For $w_i = \frac1n\ind_{\{i\leq n\}}$ we get $n_{\mathrm{eff}} = n$.
        \item (Triangular arrays.)
        Since $\normof{w}_2^2 \leq \sup_iw_i$, any triangular scheme of weights $(w_i^{(n)})_{i\in\N}$ with $\sup_iw_i^{(n)}\to0$ and $\sup_{n,i}\Vmof{Y_i^{(n)}}<\infty$ yields $L^2$ consistency for the moving target $m_2(\sum_i w_i^{(n)} \mu_i^{(n)})$, where both the weights and the distributions may vary with $n$.
		\item (Geometric weights.)
		For $\rho \in (0,1)$ and $w_i = (1-\rho)\rho^{i-1}$, $i\in\N$, the estimator $\samfm$ is the Fr\'echet mean of an exponentially weighted infinite independent sequence.
		Here $\normof{w}_2^2 = \frac{1-\rho}{1+\rho}$, so that $n_{\mathrm{eff}} = \frac{1+\rho}{1-\rho}$ and, for iid data, $\Eof{\ol{\popfm}{\samfm}^2} \leq \frac{1-\rho}{1+\rho}\Vmof{Y_1}$.
		For non-identically distributed data, e.g., under distributional drift, the target $\popfm$ automatically discounts old (large $i$) distributions geometrically. The bounds hold without stationarity, but still require independence.
		\item (Polynomially weighted windows.)
		For a window of length $n$ and $\lambda \geq 0$, let $w_i \propto (n+1-i)^\lambda\, \ind_{\{i\leq n\}}$, with $i = 1$ the most recent observation.
		Then $\normof{w}_2^2 = \frac{(\lambda+1)^2}{(2\lambda+1)\, n}\, (1 + \mathrm{o}(1))$ as $n\to\infty$, so the parametric rate is retained up to the constant $\frac{(\lambda+1)^2}{2\lambda+1}$; for $\lambda = 1$, the linearly weighted moving average common in time series analysis, $\normof{w}_2^2 = \frac{4}{3n}\, (1 + \mathrm{o}(1))$.
		\item (Polynomial weights on all of $\N$.)
		For fixed $\lambda > 0$ and $w_i = i^{-1-\lambda}/\zeta(1+\lambda)$, $i \in \N$, with $\zeta$ the Riemann zeta function, we have $\normof{w}_2^2 = \zeta(2+2\lambda)/\zeta(1+\lambda)^2 > 0$. As $\lambda \searrow 0$, $n_{\mathrm{eff}} = \zeta(1+\lambda)^2/\zeta(2+2\lambda) \sim \frac{6}{\pi^2 \lambda^2}$ diverges. Hence $\Eof{\ol{\popfm}{\samfm}^2} \xrightarrow{\lambda\searrow0} 0$ whenever $\sup_i\Vmof{Y_i}<\infty$.
	\end{enumerate}
\end{remark}

\subsection{Local Constant Fr\'echet Regression}\label{ssec:localfrechet}

Let $(\mc X,\rho)$ be a metric space, the predictor space.
Let $x_1,\ldots,x_n\in\mc X$ be fixed design points.
We consider the local constant Fr\'echet estimator, cf.\ \cite{DavisEtAl2007ShapeRegression-ICCV,Hein2009RobustRegression-NIPS,Schoetz2022Regression-EJS,PetersenMueller2019Regression-AOS}.
Let the kernel $K\colon\Rp\to\Rp$ satisfy
\begin{equation}\label{eq:localfrechet:kernel}
 \const{Kmi}^{-1}\ind_{[0,1/2]}\leq K\leq\const{Kma}\ind_{[0,1]},
 \qquad \const{Kmi},\const{Kma}\in[1,\infty).
\end{equation}
For a target $x\in\mc X$ and bandwidth $h>0$, set
\begin{equation}\label{eq:localfrechet:weights}
 w_{i,h}(x):=
 \frac{K(\rho(x_i,x)/h)}{\sum_{j=1}^n K(\rho(x_j,x)/h)},
 \qquad
 \samfm_h\!(x):=m_2\brOf{\sum_{i=1}^n w_{i,h}(x)\delta_{Y_i}}.
\end{equation}
Write $N_n(x,r):=\#\{i:\rho(x_i,x)\leq r\}$. If $N_n(x,h/2)>0$, the denominator in the definition of $w_{i,h}(x)$ is positive, the weights sum to one, and $w_{i,h}(x)=0$ for $\rho(x_i,x)>h$.

The following corollary gives finite-sample bounds for this Nadaraya--Watson-type estimator under a H\"older condition (Lipschitz for $\beta=1$) on $z\mapsto\mu_z$ with respect to the $1$-Wasserstein distance.

\begin{corollary}[Local constant Fr\'echet regression]\label{cor:two:localreg}
    Let $(\mc X, \rho)$ be a metric space.
    Let $\mu_z\in\mc P^2(\mc Q)$, $z\in\mc X$, satisfy
    \begin{equation}\label{eq:two:localreg:smooth}
        W_1(\mu_z,\mu_{z\pr})\leq L\rho(z,z\pr)^\beta
    \end{equation}
    for some $L>0$ and $\beta\in(0,1]$.
    Set $v := \sup_{z\in\mc X}\VmofP{\mu_z}$ and assume $v<\infty$.
    Let $Y_i\sim\mu_{x_i}$ be independent, and let the kernel and the estimator be as in \eqref{eq:localfrechet:kernel} and \eqref{eq:localfrechet:weights}.
    Let $x\in\mc X$.
    Assume that, for some $D \in \Rpp$ and $\const{dns}\in [1,\infty)$, the design satisfies
    \begin{equation}\label{eq:localfrechet:design}
    N_n(x,h/2)\geq \const{dns}^{-1}nh^D.
    \end{equation}
    Then
    \begin{equation}\label{eq:localfrechet:designbound}
    \br{\EOf{\ol{\samfm_h\!(x)\,}{m_2(\mu_x)}^2}}^{1/2}
    \leq L h^\beta+\sqrt{\frac{v\const{Kmi}\const{Kma}\const{dns}}{nh^D}}
    \eqfs
    \end{equation}
\end{corollary}
For $v>0$, the right-hand side of \eqref{eq:localfrechet:designbound} is minimized over $h>0$ by
\begin{equation}\label{eq:localfrechet:bandwidth}
    h_*:=\const{bw}\br{\frac{v\const{Kmi}\const{Kma}\const{dns}}{L^2n}}^{\frac{1}{2\beta+D}}
    \eqcm\qquad
    \const{bw} := \br{\frac{D}{2\beta}}^{\frac{2}{2\beta+D}}
    \eqfs
\end{equation}
If \eqref{eq:localfrechet:design} holds at $h=h_*$, the corresponding estimator satisfies
\begin{equation}\label{eq:localfrechet:optimal}
    \EOf{\ol{\samfm_{h_*}\!(x)\,}{m_2(\mu_x)}^2}
    \leq \const{rate}
    L^{\frac{2D}{2\beta+D}}
    \br{\frac{v\const{Kmi}\const{Kma}\const{dns}}n}^{\frac{2\beta}{2\beta+D}}
    \eqcm\qquad
    \const{rate} := \frac{(2\beta+D)^2}{D^{\frac{2D}{2\beta+D}}\br{2\beta}^{\frac{4\beta}{2\beta+D}}}
    \eqfs
\end{equation}
This is the classical nonparametric rate $n^{-2\beta/(2\beta+D)}$ for $\beta$-smooth regression functions on $D$-dimensional predictor spaces.

The local constant estimator appears as (4.1) in \cite{PetersenMueller2019Regression-AOS}, whose rate theorems concern local linear fits with random covariates; \cref{rem:two:smoothness} compares our assumption with a first-order local constant version of their density conditions.
For iid covariate--response pairs, \cref{cor:two:localreg} applies conditionally on the covariates whenever \eqref{eq:localfrechet:design} holds at $x$, which for a positive local count one may always arrange by taking $\const{dns}=\max\{1,nh^D/N_n(x,h/2)\}$; an unconditional risk bound would additionally require control of the random weights and a rule for empty windows.

For real predictors $x_i=i/n$ and $\beta\in(0,1]$, \cref{cor:two:localreg} covers the local constant case of \cite[Theorem 2]{Schoetz2022Regression-EJS}, which already gives finite-sample mean squared error bounds of order $h^{2\beta}+(nh)^{-1}$.
Compared to that result, we obtain better constants, allow a general metric predictor space, and weaken several assumptions: we require only a second moment instead of a moment of order larger than $2$; we need no entropy condition, which makes the result stronger in certain negatively curved spaces such as the hyperbolic plane $\mathbb H^2$ and applicable in infinite-dimensional spaces; and the Wasserstein smoothness condition \eqref{eq:two:localreg:smooth} is weaker than the assumption of \cite{Schoetz2022Regression-EJS} that the $\mu_z$ have densities that vary smoothly in $z$ at every fixed location. Supporting arguments are collected in \cref{app:comparison}.

We can view local Fr\'echet regression as a two-step procedure, in which we first estimate the $\R$-valued regression functions $x\mapsto\Eof{\ol Yq^2\mid X=x}$ for every fixed $q$, and then perform the minimization over $q$ for fixed $x$. From standard nonparametric regression theory, this requires smoothness of these $\R$-valued regression functions. Directly assuming this smoothness as in \cite{TuckerWu2025PartiallyGlobal-SS} yields a hard-to-verify condition. In \cite{PetersenMueller2019Regression-AOS,Schoetz2022Regression-EJS} it is obtained by more tractable smoothness conditions on densities, but they are restrictive. Condition \eqref{eq:two:localreg:smooth} occupies the position between the two. This discussion is further elaborated in \cref{rem:two:smoothness}.

Entropy conditions can fail in natural response spaces.
Im, Jeon, and Park \cite{ImEtAl2026Entropy-arxiv} show that the local entropy condition of \cite{PetersenMueller2019Regression-AOS}, uniform boundedness of the entropy integral over shrinking balls around the target, generally fails in the quadratic Wasserstein space of univariate distributions on a compact interval, in particular at targets with densities bounded away from zero and infinity.
The correction note \cite{PetersenMueller2026Correction-AOS} withdraws the statement that had asserted the condition in that space, replacing it by one restricted to finite-dimensional families of quantile functions. The original local-linear rate theorems also invoke this entropy condition.
For global Fr\'echet regression, the correction note recovers the parametric $O_P(n^{-1/2})$ rate using Hilbert space arguments and projection onto the closed convex set of quantile functions in $L^2(0,1)$; this representation is isometric for the quadratic Wasserstein metric \cite{PanaretosZemel2020Wasserstein-Springer}.
An isometric representation as a closed convex subset of an $L^2$ space likewise carries the finite-sample bounds of \cite{LiZhu2026Regression-StatPapers}, which are accordingly stated only for response classes that admit one.
In contrast, \cref{cor:two:localreg} needs neither an entropy condition nor a Hilbert space representation.

\begin{remark}[Density smoothness versus transport smoothness]\label{rem:two:smoothness}
For the local constant estimator with $\beta=1$, the smoothness assumption of \cite{Schoetz2022Regression-EJS} and the local constant version of the assumption of \cite{PetersenMueller2019Regression-AOS} are \emph{density} conditions.
In the setting of \cref{cor:two:localreg} with real predictors, they lead to a probability measure $\nu$ on $\mc Q$ with $\mu_x\ll\nu$ for all $x$, and densities $f(y\mid x):=\frac{\dl\mu_x}{\dl\nu}(y)$ such that $x\mapsto f(y\mid x)$ is Lipschitz with a constant $L(y)$ for $\nu$-almost every $y$.
In \cite[Theorems 1--2]{Schoetz2022Regression-EJS} (assumption \textsc{H\"olderSmoothDensity} with polynomial order zero), $\nu$ may be any probability measure dominating the conditional laws and $\normof{L}_{L^2(\nu)}\leq C_{\mathrm{SmD}}$; a second moment $\int\ol y{m_2(\mu_0)}^2\,\dl\nu(y)\leq C_{\mathrm{Int}}$ of the reference measure is assumed as well, with $C_{\mathrm{SmD}},C_{\mathrm{Int}}\geq1$.
The published condition (L1) of \cite{PetersenMueller2019Regression-AOS} requires the marginal density $f_X$ of $X$ and the conditional densities $g_y$ of $X$ given $Y=y$ to be twice continuously differentiable, with $\sup_{x,y}\absof{g_y^{\prime\prime}(x)}<\infty$; it also requires continuity of $x\mapsto\Pr(Y\in U\mid X=x)$ for every open $U\subset\mc Q$.
For the local constant comparison, the corresponding first-order condition is Lipschitz continuity of $f_X$ and a Lipschitz bound $L_g$ on $g_y$ that is uniform in $y$, retaining the continuity requirement on the conditional laws.
On any covariate interval on which $f_X\geq f_->0$, Bayes' formula with $\nu$ the law of $Y$ gives $f(y\mid x)=g_y(x)/f_X(x)$ and hence a uniform Lipschitz bound on $f(y\mid\cdot)$ depending on $L_g$, $f_-$, and $\operatorname{Lip}(f_X)$; the bound is derived in \cref{app:comparison}.
Provided the reference measure has a finite second moment, this local constant density condition therefore implies, on such an interval, the density smoothness condition of \cite{Schoetz2022Regression-EJS}, which in turn implies \eqref{eq:two:localreg:smooth} with $L=C_{\mathrm{Int}}^{1/2}C_{\mathrm{SmD}}$ and the same $\beta$.

The gap between density and transport smoothness reflects two mechanisms by which $x\mapsto\mu_x$ can vary.
Under \emph{reweighting}, mass is redistributed over a fixed support. Density conditions control this variation and imply a bound on $W_1$, by the common-mass bound $W_1(\mu_x,\mu_{x\pr})\leq\int\ol yo\,\absof{f(y\mid x)-f(y\mid x\pr)}\,\dl\nu(y)$ used in \cref{app:comparison}.
Under \emph{transport}, the support itself moves: $W_1(\delta_q,\delta_p)=\ol qp$, whereas a common dominating probability measure need not exist. When densities exist, they may be discontinuous in $x$, as for translated uniform noise.
Density conditions measure changes at fixed locations and can be much more restrictive than transport smoothness; \cref{ex:two:phase} quantifies the difference.

A third mechanism, neither a density nor a transport condition, is used in \cite[Assumption 4]{TuckerWu2025PartiallyGlobal-SS}: H\"older continuity of the conditional Fr\'echet objective,
\begin{equation}\label{eq:two:localreg:objective}
 F_x(q):=\int\ol yq^2\,\dl\mu_x(y)
 \eqcm\qquad
 \sup_{q\in\mc Q}\abs{F_x(q)-F_{x\pr}(q)}\leq L\rho(x,x\pr)^\beta
 \eqfs
\end{equation}
If $\mc Q$ is bounded, as in \cite{TuckerWu2025PartiallyGlobal-SS}, this condition is implied by \eqref{eq:two:localreg:smooth}, with $2\diam(\mc Q)L$ in place of $L$.
On unbounded response spaces it is not: for $\mc Q=\R$ and $\mu_x:=\delta_x$, \eqref{eq:two:localreg:smooth} holds with $L=\beta=1$, whereas $\sup_{q\in\R}\absof{(x-q)^2-(x\pr-q)^2}=\infty$ for $x\neq x\pr$.
The three conditions are thus ordered, \emph{density} $\Rightarrow$ \emph{transport} $\Rightarrow$ \emph{objective}, the second implication on bounded response spaces, and both implications are strict; this is verified in \cref{app:comparison}.

\end{remark}

\begin{example}[Phase noise along a geodesic]\label{ex:two:phase}
Let $\gamma\colon\R\to\mc Q$ be an isometric embedding of the real line into the Hadamard space $(\mc Q,d)$, i.e., $\ol{\gamma(s)}{\gamma(t)}=\absof{s-t}$ for all $s,t\in\R$, let $U$ be a real random variable with $\Eof{U}=0$ and $\Eof{U^2}<\infty$, let $\tau\geq0$, and consider real predictors $x\in[0,1]$ with
\begin{equation*}
 Y_x:=\gamma(x+\tau U)\sim\mu_x\eqfs
\end{equation*}
The noise acts on the parameter of the geodesic, and $\tau=0$ is the noise-free response $Y_x=\gamma(x)$.
(If only a geodesic segment $\gamma\colon[a,b]\to\mc Q$ is available and $\absof U\leq M$ almost surely, the statements below remain valid provided $[a,b]\supseteq[-\tau M,1+\tau M]$; the two noise laws in (ii) have $M=1$.)
The following statements are verified in \cref{app:comparison}.
\begin{enumerate}[label=(\roman*)]
 \item (Transport smoothness.) For all $\tau\geq0$ and $x,x\pr\in[0,1]$,
 \begin{equation}\label{eq:two:phase:w1}
  W_1(\mu_x,\mu_{x\pr})=\abs{x-x\pr}\eqcm\qquad m_2(\mu_x)=\gamma(x)\eqcm\qquad\VmOfP{\mu_x}=\tau^2\EOf{U^2}\eqfs
 \end{equation}
 Thus \eqref{eq:two:localreg:smooth} holds with $L=1$ and $\beta=1$, uniformly in $\tau$. For the design $x_i=i/n$ and $4/n\leq h\leq1$, \eqref{eq:localfrechet:design} holds with $D=1$ and $\const{dns}=4$, and \cref{cor:two:localreg} gives
 \begin{equation}\label{eq:two:phase:bound}
  \br{\EOf{\ol{\samfm_h\!(x)}{\gamma(x)}^2}}^{1/2}\leq h+2\tau \sqrt{\frac{\const{Kmi}\const{Kma} \EOf{U^2}}{nh}}\eqcm
 \end{equation}
 which for $\tau=0$ is the deterministic bound $h$.
 \item (Density smoothness fails or degrades.) Let $\tau>0$. If $U\sim\dunif{[-1,1]}$, then for every probability measure $\nu$ dominating all $\mu_x$, $x\in[0,1]$, there is no family of densities $f(\cdot\mid x)$, $x\in[0,1]$, such that $x\mapsto f(y\mid x)$ is continuous for $\nu$-almost every $y$; the smoothness assumption of \cite[Theorems 1--2]{Schoetz2022Regression-EJS} fails for every $\beta>0$, and so does the local constant analogue of (L1) in \cite{PetersenMueller2019Regression-AOS}, although each $\mu_x$ has a bounded density with respect to arc length on $\gamma$.
 If instead $U$ has the triangular density $u\mapsto(1-\absof u)_+$, the assumption of \cite{Schoetz2022Regression-EJS} holds with $\beta=1$, with $\nu$ the normalized arc-length measure on $\gamma([-\tau,1+\tau])$ and $C_{\mathrm{SmD}}=\max\{1,(1+2\tau)/\tau^2\}$, and this choice is optimal: every admissible choice of $\nu$ and $C_{\mathrm{SmD}}$ with $\beta=1$ satisfies $C_{\mathrm{SmD}}\geq(1+2\tau)/\tau^2$. For $\beta\in(0,1]$ and $\tau\leq\frac12$, the optimal $C_{\mathrm{SmD}}$ lies between $\tau^{-1-\beta}$ and $(1+2\tau)\tau^{-1-\beta}$.
\end{enumerate}
The transport constant remains $L=1$ in both families for every $\tau\geq0$, whereas no admissible density constant exists for uniform noise, the optimal one for triangular noise is of order $\tau^{-1-\beta}$ and diverges as the noise vanishes, and at $\tau=0$ the moving point masses admit no common dominating probability measure.
\end{example}

\subsection{Dirichlet-Process Fr\'echet Means}\label{ssec:dirichlet}

The Dirichlet process is the standard prior on probability measures in Bayesian nonparametrics \cite{Ferguson1973Dirichlet-AOS}.
It is defined on an arbitrary measurable space and is conjugate: the posterior given conditionally iid observations is again a Dirichlet process.
When the sample space is a Hadamard space, the Fr\'echet mean $m_2(P)$ of a random probability measure $P$ is a random point of $\mc Q$, and one may ask how concentrated its prior and posterior distributions are.
In a Hilbert space, $m_2(P)=\int y\,\dl P(y)$ is linear in $P$.
In a Hadamard space the mean is nonlinear, but the stick-breaking representation turns $P$ into a countably weighted sample of iid atoms, which is exactly the setting of \cref{thm:two:infty:independent}.

Let $a>0$ and $\mu\in\mc P^2(\mc Q)$. A random probability measure $P\sim\operatorname{DP}(a,\mu)$ has the stick-breaking representation \cite{Sethuraman1994DirichletPriors-SS}
\begin{equation}\label{eq:dirichlet:stickbreaking}
 P=\sum_{i=1}^\infty W_i\delta_{Y_i},
 \qquad W_i=V_i\prod_{j<i}(1-V_j),
\end{equation}
where $V_i\sim\operatorname{Beta}(1,a)$ are iid, $Y_i\sim\mu$ are iid, and the two sequences are independent. Each $W_i$ is the fraction $V_i$ of what is left of a unit stick after $i-1$ breaks; the remainder $\prod_{j\leq i}(1-V_j)$ tends to zero almost surely, so the weights sum to one. For nonatomic $\mu$, the atoms $Y_i$ are almost surely distinct, so $P$ has infinitely many atoms. Since $\Eof{\int \ol yo^2\,\dl P(y)}=\int \ol yo^2\,\dl\mu(y)<\infty$ for any $o\in\mc Q$, the mean $m_2(P)$ is well-defined almost surely.

Conditionally on the weights, \eqref{eq:dirichlet:stickbreaking} is the setting of \cref{thm:two:infty:independent} with $w_i=W_i$ and $\mu_i=\mu$ for all $i\in\N$: the population mixture is $\popdis=\sum_iW_i\mu=\mu$, the empirical mixture is $\samdis=P$, and the theorem gives
\begin{equation}\label{eq:dirichlet:risk}
 \EOf{\ol{m_2(P)}{m_2(\mu)}^2}
 \leq\VmOfP{\mu}\,\EOf{\sum_{i=1}^\infty W_i^2}
 =\frac{\VmOfP{\mu}}{a+1}
 \eqcm
\end{equation}
where the identity follows from $\Eof{W_i^2}=\frac{2}{(a+1)(a+2)}(\frac a{a+2})^{i-1}$ by summing the geometric series.
Thus the precision parameter $a$ controls how tightly the prior distribution of the Fr\'echet mean concentrates around $m_2(\mu)$.
Only two features of the Dirichlet process enter this prior bound: the weights are independent of the iid atoms, and $\Eof{\sum_iW_i^2}$ is known. The prior bound therefore holds for any proper species-sampling prior with $\VmofP{\mu}\,\Eof{\sum_iW_i^2}$ on the right-hand side, for instance with $(1-\varsigma)\VmofP{\mu}/(1+\vartheta)$ for the Pitman--Yor process with discount $\varsigma\in[0,1)$ and strength $\vartheta>-\varsigma$; the posterior bounds below use conjugacy of the Dirichlet process and do not transfer.
In a Hilbert space, \eqref{eq:dirichlet:risk} holds with equality and recovers the classical variance formula for the mean of a Dirichlet process, cf.\ \cite{Ferguson1973Dirichlet-AOS}; in a Hadamard space, the same expression is an upper bound.

Now let $X_1,\ldots,X_n$ be $\mc Q$-valued observations that are conditionally iid with law $P$, where $P\sim\operatorname{DP}(a,\mu)$. The posterior distribution of $P$ given $X_1,\ldots,X_n$ is $\operatorname{DP}(a+n,\tilde\mu_n)$ with
\begin{equation}\label{eq:dirichlet:posteriorbase}
 \tilde\mu_n:=\frac{a\mu+\sum_{j=1}^n\delta_{X_j}}{a+n}
\end{equation}
\cite{Ferguson1973Dirichlet-AOS}. Given the observations, $\tilde\mu_n$ is a fixed element of $\mc P^2(\mc Q)$, and \eqref{eq:dirichlet:risk} with $(a,\mu)$ replaced by $(a+n,\tilde\mu_n)$ yields
\begin{equation}\label{eq:dirichlet:posterior}
 \EOf{\ol{m_2(P)}{m_2(\tilde\mu_n)}^2\,\middle|\, X_1,\ldots,X_n}
 \leq\frac{\VmOfP{\tilde\mu_n}}{a+n+1}
 \eqcm\qquad
 \VmOfP{\tilde\mu_n}=\frac12\iint\ol y{y\pr}^2\,\dl\tilde\mu_n(y)\,\dl\tilde\mu_n(y\pr)
 \eqfs
\end{equation}
Note that $\VmofP{\tilde\mu_n}$, like the left-hand side of \eqref{eq:dirichlet:posterior}, is a $\sigma(X_1,\ldots,X_n)$-measurable random variable: it is the metric variance of the measure $\tilde\mu_n$, which depends on the observations.

The bound \eqref{eq:dirichlet:posterior} quantifies the posterior uncertainty about the Fr\'echet mean after finitely many observations.
Its center $m_2(\tilde\mu_n)$ interpolates between the prior center $m_2(\mu)$ and the Fr\'echet mean of the empirical measure $\hat\nu_n:=\frac1n\sum_{j=1}^n\delta_{X_j}$ of the observations: since $\tilde\mu_n=\frac{a}{a+n}\mu+\frac{n}{a+n}\hat\nu_n$, the contraction property \eqref{eq:cat0:2fm:CP} with the coupling that is the identity on the empirical part, respectively on the prior part, gives
\begin{equation}\label{eq:dirichlet:shrinkage}
\begin{split}
 \ol{m_2(\tilde\mu_n)}{m_2(\hat\nu_n)}&\leq W_1(\tilde\mu_n,\hat\nu_n)\leq\frac{a}{a+n}\,W_1(\mu,\hat\nu_n)
 \eqcm\\
 \ol{m_2(\tilde\mu_n)}{m_2(\mu)}&\leq W_1(\tilde\mu_n,\mu)\leq\frac{n}{a+n}\,W_1(\mu,\hat\nu_n)
 \eqfs
\end{split}
\end{equation}
Thus the influence of the prior on the center decays like $n^{-1}$, and the two bounds add up to $W_1(\mu,\hat\nu_n)$.
By the conditional Markov inequality, for $\eta\in(0,1)$, the closed ball around $m_2(\tilde\mu_n)$ of radius $\sqrt{\VmofP{\tilde\mu_n}/((a+n+1)\eta)}$ has posterior probability at least $1-\eta$; this is an explicit credible ball for the Fr\'echet mean.
With $(a,\mu)=(n,\hat\nu_n)$, \eqref{eq:dirichlet:risk} also covers Rubin's Bayesian bootstrap \cite{Rubin1981BayesianBootstrap-AOS}, whose resampled measures are $P^*\sim\operatorname{DP}(n,\hat\nu_n)$: $\Eof{\ol{m_2(P^*)}{m_2(\hat\nu_n)}^2\mid X_1,\ldots,X_n}\leq\VmofP{\hat\nu_n}/(n+1)$, with equality in Hilbert spaces.

Taking expectations in \eqref{eq:dirichlet:posterior} yields an unconditional bound.
Write $\lambda_n:=\frac a{a+n}$, so that $\tilde\mu_n=\lambda_n\mu+(1-\lambda_n)\hat\nu_n$ and
\begin{equation*}
 \VmOfP{\tilde\mu_n}
 =\lambda_n^2\,\VmOfP{\mu}
 +\lambda_n(1-\lambda_n)\iint\ol yz^2\,\dl\mu(y)\,\dl\hat\nu_n(z)
 +(1-\lambda_n)^2\,\VmOfP{\hat\nu_n}
 \eqfs
\end{equation*}
Marginally, each $X_j$ has law $\mu$, so $\Eof{\iint\ol yz^2\,\dl\mu(y)\,\dl\hat\nu_n(z)}=2\VmofP{\mu}$.
For $j\neq k$, conditioning on $P$ gives $\Eof{\ol{X_j}{X_k}^2}=2\,\Eof{\VmofP{P}}=\frac{2a}{a+1}\VmofP{\mu}$, since $\VmofP{P}=\frac12\sum_{i\neq i\pr}W_iW_{i\pr}\ol{Y_i}{Y_{i\pr}}^2$ by \eqref{eq:dirichlet:stickbreaking} and $\Eof{\sum_iW_i^2}=\frac1{a+1}$; hence $\Eof{\VmofP{\hat\nu_n}}=\frac{n-1}{n}\frac{a}{a+1}\VmofP{\mu}$.
Altogether,
\begin{equation}\label{eq:dirichlet:bayesrisk}
 \EOf{\VmOfP{\tilde\mu_n}}=\frac{a\,(a+n+1)}{(a+1)(a+n)}\,\VmOfP{\mu}
 \eqcm\qquad
 \EOf{\ol{m_2(P)}{m_2(\tilde\mu_n)}^2}\leq\frac{a}{(a+1)(a+n)}\,\VmOfP{\mu}
 \eqfs
\end{equation}
The bound coincides with \eqref{eq:dirichlet:risk} for $n=0$ and decays like $n^{-1}$.
It is an integrated Bayes squared-risk bound under the joint law of the prior and the observations: the posterior mean squared distance from $m_2(P)$ to $m_2(\tilde\mu_n)$, averaged over the prior predictive distribution of the data, is $\landauO(n^{-1})$ as $n\to\infty$, with the prior constant $\VmofP{\mu}$.

\section{\texorpdfstring{2-Fr\'echet Means in $\mathsf{CAT}(\kappa)$ Spaces, $\kappa>0$}{2-Frechet Means in CAT(kappa) Spaces, kappa > 0}}\label{sec:pos}

In this section, we assume $(\mc Q,d)$ to be a complete \catk{} space with $\kappa>0$ and again consider only the $2$-Fr\'echet mean ($\alpha=2$). Under positive curvature, $2$-Fr\'echet means are in general not unique. By restricting the distributions to a closed convex set with circumradius less than $\pi/(2\sqrt\kappa)$, we retain uniqueness as well as the quadruple, variance, and contraction inequalities, up to explicit constants (\cref{ssec:pos:prelim}). We then derive a weighted-sample bound in \cref{ssec:pos:main} (\cref{thm:pos:infty:independent}). We apply it to local constant regression in \cref{ssec:pos:localfrechet} and to Dirichlet-process Fr\'echet means in \cref{ssec:pos:dirichlet}.

\subsection{Preliminaries}\label{ssec:pos:prelim}

Let $\kappa\in\Rpp$ and let $(\mc Q,d)$ be a complete \catk{} space.
Let $\setC\subseteq\mc Q$ be a nonempty bounded closed convex set, and let $\varrho:=\inf_{o\in\mc Q}\sup_{y\in\setC}d(o,y)$ be its circumradius.
Assume $0<\varrho<\pi/(2\sqrt\kappa)$ and set 
\begin{equation}\label{eq:pos:constants}
    \Delta := 2\varrho\sqrt{\kappa}\in(0,\pi)
    \eqcm\qquad
    C_\Delta = \frac{\Delta}{\sin(\Delta)}\in(1,\infty)
    \eqfs
\end{equation}
As $\setC$ is convex, geodesics between points of $\setC$ stay in $\setC$, so that \cref{prp:prelim:qi} yields the \emph{quadruple inequality}
\begin{equation}\tag{QI$_+$}\label{eq:2fmcatk:qi}
    \frac12\abs{\ol yq^2-\ol yp^2-\ol zq^2+\ol zp^2}
    \leq C_\Delta\,\ol qp\,\ol yz
    \qquad\text{for all } y,z,q,p\in\setC
    \eqfs
\end{equation}
As probability measures $\mu\in\mc P(\setC)$ have bounded support, $\mc P(\setC)\subset\mc P^s(\mc Q)$ for all $s\in\Rpp$, and the renormalization in \eqref{eq:2fmcat0:fmdef} is not needed.
By \cite[Theorem B and Corollary 16]{Yokota2016Barycenter-JMSJ}, every $\mu\in\mc P(\setC)$ has a unique $2$-Fr\'echet mean, and this mean belongs to $\setC$.
Thus, for $Y\sim\mu\in\mc P(\setC)$,
\begin{equation*}
    m := m_2(\mu) = m_2[Y] = \argmin_{q\in\mc Q} \EOf{\ol Yq^2} = \argmin_{q\in\setC} \EOf{\ol Yq^2}
    \eqfs
\end{equation*}
By \cref{thm:pos:hemisphere:vi} in the appendix, $m$ satisfies the \emph{variance inequality}: for all $q\in\setC$,
\begin{equation}\tag{VI$_+$}\label{eq:2fmcatk:vi}
    \ol{m}q^2
    \leq
    C_\Delta
    \EOf{\ol Yq^2-\ol Ym^2}
    \eqfs
\end{equation}
Moreover, any $\mu,\nu \in \mc P(\setC)$ satisfy the \emph{contraction property} (\cref{thm:pos:hemisphere:cp} in the appendix)
\begin{equation}\tag{CP$_+$}\label{eq:2fmcatk:CP}
    \ol {m_2(\mu)}{m_2(\nu)}\leq C_\Delta^2\, W_1(\mu, \nu)
    \eqfs
\end{equation}
For $\mu\in\mc P(\setC)$, the Fr\'echet and metric variances satisfy
\begin{equation}\label{eq:pos:variances}
    \VfOf{Y}\leq\frac{2C_\Delta}{C_\Delta+1}\,\VmOf{Y}
    \qquad\text{and}\qquad
    \VmOf{Y}\leq2\,\VfOf{Y}
    \eqcm
\end{equation}
see \cref{lmm:pos:variances} in the appendix.
As $\Delta\to0$, we have $C_\Delta\to1$, and \eqref{eq:2fmcatk:qi}, \eqref{eq:2fmcatk:vi}, \eqref{eq:2fmcatk:CP}, and \eqref{eq:pos:variances} turn into the Hadamard-space inequalities of \cref{ssec:two:prelim}.

\begin{remark}[Radius versus diameter]\label{rem:pos:diameter}
    Set $\delta:=\sqrt\kappa\,\diam(\setC)$. Then $\delta\leq\Delta\leq2\delta$.
    Hence, the diameter condition $\delta<\pi/2$ implies the radius condition $\Delta<\pi$, whereas $\Delta<\pi$ is compatible with every diameter $\delta<\pi$.
    The condition $\delta<\pi$ alone still gives the quadruple inequality, but does not ensure a distribution-uniform variance inequality with a constant depending only on $\delta$ (\cref{prp:pos:diameterobstruction} in the appendix).
\end{remark}

\begin{remark}[Sharper constants]\label{rem:pos:sharper}
    Recall $C_\delta=\delta/\sin(\delta)$.
    For $0<\delta<\pi/2$, the optimal diameter-only variance and contraction constants are
    \begin{equation*}
        A_\delta^{\mathrm{diam}}:=\frac{\theta_\delta^2}{\delta^2-\theta_\delta^2},
        \qquad B_\delta^{\mathrm{diam}}:=\frac1{\sqrt{\cos(\delta)}},
        \qquad \theta_\delta:=\arccos\brOf{\sqrt{\cos(\delta)}}.
    \end{equation*}
    Both tend to $1$ as $\delta\searrow0$ and are set to $+\infty$ for $\delta\geq\pi/2$.
    \Cref{thm:pos:hemisphere:vi,thm:pos:hemisphere:cp} in the appendix give the optimal radius-only constants $A_\Delta^{\mathrm{rad}},B_\Delta^{\mathrm{rad}}$ implicitly; we do not know elementary closed forms for them.
    Restricting the base point to $\ol mq\leq\varrho$ does admit one: \cref{thm:pos:hemisphere:vi}\,(ii) gives the explicit optimal constant $A_\Delta^{\mathrm{near}}$ on that range, which is also an explicit lower bound for $A_\Delta^{\mathrm{rad}}$.
    The radius constants satisfy the non-sharp upper bounds $A^{\mathrm{rad}}_\Delta<1/\cos(\Delta/2)<C_\Delta$ and $B^{\mathrm{rad}}_\Delta<\min\{1/\cos(\Delta/2)^2,2C_\Delta\}$ (\cref{thm:pos:hemisphere:vi,thm:pos:hemisphere:cp} and \cref{lmm:pos:cp:scalar} in the appendix), which give the explicit choices
    \begin{align}
        A_{\Delta,\delta}&:=\min\{1/\cos(\Delta/2),A_\delta^{\mathrm{diam}}\},\label{eq:pos:vi:best}\\
        B_{\Delta,\delta}&:=\min\{1/\cos(\Delta/2)^2,2C_\Delta,B_\delta^{\mathrm{diam}},C_\delta A_{\Delta,\delta}\}.\label{eq:pos:cp:best}
    \end{align}
    We can replace $C_\Delta$ in \eqref{eq:2fmcatk:qi}, $C_\Delta$ in \eqref{eq:2fmcatk:vi}, and $C_\Delta^2$ in \eqref{eq:2fmcatk:CP} by $C_\delta$, $A_{\Delta,\delta}$, and $B_{\Delta,\delta}$, respectively.
    In particular, the multiplier $C_\delta A_{\Delta,\delta}B_{\Delta,\delta}$ that replaces $C_\Delta^4$ in \cref{thm:pos:infty:independent} is at most $\min\{C_\Delta/\cos(\Delta/2)^3,2C_\Delta^2/\cos(\Delta/2)\}$, which is less than $C_\Delta^4$ and, by its second term, less than $2C_\Delta^3$.
    \Cref{lmm:pos:variances} in the appendix also sharpens \eqref{eq:pos:variances} to $\Vfof{Y}\leq2A_{\Delta,\delta}(1+A_{\Delta,\delta})^{-1}\Vmof{Y}$.

    To gauge what the explicit surrogates cost, here are the optimal radius constants at three values of $\Delta$, computed from \cref{thm:pos:hemisphere:vi,thm:pos:hemisphere:cp}.
    \begin{center}
        \small
        \begin{tabular}{@{}rrrrrrrrr@{}}
            \toprule
            $\Delta$ & $A_\Delta^{\mathrm{near}}$ & $A_\Delta^{\mathrm{rad}}$ & $1/\cos(\Delta/2)$ & $C_\Delta$ & $B_\Delta^{\mathrm{rad}}$ & $1/\cos(\Delta/2)^2$ & $2C_\Delta$ & $C_\Delta^2$\\
            \midrule
            $1$ & $1.094$ & $1.094$ & $1.139$ & $1.188$ & $1.188$ & $1.298$ & $2.377$ & $1.412$\\
            $2$ & $1.601$ & $1.605$ & $1.851$ & $2.200$ & $2.208$ & $3.426$ & $4.399$ & $4.838$\\
            $3$ & $11.15$ & $11.55$ & $14.14$ & $21.26$ & $23.41$ & $199.9$ & $42.52$ & $451.9$\\
            \bottomrule
        \end{tabular}
    \end{center}
    At $\Delta=1$ the pairs $A_\Delta^{\mathrm{near}},A_\Delta^{\mathrm{rad}}$ and $C_\Delta,B_\Delta^{\mathrm{rad}}$ each agree to the digits shown; nevertheless, $A_\Delta^{\mathrm{near}}<A_\Delta^{\mathrm{rad}}<1/\cos(\Delta/2)<C_\Delta<B_\Delta^{\mathrm{rad}}<\min\{1/\cos(\Delta/2)^2,2C_\Delta\}$ and $1/\cos(\Delta/2)^2<C_\Delta^2$ for every $0<\Delta<\pi$, by \cref{thm:pos:hemisphere:vi}\,(ii), \eqref{eq:pos:vi:radius:upper}, and \eqref{eq:pos:cp:radius:upper} in the appendix. The bound $2C_\Delta$ is the smaller one from $\Delta\approx2.33$ on.
\end{remark}

\subsection{The Weighted-Sample Bound}\label{ssec:pos:main}

The following theorem is the analogue of \cref{thm:two:infty:independent} under positive curvature. Its proof, given in \cref{app:pos}, is the same leave-one-out argument with the inequalities of \cref{ssec:pos:prelim} in place of those of \cref{ssec:two:prelim}.
\begin{theorem}[Weighted-sample bound under positive curvature]\label{thm:pos:infty:independent}
    In the setting of \cref{ssec:pos:prelim}, let $(w_i)_{i\in\N} \subset [0,1]$ such that $\sum_{i=1}^\infty w_i = 1$.
    Let $(Y_i)_{i\in\N}$ be a sequence of independent random variables with distributions $\mu_1, \mu_2, \dots \in \mc P(\mc C)$.
    Let $\popdis := \sum_{i=1}^\infty w_i \mu_i$ and $\samdis := \sum_{i=1}^\infty w_i \delta_{Y_i}$. 
    Set $\popfm := m_2(\popdis)$ and $\samfm := m_2(\samdis)$.
    Then
	\begin{equation}\label{eq:pos:infty:independent}
		\EOf{\ol{\popfm}{\samfm}^2}
		\leq
		C_\Delta^4\sum_{i=1}^\infty w_i^2\, \VmOf{Y_i}
		\eqfs
	\end{equation}
\end{theorem}

The right-hand side is that of \cref{thm:two:infty:independent} multiplied by $C_\Delta^4$.
This factor is exactly the product of the constants in \eqref{eq:2fmcatk:vi}, \eqref{eq:2fmcatk:qi}, and \eqref{eq:2fmcatk:CP}, which are used successively in the proof of \cref{thm:pos:infty:independent}.
Using the constants \eqref{eq:pos:vi:best} and \eqref{eq:pos:cp:best} of \cref{rem:pos:sharper} gives the sharper bound
\begin{equation}\label{eq:pos:infty:refined}
    \EOf{\ol{\popfm}{\samfm}^2}
    \leq C_\delta A_{\Delta,\delta}B_{\Delta,\delta}\sum_{i=1}^{\infty}w_i^2\,\VmOf{Y_i}.
\end{equation}
The optimal one-parameter variance and contraction constants give the multipliers
$C_\Delta A_\Delta^{\mathrm{rad}}B_\Delta^{\mathrm{rad}}$ and, if $\delta<\pi/2$,
$C_\delta A_\delta^{\mathrm{diam}}B_\delta^{\mathrm{diam}}$.

Neither identical distributions nor equal Fr\'echet means of the $\mu_i$ are assumed.
Items (ii)--(vi) of \cref{rem:weights} apply verbatim, with each bound multiplied by $C_\Delta^4$; in particular, iid data with $w_i = \frac1n\ind_{\{i\leq n\}}$ give $\Eof{\ol{\popfm}{\samnfm}^2}\leq\frac{C_\Delta^4}{n}\Vmof{Y_1}$.

\begin{remark}[Strong convexity and the variance inequality in the literature]\label{rem:pos:literature}
    A key tool for obtaining \cref{thm:pos:infty:independent} is the variance inequality \eqref{eq:2fmcatk:vi} (see also \cref{thm:pos:hemisphere:vi} in the appendix) in \catk{} spaces. This article improves the previously known inequalities of this kind.

    To compare them, work in the notation of \cref{ssec:pos:prelim}, let $\mu\in\mc P(\setC)$, $Y\sim\mu$ and $m:=m_2(\mu)$, and write
    \begin{equation*}
        \mathcal E_\mu(q):=\EOf{\ol Yq^2-\ol Ym^2}
        \eqcm\qquad q\in\setC
        \eqcm
    \end{equation*}
    for the excess of the objective over its minimum. In this notation the variance inequality with constant $A$ is 
    \begin{equation}\label{eq:pos:literature:vi}
        \ol mq^2\leq A\,\mathcal E_\mu(q)
        \eqfs
    \end{equation}

    A function $f\colon\setC\to\R$ is \emph{$k$-strongly convex}, $k\in\Rp$, if
    \begin{equation}\label{eq:pos:strongconvexity}
        f(\gamma_t)\leq(1-t)f(\gamma_0)+tf(\gamma_1)-\frac k2\,t(1-t)\,\ol{\gamma_0}{\gamma_1}^2
    \end{equation}
    for every geodesic $\gamma\colon[0,1]\to\setC$ and every $t\in[0,1]$; for $k=0$ this is convexity. If the objective $q\mapsto\int\ol yq^2\,\dl\mu(y)$ is $k$-strongly convex with $k>0$, then applying \eqref{eq:pos:strongconvexity} along the geodesic from $m$ to $q$, using $f(\gamma_t)\geq f(m)$, and letting $t\searrow0$ gives $\mathcal E_\mu(q)\geq\frac k2\ol mq^2$, which is \eqref{eq:pos:literature:vi} with $A=2/k$.

    The following table compares different results in the literature; here $\cot=\cos/\sin$.
    \begin{center}
        \small
        \begin{tabular}{@{}lll@{}}
            \toprule
            source & hypothesis & statement\\
            \midrule
            \cite[Theorem 1.2]{Karcher1977CenterOfMass-CPAM} & Riemannian, $\Delta<\pi/2$ & $q\mapsto\int\ol yq^2\,\dl\mu(y)$ is $2\Delta\cot(\Delta)$-strongly convex\\
            \cite[Proposition 3.1(i)]{Ohta2007Convexities-GD} & $\delta<\pi/2$ & $q\mapsto\ol yq^2$ is $2\delta\cot(\delta)$-strongly convex, each $y$\\
            \cite[Corollary 16]{Yokota2017PBarycenter-TJM} & $\delta<\pi/2$ & $\mathcal E_\mu(q)\geq\delta\cot(\delta)\,\ol mq^2$\\
            \cite[Proposition 3]{BrunelSerres2026FiniteSample-arxiv} & $\Delta<\pi/2$ & $\mathcal E_\mu(q)\geq\Delta\cot(\Delta)\,\ol mq^2$\\
            \cite[Proposition 19]{Yokota2016Barycenter-JMSJ} & $\Delta<\pi$ & $\mathcal E_\mu(q)\geq c\,\ol mq^{a}$ for some $a>2$, $c>0$\\
            \cite[Theorem 1.1]{Gietl2026VarianceInequalities-AGMS} & $\Delta<\pi$ & $\ol mo^2\leq\mathcal E_\mu(o)+(1-\tfrac\Delta2\cot(\tfrac\Delta2))\VfofP\mu$\\
            \Cref{thm:pos:hemisphere:vi} & $\Delta<\pi$ & $\mathcal E_\mu(q)\geq\min\{A^{\mathrm{rad}}_\Delta,A^{\mathrm{diam}}_\delta\}^{-1}\ol mq^2$\\
            \bottomrule
        \end{tabular}
    \end{center}
    Each statement is asserted for every $\mu\in\mc P(\setC)$ and every $q\in\setC$, with one exception: Gietl's bound is asserted only at the center $o$ of a ball containing $\setC$. In the displayed Gietl bound, $o$ is a point with $\setC\subset\overline B(o,\varrho)$, if one exists, and $\mathcal E_\mu(o)$ is given by the same formula; the coefficient follows by applying the cited theorem to balls about $o$ with radii decreasing to $\varrho$. In \cite[Corollary 16]{Yokota2017PBarycenter-TJM} the modulus is left unspecified; the explicit constant in that row uses Ohta's modulus.
    By the argument above, the strong convexity of Karcher and of Ohta yields \eqref{eq:pos:literature:vi} with $A=\tan(\Delta)/\Delta$ and $A=\tan(\delta)/\delta$, respectively, which are also the constants of the variance inequalities \cite[Proposition 3]{BrunelSerres2026FiniteSample-arxiv} and \cite[Corollary 16]{Yokota2017PBarycenter-TJM}, respectively.
    The inequalities of Yokota \cite[Proposition 19]{Yokota2016Barycenter-JMSJ} and Gietl \cite[Theorem 1.1]{Gietl2026VarianceInequalities-AGMS} are of a different form. In the former, the growth is of order $\ol mq^{a}$ with an exponent $a>2$ and a constant depending only on $\kappa$ and the radius, neither made explicit, so it is weaker than quadratic near $m$.
    In the latter, the constant multiplies the variance rather than the excess, and only the center of the ball is admitted as a base point; its Section~1 notes that an arbitrary base point would require a diameter condition.
    Among the earlier results listed, only these two reach past $\delta=\pi/2$.
    \Cref{thm:pos:hemisphere:vi} is a variance inequality on the maximal range of radii (\cref{prp:pos:diameterobstruction}\,\ref{itm:pos:obstruction:boundary} in the appendix) with optimal constants, each optimal among constants depending on $\Delta$ alone, respectively on $\delta$ alone.
\end{remark}

\begin{remark}[Comparison with \cite{BrunelSerres2026FiniteSample-arxiv}]\label{rem:pos:BrunelSerres2026}
    Theorems 3 and 4 of \cite{BrunelSerres2026FiniteSample-arxiv} bound the risk of the empirical $2$-Fr\'echet mean of $n$ observations in $\setC$ by
    \begin{equation}\label{eq:pos:bs}
        \EOf{\ol{\popfm}{\samnfm}^2}\leq\frac{A_{\mathsf{BS}}}{n}\,\VfOf{Y_1}
        \qquad\text{and}\qquad
        \EOf{\ol{\popfm}{\samnfm}^2}\leq\frac{\tilde A_{\mathsf{BS}}}{n^2}\sum_{i=1}^n\VfOf{Y_i}
        \eqcm
    \end{equation}
    the first for iid observations, the second for independent observations with a common $2$-Fr\'echet mean.
    Both hold under $\varrho<\pi/(4\sqrt\kappa)$, half the range assumed in \cref{ssec:pos:prelim}, and their constants are
    \begin{equation*}
        A_{\mathsf{BS}}=\frac{16\pi\tan(\Delta)^2}{\Delta^2}
        \qquad\text{and}\qquad
        \tilde A_{\mathsf{BS}}=\frac{8\sqrt2\,\pi\tan(\Delta)^2}{\Delta^2}
        \eqfs
    \end{equation*}
    The bound \eqref{eq:pos:infty:independent} with $w_i=\frac1n\ind_{\{i\leq n\}}$ and \cref{lmm:pos:variances} in the appendix give
    \begin{equation}\label{eq:pos:openquestion}
        \EOf{\ol{\popfm}{\samnfm}^2}
        \leq\frac{C_\Delta^4}{n^2}\sum_{i=1}^n\VmOf{Y_i}
        \leq\frac{2C_\Delta^4}{n^2}\sum_{i=1}^n\VfOf{Y_i}
        \eqcm
    \end{equation}
    that is, both inequalities of \eqref{eq:pos:bs} with the single constant $2C_\Delta^4$, without the requirement of a common $2$-Fr\'echet mean, on the full range $\varrho<\pi/(2\sqrt\kappa)$, and, by \cref{thm:pos:infty:independent} itself, for weighted and countable samples.
    Already the smaller of the two published constants exceeds ours by a factor of at least $4\sqrt2\,\pi\approx17.8$:
    \begin{equation*}
        \frac{\tilde A_{\mathsf{BS}}}{2C_\Delta^4}
        =\frac{4\sqrt2\,\pi\sin(\Delta)^6}{\Delta^6\cos(\Delta)^2}
        \geq4\sqrt2\,\pi,
        \qquad 0<\Delta<\pi/2.
    \end{equation*}
    The ratio tends to $4\sqrt2\,\pi$ as $\Delta\searrow0$ and increases to infinity as $\Delta\nearrow\pi/2$, since its logarithmic derivative has the sign of $h(\Delta):=\Delta(1+2\cos(\Delta)^2)-3\sin(\Delta)\cos(\Delta)$, and $h(0)=0$, $h'(\Delta)=4\sin(\Delta)(\sin(\Delta)-\Delta\cos(\Delta))>0$.
    The main ingredients that allow for the improved result are the variance inequality \eqref{eq:2fmcatk:vi} of \cref{thm:pos:hemisphere:vi}, the quadruple inequality \eqref{eq:2fmcatk:qi} of \cite{Schoetz2026Trapezoid-arxiv}, and the contraction property \eqref{eq:2fmcatk:CP} of \cref{thm:pos:hemisphere:cp}, both theorems in the appendix.
    Theorem 1 of \cite{BrunelSerres2026FiniteSample-arxiv} makes the $n$-point empirical barycenter map $\frac Ln$-Lipschitz from $(\setC^n,\ell^1)$ to $\setC$ with $L=\pi\tan(\Delta)/(2\Delta)$, which is \eqref{eq:2fmcatk:CP} for pairs of uniform empirical measures on $n$ atoms, whereas \cref{thm:pos:hemisphere:cp} covers all $\mu,\nu\in\mc P(\setC)$ with the optimal constant $B_\Delta^{\mathrm{rad}}$.
    On the full range, Yokota \cite[Proposition 24]{Yokota2016Barycenter-JMSJ} shows that $\mu\mapsto m_2(\mu)$ is Lipschitz continuous with respect to the $p$-Wasserstein distance for some $p>1$, with a constant depending only on $\kappa$ and the radius; neither $p$ nor the constant is made explicit.

    Open Question 2 in \cite[Section 4.2]{BrunelSerres2026FiniteSample-arxiv} asks (among other things) whether the constants in \eqref{eq:pos:bs} must diverge as $\varrho\nearrow\pi/(4\sqrt\kappa)$. Since $2C_\Delta^4\leq\pi^4/8$ on that range, \eqref{eq:pos:openquestion} answers this in the negative. At the larger threshold $\varrho\nearrow\pi/(2\sqrt\kappa)$, however, divergence \emph{is} necessary: \cref{prp:pos:diameterobstruction}\,\ref{itm:pos:obstruction:order} in the appendix gives a family of fixed diameter along which the smallest admissible coefficient is of order $(\pi/2-\varrho\sqrt\kappa)^{-2}$.
\end{remark}

Under bounded support, a bounded-differences argument turns \cref{thm:pos:infty:independent} into a high-probability bound; its proof is in \cref{app:pos}.

\begin{corollary}[Bounded-support concentration]\label{cor:pos:concentration}
 In the setting of \cref{ssec:pos:prelim}, let $(Y_i)_{i\in\N}$ be a sequence of independent random variables with laws in $\mc P(\setC)$, and let $(w_i)_{i\in\N}\subset[0,1]$ be deterministic weights with $\sum_{i=1}^{\infty}w_i=1$.
 Define $\popfm,\samfm$ as in \cref{thm:pos:infty:independent}. For $\eta\in(0,1)$, with probability at least $1-\eta$,
 \begin{equation}\label{eq:pos:concentration}
 d(\samfm,\popfm)\leq C_\Delta^2\br{
 \br{\sum_{i=1}^{\infty} w_i^2\VmOf{Y_i}}^{1/2}
 +\diam(\setC)\br{\frac{\normOf{w}_2^2}{2}\log\frac1\eta}^{1/2}}.
 \end{equation}
\end{corollary}

Analogously to the comparison of $L^2$ risk bounds in \cref{rem:pos:BrunelSerres2026}, \cref{cor:pos:concentration} strengthens the corresponding concentration bound \cite[Corollary 1]{BrunelSerres2026FiniteSample-arxiv}.

\subsection{Local Constant Fr\'echet Regression}\label{ssec:pos:localfrechet}

We consider the local constant estimator \eqref{eq:localfrechet:weights} of \cref{ssec:localfrechet} for responses in $\setC$.

\begin{corollary}[Local constant Fr\'echet regression under an upper curvature bound]\label{cor:pos:localfrechet}
    In the setting of \cref{ssec:pos:prelim}, let $(\mc X, \rho)$ be a metric space.
    Let $\mu_z\in\mc P(\setC)$, $z\in\mc X$, satisfy the transport smoothness condition \eqref{eq:two:localreg:smooth} of \cref{cor:two:localreg} for some $L>0$ and $\beta\in(0,1]$.
    Set $v := \sup_{z\in\mc X}\VmofP{\mu_z}$.
    Let $Y_i\sim\mu_{x_i}$ be independent, and let the kernel and the estimator be as in \cref{ssec:localfrechet}.
    Let $x\in\mc X$, and assume that the design satisfies the condition \eqref{eq:localfrechet:design} of \cref{cor:two:localreg} for some $D \in \Rpp$ and $\const{dns}\in [1,\infty)$.
    Then
    \begin{equation}\label{eq:pos:localfrechet}
    \br{\EOf{\ol{\samfm_h\!(x)\,}{m_2(\mu_x)}^2}}^{1/2}
    \leq C_\Delta^2\br{L h^\beta+\sqrt{\frac{v\const{Kmi}\const{Kma}\const{dns}}{nh^D}}}
    \eqfs
    \end{equation}
\end{corollary}
The hypothesis $v<\infty$ of \cref{cor:two:localreg} is automatic here, as $v\leq\diam(\setC)^2/2$.
The right-hand side of \eqref{eq:pos:localfrechet} is $C_\Delta^2$ times that of \cref{cor:two:localreg}. Hence, for $v>0$, it is minimized over $h>0$ by the bandwidth $h_*$ of \eqref{eq:localfrechet:bandwidth}, and if the design condition holds at $h=h_*$, then \eqref{eq:localfrechet:optimal} holds with its right-hand side multiplied by $C_\Delta^4$. The nonparametric rate $n^{-2\beta/(2\beta+D)}$ is unchanged; only the constant carries the curvature factor.

Compared with \cite[Theorem 1]{Schoetz2022Regression-EJS} on the bounded response space $(\setC,d)$, beyond the points made in \cref{ssec:localfrechet}, the variance inequality assumed there (with constant $C_{\mathrm{Vlo}}$) holds automatically with $C_{\mathrm{Vlo}}=C_\Delta$ by \eqref{eq:2fmcatk:vi}, and the stochastic term retains the noise variance $v$ in place of the factor $\diam(\setC)^2$ appearing there.

\subsection{Dirichlet-Process Fr\'echet Means}\label{ssec:pos:dirichlet}

Let $a>0$ and $\mu\in\mc P(\setC)$, and let $P\sim\operatorname{DP}(a,\mu)$ and the observations $X_1,\ldots,X_n$ be as in \cref{ssec:dirichlet}.
The atoms of the stick-breaking representation \eqref{eq:dirichlet:stickbreaking} are iid with law $\mu$, so $P\in\mc P(\setC)$ almost surely, and the posterior base measure $\tilde\mu_n$ of \eqref{eq:dirichlet:posteriorbase} lies in $\mc P(\setC)$ as well.
Replacing \cref{thm:two:infty:independent} by \cref{thm:pos:infty:independent} and \eqref{eq:cat0:2fm:CP} by \eqref{eq:2fmcatk:CP}, the arguments of \cref{ssec:dirichlet} go through verbatim: the right-hand sides of the risk inequalities in \eqref{eq:dirichlet:risk}, \eqref{eq:dirichlet:posterior}, and \eqref{eq:dirichlet:bayesrisk} are multiplied by $C_\Delta^4$, those of \eqref{eq:dirichlet:shrinkage} by $C_\Delta^2$, and the radius of the credible ball by $C_\Delta^2$.
In particular,
\begin{equation*}
 \begin{aligned}
 \EOf{\ol{m_2(P)}{m_2(\tilde\mu_n)}^2\,\middle|\, X_1,\ldots,X_n}
 &\leq C_\Delta^4\,\frac{\VmOfP{\tilde\mu_n}}{a+n+1}
 \eqcm\\
 \EOf{\ol{m_2(P)}{m_2(\tilde\mu_n)}^2}
 &\leq C_\Delta^4\,\frac{a}{(a+1)(a+n)}\,\VmOfP{\mu}
 \eqfs
 \end{aligned}
\end{equation*}
The computation of $\Eof{\VmofP{\tilde\mu_n}}$ behind the second bound uses only the prior-predictive law of the observations and the stick-breaking representation, and is therefore valid in any metric space.

\section{Power Fr\'echet Means in Hadamard Spaces}\label{sec:main}

We return to Hadamard spaces and consider power Fr\'echet means of order $\alpha\in(1,2]$.
Throughout, put $\phi:=(2-\alpha)/(\alpha-1)$ and use the convention $s^0:=1$ also at $s=0$.
As in \cref{sec:two,sec:pos}, we first state the geometric tools, then the weighted-sample bound, and finally the applications to local constant regression and Dirichlet-process means.
At $\alpha=2$, the bounds reduce to those of \cref{sec:two} up to constants.

\subsection{Preliminaries}\label{ssec:power:prelim}

Let $(\mc Q,d)$ be a Hadamard space and $\alpha\in(1,2]$.
By \cref{prp:prelim:qi} with $\kappa=0$, we have, for all $y,z,q,p\in\mc Q$, the \emph{quadruple inequality}
\begin{equation}\tag{QI$_\alpha$}\label{eq:power:qi}
    \abs{\ol yq^\alpha-\ol yp^\alpha-\ol zq^\alpha+\ol zp^\alpha}
    \leq
    2^{2-\alpha}\alpha\,\ol qp\,\ol yz\,\diam\cb{y,z,q,p}^{\alpha-2}
    \eqfs
\end{equation}
Let $\mu\in\mc P^{\alpha-1}(\mc Q)$ and $Y\sim\mu$.
The $\alpha$-Fr\'echet mean of $\mu$,
\begin{equation}\label{eq:power:fmdef}
    m := m_\alpha(\mu) = m_\alpha[Y] = \argmin_{q\in\mc Q}\EOf{\ol Yq^\alpha-\ol Yo^\alpha}
    \eqcm
\end{equation}
exists by \cite[Proposition 5.2(ii)]{Schoetz2025VarianceInequalities-EJP} and is unique by \cite[Corollary 5.8]{Schoetz2025VarianceInequalities-EJP}.
For $\mu,\nu\in\mc P^{\alpha-1}(\mc Q)$, define the \emph{location-aware $(\alpha-1)$-transport cost}
\begin{equation}\label{eq:power:Wtilde}
    \widetilde W_{\alpha-1}(\mu,\nu)
    :=\inf_{\pi\in\Pi(\mu,\nu)}\int
    \ol yz\,\diam\cb{y,z,m_\alpha(\mu),m_\alpha(\nu)}^{\alpha-2}\,\dl\pi(y,z)
    \eqcm
\end{equation}
where the integrand is read as $0$ whenever $\ol yz=0$.
We have $\widetilde W_{\alpha-1}\leq W_{\alpha-1}$ with $\widetilde W_1=W_1$.
The cost $\widetilde W_{\alpha-1}$ is symmetric and vanishes exactly when $\mu=\nu$, but for $\alpha<2$ it is not a metric since the triangle inequality may fail (\cref{lmm:power:Wtilde} in the appendix).

For $\mu\in\mc P^{\alpha-1}(\mc Q)$ and $t\in\Rp$, define the \emph{regularized inverse moment}
\begin{equation}\label{eq:power:Lambda}
    \LamOf{\mu}{t}:=\int\br{\ol y{m_\alpha(\mu)}+t}^{\alpha-2}\,\dl\mu(y)\in(0,\infty]
    \eqfs
\end{equation}
It is nonincreasing in $t$ and finite for every $t>0$, whereas $\Lamof{\mu}{0}$ may be infinite; the products $t\Lamof{\mu}{t}$ and $t^2\Lamof{\mu}{t}$ are read as zero at $t=0$.
At $\alpha=2$, $\Lamof{\mu}{t}=1$ for all $t$.
For $\alpha<2$, $t^2\Lamof{\mu}{t}\leq t^\alpha$ for all $t$, $t^{2-\alpha}\Lamof{\mu}{t}\to1$ as $t\to\infty$, and $\Lamof{\mu}{t}\nearrow\Lamof{\mu}{0}$ as $t\searrow0$.

The mean $m$ satisfies the \emph{variance inequality} \cite[Theorem 5.4]{Schoetz2025VarianceInequalities-EJP} (see also \cref{lmm:power:vi} in the appendix): for all $q\in\mc Q$,
\begin{equation}\tag{VI$_\alpha$}\label{eq:power:vi}
    \frac{\alpha(\alpha-1)}{2}\,\ol mq^2\,\LamOf{\mu}{\ol mq}
    \leq
    \EOf{\ol Yq^\alpha-\ol Ym^\alpha}
    \eqfs
\end{equation}
Moreover, any $\mu,\nu\in\mc P^{\alpha-1}(\mc Q)$ satisfy the \emph{contraction property} (\cref{lmm:power:cp} in the appendix)
\begin{equation}\tag{CP$_\alpha$}\label{eq:power:cp}
    \ol{m_\alpha(\mu)}{m_\alpha(\nu)}
    \br{\LamOf{\mu}{\ol{m_\alpha(\mu)}{m_\alpha(\nu)}}+\LamOf{\nu}{\ol{m_\alpha(\mu)}{m_\alpha(\nu)}}}
    \leq
    \frac{2^{3-\alpha}}{\alpha-1}\,\widetilde W_{\alpha-1}(\mu,\nu)
    \eqfs
\end{equation}
At $\alpha=2$, \eqref{eq:power:qi}, \eqref{eq:power:vi} and \eqref{eq:power:cp} reduce to \eqref{eq:2fmcat0:qi}, \eqref{eq:cat0:2fm:vi} and \eqref{eq:cat0:2fm:CP}, respectively.
The constant $2^{3-\alpha}/(\alpha-1)$ in \CPa cannot be improved (\cref{rem:power:cp:sharp} in the appendix).

As in the previous sections, the main tools are \eqref{eq:power:qi}, \eqref{eq:power:vi}, and \eqref{eq:power:cp}, but their more complicated forms make the proof substantially harder.
In particular, for $\alpha<2$, the moduli in \VIa and \CPa depend on the distribution(s).
Applied to an empirical distribution, these moduli are random and depend on the observations being resampled in the stability argument.

\subsection{The Weighted-Sample Bound}\label{ssec:main:bound}

Let $(w_i)_{i\in\N} \subset [0,1]$ such that $\sum_{i=1}^\infty w_i = 1$.
Let $(Y_i)_{i\in\N}$ be a sequence of independent random variables with distributions $\mu_1, \mu_2, \dots \in \mc P(\mc Q)$.
Define
\begin{equation*}
	\popdis := \sum_{i=1}^\infty w_i\, \mu_i
	\eqcm\qquad
	\samdis := \sum_{i=1}^\infty w_i\, \delta_{Y_i}
	\eqcm
\end{equation*}
assume $\popdis\in\mc P^{\alpha-1}(\mc Q)$, and set $\popfm:=m_\alpha(\popdis)$ and $\samfm:=m_\alpha(\samdis)$. By Tonelli's theorem, $\samdis\in\mc P^{\alpha-1}(\mc Q)$ almost surely, so both means exist.
For $s \in \Rpp$, write
\begin{equation}\label{eq:power:sigma}
	\sigma_s := \sigma_\alpha^s(\popdis) = \sum_{i=1}^\infty w_i\, \EOf{\ol{Y_i}{\popfm}^{s}}
\end{equation}
as shorthand for the $s$-th moment of $\popdis$ about $\popfm$, and define the \emph{weighted moments}
\begin{equation}\label{eq:power:sigmaw}
	\Sigma_s := \sum_{i=1}^\infty w_i^{\frac{s}{\alpha-1}}\, \VmpOf{s}{Y_i}
	\eqfs
\end{equation}
The metric $s$-variance $\Vmpof{s}{Y_i}$ is defined in \eqref{eq:variances}.
Summands corresponding to indices $i$ with $w_i = 0$ are always read as $0$.

Only two weighted moments appear in the main result,
\begin{equation*}
	\Sigma_{2\alpha-2} = \sum_{i=1}^\infty w_i^2\, \VmpOf{2\alpha-2}{Y_i}
	\eqcm\qquad
	\Sigma_{\alpha} = \sum_{i=1}^\infty w_i^{2+\phi}\, \VmpOf{\alpha}{Y_i}
	\eqcm
\end{equation*}
where $\phi=(2-\alpha)/(\alpha-1)$.
Other weighted moments are used only in the proofs.
\Needspace{10\baselineskip}
\begin{theorem}[Weighted-sample bound for the sample $\alpha$-Fr\'echet mean]\label{thm:power:main}
	In the setting above, there is a constant $c_\alpha$ depending only on $\alpha$ such that
	\begin{equation}\label{eq:power:rate:sharp}
		\EOf{\ol{\popfm}{\samfm}^\alpha}^{\frac1\alpha}
		\leq
		c_\alpha\br{
			\sigma_{\alpha-1}^{\phi}\, \Sigma_{2\alpha-2}^{\frac12}
			+
			\sigma_{\alpha-1}^{\frac\phi2}\, \Sigma_{\alpha}^{\frac12}
			+
			\Sigma_\alpha^{\frac1\alpha}
		}
		\eqfs
	\end{equation}
\end{theorem}

The theorem is proved in \cref{app:power} with a constant $c_\alpha$ that is bounded on $[1+\epsilon,2]$ for every $\epsilon\in(0,1)$; conversely, for $1<\alpha<2$, every admissible constant satisfies $c_\alpha\geq\max\{2^\phi,2^{2-\alpha}/(\alpha-1)\}$ and hence diverges as $\alpha\searrow1$ (\cref{rem:power:calpha} in the appendix).
The assumption $\popdis\in\mc P^{\alpha-1}(\mc Q)$ ensures existence of the means.
The right-hand side of \eqref{eq:power:rate:sharp} is finite if and only if $\Sigma_\alpha<\infty$; in particular, $\popdis\in\mc P^\alpha(\mc Q)$ suffices; for finitely many positive weights it is also necessary, but for infinitely many it is not (\cref{lmm:power:moments}\,(i) in the appendix).
The component distributions need not have a common mean: $\sigma_{\alpha-1}$ measures distances to the common target $\popfm$, whereas $\Sigma_{2\alpha-2}$ and $\Sigma_\alpha$ measure fluctuations within each component.
The exponent $\phi = (2-\alpha)/(\alpha-1)$ of the common-center factor tends to zero as $\alpha\nearrow2$; at $\alpha=2$ that factor disappears.
For each fixed $1<\alpha<2$, however, $\sigma_{\alpha-1}$ cannot be replaced by $\Sigma_{\alpha-1}$ (\cref{lmm:power:moments}\,(ii) in the appendix).
Moreover, none of the three summands is uniformly bounded by an $\alpha$-dependent constant times the sum of the other two (\cref{lmm:power:moments}\,(iii) in the appendix).
At $\alpha=2$ all three summands coincide.

For uniform weights this gives the parametric rate $n^{-1/2}$ in $L^\alpha$ under an $\alpha$-th moment condition.
\begin{corollary}[Parametric $L^\alpha$ rate]\label{cor:power:rate}
	Let $\alpha\in(1,2]$ and $n\in\N$.
	Let $Y_1,\dots,Y_n$ be independent random variables with distributions $\mu_1,\dots,\mu_n\in\mc P^{\alpha-1}(\mc Q)$, and set $\popdis := \frac1n\sum_{i=1}^n\mu_i$, $\popfm := m_\alpha(\popdis)$, $\samndis := \frac1n\sum_{i=1}^n \delta_{Y_i}$, and $\samnfm := m_\alpha(\samndis)$.
	For $s\in\Rpp$, let
	\begin{equation*}
		\sigma_s := \sigma_\alpha^s(\popdis)
        \qquad\text{and}\qquad
        \bar V_s := \frac1n\sum_{i=1}^n \VmpOf{s}{Y_i}
		\eqfs
	\end{equation*}
	Then, with a constant $c_\alpha$ depending only on $\alpha$, and with $\phi = (2-\alpha)/(\alpha-1)$, 
	\begin{equation}\label{eq:power:rate:iid}
		\EOf{\ol{\popfm}{\samnfm}^\alpha}^{\frac1\alpha}
		\leq
		c_\alpha\br{
			\sigma_{\alpha-1}^{\phi}\, \bar V_{2\alpha-2}^{\frac12}\, n^{-\frac12}
			+
			\sigma_{\alpha-1}^{\frac\phi2}\, \bar V_{\alpha}^{\frac12}\, n^{-\frac1{2(\alpha-1)}}
			+
			\bar V_\alpha^{\frac1\alpha}\, n^{-\frac1{\alpha(\alpha-1)}}
		}
		\eqfs
	\end{equation}
\end{corollary}

The right-hand side of \eqref{eq:power:rate:iid} is finite if and only if $\bar V_\alpha<\infty$, equivalently if every $\mu_i$ has a finite $\alpha$-moment.
Since $\frac1{2(\alpha-1)}\geq\frac12$ and $\frac1{\alpha(\alpha-1)}\geq\frac12$ for $\alpha\in(1,2]$, the right-hand side is $\landauO(n^{-1/2})$ as $n\to\infty$ whenever $\sigma_{\alpha-1}, \bar V_{2\alpha-2}, \bar V_\alpha$ stay bounded; in particular this holds for iid data with law $\mu\in\mc P^\alpha(\mc Q)$.

An $\alpha$-th moment thus yields the order $n^{-\frac12}$ for the $L^\alpha$ risk of the sample $\alpha$-Fr\'echet mean with the population $\alpha$-Fr\'echet mean as target.
For ordinary mean estimation (2-Fr\'echet mean), a bounded $\alpha$-moment alone, $1<\alpha<2$, does not give the parametric rate uniformly over the class of such laws \cite[Theorem 3.1]{DevroyeEtAl2016SubGaussian-AOS}. 

For iid observations, \cite[Corollary 3.5]{Schoetz2026Transformed-arxiv} gives the same three terms with centered instead of pairwise moments; optimality of this rate is discussed in \cite[Section 5.4]{Schoetz2026Transformed-arxiv}.
Mass near the population mean can lead to faster rates for $1<\alpha<2$; for iid observations this is investigated in \cite[Sections 5.1 and 5.2]{Schoetz2026Transformed-arxiv}, and we do not pursue such refinements here.

\Needspace{8\baselineskip}
\begin{remark}[Weight schemes]\label{rem:power:weights}
The schemes of \cref{rem:weights} also apply to power Fr\'echet means.
Consider iid observations with law $\mu$.
Then the two weighted moments in \cref{thm:power:main} are
\begin{equation*}
 \Sigma_{2\alpha-2}=\VmpOfP{2\alpha-2}{\mu}\normOf{w}_2^2,
 \qquad
 \Sigma_\alpha=\VmpOfP{\alpha}{\mu}\normOf{w}_{2+\phi}^{2+\phi}.
\end{equation*}
The first term of the bound \eqref{eq:power:rate:sharp} is proportional to $\normof{w}_2$, while the other two involve the higher weight norm $\normof{w}_{2+\phi}$.
For uniform weights, $\normof{w}_2^2=n^{-1}$ and $\normof{w}_{2+\phi}^{2+\phi}=n^{-1-\phi}$, giving \cref{cor:power:rate}.
For geometric weights $w_i=(1-\rho)\rho^{i-1}$ with $0<\rho<1$,
\[
 \normOf{w}_2^2=\frac{1-\rho}{1+\rho},\qquad
 \normOf{w}_{2+\phi}^{2+\phi}=\frac{(1-\rho)^{2+\phi}}{1-\rho^{2+\phi}}
 \in\ab{\frac{(1-\rho)^{1+\phi}}{2+\phi},(1-\rho)^{1+\phi}}.
\]
For non-identically distributed observations, each component retains its own pairwise moment in $\Sigma_s$.
Set $w_{\max} := \sup_{i\in\N}w_i$.
The bounds $\Sigma_{2\alpha-2}\leq2w_{\max}\sigma_{2\alpha-2}$ and $\Sigma_\alpha\leq2w_{\max}^{1+\phi}\sigma_\alpha$, proved at the start of \cref{app:power}, show $L^\alpha$ consistency when $w_{\max}\to0$ and the relevant mixture moments stay bounded, as in \cref{rem:weights}\,(iii).
\end{remark}

\subsection{Local Constant Fr\'echet Regression}\label{ssec:power:localfrechet}
For $1<\alpha<2$, the weighted-sample bound yields nonparametric regression under fewer than two moments.
In particular, it permits conditional laws with infinite variance while retaining the usual rate for a H\"older regression function.
The target is the conditional $\alpha$-Fr\'echet mean, and its stability requires a smoothness condition adapted to \CPa: the transport cost between two conditional laws is bounded by the modulus of \CPa at the desired distance scale.
Kernel estimators with robust losses such as $d^p$, $p\geq1$, and Huber-type losses for metric-space-valued responses were introduced in \cite{Hein2009RobustRegression-NIPS}, with almost-sure consistency and, under a Lipschitz condition on the density, rates for the pointwise excess risk, for compact manifold-valued predictors and responses; \cref{cor:power:localfrechet} bounds the distance to the target without compactness or entropy conditions and allows infinite variance.

\Needspace{20\baselineskip}
\begin{corollary}[Local constant power Fr\'echet regression]\label{cor:power:localfrechet}
 Let $(\mc X,\rho)$ be a metric space, let $x\in\mc X$, and let $\mu_z\in\mc P^\alpha(\mc Q)$, $z\in\mc X$.
 Suppose that, for some $L>0$ and $\beta\in(0,1]$, all $z\in\mc X\setminus\{x\}$ satisfy, with $b:=L\rho(z,x)^\beta$,
 \begin{equation}\label{eq:power:localfrechet:smooth}
 \widetilde W_{\alpha-1}(\mu_z,\mu_x)
 \leq b\,\LamOf{\mu_x}{b}
 \eqfs
 \end{equation}
 Set $u:=\sup_{z\in\mc X}\VmpofP{\alpha}{\mu_z}$ and assume $u<\infty$.
 Let $Y_i\sim\mu_{x_i}$ be independent, let the kernel and weights be as in \cref{ssec:localfrechet}, and set
 $\samfm_h\!(x):=m_\alpha(\sum_iw_{i,h}(x)\delta_{Y_i})$.
 Assume that the design satisfies the condition \eqref{eq:localfrechet:design} of \cref{cor:two:localreg} at $x$ for some $D>0$ and $\const{dns}\in[1,\infty)$.
 Then, with $c_\alpha$ depending only on $\alpha$,
 \begin{equation}\label{eq:power:localfrechet}
 \EOf{\ol{\samfm_h\!(x)}{m_\alpha(\mu_x)}^\alpha}^{1/\alpha}
 \leq c_\alpha\br{Lh^\beta+u^{1/\alpha}
       \sqrt{\frac{\const{Kmi}\const{Kma}\const{dns}}{nh^D}}}.
 \end{equation}
\end{corollary}
For $u>0$, the right-hand side is minimized over $h>0$ by $h_*$ of \eqref{eq:localfrechet:bandwidth} with $v$ replaced by $u^{2/\alpha}$.
If the design condition holds at $h=h_*$, then
\begin{equation}\label{eq:power:localfrechet:optimal}
 \EOf{\ol{\samfm_{h_*}\!(x)}{m_\alpha(\mu_x)}^\alpha}^{1/\alpha}
 \leq c_\alpha\sqrt{\const{rate}}\,
 L^{\frac{D}{2\beta+D}}
 \br{\frac{u^{2/\alpha}\const{Kmi}\const{Kma}\const{dns}}n}^{\frac{\beta}{2\beta+D}},
\end{equation}
with $\const{rate}$ as in \eqref{eq:localfrechet:optimal}.
Thus the $L^\alpha$ risk has the classical nonparametric order $n^{-\beta/(2\beta+D)}$, using only a conditional $\alpha$-moment.
\Cref{rem:power:smoothness} discusses the smoothness condition.

\begin{remark}[Reading the smoothness condition]\label{rem:power:smoothness}
At $\alpha=2$ we have $\Lamof{\mu}{t}=1$ and $\widetilde W_1=W_1$, so \eqref{eq:power:localfrechet:smooth} is the smoothness condition \eqref{eq:two:localreg:smooth} of \cref{cor:two:localreg,cor:pos:localfrechet} with $z\pr=x$.

\emph{The right-hand side.} Condition \eqref{eq:power:localfrechet:smooth} bounds the transport cost between $\mu_z$ and the target law $\mu_x$ by the modulus $b\Lamof{\mu_x}{b}$ of \CPa at the desired distance scale $b$, and that modulus has two regimes.
It is at most $b^{\alpha-1}$ for every law, and at least $2^{-\phi}b^{\alpha-1}$ as soon as $b^{\alpha-1}\geq\sigma_\alpha^{\alpha-1}(\mu_x)$, by Jensen's inequality (see \cref{app:power:moment}); thus, for $b$ beyond the spread of $\mu_x$, it is enough that $W_{\alpha-1}(\mu_z,\mu_x)\leq2^{-\phi}b^{\alpha-1}$, a H\"older condition in the plain cost, since $W_{\alpha-1}$ dominates $\widetilde W_{\alpha-1}$.
For small $b$, and assuming $\Lamof{\mu_x}{0}<\infty$ for simplicity --- a criterion is given in \cref{app:power:moment} --- the modulus is instead of order $b$.
The plain cost $W_{\alpha-1}$ can satisfy this linear bound for smooth families of densities, as in \cref{ex:power:hyperbolic}, but it excludes some elementary families of moving atoms that the location-aware cost admits.

\emph{The left-hand side.} The diameter factor in \eqref{eq:power:Wtilde} reduces the price of moving mass that sits far from the two means, as quantified in \cref{rem:power:diffuse} in the appendix.
For $1<\alpha<2$, replacing $\widetilde W_{\alpha-1}$ by $W_{\alpha-1}$ in \eqref{eq:power:localfrechet:smooth} gives a stronger sufficient condition that can fail even for translations of a two-point law: in $\R$, let $\mu_x:=\frac12(\delta_{x-1}+\delta_{x+1})$.
For $x\neq x\pr$ and $t:=\absof{x-x\pr}$, the coupling of corresponding atoms gives
\begin{equation*}
 \widetilde W_{\alpha-1}(\mu_x,\mu_{x\pr})
 \leq t(1+t)^{\alpha-2}=t\LamOf{\mu_x}{t},
 \qquad W_{\alpha-1}(\mu_x,\mu_{x\pr})=t^{\alpha-1}\quad(t\leq1/2),
\end{equation*}
with equality in the first bound for $t\leq1/2$ (\cref{app:power:euclidean}).
Thus \eqref{eq:power:localfrechet:smooth} holds at every $x$ with $L=1$ and $\beta=1$, whereas with $W_{\alpha-1}$ in place of $\widetilde W_{\alpha-1}$ it fails for every finite $L$ at $\beta=1$, since $t^{\alpha-1}/(Lt\Lamof{\mu_x}{Lt})\to\infty$ as $t\searrow0$.

\emph{A moment-only criterion.} Replacing the inverse moment by its Jensen lower bound turns \eqref{eq:power:localfrechet:smooth} into a condition on the centered moments alone; \cref{lmm:power:sufficient}\,(i) in the appendix gives the Jensen sufficient form with the same $L$, while part\,(ii) gives a two-regime form with $L$ replaced by $2^{\phi/(\alpha-1)}L$.
The latter is the two-regime description above with the slope $\Lamof{\mu_x}{0}$ replaced by the smaller $\max\{\sigma_\alpha^{\alpha-1}(\mu_z),\sigma_\alpha^{\alpha-1}(\mu_x)\}^{-\phi}$.
It is convenient for verification, as in \cref{ex:power:hyperbolic}, but it can fail for additive noise with an atom at its mean (\cref{app:power:euclidean}).
\end{remark}

The next four items illustrate \cref{cor:power:localfrechet}: a minimal check in the noiseless case (\cref{ex:power:noiseless}), additive infinite-variance noise in a Hilbert space (\cref{cor:power:additive}), its specialization to distributions on the line as responses (\cref{ex:power:wasserstein}), and a non-flat example (\cref{ex:power:hyperbolic}).

\begin{example}[Noiseless responses]\label{ex:power:noiseless}
 If $\mu_x=\delta_{\gamma(x)}$ for a map $\gamma\colon\mc X\to\mc Q$, then $\Lamof{\mu_x}{b}=b^{\alpha-2}$ and $\widetilde W_{\alpha-1}(\mu_x,\mu_{x\pr})=\ol{\gamma(x)}{\gamma(x\pr)}^{\alpha-1}$, so \eqref{eq:power:localfrechet:smooth} is equivalent to $\ol{\gamma(x)}{\gamma(x\pr)}\leq L\rho(x,x\pr)^\beta$ for all $x\pr\in\mc X$.
 Every $\beta$-H\"older response map is included; for real predictors, a unit-speed geodesic has $L=1$, $\beta=1$.
 Here $u=0$, the stochastic term vanishes, and \eqref{eq:power:localfrechet} is the deterministic approximation bound $c_\alpha Lh^\beta$; no noise--bias bandwidth balance is required.
\end{example}

\begin{corollary}[Additive noise in Hilbert spaces]\label{cor:power:additive}
 Let $\mc H$ be a Hilbert space, let $\mc Q\subseteq \mc H$ be nonempty, closed, and convex, with the induced metric, and let $\alpha\in(1,2]$.
 Let $\gamma\colon\mc X\to\mc Q$ be $\beta$-H\"older with constant $L_\gamma$, and let $\varepsilon$ be an $\mc H$-valued random variable with $\Eof{\normof\varepsilon^\alpha}<\infty$, with $\alpha$-Fr\'echet mean $m_\alpha[\varepsilon]=0$ in $\mc H$, and with $\gamma(x)+\varepsilon\in\mc Q$ almost surely for every $x\in\mc X$.
 For $\gamma(x)+\varepsilon\sim\mu_x$, we have $m_\alpha(\mu_x)=\gamma(x)$, $\VmpofP{\alpha}{\mu_x}\leq2^{\alpha-1}\Eof{\normof\varepsilon^\alpha}$, and \eqref{eq:power:localfrechet:smooth} holds at every $x\in\mc X$ with $L=2^{\phi/2}L_\gamma$.
 Consequently, under the kernel and design assumptions of \cref{cor:power:localfrechet},
 \begin{equation}\label{eq:power:additive}
  \EOf{\ol{\samfm_h\!(x)}{\gamma(x)}^\alpha}^{1/\alpha}
  \leq c_\alpha\br{L_\gamma h^\beta+\EOf{\normOf{\varepsilon}^\alpha}^{1/\alpha}
  \sqrt{\frac{\const{Kmi}\const{Kma}\const{dns}}{nh^D}}}
  \eqfs
 \end{equation}
\end{corollary}
For $1<\alpha<2$ the noise may have infinite variance.
In particular, under the design condition with $D=1$, Lipschitz curves attain $L^\alpha$ risk $\landauO(n^{-1/3})$ as $n\to\infty$.
The proof is given in \cref{app:power:euclidean}.

\begin{example}[Distributions on the line as responses]\label{ex:power:wasserstein}
 Let $\mc Q = \mc P^2(\R)$ be the set of probability measures on $\R$ with finite second moment, equipped with the quadratic Wasserstein distance
 \begin{equation*}
  \overline W_2(\mu, \nu) := \br{\inf_{\pi\in\Pi(\mu,\nu)}\int\abs{y-z}^2\,\dl\pi(y,z)}^{\frac12}
  \eqfs
 \end{equation*}
 The quantile map $\nu\mapsto F_\nu^{-1}$ is an isometry of $\mc Q$ onto the closed convex cone of nondecreasing functions in $L^2(0,1)$ \cite{PanaretosZemel2020Wasserstein-Springer}, so \cref{cor:power:additive} applies with $\mc H=L^2(0,1)$.
 For instance, let $\gamma\colon\mc X\to\mc P^2(\R)$ be $\beta$-H\"older and let $\xi$ be a real random variable with $m_\alpha[\xi]=0$ and $\Eof{\absof\xi^\alpha}<\infty$.
 As response in $\mc P^2(\R)$, take $Y_x:=\gamma(x)\ast\delta_\xi$, where $\ast$ denotes the convolution of measures and $\delta_\xi$ is the Dirac measure at $\xi$.
 Convolution with a Dirac measure is a translation, $Y_x(B)=\gamma(x)(B-\xi)$ for Borel sets $B\subseteq\R$, so $Y_x$ is the law $\gamma(x)$ shifted by the random amount $\xi$; equivalently, $Y_x$ is the conditional law of $Z+\xi$ given $\xi$, where $Z\sim\gamma(x)$ is independent of $\xi$.
 Shifting a measure by $\xi$ shifts its quantile function by $\xi$, i.e., $F_{Y_x}^{-1}=F_{\gamma(x)}^{-1}+\xi$ on $(0,1)$.
 In $L^2(0,1)$, the response therefore has the additive form $Y_x=\gamma(x)+\varepsilon$ with the constant function $\varepsilon:=\xi\ind_{(0,1)}$; the sum $F_{\gamma(x)}^{-1}+\xi$ is nondecreasing and square integrable, so it lies in the cone of $L^2(0,1)$ identified with $\mc P^2(\R)$.
 The $L^2(0,1)$ norm of a constant function is its absolute value, so $\normof\varepsilon=\absof\xi$ and $\Eof{\normof\varepsilon^\alpha}=\Eof{\absof\xi^\alpha}<\infty$.
 It remains to check $m_\alpha[\varepsilon]=0$.
 Write $f\in L^2(0,1)$ as $f=c\ind_{(0,1)}+f^\perp$ with $c\in\R$ and $f^\perp$ orthogonal to the constant functions.
 Then $\normof{\varepsilon-f}^2=(\xi-c)^2+\normof{f^\perp}^2$, so that $\Eof{\normof{\varepsilon-f}^\alpha}\geq\Eof{\absof{\xi-c}^\alpha}$ with equality if and only if $f^\perp=0$.
 Minimizing over $c$ gives $m_\alpha[\varepsilon]=m_\alpha[\xi]\ind_{(0,1)}=0$.
 Hence, \cref{cor:power:additive} applies with $\Eof{\normof\varepsilon^\alpha}=\Eof{\absof\xi^\alpha}$.
 More generally, $\varepsilon$ may be any random element of $L^2(0,1)$ with $\Eof{\normof\varepsilon^\alpha}<\infty$ and $m_\alpha[\varepsilon]=0$ (for instance, symmetric) that keeps $F_{\gamma(x)}^{-1}+\varepsilon$ nondecreasing, e.g., a $\lambda$-Lipschitz perturbation when every $F_{\gamma(x)}^{-1}$ increases with slope at least $\lambda$.
 For $1<\alpha<2$ and $\Eof{\xi^2}=\infty$, the responses have infinite second moment in $\mc Q$, and \eqref{eq:power:additive} still gives the rate $n^{-\beta/(2\beta+D)}$; recall from \cref{ssec:localfrechet} that the entropy condition of \cite{PetersenMueller2019Regression-AOS} fails in this space \cite{ImEtAl2026Entropy-arxiv}, and note that the projection route of \cite{PetersenMueller2026Correction-AOS} is tied to $\alpha=2$.
\end{example}

\begin{example}[Hyperbolic location--scale families]\label{ex:power:hyperbolic}
 Let $\mc Q=\mathbb H^k$, $k\geq2$, be the hyperbolic space of dimension $k$ and constant curvature $-1$, a Hadamard manifold, in its hyperboloid model \cite[Chapter~I.2]{BridsonHaefliger1999NPC-Springer}: with the Minkowski form $\mink pq:=-p_0q_0+\sum_{j=1}^kp_jq_j$ on $\R^{k+1}$,
 \begin{equation*}
  \mathbb H^k:=\cb{p\in\R^{k+1}\colon\mink pp=-1,\ p_0>0}
  \eqcm\qquad
  d(p,q):=\operatorname{arcosh}\brOf{-\mink pq}
  \eqfs
 \end{equation*}
 The tangent space at $p\in\mathbb H^k$ is $T_p\mathbb H^k:=\{v\in\R^{k+1}\colon\mink pv=0\}$.
 The Minkowski form is positive definite on $T_p\mathbb H^k$, so $(T_p\mathbb H^k,\mink\cdot\cdot)$ is a $k$-dimensional Euclidean space with norm $\normof v:=\mink vv^{1/2}$.
 Hence, a linear isometry $I\colon\R^k\to T_p\mathbb H^k$ exists for every $p$; it amounts to choosing an orthonormal basis of $T_p\mathbb H^k$.
 At $o:=(1,0,\dots,0)$, for instance, $T_o\mathbb H^k=\{0\}\times\R^k$ and $I(z):=(0,z)$ is a linear isometry.
 The exponential map $\exp_p\colon T_p\mathbb H^k\to\mathbb H^k$ sends $v\neq0$ to the point at distance $\normof v$ from $p$ on the geodesic ray from $p$ in direction $v$, and $0$ to $p$; explicitly,
 \begin{equation}\label{eq:power:hyperbolic:exp}
  \exp_p(v)=\cosh\brOf{\normOf v}\,p+\frac{\sinh\brOf{\normOf v}}{\normOf v}\,v
  \eqcm
 \end{equation}
 with the quotient read as $1$ at $v=0$.
 It is a diffeomorphism onto $\mathbb H^k$, and $d(\exp_p(v),p)=\normof v$ because $\mink p{\exp_p(v)}=-\cosh(\normof v)$.
 Let $\alpha\in(1,2)$ and choose $\theta\in(\alpha,2)$.
 Let $Z$ have the spherically symmetric Student distribution in $\R^k$ with $\theta$ degrees of freedom.
 Let $\gamma\colon[0,1]\to\mathbb H^k$ and $\zeta\colon[0,1]\to[\zeta_-,\zeta_+]$ be arbitrary $\beta$-H\"older functions, with $\beta\in(0,1]$ and $0<\zeta_-\leq\zeta_+<\infty$; for instance, $\gamma=\exp_o\circ\gamma_0$ for a $\beta$-H\"older function $\gamma_0\colon[0,1]\to T_o\mathbb H^k$, see \cref{app:power:hyperbolic}.
 For every $x\in[0,1]$, fix a linear isometry $I_x\colon\R^k\to T_{\gamma(x)}\mathbb H^k$ and set
 \begin{equation}\label{eq:power:hyperbolic:model}
  Y_x:=\exp_{\gamma(x)}\brOf{\zeta(x)I_xZ}\sim\mu_x\eqfs
 \end{equation}
 This is the exponential-wrapped construction of \cite{NaganoEtAl2019WrappedNormal-ICML}, where its geometric change of variables is described; by spherical symmetry of $Z$, the law $\mu_x$ does not depend on the choice of $I_x$.
 Rotational symmetry gives $m_\alpha(\mu_x)=\gamma(x)$, and
 $\ol{Y_x}{\gamma(x)}=\zeta(x)\euclof Z$.
 The conditional $\alpha$-moments are uniformly finite, every conditional second moment is infinite, and \eqref{eq:power:localfrechet:smooth} holds at every $x\in[0,1]$ with this $\beta$; see \cref{app:power:hyperbolic}.
 Consequently, under the kernel and design assumptions with $D=1$, the $L^\alpha$ risk is $\landauO(n^{-\beta/(2\beta+1)})$ for $h\asymp n^{-1/(2\beta+1)}$ as $n\to\infty$.
\end{example}

The regression theorems in \cite{Schoetz2022Regression-EJS,PetersenMueller2019Regression-AOS} do not cover this heavy-tailed power Fr\'echet mean setting.
They concern the squared-distance target under stronger moment or boundedness assumptions.
Changing the metric to $d^{\alpha/2}$ encodes the power target, but loses the Hadamard structure and can violate the quadratic variance inequality required in \cite{Schoetz2022Regression-EJS}, including for \cref{ex:power:hyperbolic}.
The supporting calculations are given in \cref{app:comparison}.

\subsection{Dirichlet-Process Fr\'echet Means}\label{ssec:power:dirichlet}

The stick-breaking argument of \cref{ssec:dirichlet} also applies to power Fr\'echet means, now under a finite $\alpha$-moment.
Besides the second moment of the weights, the bound involves their moment of order $2+\phi=\alpha/(\alpha-1)$.
For the stick-breaking weights of \cref{ssec:dirichlet} and $s\in\Rpp$, define
\begin{equation}\label{eq:power:dirichlet:weights}
 H_s(a):=\EOf{\sum_{i=1}^\infty W_i^s}
 =\frac{\Gamma(s)\Gamma(a+1)}{\Gamma(a+s)},
 \qquad\text{so that}\qquad H_2(a)=\frac1{a+1},
\end{equation}
where $\Gamma$ is the gamma function.
To state the prior and posterior bounds together, for $b>0$ and $\nu\in\mc P^\alpha(\mc Q)$ put
\begin{equation}\label{eq:power:dirichlet:R}
 \begin{split}
 \mathcal R_\alpha(b,\nu):={}&
 \sigma_\alpha^{\alpha-1}(\nu)^\phi
       \br{\frac{\VmpOfP{2\alpha-2}{\nu}}{b+1}}^{1/2}\\
 &+\sigma_\alpha^{\alpha-1}(\nu)^{\phi/2}
       \br{H_{2+\phi}(b)\VmpOfP{\alpha}{\nu}}^{1/2}
   +\br{H_{2+\phi}(b)\VmpOfP{\alpha}{\nu}}^{1/\alpha}.
 \end{split}
\end{equation}
\begin{corollary}[Dirichlet-process power Fr\'echet means]\label{cor:power:dirichlet}
 Let $\alpha\in(1,2]$, $a>0$, $\mu\in\mc P^\alpha(\mc Q)$, and $P\sim\operatorname{DP}(a,\mu)$.
 Then $m_\alpha(P)$ exists almost surely, and, with constants $c_\alpha$ depending only on $\alpha$,
 \begin{equation}\label{eq:power:dirichlet:prior}
 \br{\EOf{\ol{m_\alpha(P)}{m_\alpha(\mu)}^\alpha}}^{1/\alpha}
 \leq c_\alpha\mathcal R_\alpha(a,\mu).
 \end{equation}
 If $X_1,\ldots,X_n$ are conditionally iid with law $P$, and $\tilde\mu_n$ is the posterior base measure of \cref{ssec:dirichlet}, then almost surely
 \begin{equation}\label{eq:power:dirichlet:posterior}
 \br{\EOf{\ol{m_\alpha(P)}{m_\alpha(\tilde\mu_n)}^\alpha
       \,\middle|\,X_1,\ldots,X_n}}^{1/\alpha}
 \leq c_\alpha\mathcal R_\alpha(a+n,\tilde\mu_n).
 \end{equation}
 Averaging over the observations under this model gives the simpler bound
 \begin{equation}\label{eq:power:dirichlet:bayes}
 \br{\EOf{\ol{m_\alpha(P)}{m_\alpha(\tilde\mu_n)}^\alpha}}^{1/\alpha}
 \leq c_\alpha\br{\frac{a\,(a+n+1)}{(a+1)(a+n)}\VmpOfP{\alpha}{\mu}}^{1/\alpha}\frac1{\sqrt{a+n+1}}.
 \end{equation}
 The same bound with $n=0$ holds for the prior, with $\tilde\mu_0:=\mu$.
\end{corollary}
The proof conditions on the weights and applies \cref{thm:power:main}; it is given in \cref{app:power:dirichlet}.
By Markov's inequality, the closed ball centered at $m_\alpha(\tilde\mu_n)$ with radius
$c_\alpha\eta^{-1/\alpha}\mathcal R_\alpha(a+n,\tilde\mu_n)$ has posterior probability at least $1-\eta$, for $\eta\in(0,1)$.
The expectation in \eqref{eq:power:dirichlet:bayes} is with respect to the joint law of $P$ and the observations.
At $\alpha=2$, \cref{ssec:dirichlet} gives sharper constants.

\section*{Declarations}
\paragraph{Use of AI tools.}
OpenAI Codex and Anthropic Claude were used to assist with manuscript organization and editing, literature searches and comparisons, and the development and checking of mathematical arguments.
Responsibility for the mathematical results, references, and final text rests with the author.

\appendix
\section{Tools for Positive Curvature}\label{app:pos:tools}
Here the notation of \cref{ssec:pos:prelim} applies: $\setC$ is a nonempty bounded closed convex set with circumradius $0<\varrho<\pi/(2\sqrt{\kappa})$, and $\Delta=2\varrho\sqrt\kappa\in(0,\pi)$.
Moreover, $\delta=\sqrt\kappa\,\diam(\setC)\leq\Delta$ by \cref{rem:pos:diameter}.
We use the factor of \cref{prp:prelim:qi},
\begin{equation}\label{eq:pos:const:C}
    C_x=x/\sin(x)\quad\text{for}\quad x \in (0,\pi)\qquad\text{and}\qquad C_0 = 1
    \eqcm
\end{equation}
which is increasing, so that $C_\delta\leq C_\Delta$.

\Needspace{18\baselineskip}
\textbf{Notation.}
\begin{center}
    \begin{tabularx}{\linewidth}{@{}l@{\hspace{1.5em}}>{\raggedright\arraybackslash}X@{}}
        \toprule
        $\kappa$, $\setC$, $\varrho$ & curvature bound, data set, and circumradius; \cref{ssec:pos:prelim}\\
        $\Delta=2\varrho\sqrt\kappa$, $\delta=\sqrt\kappa\,\diam(\setC)$ & twice the rescaled circumradius, and the rescaled diameter; \cref{thm:pos:hemisphere:vi}\\
        $C_x=x/\sin(x)$, $C_0=1$ & quadruple-inequality factor; \eqref{eq:pos:const:C}\\
        $A_\Delta^{\mathrm{rad}}$, $A_\Delta^{\mathrm{near}}$, $A_\delta^{\mathrm{diam}}$ & optimal variance constants; \eqref{eq:pos:vi:radius:optimal}, \eqref{eq:pos:vi:near:constant}, \eqref{eq:pos:vi:constants}\\
        $B_\Delta^{\mathrm{rad}}$, $B_\delta^{\mathrm{diam}}$ & optimal contraction constants; \eqref{eq:pos:cp:radius:optimal}, \eqref{eq:pos:cp:diameter:optimal}\\
        $E_\Delta$, $t_\Delta$, $s_\Delta$, $x_\Delta$ & radius profiles and optimizer; see \eqref{eq:pos:vi:radius:functions}--\eqref{eq:pos:vi:radius:optimal}, \eqref{eq:pos:vi:radius:ts}\\
        $g$, $H$ & cosine profiles in \cref{lmm:pos:cosine-average}\\
        $U$, $K$, $z=K/U$, $\lambda$ & reweighting quantities in the proof of \cref{thm:pos:hemisphere:vi}\\
        $\theta_\delta$, $T_\delta$ & diameter profiles; see \eqref{eq:pos:vi:constants}, \eqref{eq:pos:cp:local:diameter}\\
        $m_\eta=m_2(\eta)$, $K_\eta$ & mean and cosine-growth coefficient of $\eta$; proof of \cref{thm:pos:hemisphere:cp}\\
        \bottomrule
    \end{tabularx}
\end{center}

The quadruple inequality \eqref{eq:2fmcatk:qi} follows from \cref{prp:prelim:qi} or equivalently \cite[Theorem 1.4]{Schoetz2026Trapezoid-arxiv}.
For the constant, we use $\diam(\seg yz\cup\seg qp) \leq \diam(\setC)$, since both segments stay in $\setC$ as $\setC$ is convex.
This gives \eqref{eq:2fmcatk:qi} with constant $C_\delta$ (\cref{rem:pos:sharper}), which in turn implies \eqref{eq:2fmcatk:qi} with $C_\Delta$.

The variance inequality \eqref{eq:2fmcatk:vi} is proved in \cref{thm:pos:hemisphere:vi} and the contraction property \eqref{eq:2fmcatk:CP} in \cref{thm:pos:hemisphere:cp}.
The weighted-sample and local-regression arguments of \cref{sec:two} remain unchanged apart from their explicit geometric constants.

\subsection{The Variance Inequality}\label{ssec:pos:vi}
Strong convexity, on which most earlier inequalities of \cref{rem:pos:literature} rest, is available only for $\Delta<\pi/2$ or $\delta<\pi/2$; to obtain a quadratic variance inequality throughout $\Delta<\pi$, we instead establish quadratic growth of the objective \emph{away from its minimizer}.
We first illustrate the proof of our variance inequality on a Hilbert sphere in \cref{prp:sphere:hemisphere:vi}, and then extend the proof to general \catk{} spaces in \cref{thm:pos:hemisphere:vi}.
The supporting lemmas are collected in \cref{ssec:pos:vi:lemmas}.

\begin{proposition}[Variance inequality on a Hilbert sphere]\label{prp:sphere:hemisphere:vi}
    Let $\mc H$ be a real Hilbert space and equip its unit sphere
    $\mc Q:=\{u\in\mc H:\normof u=1\}$ with the spherical distance
    $d(u,v):=\arccos(\ip uv)$.
    Let $\setC\subseteq\mc Q$ be nonempty, closed, and geodesically convex, and suppose that
    $\setC\subset B(o,R)$ for some $o\in\mc Q$ and $0<R<\pi/2$.
    Define
    \begin{equation}\label{eq:sphere:vi:constant}
        A_{2R}^{\mathrm{rad}}:=\max_{0\leq x\leq R}
        \frac{(R+x)^2}{
        \arccos\brOf{\frac{\cos(R)\cos(R+x)}{\cos(x)}}^2
        -\arccos\brOf{\frac{\cos(R)}{\cos(x)}}^2}
        \eqfs
    \end{equation}
    Then $A_{2R}^{\mathrm{rad}} \in [1,\infty)$ and, for every $\mu\in\mc P(\setC)$ and $q\in\setC$,
    \begin{equation}\label{eq:sphere:vi:radius}
        \ol{m_2(\mu)}q^2
        \leq A_{2R}^{\mathrm{rad}}
        \int\br{\ol yq^2-\ol y{m_2(\mu)}^2}\,\dl\mu(y)
        \eqfs
    \end{equation}
\end{proposition}
In \cref{thm:pos:hemisphere:vi} we establish optimality in terms of the circumradius and show that $A_{2R}^{\mathrm{rad}}\leq C_{2R}$, with $C_x$ from \eqref{eq:pos:const:C}.

The main idea is to apply Jensen's inequality jointly to the two squared distances in the excess risk.
Writing $m=m_2(\mu)$ and $g=\arccos^2$, the spherical distance gives $\ol yq^2=g(\ip yq)$ and $\ol ym^2=g(\ip ym)$.
The map $(x,v)\mapsto g(v)-g(x)$ is not convex, so Jensen's inequality does not apply directly to their difference.
However, for $H(x)=\sin(\arccos(x))/\arccos(x)$, with $H(1)=1$, the map
\begin{equation*}
    (x,v)\mapsto H(x)\br{g(v)-g(x)}
\end{equation*}
is jointly convex, as proved in \cref{lmm:pos:cosine-conversion}\,(i).
A change of measure, with density proportional to $1/H(\ip ym)$, inserts the factor $H(\ip ym)$ into the integrand and makes this joint Jensen inequality applicable.
\begin{proof}
    Put $m:=m_2(\mu)$ and $F(q):=\int\ol yq^2\,\dl\mu(y)$.
    By \cite[Theorem B and Corollary 16]{Yokota2016Barycenter-JMSJ}, $m$ is the unique mean in the whole sphere and belongs to $\setC$.
    Use $g(v)=\arccos(v)^2$ and
    \begin{equation*}
        H(v)=\frac{\sin(\arccos(v))}{\arccos(v)}
        =\frac{\sqrt{1-v^2}}{\arccos(v)}
        \qquad(-1<v<1)\eqcm\qquad H(1)=1
        \eqcm
    \end{equation*}
    from \cref{lmm:pos:cosine-average}.
    Since $\ol ym\leq2R<\pi$, all weights
    $1/H(\ip ym)=\ol ym/\sqrt{1-\ip ym^2}$ below are bounded, with value $1$ at $y=m$.
    Define the normalizing constant $U$ and the probability measure $\lambda$ by
    \begin{equation*}
        U:=\int\frac1{H(\ip ym)}\,\dl\mu(y)
        \eqcm\qquad
        \dl\lambda(y):=\frac{\dl\mu(y)}{UH(\ip ym)}
        \eqfs
    \end{equation*}
    Thus $\dl\mu(y)=UH(\ip ym)\,\dl\lambda(y)$, and the change of measure gives
    \begin{equation}\label{eq:sphere:vi:change}
        \begin{split}
            F(q)-F(m)
            &=\int\br{g(\ip yq)-g(\ip ym)}\,\dl\mu(y)\\
            &=U\int H(\ip ym)\br{g(\ip yq)-g(\ip ym)}\,\dl\lambda(y)
            \eqfs
        \end{split}
    \end{equation}
    Applying the joint Jensen inequality of \cref{lmm:pos:cosine-conversion}\,(i) now yields
    \begin{equation}\label{eq:sphere:vi:jensen}
        \begin{split}
            F(q)-F(m)
            &\geq U H\brOf{\int\ip ym\,\dl\lambda(y)}\\
            &\qquad{}\times\br{g\brOf{\int\ip yq\,\dl\lambda(y)}
            -g\brOf{\int\ip ym\,\dl\lambda(y)}}
            \eqfs
        \end{split}
    \end{equation}

    To evaluate these averages, we use the minimizing property of $m$.
    Put
    \begin{equation*}
        K:=\int\frac{\ip ym}{H(\ip ym)}\,\dl\mu(y)
        \eqcm\qquad z:=\frac KU
        \eqfs
    \end{equation*}
    For any unit vector $h\in\mc H$ with $\ip mh=0$, the great circle
    $\gamma(s):=\cos(s)m+\sin(s)h$ lies on the sphere and satisfies $\gamma(0)=m$ and $\gamma'(0)=h$.
    Since $m$ minimizes $F$ on the whole sphere, the derivative of $F\circ\gamma$ at zero vanishes.
    The bound $\ol y{\gamma(s)}\leq2R+|s|<\pi$ for small $|s|$ justifies differentiation under the integral.
    Using $g'=-2/H$ gives
    \begin{equation*}
        0=\left.\frac{\dl}{\dl s}F(\gamma(s))\right|_{s=0}
        =-2\int\frac{\ip yh}{H(\ip ym)}\,\dl\mu(y)
        \eqfs
    \end{equation*}
    Thus $\int y/H(\ip ym)\,\dl\mu(y)$ is orthogonal to every tangent direction $h\perp m$, so it is a multiple of $m$.
    Taking its inner product with $m$ identifies that multiple as $K$, and hence
    \begin{equation*}
        \int\frac{y}{H(\ip ym)}\,\dl\mu(y)=Km
        \eqfs
    \end{equation*}
    Taking the Hilbert inner product with $p\in\mc Q$ yields the identity
    \begin{equation}\label{eq:sphere:vi:comparison}
        \int\frac{\ip yp}{H(\ip ym)}\,\dl\mu(y)
        =K\ip mp
        \eqfs
    \end{equation}
    By the definition of $z$ and \eqref{eq:sphere:vi:comparison},
    \begin{equation}\label{eq:sphere:vi:averages}
        \int\ip ym\,\dl\lambda(y)=z
        \eqcm\qquad
        \int\ip yq\,\dl\lambda(y)=z\ip mq
        \eqfs
    \end{equation}
    At $p=o$, \eqref{eq:sphere:vi:comparison} implies $\cos(R)U\leq\ip mo\,K$.
    Thus, using $\ol mo\geq(\ol mq-R)_+$,
    \begin{equation}\label{eq:sphere:vi:z}
        z\geq\frac{\cos(R)}{\ip mo}
        \geq\frac{\cos(R)}{\cos((\ol mq-R)_+)}\geq\cos(R)>0
        \eqfs
    \end{equation}
    In particular, $\int y\,\dl\lambda(y)=zm$, so $z\in(0,1]$ is the Hilbert norm of this weighted average.
    Concavity of $H$ gives
    \begin{equation}\label{eq:sphere:vi:concavity}
        1=U\int H(\ip ym)\,\dl\lambda(y)\leq UH(z)
        \eqfs
    \end{equation}
    Substituting \eqref{eq:sphere:vi:averages} into \eqref{eq:sphere:vi:jensen} and using \eqref{eq:sphere:vi:concavity}, we obtain
    \begin{equation}\label{eq:sphere:vi:nonlinear}
        \begin{split}
            F(q)-F(m)
            &\geq UH(z)\br{g(z\ip mq)-g(z)}\\
            &\geq g(z\ip mq)-g(z)
            \eqcm
        \end{split}
    \end{equation}
    where $g(z\ip mq)-g(z)\geq0$ because $z\ip mq\leq z$ and $g$ is decreasing.
    For $-1<a\leq1$, the function $z\mapsto g(az)-g(z)$ is nondecreasing on $(0,1]$:
    its derivative is $2/H(z)-2a/H(az)\geq0$, because for $a\geq0$ concavity of $H$ and $H(0)=2/\pi>0$ give $H(az)\geq aH(z)+(1-a)H(0)\geq aH(z)$, and for $a<0$ the sign is immediate.
    Substitution of \eqref{eq:sphere:vi:z} into \eqref{eq:sphere:vi:nonlinear} therefore yields
    \begin{equation}\label{eq:sphere:vi:master}
        F(q)-F(m)
        \geq g\brOf{\frac{\cos(R)\ip mq}{\cos((\ol mq-R)_+)}}
        -g\brOf{\frac{\cos(R)}{\cos((\ol mq-R)_+)}}
        \eqfs
    \end{equation}

    To obtain \eqref{eq:sphere:vi:radius}, the case $q=m$ is immediate.
    For $0\leq x\leq R$, write
    \begin{equation*}
        E_{2R}(x):=g\brOf{\frac{\cos(R)\cos(R+x)}{\cos(x)}}
        -g\brOf{\frac{\cos(R)}{\cos(x)}}
        \eqfs
    \end{equation*}
    This is continuous and positive on $[0,R]$, since $g$ is decreasing and $\cos(R+x)<1$, so the maximum in \eqref{eq:sphere:vi:constant} is finite.
    At $x=R$, we have $E_{2R}(R)=4R^2$, so $A_{2R}^{\mathrm{rad}}\geq1$.
    If $0<\ol mq\leq R$, \cref{lmm:pos:cosine-conversion}\,(ii) gives
    \begin{equation*}
        \frac{F(q)-F(m)}{\ol mq^2}
        \geq\frac{g(\cos(R)\ip mq)-g(\cos(R))}{\ol mq^2}
        \geq\frac{E_{2R}(0)}{R^2}
        \eqfs
    \end{equation*}
    If $R\leq\ol mq\leq2R$, put $x:=\ol mq-R$.
    Then \eqref{eq:sphere:vi:master} reads
    $F(q)-F(m)\geq E_{2R}(x)$.
    The definition \eqref{eq:sphere:vi:constant} now gives \eqref{eq:sphere:vi:radius} in both ranges.
\end{proof}
The proof of \cref{thm:pos:hemisphere:vi} below carries this argument over to complete \catk{} spaces, with $\cos(\ol xy)$ in place of $\ip xy$ and the curved first-order condition of \cref{lmm:pos:curved:foc} in place of the identity \eqref{eq:sphere:vi:comparison}, and then lets $R\searrow\varrho$.

\begin{theorem}[Variance inequality under positive curvature upper bound]\label{thm:pos:hemisphere:vi}
    Let $\kappa\in\Rpp$ and let $(\mc Q,d)$ be a complete \catk{} space.
    Let $\setC\subseteq\mc Q$ be a nonempty bounded closed convex set, and let $\varrho:=\inf_{o\in\mc Q}\sup_{y\in\setC}d(o,y)$ be its circumradius.
    Set $\Delta := 2\varrho\sqrt{\kappa}$.
    Suppose $0<\Delta<\pi$.
    
    \begin{enumerate}[label=(\roman*)]
        \item
        Define, for $0\leq x\leq\Delta/2$,
        \begin{equation}\label{eq:pos:vi:radius:functions}
            E_\Delta(x):=
            \arccos\brOf{\cos(\Delta/2)^2-\cos(\Delta/2)\sin(\Delta/2)\tan(x)}^2
            -\arccos\brOf{\frac{\cos(\Delta/2)}{\cos(x)}}^2
            \eqfs
        \end{equation}
        There is a unique $x_\Delta\in(0,\Delta/2)$ satisfying
        \begin{equation}\label{eq:pos:vi:radius:root}
            (\Delta/2+x_\Delta)E_\Delta'(x_\Delta)=2E_\Delta(x_\Delta).
        \end{equation}
        Set
        \begin{equation}\label{eq:pos:vi:radius:optimal}
            A_\Delta^{\mathrm{rad}}:=\frac{(\Delta/2+x_\Delta)^2}{E_\Delta(x_\Delta)}
            =\max_{0\leq x\leq\Delta/2}\frac{(\Delta/2+x)^2}{E_\Delta(x)}.
        \end{equation}
        Then, for every $\mu\in\mc P(\setC)$ and $q\in\setC$,
        \begin{equation}\label{eq:pos:vi}
            \ol{m_2(\mu)}q^2
            \leq
            A_\Delta^{\mathrm{rad}}
            \int\br{\ol yq^2-\ol y{m_2(\mu)}^2}\,\dl\mu(y)
            \eqfs
        \end{equation}
        The constant is optimal among constants depending only on $\Delta$, uniformly over all such spaces, sets, and measures, for each positive value of $\Delta$.
        Moreover,
        \begin{equation}\label{eq:pos:vi:radius:upper}
            A_\Delta^{\mathrm{rad}}<\frac1{\cos(\Delta/2)}<C_\Delta
            \eqcm
        \end{equation}
        so \eqref{eq:2fmcatk:vi} follows.
        \item 
        On the smaller and mean-dependent range $\ol{m_2(\mu)}q\leq\varrho$, the optimal radius-only constant is explicit and smaller. Writing
        \begin{equation}\label{eq:pos:vi:near:constant}
            A_\Delta^{\mathrm{near}}:=\frac{(\Delta/2)^2}{E_\Delta(0)}
            =\frac{(\Delta/2)^2}{\arccos\brOf{\cos(\Delta/2)^2}^2-(\Delta/2)^2}
            \eqcm
        \end{equation}
        one has $1\leq A_\Delta^{\mathrm{near}}<A_\Delta^{\mathrm{rad}}$ and
        \begin{equation}\label{eq:pos:vi:near}
            \ol{m_2(\mu)}q^2
            \leq
            A_\Delta^{\mathrm{near}}
            \int\br{\ol yq^2-\ol y{m_2(\mu)}^2}\,\dl\mu(y)
            \qquad\text{whenever}\quad\ol{m_2(\mu)}q\leq\varrho
            \eqcm
        \end{equation}
        and $A_\Delta^{\mathrm{near}}$ is optimal among constants depending only on $\Delta$ on that range.
        \item 
        For diameter alone, set $\delta:=\sqrt\kappa\,\diam(\setC)$ and define
        \begin{equation}\label{eq:pos:vi:constants}
            A_\delta^{\mathrm{diam}}:=
            \frac{\theta_\delta^2}{\delta^2-\theta_\delta^2},
            \qquad
            \theta_\delta:=\arccos\brOf{\sqrt{\cos(\delta)}}
            \quad(0<\delta<\pi/2).
        \end{equation}
        Set $A_\delta^{\mathrm{diam}}:=+\infty$ for $\pi/2\leq\delta<\pi$.
        Then, for $0<\delta<\pi/2$, every $\mu\in\mc P(\setC)$ and $q\in\setC$ satisfy
        \begin{equation}\label{eq:pos:vi:diameter}
            \ol{m_2(\mu)}q^2
            \leq
            A_\delta^{\mathrm{diam}}
            \int\br{\ol yq^2-\ol y{m_2(\mu)}^2}\,\dl\mu(y)
            \eqfs
        \end{equation}
        The constant $A_\delta^{\mathrm{diam}}$ is optimal among constants depending only on $\delta$, uniformly over all such spaces, sets, and measures, for each positive value of $\delta$.
        Moreover, no finite diameter-only constant exists at any fixed $\pi/2\leq\delta<\pi$ (\cref{prp:pos:diameterobstruction}\,\ref{itm:pos:obstruction:vi}).
    \end{enumerate} 
\end{theorem}
\begin{proof}[Proof of \cref{thm:pos:hemisphere:vi}]
    Rescaling the metric by $\sqrt\kappa$ multiplies the squared distances on both sides of the desired inequality by $\kappa$.
    We may therefore assume $\kappa=1$, so that $\Delta=2\varrho$.
    Fix any $R\in(\varrho,\pi/2)$ and choose $o\in\mc Q$ such that $\setC\subset B(o,R)$.
    We first prove the bounds with parameter $2R$ and then let $R\searrow\varrho$ to recover the parameter $\Delta=2\varrho$.
    Put $m:=m_2(\mu)$ and $F(q):=\int\ol yq^2\,\dl\mu(y)$.
    Existence, uniqueness, and $m\in\setC\subset B(o,R)$ follow from \cite[Theorem B and Corollary 16]{Yokota2016Barycenter-JMSJ}.

    \textbf{Part 1: Nonlinear lower bound.}

    \emph{Change of measure and Jensen.}
    Use $g=\arccos^2$ and $H$ from \cref{lmm:pos:cosine-average}, and define
    \begin{equation*}
        U:=\int\frac1{H(\cos(\ol ym))}\,\dl\mu(y)
        \eqcm\qquad
        \dl\lambda(y):=\frac{\dl\mu(y)}{UH(\cos(\ol ym))}
        \eqfs
    \end{equation*}
    Here $1/H(\cos(\ol ym))=\ol ym/\sin(\ol ym)$ is bounded because $\ol ym\leq2R<\pi$, with value $1$ at $y=m$.
    The change of measure and joint Jensen steps \eqref{eq:sphere:vi:change}--\eqref{eq:sphere:vi:jensen} apply with each inner product replaced by the cosine of the corresponding distance.
    They give
    \begin{equation}\label{eq:pos:curved:jensen}
        \begin{split}
            F(q)-F(m)
            &\geq UH\brOf{\int\cos(\ol ym)\,\dl\lambda(y)}\\
            &\qquad{}\times\br{g\brOf{\int\cos(\ol yq)\,\dl\lambda(y)}
            -g\brOf{\int\cos(\ol ym)\,\dl\lambda(y)}}
            \eqfs
        \end{split}
    \end{equation}

    \emph{Curved comparison.}
    To bound these averages, put
    \begin{equation}\label{eq:pos:curved:AU}
        K:=\int\frac{\cos(\ol ym)}{H(\cos(\ol ym))}\,\dl\mu(y)
        \eqcm\qquad z:=\frac KU
        \eqfs
    \end{equation}
    In place of the spherical identity \eqref{eq:sphere:vi:comparison}, \cref{lmm:pos:curved:foc} (a version of \cite[Proposition 23]{Yokota2016Barycenter-JMSJ}) gives, for every $p\in\overline B(o,R)$,
    \begin{equation}\label{eq:pos:curved:comparison}
        \int\frac{\cos(\ol yp)}{H(\cos(\ol ym))}\,\dl\mu(y)
        \leq K\cos(\ol mp)
        \eqfs
    \end{equation}
    Thus the averages in \eqref{eq:pos:curved:jensen} satisfy
    \begin{equation}\label{eq:pos:curved:averages}
        \int\cos(\ol ym)\,\dl\lambda(y)=z
        \eqcm\qquad
        \int\cos(\ol yq)\,\dl\lambda(y)\leq z\cos(\ol mq)
        \eqfs
    \end{equation}
    At $p=o$, \eqref{eq:pos:curved:comparison} gives $\cos(R)U\leq K\cos(\ol mo)$.
    The enclosing-ball argument of \eqref{eq:sphere:vi:z} therefore applies unchanged and yields
    \begin{equation}\label{eq:pos:curved:z}
        z\geq\frac{\cos(R)}{\cos(\ol mo)}
        \geq\frac{\cos(R)}{\cos((\ol mq-R)_+)}\geq\cos(R)>0
        \eqfs
    \end{equation}
    Also $z\leq1$ by the first identity in \eqref{eq:pos:curved:averages}.
    The concavity argument of \eqref{eq:sphere:vi:concavity} still gives $UH(z)\geq1$.
    Since $g$ is decreasing, the inequality for the second average in \eqref{eq:pos:curved:averages} has the required direction in \eqref{eq:pos:curved:jensen}.
    As in \eqref{eq:sphere:vi:nonlinear}, we conclude that
    \begin{equation}\label{eq:pos:curved:nonlinear}
        \begin{split}
            F(q)-F(m)
            &\geq UH(z)\br{g(z\cos(\ol mq))-g(z)}\\
            &\geq g(z\cos(\ol mq))-g(z)
            \eqfs
        \end{split}
    \end{equation}

    \textbf{Part 2: Optimal radius constant.}

    The identity
    $\cos(R)\cos(R+x)/\cos(x)=\cos(R)^2-\cos(R)\sin(R)\tan(x)$
    identifies $E_{2R}$ from \eqref{eq:pos:vi:radius:functions} with the function used in the proof of \cref{prp:sphere:hemisphere:vi}.
    The monotonicity in $z$ and the two distance ranges in that proof apply unchanged to \eqref{eq:pos:curved:nonlinear}, using \eqref{eq:pos:curved:z}.
    They give
    \begin{equation}\label{eq:pos:vi:radius:near}
        F(q)-F(m)\geq\frac{E_{2R}(0)}{R^2}\ol mq^2
        \qquad\text{if }\ol mq\leq R
        \eqfs
    \end{equation}
    Both ranges, including the immediate case $q=m$, are covered by
    \begin{equation}\label{eq:pos:vi:radius:enclosing}
        \ol mq^2\leq A_{2R}^{\mathrm{rad}}\br{F(q)-F(m)}
        \eqcm\qquad
        A_{2R}^{\mathrm{rad}}=\max_{0\leq x\leq R}\frac{(R+x)^2}{E_{2R}(x)}
        \eqcm
    \end{equation}
    with the same finite maximum as in \eqref{eq:sphere:vi:constant}.
    After the substitution $x=Ru$, the maximum is taken over the fixed interval $u\in[0,1]$, and the maximum theorem shows that $R\mapsto A_{2R}^{\mathrm{rad}}$ is continuous.
    Letting $R\searrow\varrho$ in \eqref{eq:pos:vi:radius:enclosing} therefore gives $\ol mq^2\leq A_\Delta^{\mathrm{rad}}(F(q)-F(m))$, which is \eqref{eq:pos:vi}.
    If $\ol mq\leq\varrho$, then $q$ lies in the first range for every $R>\varrho$.
    Letting $R\searrow\varrho$ in \eqref{eq:pos:vi:radius:near}, and using the continuity of $R^2/E_{2R}(0)$, gives \eqref{eq:pos:vi:near}.
    Finally, \cref{lmm:pos:vi:root} with parameter $\Delta$ gives the unique-root characterization \eqref{eq:pos:vi:radius:root} and the maximum formula \eqref{eq:pos:vi:radius:optimal}.
    
    \textbf{Part 3: Optimal diameter constant.}

    Suppose $0<\delta<\pi/2$; the case $q=m$ is immediate.
    Integrate \eqref{eq:pos:curved:comparison} with $p=r$ against $\frac{\ol rm}{\sin(\ol rm)}\,\dl\mu(r)$.
    Since $\cos(\ol yr)\geq \cos(\delta)$ for $y,r\in\setC$, this gives $\cos(\delta)U^2\leq K^2$.
    The same pointwise bound in \eqref{eq:pos:curved:comparison} gives $\cos(\delta)U\leq\cos(\ol mq)K$ for each $q\in\setC$.
    Consequently,
    \begin{equation}\label{eq:pos:curved:diameter}
        z\geq b(\ol mq):=\max\{\sqrt{\cos(\delta)},\cos(\delta)/\cos(\ol mq)\}.
    \end{equation}
    By the monotonicity in $z$ used in Part 2, we may replace $z$ in \eqref{eq:pos:curved:nonlinear} by $b(\ol mq)$.

    For $0<\ol mq\leq\theta_\delta$, we have $b(\ol mq)=\cos(\theta_\delta)$.
    By \cref{lmm:pos:cosine-conversion}\,(ii),
    \begin{equation*}
        \frac{g(b(\ol mq)\cos(\ol mq))-g(b(\ol mq))}{\ol mq^2}
        \geq\frac{g(\cos(\theta_\delta)^2)-\theta_\delta^2}{\theta_\delta^2}
        =\frac{\delta^2-\theta_\delta^2}{\theta_\delta^2}.
    \end{equation*}
    For $\theta_\delta\leq\ol mq\leq\delta$, put $t:=\arccos(\cos(\delta)/\cos(\ol mq))\leq\ol mq$, regarded as a function of $\ol mq$.
    The coefficient is now $(\delta^2-t^2)/\ol mq^2$, and it is nondecreasing in $\ol mq$.
    Indeed, in the interior of this interval, with $u=t^2$ and $v=\ol mq^2$,
    \begin{equation*}
        -\frac{\dl u}{\dl v}=\frac{t\cot(t)}{\ol mq\cot(\ol mq)}\geq1
        \qquad\text{and}\qquad
        \frac{\delta^2-u}{v}\leq1
        \eqfs
    \end{equation*}
    The first bound follows from $t\leq\ol mq$ and the monotonicity of $s\mapsto s\cot(s)$ in \cref{lmm:pos:cosine-average}\,(ii).
    The second follows from \cref{lmm:pos:cosine-conversion}\,(ii), since $\cos(t)\cos(\ol mq)=\cos(\delta)$.
    Differentiating $(\delta^2-u)/v$ with respect to $v$ shows that its derivative is nonnegative precisely when $-\,\dl u/\dl v\geq(\delta^2-u)/v$, which follows from the two displayed inequalities.
    By continuity at the endpoints, its minimum therefore again occurs at $\ol mq=\theta_\delta$.
    This proves \eqref{eq:pos:vi:diameter}.

    \textbf{Part 4: Sharpness and bounds on the constants.}
    
    Both $\Delta$ and $\delta$ are invariant under rescaling, so examples with $\kappa=1$ suffice.
    \Cref{lmm:pos:vi:sharp:radius} realizes, in a spherical ball of circumradius exactly $\Delta/2$, the ratio $(\Delta/2+x)^2/E_\Delta(x)$ at every $x\in[0,\Delta/2]$, with $\ol mq=\Delta/2+x$.
    Taking $x=x_\Delta$ attains $A_\Delta^{\mathrm{rad}}$ and proves radius-only optimality.
    Taking $x=0$ gives $\ol mq=\Delta/2=\varrho$ and attains $A_\Delta^{\mathrm{near}}$, proving optimality on the restricted range $\ol mq\leq\varrho$.
    
    The proof of \cref{thm:pos:hemisphere:cp} uses two further consequences of Part~1. Convexity of $g$ and $g'=-2/H$ applied to the first inequality in \eqref{eq:pos:curved:nonlinear} give
    \begin{equation}\label{eq:pos:curved:intermediate}
        \begin{split}
            F(q)-F(m)
            &\geq UH(z)g'(z)z\br{\cos(\ol mq)-1}\\
            &=2K\br{1-\cos(\ol mq)}
            \eqfs
        \end{split}
    \end{equation}
    Set $t:=\arccos(\cos(R)/\cos(\ol mo))\in[0,R]$, with $t\cot(t)=1$ at $t=0$.
    Since $z\geq\cos(t)$ by \eqref{eq:pos:curved:z}, \cref{lmm:pos:cosine-average}\,(iii), applied to $X=\cos(\ol Ym)$ with $Y\sim\mu$, gives
    \begin{equation}\label{eq:pos:curved:A-lower}
        K=Uz\geq\frac z{H(z)}\geq t\cot(t)>0
        \eqfs
    \end{equation}

    For the strict upper bound on $A_\Delta^{\mathrm{rad}}$, let $x\in[0,\Delta/2]$, put $u_x:=(\Delta/2+x)/2$, and let $t_x:=t_\Delta(x)$ be as in \eqref{eq:pos:vi:radius:ts}. Then $\cos(t_x)\cos(2u_x)$ is the argument of the first arccosine in \eqref{eq:pos:vi:radius:functions}, so $E_\Delta(x)=g(\cos(t_x)\cos(2u_x))-g(\cos(t_x))$.
    Convexity of $g$ and $g'(\cos(t_x))=-2C_{t_x}$ (\cref{lmm:pos:cosine-average}\,(i)) give
    \begin{equation*}
        E_\Delta(x)
        \geq2t_x\cot(t_x)\br{1-\cos(2u_x)}
        =t_x\cot(t_x)\br{\frac{\sin(u_x)}{u_x}}^2\br{\frac\Delta2+x}^2
        >\cos(\Delta/2)\br{\frac\Delta2+x}^2
        \eqcm
    \end{equation*}
    with $t_x\cot(t_x)=1$ if $t_x=0$, where the last inequality is the strict form of \cref{lmm:pos:radius-scalar} with $a=x$ and $\Delta/2>0$ in place of $R$.
    Hence the maximand in \eqref{eq:pos:vi:radius:optimal} is less than $1/\cos(\Delta/2)$ on the compact interval $[0,\Delta/2]$, so $A_\Delta^{\mathrm{rad}}<1/\cos(\Delta/2)$, and $\sin(\Delta/2)<\Delta/2$ turns this into \eqref{eq:pos:vi:radius:upper}.
    Finally $E_\Delta(\Delta/2)=\Delta^2$ by \eqref{eq:pos:vi:radius:functions}, so the maximand in \eqref{eq:pos:vi:radius:optimal} equals $1$ there and $A_\Delta^{\mathrm{rad}}\geq1$.
    Together with $A_\Delta^{\mathrm{rad}}<1/\cos(\Delta/2)$, this gives its continuous extension with value $1$ at $\Delta=0$.
    For $R=\Delta/2$, \cref{lmm:pos:cosine-conversion}\,(ii) gives $E_\Delta(0)=g(\cos(R)^2)-R^2\leq R^2$, hence $A_\Delta^{\mathrm{near}}\geq1$.
    Since $x_\Delta>0$ is the unique maximizer in \eqref{eq:pos:vi:radius:optimal}, the value $A_\Delta^{\mathrm{near}}$ of the maximand at $x=0$ is strictly smaller than $A_\Delta^{\mathrm{rad}}$; and $1\leq A_\Delta^{\mathrm{near}}\leq A_\Delta^{\mathrm{rad}}$ shows that it also tends to $1$ as $\Delta\searrow0$.
    Diameter-only optimality at every $0<\delta<\pi/2$ is \cref{lmm:pos:vi:sharp:diameter}, and at each fixed $\pi/2\leq\delta<\pi$, \cref{prp:pos:diameterobstruction}\,\ref{itm:pos:obstruction:vi} below shows that no finite variance constant depending only on $\delta$ exists.
\end{proof}

\Needspace{8\baselineskip}
\subsection{Lemmas for the Variance Inequality}\label{ssec:pos:vi:lemmas}
We first establish the curved first-order condition and the scalar properties used in \cref{ssec:pos:vi}, then give the sharpness constructions.

\begin{lemma}[Curved first-order condition {\cite[Proposition 23]{Yokota2016Barycenter-JMSJ}}]\label{lmm:pos:curved:foc}
    Let $(\mc Q,d)$ be a complete \catof{1} space, $o\in\mc Q$, $0<R<\pi/2$, let $\mu$ be a probability measure concentrated on $B(o,R)$, and let $m:=m_2(\mu)$. Then, with $t/\sin(t):=1$ at $t=0$, every $p\in\overline B(o,R)$ satisfies
    \begin{equation*}
        \int\cos(\ol yp)\,\frac{\ol ym}{\sin(\ol ym)}\,\dl\mu(y)
        \leq\cos(\ol mp)\int\cos(\ol ym)\,\frac{\ol ym}{\sin(\ol ym)}\,\dl\mu(y)
        \eqfs
    \end{equation*}
\end{lemma}
The statement is a special case of \cite[Proposition 23]{Yokota2016Barycenter-JMSJ}, which allows every $p$ with $\ol mp<\pi$. We include a proof because Yokota derives it from \cite[(22)]{Yokota2016Barycenter-JMSJ}, an inequality in which the comparison inner product is set to $-\infty$ for triangles of perimeter at least $2\pi$; such triangles occur here once $R>\pi/3$, and for them the proposition needs the pointwise bound in the proof below.
\begin{proof}
    The mean $m$ exists, is unique and lies in $B(o,R)$ \cite[Theorem B and Corollary 16]{Yokota2016Barycenter-JMSJ}. Put $a:=\ol mp$; if $a=0$ there is nothing to prove. All distances between points of $\overline B(o,R)$ are at most $2R<\pi$, so geodesics between them are unique. By the first variation formula \cite[Corollary II.3.6]{BridsonHaefliger1999NPC-Springer} and dominated convergence (the difference quotients of $\ol yq^2$ along $[m,p]$ are bounded by $4R$, since $[m,p]\subset\overline B(o,R)$ by convexity of balls of radius less than $\pi/2$ \cite[Proposition II.1.4(3)]{BridsonHaefliger1999NPC-Springer}), the right derivative at $m$ of $q\mapsto\frac12\int\ol yq^2\,\dl\mu(y)$ along $[m,p]$ is $-\int\ol ym\cos(\angle_m(p,y))\,\dl\mu(y)$, with the integrand read as $0$ at $y=m$. Minimality of $m$ gives
    \begin{equation*}
        \int\ol ym\cos(\angle_m(p,y))\,\dl\mu(y)\leq0
        \eqfs
    \end{equation*}
    Fix $y\neq m$ and put $c:=\ol ym$ and $e:=\ol yp$, so that $a,c\in(0,\pi)$ and $e\in[0,\pi)$. We claim
    \begin{equation*}
        \sin(a)\sin(c)\cos(\angle_m(p,y))\geq\cos(e)-\cos(a)\cos(c)
        \eqfs
    \end{equation*}
    If $a+c+e<2\pi$, this is the comparison $\angle_m(p,y)\leq\tilde\angle_m(p,y)$ \cite[Proposition II.1.7(4)]{BridsonHaefliger1999NPC-Springer} together with the spherical law of cosines.
    If $a+c+e\geq2\pi$, then $a+c>\pi$ because $e<\pi$, so $e$ and $2\pi-a-c$ both lie in $[0,\pi]$ with $e\geq2\pi-a-c$; hence $\cos(e)\leq\cos(a+c)=\cos(a)\cos(c)-\sin(a)\sin(c)$, and the claim follows from $\cos(\angle_m(p,y))\geq-1$.
    Multiplying the claim by $c/\sin(c)>0$, integrating, and noting that both integrands vanish at $y=m$, we obtain
    \begin{equation*}
        \int\frac{\ol ym}{\sin(\ol ym)}\br{\cos(\ol yp)-\cos(a)\cos(\ol ym)}\dl\mu(y)
        \leq\sin(a)\int\ol ym\cos(\angle_m(p,y))\,\dl\mu(y)
        \leq0
        \eqcm
    \end{equation*}
    which is the assertion. All integrands are bounded because $\ol ym/\sin(\ol ym)\leq2R/\sin(2R)$.
\end{proof}

\begin{lemma}[The cosine profile]\label{lmm:pos:cosine-average}
    Use $C_x$ as in \eqref{eq:pos:const:C}.
    Let $g(v):=\arccos(v)^2$ and
    \begin{equation*}
        H(v):=\frac{\sin(\arccos(v))}{\arccos(v)}
        =\frac{\sqrt{1-v^2}}{\arccos(v)}
        \qquad(-1<v<1)\eqcm
    \end{equation*}
    with $H(1):=1$ and derivatives at $v=1$ understood from the left.
    \begin{enumerate}[label=(\roman*)]
        \item $g$ is nonnegative, decreasing, and convex, with $g(1)=0$, $g\pr=-2/H$ and $g\pr(\cos(x))=-2C_x$ for $0\leq x<\pi$, and $H$ is positive, increasing, and concave.
        \item The map $v\mapsto v/H(v)$ is increasing, with value $t\cot(t)$ at $v=\cos(t)$, $t\in(0,\pi)$, and it tends to $1$ as $v\nearrow1$.
        \item Let $X$ be a random variable with values in $(-1,1]$ such that $U:=\Eof{1/H(X)}<\infty$, and put $K:=\Eof{X/H(X)}$ and $z:=K/U$. Then $z\in(-1,1]$, and if $z>0$, then
        \begin{equation*}
            K\geq\frac z{H(z)}
            \eqfs
        \end{equation*}
        In particular, $z\geq\cos(t)$ for some $t\in[0,\pi/2)$ implies $K\geq t\cot(t)>0$, with $t\cot(t)=1$ at $t=0$.
    \end{enumerate}
\end{lemma}
\begin{proof}
    \emph{Proof of (i).}
    Nonnegativity and monotonicity of $g$ follow from those of $\arccos$, and
    $g\pr(v)=-2\arccos(v)/\sqrt{1-v^2}=-2/H(v)$, which equals $-2C_x$ at $v=\cos(x)$.
    At $v=\cos(t)$ with $t\in(0,\pi)$,
    \begin{equation*}
        g^{\prime\prime}(\cos(t))=\frac{2\br{\sin(t)-t\cos(t)}}{\sin(t)^3}>0
        \eqcm
    \end{equation*}
    since $\sin(t)-t\cos(t)$ vanishes at zero and has derivative $t\sin(t)>0$; hence $g$ is convex.
    Next, $H(\cos(t))=\sin(t)/t$ is positive and decreasing in $t\in(0,\pi)$, so $H$ is positive and increasing.
    For concavity,
    \begin{equation*}
        H^{\prime\prime}(\cos(t))=-\frac{G(t)}{t^3\sin(t)^3}
        \eqcm\qquad
        G(t):=t^2+t\sin(t)\cos(t)-2\sin(t)^2
        \eqcm
    \end{equation*}
    where $G(0)=G\pr(0)=0$ and $G^{\prime\prime}(t)=4\sin(t)(\sin(t)-t\cos(t))>0$ on $(0,\pi)$, so that $G>0$ and $H^{\prime\prime}<0$ there.

    \emph{Proof of (ii).}
    At $v=\cos(t)$ the quotient equals $\cos(t)\arccos(\cos(t))/\sqrt{1-\cos(t)^2}=t\cot(t)$, whose derivative in $t$ is $(\sin(t)\cos(t)-t)/\sin(t)^2<0$ on $(0,\pi)$ because $\sin(2t)<2t$.
    As $t\mapsto\cos(t)$ is decreasing, the quotient is increasing in $v$.

    \emph{Proof of (iii).}
    Since $1/H>0$, the law $\lambda$ of $X$ reweighted by the density $1/(UH(X))$ is a probability measure, and $z=\E_\lambda[X]$ is an average of values in $(-1,1]$, hence lies in $(-1,1]$.
    Jensen's inequality for the concave $H$ gives
    \begin{equation*}
        \frac1U=\E_\lambda\abOf{H(X)}\leq H\brOf{\E_\lambda\abOf{X}}=H(z)
        \eqcm
    \end{equation*}
    so that $U\geq1/H(z)$ and therefore $K=Uz\geq z/H(z)$ whenever $z>0$.
    By (ii), $z\geq\cos(t)$ with $0\leq t<\pi/2$ gives $z/H(z)\geq\cos(t)/H(\cos(t))=t\cot(t)>0$.
\end{proof}

\begin{lemma}[Conversion from cosine comparison]\label{lmm:pos:cosine-conversion}
    Let $g$ and $H$ be as in \cref{lmm:pos:cosine-average}.
    \begin{enumerate}[label=(\roman*)]
        \item The map $\Phi(x,v):=H(x)(g(v)-g(x))$ is jointly convex on $(-1,1]^2$.
        Consequently, if $X,V$ are random variables taking values in a compact subinterval of $(-1,1]$, then
        \begin{equation}\label{eq:pos:scalar:jensen}
            \EOf{H(X)\br{g(V)-g(X)}}
            \geq H\brOf{\EOf X}\br{g\brOf{\EOf V}-g\brOf{\EOf X}}
            \eqfs
        \end{equation}
        \item For $0\leq t<\pi/2$, the function
        $G_t(\ell):=g(\cos(t)\cos(\ell))$ is concave as a function of $\ell^2$ on $0\leq \ell<\pi/2$.
        Concavity is strict for $t>0$.
        Consequently $(G_t(\ell)-t^2)/\ell^2$ is nonincreasing for $0<\ell<\pi/2$, and $G_t(\ell)\leq t^2+\ell^2$ for $0\leq \ell<\pi/2$.
    \end{enumerate}
\end{lemma}
\begin{proof}
    \emph{A prerequisite for (i).} We show that $\chi:=-HH''/(H')^2$ is decreasing on $(-1,1)$.
    Let $v=\cos(t)$ with $0<t<\pi$, and put $w:=t^2\in(0,\pi^2)$.
    Differentiating $H(\cos(t))=\sin(t)/t$ in $t$ gives $H'(\cos(t))=(1-t\cot(t))/t^2$, and Euler's partial fraction expansion $t\cot(t)=1-2\sum_{n=1}^\infty t^2/(n^2\pi^2-t^2)$ \cite[(4.22.3)]{DLMF2026-NIST} yields
    \begin{equation*}
        H'(\cos(t))=2S(w),\qquad S(w):=\sum_{n=1}^{\infty}\frac1{n^2\pi^2-w}
        \eqfs
    \end{equation*}
    This series and its termwise derivatives converge locally uniformly on $(-\infty,\pi^2)$.
    Since $\dl w/\dl v=-2t/\sin(t)=-2/H(v)$, we get $H''(\cos(t))=-4S'(w)/H(\cos(t))$, hence $\chi(\cos(t))=S'(w)/S(w)^2$, and
    \begin{equation*}
        \begin{aligned}
        \frac{\dl}{\dl w}\frac{S'(w)}{S(w)^2}
        &=\frac{S(w)S''(w)-2S'(w)^2}{S(w)^3}
        \\&=\frac2{S(w)^3}\br{\sum_{n=1}^{\infty}\frac1{n^2\pi^2-w}\sum_{n=1}^{\infty}\frac1{(n^2\pi^2-w)^3}-\br{\sum_{n=1}^{\infty}\frac1{(n^2\pi^2-w)^2}}^2}
        >0
        \end{aligned}
    \end{equation*}
    by the Cauchy--Schwarz inequality, strictly because the numbers $n^2\pi^2-w$ are not all equal.
    Hence $\chi(\cos(t))$ increases with $t$, that is, $\chi$ decreases with its argument.

    \emph{Proof of (i).}
    For $-1<x,v<1$, put $I:=\int_v^xH(u)^{-1}\,\dl u$, so that $g(v)-g(x)=2I$ by $g\pr=-2/H$.
    Since $(1/H')'=-H''/(H')^2=\chi/H$,
    \begin{equation*}
        \frac1{H'(x)}-\frac1{H'(v)}-\chi(x)I
        =\int_v^x\frac{\chi(u)-\chi(x)}{H(u)}\,\dl u
        \geq0
        \eqcm
    \end{equation*}
    in both orientations, because $\chi$ is decreasing: for $v\leq x$ the integrand is nonnegative on $[v,x]$, and for $v>x$ it is nonpositive on $[x,v]$ while the orientation of the integral is reversed.
    By \cref{lmm:pos:cosine-average}\,(i), $g\pr=-2/H$, hence $g''=2H'/H^2$, and $H$ is positive and increasing.
    The second derivatives of $\Phi$ are
    \begin{equation*}
        \Phi_{xx}=2H''(x)I+\frac{2H'(x)}{H(x)}
        \eqcm\qquad
        \Phi_{xv}=-\frac{2H'(x)}{H(v)}
        \eqcm\qquad
        \Phi_{vv}=\frac{2H(x)H'(v)}{H(v)^2}>0
        \eqcm
    \end{equation*}
    and $H(x)H''(x)=-\chi(x)H'(x)^2$ gives
    \begin{equation*}
        \begin{split}
        \frac{\det\nabla^2\Phi(x,v)}{g'(v)^2}
        &=H(x)H''(x)H'(v)I+H'(x)H'(v)-H'(x)^2\\
        &=H'(x)^2H'(v)
        \br{\frac1{H'(x)}-\frac1{H'(v)}-\chi(x)I}\geq0
        \eqfs
        \end{split}
    \end{equation*}
    Thus $\nabla^2\Phi$ is positive semidefinite on $(-1,1)^2$.
    Continuity extends convexity to $(-1,1]^2$, including both edges $x=1$ and $v=1$.
    Jensen's inequality for $\Phi$ applied to $(X,V)$ now gives \eqref{eq:pos:scalar:jensen}.

    \emph{Proof of (ii).}
    The case $t=0$ is immediate.
    Otherwise set $s:=\arccos(\cos(t)\cos(\ell))$ and $a(u):=u/\sin(u)$.
    Differentiation yields
    \begin{equation*}
        \frac{\dl G_t}{\dl(\ell^2)}=\cos(t)\,\frac{a(s)}{a(\ell)},\qquad
        \frac{\dl}{\dl \ell}\log\br{\frac{\dl G_t}{\dl(\ell^2)}}
        =\tan(\ell)\br{k(s)-k(\ell)},
        \quad k(u):=\frac{\cot(u)}{u}-\cot(u)^2.
    \end{equation*}
    On $(0,\pi/2)$,
    \begin{equation*}
        k'(u)=\frac{\cos(u)}{u^2\sin(u)^3}
        \br{2u^2-u\tan(u)-\sin(u)^2}<0
        \eqcm
    \end{equation*}
    because
    \begin{equation*}
        u\tan(u)+\sin(u)^2
        \geq u^2\br{1+\frac{u^2}3}+u^2\br{1-\frac{u^2}3}
        =2u^2
        \eqcm
    \end{equation*}
    with strict inequality for $u>0$.
    Here $\tan(u)\geq u+u^3/3$ follows by integrating $\tan'(u)=1+\tan(u)^2\geq1+u^2$, strictly for $u>0$, and $(\sin(u)/u)^2\geq(1-u^2/6)^2\geq1-u^2/3$ follows by squaring $\sin(u)\geq u-u^3/6$, whose right-hand side is nonnegative for $0<u<\pi/2$.
    Since $\ell<s<\pi/2$, the logarithmic derivative is strictly negative, proving strict concavity.
    The secant slopes from zero therefore decrease strictly when $t>0$ and are constant when $t=0$.
    Finally the derivative at $\ell^2=0$ is $t\cot(t)\leq1$, so $G_t(\ell)\leq G_t(0)+\ell^2=t^2+\ell^2$.
\end{proof}

\Needspace{10\baselineskip}
\begin{lemma}[Unique-root characterization of the radius constant]\label{lmm:pos:vi:root}
    Let $0<\Delta<\pi$ and let $E_\Delta$ be as in \eqref{eq:pos:vi:radius:functions}, so that $E_\Delta=s_\Delta^2-t_\Delta^2$ for
    \begin{equation}\label{eq:pos:vi:radius:ts}
        t_\Delta(x):=\arccos\brOf{\frac{\cos(\Delta/2)}{\cos(x)}}
        \eqcm\qquad
        s_\Delta(x):=\arccos\brOf{\cos(\Delta/2)^2-\cos(\Delta/2)\sin(\Delta/2)\tan(x)}
        \eqcm
    \end{equation}
    defined for $0\leq x\leq\Delta/2$.
    Then \eqref{eq:pos:vi:radius:root} has a unique solution $x_\Delta\in(0,\Delta/2)$, and $x_\Delta$ is the unique maximizer of $x\mapsto(\Delta/2+x)^2/E_\Delta(x)$ on $[0,\Delta/2]$, so that \eqref{eq:pos:vi:radius:optimal} holds.
\end{lemma}
\begin{proof}
    Write $R:=\Delta/2$.
    Derivatives of $E_\Delta$ at the endpoints are understood from within $[0,R]$.

    \emph{Positivity.}
    Differentiating $\cos(s_\Delta(x))=\cos(R)^2-\cos(R)\sin(R)\tan(x)$ and $\cos(t_\Delta(x))=\cos(R)/\cos(x)$ gives
    \begin{equation*}
        s_\Delta'(x)=\frac{\cos(R)\sin(R)}{\cos(x)^2\sin(s_\Delta(x))}
        \eqcm\qquad
        t_\Delta'(x)=-\cot(t_\Delta(x))\tan(x)
        \eqcm
    \end{equation*}
    and $E_\Delta'/2=s_\Delta s_\Delta'-t_\Delta t_\Delta'$ becomes
    \begin{equation}\label{eq:pos:vi:radius:derivative}
        \frac{E_\Delta'(x)}2
        =\frac{\cos(R)\sin(R)}{\cos(x)^2}\frac{s_\Delta(x)}{\sin(s_\Delta(x))}
        +t_\Delta(x)\cot(t_\Delta(x))\tan(x)
        \eqcm
    \end{equation}
    where $t\cot(t)=1$ at $t=0$.
    On $[0,R]$ we have $s_\Delta(x)\in(0,\pi)$ and $t_\Delta(x)\in[0,R]\subset[0,\pi/2)$, so the first summand is positive and the second is nonnegative; hence $E_\Delta$ is strictly increasing there.
    Since $\cos(R)^2<\cos(R)$, we have $E_\Delta(0)=\arccos(\cos(R)^2)^2-R^2>0$, and therefore $E_\Delta>0$ on $[0,R]$.

    \emph{Monotonicity of $J_\Delta$.}
    Put $J_\Delta(x):=(R+x)E_\Delta'(x)/2-E_\Delta(x)$.
    With $g$ as in \cref{lmm:pos:cosine-average}, the identity
    \begin{equation*}
        \frac{\arccos(z)}{\sqrt{1-z^2}}
        =\int_0^\infty\frac{1}{\cosh(u)+z}\,\dl u
        \qquad(-1<z\leq1)
    \end{equation*}
    follows by the substitution $v=\tanh(u/2)$, with the value at $z=1$ interpreted continuously.
    Integrating $g\pr(z)=-2\arccos(z)/\sqrt{1-z^2}$ yields
    \begin{equation*}
        E_\Delta(x)=2\int_0^\infty f_{\cosh(u)}(x)\,\dl u,
        \qquad
        f_a(x):=\log\frac{a\cos(x)+\cos(R)}{a\cos(x)+\cos(R)\cos(R+x)}.
    \end{equation*}
    For $a\geq1$, put $N_a:=a\cos(x)+\cos(R)$ and $M_a:=a\cos(x)+\cos(R)\cos(R+x)$.
    Both are positive and nonincreasing on $[0,R]$, with $M_a$ strictly decreasing.
    Moreover $M_a(x)\geq M_1(x)\geq M_1(R)=2\cos(R)^3>0$.
    Differentiation gives
    \begin{equation*}
        f_a'(x)=\frac{N_a'}{N_a}-\frac{M_a'}{M_a}
        =\frac{\cos(R)\br{a\sin(x)+a\sin(R)+\cos(R)\sin(R+x)}}{N_aM_a}
        \eqfs
    \end{equation*}
    With $v:=(R+x)/2$, the identities $\sin(x)+\sin(R)=2\sin(v)\cos((R-x)/2)$, $\sin(R+x)=2\sin(v)\cos(v)$, $2\cos(v)\cos((R-x)/2)=\cos(R)+\cos(x)$, and $2\cos(v)^2=1+\cos(R+x)$ give
    \begin{equation*}
        \begin{split}
        a\br{\sin(x)+\sin(R)}+\cos(R)\sin(R+x)
        &=\tan(v)\br{a\br{\cos(R)+\cos(x)}+\cos(R)\br{1+\cos(R+x)}}\\
        &=\tan(v)\br{M_a+(a+1)\cos(R)}
        \eqcm
        \end{split}
    \end{equation*}
    and therefore
    \begin{equation*}
        \frac{f_a'(x)}{\tan((R+x)/2)}
        =\frac{\cos(R)}{N_a}+\frac{\cos(R)^2(a+1)}{N_aM_a}.
    \end{equation*}
    The right-hand side is positive and has a positive derivative in $x$, since $N_a$ and $M_a$ are positive, $N_a'\leq0$, and $M_a'(x)=-a\sin(x)-\cos(R)\sin(R+x)<0$.
    Also $s\mapsto\tan(s/2)/s$ is positive with a positive derivative for $0<s<\pi$.
    Hence $f_a'(x)/(R+x)$, the product of $f_a'(x)/\tan((R+x)/2)$ and $\tan((R+x)/2)/(R+x)$, has a positive derivative in $x$.
    To differentiate under the integral, write $f_a=\log(1+D/M_a)$ with $D(x):=\cos(R)(1-\cos(R+x))$, which does not depend on $a$.
    For $a\geq1$, we have $M_a\geq a\cos(R)^3$: if $\cos(R+x)\geq0$, then $M_a\geq a\cos(x)\geq a\cos(R)$, and otherwise $M_a\geq aM_1(x)\geq2a\cos(R)^3$; moreover, $M_a'$ and $M_a''$ are bounded by $a+1$ in absolute value; since $D$, $D'$, and $D''$ are bounded, $f_a$, $f_a'$, and $f_a''$ are of order $1/a$ as $a\to\infty$, uniformly in $x\in[0,R]$.
    With $a=\cosh(u)\geq\euler^u/2$, this is an integrable $\landauO(\euler^{-u})$ bound, so $E_\Delta'(x)=2\int_0^\infty f_{\cosh(u)}'(x)\,\dl u$ and similarly for $E_\Delta''$.
    Consequently, $E_\Delta'(x)/(R+x)=2\int_0^\infty f_{\cosh(u)}'(x)/(R+x)\,\dl u$ has a positive derivative, and
    \begin{equation*}
        J_\Delta'(x)=\frac{(R+x)E_\Delta''(x)-E_\Delta'(x)}2
        =\frac{(R+x)^2}2\,\frac{\dl}{\dl x}\frac{E_\Delta'(x)}{R+x}>0.
    \end{equation*}

    \emph{The endpoints.}
    At $x=0$, $t_\Delta'(0)=0$.
    The strict concavity in \cref{lmm:pos:cosine-conversion}\,(ii) for $t=R>0$ compares the secant slope of $v\mapsto g(\cos(R)\cos(\sqrt v))$ on $[0,R^2]$ with its derivative at $R^2$, giving $J_\Delta(0)<0$.
    At the other endpoint, \eqref{eq:pos:vi:radius:derivative} gives
    \begin{equation*}
        J_\Delta(R)=2R\bigl(\tan(R)-R+R\tan(R)^2\bigr)>0
        \eqcm
    \end{equation*}
    because $\tan(R)>R$.

    \emph{Conclusion.}
    \begin{equation*}
        \frac{\dl}{\dl x}\frac{(R+x)^2}{E_\Delta(x)}
        =-\frac{2(R+x)J_\Delta(x)}{E_\Delta(x)^2}.
    \end{equation*}
    Thus the unique zero $x_\Delta$ of $J_\Delta$ is the unique maximizer of $x\mapsto(R+x)^2/E_\Delta(x)$ on $[0,R]$.
\end{proof}

\begingroup\mdfsetup{nobreak=true}\begin{lemma}[A scalar radius bound]\label{lmm:pos:radius-scalar}
    Let $0\leq a\leq R<\pi/2$, put $t:=\arccos(\cos(R)/\cos(a))$ and $u:=(R+a)/2$. Then
    \begin{equation}\label{eq:pos:scalar:radius}
        t\cot(t)\br{\frac{\sin(u)}u}^2\geq\cos(R)
        \eqcm
    \end{equation}
    with all ratios at zero understood by continuity. The inequality is strict for $R>0$.
\end{lemma}\endgroup
\begin{proof}
    Put $v:=(R-a)/2$, so that $a=u-v$ and $R=u+v$ with $0\leq v\leq u<\pi/2$. Then
    \begin{equation*}
        \begin{split}
            \cos(a)\,\frac{\sin(t)}{t}
            &\leq\frac{2\cos(a)+\cos(R)}{3}
            =\frac{3\cos(u)\cos(v)+\sin(u)\sin(v)}{3}\\
            &\leq\frac{\sqrt{1+8\cos(u)^2}}3
            \leq\frac{1+2\cos(u)}{3}
            \leq\br{\frac{\sin(u)}{u}}^2.
        \end{split}
    \end{equation*}
    The first step uses $\cos(t)=\cos(R)/\cos(a)$ together with $\sin(t)/t\leq(2+\cos(t))/3$: the difference
    $t(2+\cos(t))-3\sin(t)$ and its first derivative vanish at zero, and its second derivative is $\sin(t)-t\cos(t)\geq0$.
    The next two inequalities follow from Cauchy--Schwarz and squaring, respectively.
    For the last, the alternating Taylor bounds on $[0,\pi/2]$ give
    \begin{equation*}
        \br{\frac{\sin(u)}{u}}^2-\frac{1+2\cos(u)}{3}
        \geq
        \br{1-\frac{u^2}3+\frac{2u^4}{45}-\frac{u^6}{315}}
        -\br{1-\frac{u^2}3+\frac{u^4}{36}}
        =u^4\br{\frac1{60}-\frac{u^2}{315}}\geq0
        \eqcm
    \end{equation*}
    which is strict for $u>0$, hence for $R>0$.
    For $t>0$, this proves \eqref{eq:pos:scalar:radius} by the identity
    \[
      t\cot(t)=\frac{\cos(R)}{\cos(a)\sin(t)/t}.
    \]
    At $t=0$, that is, at $a=R$ and $u=R$, the required bound reads $\cos(R)\leq(\sin(R)/R)^2$, which follows from $\cos(R)\leq(1+2\cos(R))/3$ and the last inequality of the chain.
\end{proof}

\begin{lemma}[Sharpness of the radius constant]\label{lmm:pos:vi:sharp:radius}
    Let $0<\Delta<\pi$, let $0\leq x\leq\Delta/2$, and let $E_\Delta$ be as in \eqref{eq:pos:vi:radius:functions}.
    On the unit sphere in $\R^3$, a compact \catof{1} space, there are a closed ball $\setC=\overline B(o,\Delta/2)$, which is convex with circumradius exactly $\Delta/2$, a measure $\mu\in\mc P(\setC)$ and a point $q\in\setC$ such that $m:=m_2(\mu)$ satisfies $\ol mo=x$, $\ol mq=\Delta/2+x$, and
    \begin{equation}\label{eq:pos:vi:sharp:radius}
        \ol mq^2=\frac{(\Delta/2+x)^2}{E_\Delta(x)}\int\br{\ol yq^2-\ol ym^2}\,\dl\mu(y)
        \eqfs
    \end{equation}
    In particular, \eqref{eq:pos:vi:radius:optimal} gives $A_\Delta^{\mathrm{rad}}$ at $x=x_\Delta$ and \eqref{eq:pos:vi:near:constant} gives $A_\Delta^{\mathrm{near}}$ at $x=0$.
\end{lemma}
\begin{proof}
    Write $R:=\Delta/2$ and $t:=t_\Delta(x)$, with $t_\Delta$ and $s_\Delta$ as in \eqref{eq:pos:vi:radius:ts}, and take the following points on the unit sphere in $\R^3$:
    \begin{equation*}
        \begin{split}
            m&=(0,0,1),\qquad o=(\sin(x),0,\cos(x)),\\
            y_\pm&=(0,\pm\sin(t),\cos(t)),\qquad
            q=(\sin(R+x),0,\cos(R+x)).
        \end{split}
    \end{equation*}
    The closed spherical ball $\setC=\overline B(o,R)$ is convex and has circumradius exactly $R$, since its diameter is $2R$.
    The points $y_\pm$ and $q$ lie on its boundary.
    Put $\mu:=(\delta_{y_+}+\delta_{y_-})/2$ and $F(p):=\int\ol yp^2\,\dl\mu(y)$.
    Since $m$ is the midpoint of the segment from $y_-$ to $y_+$ of length $2t<\pi$, it is the mean of $\mu$.
    Indeed, for every $p$ on the sphere, the triangle inequality gives
    \begin{equation*}
        F(p)=\tfrac12\br{\ol{y_+}p^2+\ol{y_-}p^2}
        \geq\tfrac14\ol{y_+}{y_-}^2=t^2=F(m)
        \eqfs
    \end{equation*}
    Equality forces $\ol{y_+}p=\ol{y_-}p=t$, so $p$ is a midpoint of $y_-$ and $y_+$, which is $m$ because the geodesic between them is unique, its length $2t$ being less than $\pi$; when $t=0$, $\mu$ is the point mass at $m$.
    Consequently
    \begin{equation*}
        F(m)=t^2,\qquad F(q)=s_\Delta(x)^2,\qquad
        \frac{\ol mq^2}{F(q)-F(m)}
        =\frac{(R+x)^2}{E_\Delta(x)}
        \eqcm
    \end{equation*}
    which is \eqref{eq:pos:vi:sharp:radius}.
\end{proof}

\begin{lemma}[Sharpness of the diameter constant]\label{lmm:pos:vi:sharp:diameter}
    Let $0<\delta<\pi/2$, and let $\theta_\delta$ and $A_\delta^{\mathrm{diam}}$ be as in \eqref{eq:pos:vi:constants}.
    There are compact \catof{1} spaces $\mc Q_n$, $n\geq2$, realized in a common Hilbert sphere, each of diameter exactly $\delta$ and circumradius at most $\theta_\delta<\pi/2$, measures $\mu_n\in\mc P(\mc Q_n)$, and fixed points $m,q\in\bigcap_{n\geq2}\mc Q_n$ such that $m_2(\mu_n)=m$ and
    \begin{equation*}
        \frac{\ol mq^2}{\int\br{\ol yq^2-\ol ym^2}\,\dl\mu_n(y)}
        \longrightarrow A_\delta^{\mathrm{diam}}
        \qquad(n\to\infty)
        \eqfs
    \end{equation*}
\end{lemma}
\begin{proof}
    Write $\theta:=\theta_\delta$ and use $g$ from \cref{lmm:pos:cosine-average}.
    In $\ell^2$, fix orthonormal vectors $m,e$.
    For each $n\geq2$, choose a regular simplex of unit vectors $v_1,\ldots,v_n$ perpendicular to both, with $\sum_i v_i=0$ and $\ip{v_i}{v_j}=-1/(n-1)$ for $i\ne j$.
    Put
    \begin{equation*}
        q=\cos(\theta)\,m+\sin(\theta)\,e,\qquad
        y_i=\cos(t_n)\,m+\sin(t_n)\,v_i,\qquad
        \cos(t_n)^2=\frac{1+(n-1)\cos(\delta)}{n},
    \end{equation*}
    where $0<t_n<\theta$.
    Let $\mu_n$ be uniform on the $y_i$, with $F_n(p):=\int\ol yp^2\,\dl\mu_n(y)$.
    The spherical convex hull of the $y_i$ together with $q$ is a compact \catof{1} space $\mc Q_n$.
    All vertex inner products are at least $\cos(\delta)$, and normalization of nonnegative linear combinations preserves this lower bound.
    Hence $\diam(\mc Q_n)=\delta$, with equality already between two $y_i$.
    Also $\sum_i y_i=n\cos(t_n)m$, so $m\in\mc Q_n$.
    All vertices lie in $\overline B(m,\theta)$, which is convex, so the circumradius of $\mc Q_n$ is at most $\theta<\pi/2$.
    Convexity and monotonicity of $g$ give, for every $p\in\mc Q_n$,
    \begin{equation*}
        F_n(p)=\frac1n\sum_{i=1}^{n} g(\ip{y_i}p)
        \geq g(\cos(t_n)\ip mp)\geq t_n^2,
    \end{equation*}
    with equality only if $\ip mp=1$, since $g$ is strictly decreasing; so $m$ is its unique mean, independently of $n$, and the same bound for every $p$ on the whole sphere shows that $m$ is also the unique mean there.
    As $n\to\infty$, $\cos(t_n)^2=\cos(\delta)+(1-\cos(\delta))/n$ decreases to $\cos(\theta)^2$, so $t_n\nearrow\theta$ and therefore
    \begin{equation*}
        \frac{\ol mq^2}{F_n(q)-F_n(m)}
        =\frac{\theta^2}{g(\cos(t_n)\cos(\theta))-t_n^2}
        \longrightarrow\frac{\theta^2}{g(\cos(\theta)^2)-\theta^2}
        =\frac{\theta^2}{\delta^2-\theta^2}
        =A_\delta^{\mathrm{diam}}
        \eqcm
    \end{equation*}
    using $\cos(\theta)^2=\cos(\delta)$ in the last step but one.
\end{proof}

\Needspace{10\baselineskip}
\subsection{Variance Comparison}

\begin{lemma}[Quadratic variances under positive curvature]\label{lmm:pos:variances}
    In the setting of \cref{ssec:pos:prelim}, let $Y\sim\mu\in\mc P(\setC)$.
    If \eqref{eq:2fmcatk:vi} holds for $\mu$ with a constant $b>0$ in place of $C_\Delta$, then
    \begin{equation*}
        \VfOf{Y}\leq\frac{2b}{1+b}\,\VmOf{Y}
        \qquad\text{and}\qquad
        \VmOf{Y}\leq2\,\VfOf{Y}
        \eqfs
    \end{equation*}
    In particular, $b=C_\Delta$ gives \eqref{eq:pos:variances}, and $b=A_{\Delta,\delta}$ gives the sharper comparison in \cref{rem:pos:sharper}.
\end{lemma}
\begin{proof}
    Let $Y\pr$ be an independent copy of $Y$ and put $m=m_2(\mu)$.
    Integrating the assumed variance inequality with respect to $q=Y\pr$ gives
    \begin{equation*}
        \VfOf{Y}=\EOf{\ol m{Y\pr}^2}
        \leq b\,\EOf{\ol Y{Y\pr}^2-\ol Ym^2}
        =b\br{2\VmOf{Y}-\VfOf{Y}}.
    \end{equation*}
    Rearranging proves the first bound; the second follows from \cref{lmm:variances} with $s=2$.
\end{proof}

\subsection{The Contraction Property}

The contraction property already follows from the variance and quadruple inequalities.
\begin{lemma}[Product bound]\label{lmm:pos:cp:product}
    In the setting of \cref{ssec:pos:prelim}, all $\mu,\nu\in\mc P(\setC)$ satisfy
    \begin{equation}\label{eq:pos:cp:product}
        \ol{m_2(\mu)}{m_2(\nu)}
        \leq C_\delta\min\{A_\Delta^{\mathrm{rad}},A_\delta^{\mathrm{diam}}\}\,W_1(\mu,\nu)
        \eqfs
    \end{equation}
\end{lemma}
\begin{proof}
    Put $m_\mu:=m_2(\mu)$, $m_\nu:=m_2(\nu)$, and $A:=\min\{A_\Delta^{\mathrm{rad}},A_\delta^{\mathrm{diam}}\}$; the case $m_\mu=m_\nu$ is trivial.
    Let $\pi$ be a coupling of $\mu$ and $\nu$.
    \Cref{thm:pos:hemisphere:vi} for $\mu$ at $q=m_\nu$ and for $\nu$ at $q=m_\mu$, added, and \cref{prp:prelim:qi} with $\alpha=2$ give
    \begin{equation*}
        2\,\ol{m_\mu}{m_\nu}^2
        \leq A\int\br{\ol y{m_\nu}^2-\ol y{m_\mu}^2-\ol z{m_\nu}^2+\ol z{m_\mu}^2}\dl\pi(y,z)
        \leq 2AC_\delta\,\ol{m_\mu}{m_\nu}\int\ol yz\,\dl\pi(y,z)
        \eqcm
    \end{equation*}
    because all four points lie in the convex set $\setC$, so that the geodesic segments between them have $\sqrt\kappa\,\diam(\seg yz\cup\seg{m_\mu}{m_\nu})\leq\delta$, and $C_x$ is increasing.
    Dividing by $2\,\ol{m_\mu}{m_\nu}$ and taking the infimum over $\pi$ proves \eqref{eq:pos:cp:product}.
\end{proof}
This uses the strongest global variance bounds established above and the quadruple constant $C_\delta$.
Using only the radius gives $C_\Delta A_\Delta^{\mathrm{rad}}$, and using only the diameter gives $C_\delta A_\delta^{\mathrm{diam}}$ when $\delta<\pi/2$.
These products do not in general give the optimal contraction constant depending only on the radius or only on the diameter.
The following theorem determines these two constants.

The main idea is to retain the dependence of both comparison estimates on the location of the mean.
Starting from a coupling, move each observation along the geodesic joining its two endpoints, obtaining a path of probability measures.
For nearby measures on this path, combine the cosine form of the variance inequality from \cref{ssec:pos:vi} with the quadruple inequality applied to the short observation and mean segments.
This gives a local contraction estimate.
The geometric bounds from \cref{ssec:pos:vi} control its coefficient, and summing over small steps gives the contraction property for the original measures.

\begin{theorem}[Contraction of Fr\'echet means]\label{thm:pos:hemisphere:cp}
    In the setting of \cref{ssec:pos:prelim}, suppose $0<\Delta<\pi$, and let $t_\Delta(x)=\arccos(\cos(\Delta/2)/\cos(x))$ be as in \eqref{eq:pos:vi:radius:ts}.
    Define
    \begin{equation}\label{eq:pos:cp:radius:optimal}
        B_\Delta^{\mathrm{rad}}
        :=\max_{0\leq x\leq\Delta/2}
        C_{\Delta/2+x}\,\frac{\tan(t_\Delta(x))}{t_\Delta(x)},
    \end{equation}
    with $\tan(t)/t:=1$ at zero, and
    \begin{equation}\label{eq:pos:cp:diameter:optimal}
        B_\delta^{\mathrm{diam}}:=
        \begin{cases}
            1/\sqrt{\cos(\delta)},&0<\delta<\pi/2,\\
            +\infty,&\pi/2\leq\delta<\pi.
        \end{cases}
    \end{equation}
    Set $B_0^{\mathrm{rad}}=B_0^{\mathrm{diam}}:=1$ by continuity.
    Then, for all $\mu,\nu\in\mc P(\setC)$,
    \begin{equation}\label{eq:pos:cp}
        \ol{m_2(\mu)}{m_2(\nu)}
        \leq\min\{B_\Delta^{\mathrm{rad}},B_\delta^{\mathrm{diam}}\}\,W_1(\mu,\nu)
        \eqfs
    \end{equation}
    The two constants are optimal among constants depending only on $\Delta$ and only on $\delta$, respectively, uniformly over all such spaces, sets, and measures, at each positive value of the corresponding parameter.
    The radius constant satisfies
    \begin{equation}\label{eq:pos:cp:radius:upper}
        C_\Delta<B_\Delta^{\mathrm{rad}}<\min\cb{\frac1{\cos(\Delta/2)^2},\,2C_\Delta}
        \qquad\text{and}\qquad
        \frac1{\cos(\Delta/2)^2}<C_\Delta^2
        \qquad(0<\Delta<\pi).
    \end{equation}
    The maximum in \eqref{eq:pos:cp:radius:optimal} is attained in $(0,\Delta/2)$; its maximizer need not be the $x_\Delta$ in \eqref{eq:pos:vi:radius:root}.
\end{theorem}
\begin{proof}
    For this proof and the following lemmas, we write $C(x):=C_x$ for the factor \eqref{eq:pos:const:C}.
    Rescaling permits $\kappa=1$, so that $\Delta=2\varrho$ and $\delta=\diam(\setC)<\pi$.
    For $\eta\in\mc P(\setC)$, write $m_\eta:=m_2(\eta)$.
    Choose $R\in(\varrho,\pi/2)$ and $o\in\mc Q$ such that $\setC\subset B(o,R)$; the estimates of the proof of \cref{thm:pos:hemisphere:vi} used below refer to this enclosing ball.

    \emph{Transport interpolation and local comparison.}
    Fix a coupling $\varpi$ of $\mu,\nu$ with cost $I:=\int\ol yz\,\dl\varpi(y,z)$; the case $I=0$ is immediate.
    Let $\gamma_{y,z}$ be the constant-speed geodesic from $y$ to $z$, and let $\eta_s$ be the distribution of $\gamma_{y,z}(s)$ under $\varpi$, for $s\in[0,1]$.
    These geodesics are unique, depend continuously on their endpoints, and stay in $\setC$.
    Put $m_s:=m_{\eta_s}$.
    The coupling $(\gamma_{y,z}(s),\gamma_{y,z}(t))_\#\varpi$ of $\eta_s$ and $\eta_t$ gives the first bound below, and the preliminary bound \eqref{eq:pos:cp:product} of \cref{lmm:pos:cp:product}, with $L_0:=C(\delta)\min\{A_\Delta^{\mathrm{rad}},A_\delta^{\mathrm{diam}}\}<\infty$, then gives the second:
    \begin{equation*}
        W_1(\eta_s,\eta_t)\leq I|t-s|,
        \qquad \ol{m_s}{m_t}\leq L_0I|t-s|.
    \end{equation*}
    As in \eqref{eq:pos:curved:AU}, define
    \begin{equation*}
        K_\eta:=\int\ol y{m_\eta}\cot(\ol y{m_\eta})\,\dl\eta(y),
    \end{equation*}
    with $r\cot(r):=1$ at zero.
    The cosine form \eqref{eq:pos:curved:intermediate} of the variance inequality gives
    \begin{equation}\label{eq:pos:cp:cosine}
        \int\br{\ol yq^2-\ol y{m_\eta}^2}\,\dl\eta(y)
        \geq2K_\eta\br{1-\cos(\ol{m_\eta}q)}.
    \end{equation}
    To apply the quadruple inequality to two nearby times, take $0\leq s<t\leq1$.
    Since $I\leq\delta$, the relevant segments satisfy
    \begin{equation*}
        \diam\bigl(\seg{\gamma_{y,z}(s)}{\gamma_{y,z}(t)}
                   \cup\seg{m_s}{m_t}\bigr)
        \leq\sup_{v\in\setC}\ol{m_s}v+(1+L_0)\delta(t-s)<\pi
    \end{equation*}
    for all sufficiently small $t-s>0$, uniformly in $s,y,z$.
    Apply \eqref{eq:pos:cp:cosine} to $\eta_s,\eta_t$, using each mean as the other measure's comparison point, and add.
    Integrating the quadruple inequality of \cref{prp:prelim:qi} against $\varpi$ then yields
    \begin{equation}\label{eq:pos:cp:local:transport}
        \begin{split}
            &2(K_{\eta_s}+K_{\eta_t})\br{1-\cos(\ol{m_s}{m_t})}\\
            &\qquad\leq2C\br{\sup_{v\in\setC}\ol{m_s}v+(1+L_0)\delta(t-s)}\,\ol{m_s}{m_t}\,I(t-s).
        \end{split}
    \end{equation}
    The function $r\mapsto r\cot(r)$ is continuously differentiable on $[0,\delta]$, and both the observation paths and the mean path are uniformly Lipschitz.
    Thus $s\mapsto K_{\eta_s}$ is Lipschitz.
    It is positive by \eqref{eq:pos:curved:A-lower}, hence bounded away from zero on $[0,1]$.
    For $m_s\ne m_t$, divide \eqref{eq:pos:cp:local:transport} and apply the half-angle identity
    \begin{equation*}
        2(1-\cos(x))=4\sin(x/2)^2
        =x^2\br{\frac{\sin(x/2)}{x/2}}^2
        \qquad(x>0)
    \end{equation*}
    with $x=\ol{m_s}{m_t}$ to obtain
    \begin{equation}\label{eq:pos:cp:local:ratio}
        \ol{m_s}{m_t}\leq\br{\frac{C\br{\sup_{v\in\setC}\ol{m_s}v}}{K_{\eta_s}}+\mathrm{o}(1)}I(t-s),
    \end{equation}
    where $\mathrm{o}(1)\to0$ as $t-s\searrow0$, uniformly for $0\leq s<t\leq1$.
    The same bound holds when $m_s=m_t$.

    \emph{Geometric bounds and summation.}
    For $x:=\ol o{m_\eta}$, \eqref{eq:pos:curved:A-lower} and the triangle inequality give
    \begin{equation*}
        K_\eta\geq t_{2R}(x)\cot(t_{2R}(x)),\qquad
        \sup_{v\in\setC}\ol{m_\eta}v\leq R+x.
    \end{equation*}
    Since $C(u)$ increases with $u$, this bounds $C(\sup_{v\in\setC}\ol{m_\eta}v)/K_\eta$ by $B_{2R}^{\mathrm{rad}}$.
    If $\delta<\pi/2$, combine the bound $K\geq z/H(z)$ in \eqref{eq:pos:curved:A-lower} with \eqref{eq:pos:curved:diameter} and the monotonicity of $z/H(z)$ from \cref{lmm:pos:cosine-average} to obtain
    \begin{equation}\label{eq:pos:cp:local:diameter}
        \begin{split}
            K_\eta&\geq T_\delta\brOf{\sup_{v\in\setC}\ol{m_\eta}v}\cot\brOf{T_\delta\brOf{\sup_{v\in\setC}\ol{m_\eta}v}},\\
            T_\delta(\ell)&:=\arccos\brOf{\max\{\sqrt{\cos(\delta)},\cos(\delta)/\cos(\ell)\}}.
        \end{split}
    \end{equation}
    The comparison holds for every $q\in\setC$, so taking the supremum of $\ol{m_\eta}q$ gives this bound without requiring a farthest point.
    By \cref{lmm:pos:cp:scalar}, it bounds $C(\sup_{v\in\setC}\ol{m_\eta}v)/K_\eta$ by $B_\delta^{\mathrm{diam}}$.
    Therefore \eqref{eq:pos:cp:local:ratio} gives
    \begin{equation*}
        \ol{m_s}{m_t}
        \leq\br{\min\{B_{2R}^{\mathrm{rad}},B_\delta^{\mathrm{diam}}\}+\mathrm{o}(1)}I(t-s),
    \end{equation*}
    where again $\mathrm{o}(1)\to0$ as $t-s\searrow0$, uniformly for $0\leq s<t\leq1$.
    This also covers $B_\delta^{\mathrm{diam}}=+\infty$.
    Partition $[0,1]$ into equal intervals, sum, and pass to the limit to get
    \begin{equation*}
        \ol{m_\mu}{m_\nu}
        \leq\min\{B_{2R}^{\mathrm{rad}},B_\delta^{\mathrm{diam}}\}\,I.
    \end{equation*}
    Take the infimum over couplings and let $R\searrow\varrho$.
    Continuity of the radius maximum follows by writing $x=Ru$, $u\in[0,1]$, as in the variance proof.
    This proves \eqref{eq:pos:cp}.
    The remaining claims follow from \cref{lmm:pos:cp:scalar,lmm:pos:cp:sharp:radius,lmm:pos:cp:sharp:diameter}.
\end{proof}

\Needspace{14\baselineskip}
\begin{lemma}[Scalar contraction constants]\label{lmm:pos:cp:scalar}
    For $0<\delta<\pi/2$ and $T_\delta$ from \eqref{eq:pos:cp:local:diameter},
    \begin{equation*}
        \max_{0\leq \ell\leq\delta}
        \frac{C(\ell)}{T_\delta(\ell)\cot(T_\delta(\ell))}
        =\frac1{\sqrt{\cos(\delta)}}.
    \end{equation*}
    The radius constant satisfies \eqref{eq:pos:cp:radius:upper}, and every maximizer in \eqref{eq:pos:cp:radius:optimal} lies in $(0,\Delta/2)$.
    Also $B_\Delta^{\mathrm{rad}}\to1$ as $\Delta\searrow0$ and $B_\delta^{\mathrm{diam}}\to1$ as $\delta\searrow0$.
\end{lemma}
\begin{proof}
    Recall from the proof of \cref{lmm:pos:cosine-conversion} that
    $k(u):=\cot(u)/u-\cot(u)^2$, with $k(0):=1/3$, is strictly decreasing on $(0,\pi/2)$.

    For the diameter maximum, put $\theta:=\theta_\delta$.
    On $0\leq \ell\leq\theta$, $T_\delta(\ell)=\theta$ and $C(\ell)$ increases, giving the maximum $C(\theta)/(\theta\cot(\theta))=1/\cos(\theta)$.
    For $\theta<\ell<\delta$, write $t:=T_\delta(\ell)$.
    Differentiating $\cos(t)=\cos(\delta)/\cos(\ell)$ gives $t'=-\cot(t)\tan(\ell)$, hence
    \begin{equation*}
        \frac{\dl}{\dl \ell}\log\frac{C(\ell)}{t\cot(t)}
        =\tan(\ell)\,\br{k(\ell)+k(t)-1}<0.
    \end{equation*}
    The sign follows from $k(\ell),k(t)\leq1/3$.
    The maximum is therefore $1/\cos(\theta)=1/\sqrt{\cos(\delta)}$.

    For the radius, put $R:=\Delta/2$ and
    $f_{2R}(x):=C(R+x)/(t_{2R}(x)\cot(t_{2R}(x)))$.
    Since $C(u)$ increases, $t\cot(t)$ decreases, and $0\leq t_{2R}(x)\leq R$,
    \begin{equation*}
        f_{2R}(x)\leq\frac{C(2R)}{R\cot(R)}=\frac1{\cos(R)^2}.
    \end{equation*}
    Equality in the numerator bound requires $x=R$, whereas equality in the denominator bound requires $x=0$.
    For $R>0$, compactness therefore makes the upper bound strict.
    Also $1/\cos(R)^2<C(2R)^2$ follows from $\sin(R)<R$.
    For the second upper bound, $t:=t_{2R}(x)$ satisfies $\cos(t)=\cos(R)/\cos(x)$, so $t\cot(t)\geq\cos(t)$ gives $\tan(t)/t\leq\cos(x)/\cos(R)$; with $\sin(R+x)\geq\sin(R)\cos(x)$, equality only at $x=0$, this yields $f_{2R}(x)\leq(R+x)/(\sin(R)\cos(R))\leq4R/\sin(2R)=2C(2R)$, equality in the last step only at $x=R$. Hence $B_\Delta^{\mathrm{rad}}<2C_\Delta$.
    Using $t_{2R}'(x)=-\cot(t_{2R}(x))\tan(x)$ as above gives
    \begin{equation*}
        \frac{f_{2R}'(x)}{f_{2R}(x)}
        =\frac1{R+x}-\cot(R+x)-\tan(x)\,\br{1-k(t_{2R}(x))}.
    \end{equation*}
    Its continuous endpoint value at $x=0$ is positive, whereas at $x=R$ it is
    $(3-3R\cot(R)-R\tan(R))/(6R)<0$.
    For the latter sign, put $T:=\tan(R)>0$; multiplying by $6RT$ shows that the sign condition is equivalent to $R(T^2+3)>3T$, that is, to $\arctan(T)>3T/(T^2+3)$, and the function
    $\arctan(T)-3T/(T^2+3)$ vanishes at zero and has derivative
    $4T^4/((1+T^2)(T^2+3)^2)>0$.
    Thus every maximizer is interior, and the maximum strictly exceeds $f_{2R}(R)=C(2R)$.
    This proves \eqref{eq:pos:cp:radius:upper}; squeezing gives the limit at zero.
    The diameter limit is immediate from its formula.
\end{proof}

\Needspace{12\baselineskip}
\begin{lemma}[Spherical sensitivity]\label{lmm:pos:cp:sensitivity}
    Let $n\geq2$ and $N\geq n+2$. On the unit sphere of $\R^N$, let $m,e$ be orthonormal and let $a_1,\ldots,a_n$ be unit vectors orthogonal to $m,e$ that form a regular simplex with zero sum.
    Fix $0\leq t<\pi/2$ and $0<\ell<\pi$, and set
    \begin{equation*}
        y_j:=\cos(t)\,m+\sin(t)\,a_j,
        \qquad \eta_0:=\frac1n\sum_{j=1}^n\delta_{y_j}.
    \end{equation*}
    Let $p(u)$ be a smooth spherical curve, defined near zero and orthogonal to every $a_j$, such that
    $p(0)=\cos(\ell)\,m+\sin(\ell)\,e$ and $p'(0)\ne0$ is orthogonal to $m,e$.
    For $0<\epsilon<1$, set $\eta_{\epsilon,u}:=(1-\epsilon)\eta_0+\epsilon\delta_{p(u)}$. Then
    \begin{equation}\label{eq:pos:cp:sensitivity}
        \lim_{\epsilon\searrow0}\lim_{u\to0}
        \frac{\ol{m_2(\eta_{\epsilon,-u})}{m_2(\eta_{\epsilon,u})}}
             {W_1(\eta_{\epsilon,-u},\eta_{\epsilon,u})}
        =\frac{C(\ell)}{t\cot(t)},
    \end{equation}
    with $t\cot(t):=1$ at zero.
\end{lemma}
\begin{proof}
    \emph{Localization.}
    Choose $\omega\in(0,\pi/2)$ such that $\absof{\ell-\omega}<\pi/2$, and put $c:=\cos(\omega)\,m+\sin(\omega)\,e$.
    Then $\ip c{y_j}=\cos(\omega)\cos(t)>0$ and $\ip c{p(0)}=\cos(\ell-\omega)>0$, so for sufficiently small $\absof u$ all observations lie in a closed spherical ball $\overline B(c,r)$ with $r<\pi/2$, which is convex with circumradius at most $r$.
    The means of $\eta_0$ and $\eta_{\epsilon,u}$ are therefore unique and lie in this ball, and \cref{lmm:pos:cp:product} applies there.
    The Jensen argument from \cref{lmm:pos:vi:sharp:diameter} gives $m_2(\eta_0)=m$, and \eqref{eq:pos:cp:product} with $W_1(\eta_{\epsilon,u},\eta_0)\leq\epsilon\pi$ shows that $m_2(\eta_{\epsilon,u})\to m$ as $\epsilon\to0$, uniformly for small $\absof u$.

    \emph{Symmetry reduction.}
    Put $\ell_u:=\ol m{p(u)}$ and $e(u):=(p(u)-\cos(\ell_u)m)/\sin(\ell_u)$, so that $e(u)$ is a unit vector orthogonal to $m$ and to every $a_j$, and $p(u)=\cos(\ell_u)m+\sin(\ell_u)e(u)$.
    Let $V:=\operatorname{span}\{a_1,\ldots,a_n\}$ and $E:=\operatorname{span}\{m,e(u)\}$, which are orthogonal.
    Every orthogonal map of $\R^N$ that fixes $E$ pointwise and maps $V$ onto itself by permuting the $a_j$ preserves $\eta_{\epsilon,u}$, and so does every reflection in a hyperplane containing $V+E$.
    By uniqueness, these maps fix $m_2(\eta_{\epsilon,u})$.
    The permutations of the $a_j$ fix no nonzero vector of $V$, since a fixed vector $w\in V$ has $\ip w{a_j}$ independent of $j$, hence $\ip w{a_j}=\frac1n\ip w{\sum_ia_i}=0$ for all $j$; and the reflections remove every component orthogonal to $V+E$.
    Hence $m_2(\eta_{\epsilon,u})\in E$, that is, $m_2(\eta_{\epsilon,u})=\cos(\vartheta)m+\sin(\vartheta)e(u)$ for an angle $\vartheta$ near zero.

    \emph{Reduced objective.}
    For $q_\vartheta:=\cos(\vartheta)m+\sin(\vartheta)e(u)$, we have $\ip{y_j}{q_\vartheta}=\cos(t)\cos(\vartheta)$ and $\ol{p(u)}{q_\vartheta}=\absof{\ell_u-\vartheta}$, so $\vartheta$ minimizes, near zero,
    \begin{equation*}
        F_{\epsilon,\ell_u}(\vartheta):=(1-\epsilon)\arccos(\cos(t)\cos(\vartheta))^2
        +\epsilon(\ell_u-\vartheta)^2.
    \end{equation*}
    The first term is a smooth even function of $\vartheta$ whose derivative with respect to $\vartheta^2$ at zero is $t\cot(t)$, as in the proof of \cref{lmm:pos:cosine-conversion}\,(ii); so at $\epsilon=\vartheta=0$ the second derivative of $F_{\epsilon,\ell}$ in $\vartheta$ is $2t\cot(t)>0$, while $\partial_\epsilon\partial_\vartheta F_{\epsilon,\ell}(0)=-2\ell$.
    The implicit function theorem applied to $\partial_\vartheta F_{\epsilon,\ell}(\vartheta)=0$ therefore gives a smooth angular coordinate $\vartheta(\epsilon,\ell_u)$, with
    \begin{equation*}
        \vartheta_\epsilon:=\vartheta(\epsilon,\ell)
        =\frac{\epsilon \ell}{t\cot(t)}+\landauO(\epsilon^2)
        \qquad(\epsilon\searrow0).
    \end{equation*}

    \emph{Speed and transport cost.}
    Since $p'(0)$ is orthogonal to $m$ and $e$, the derivative of $u\mapsto\ell_u=\arccos(\ip m{p(u)})$ vanishes at zero, and $e'(0)=p'(0)/\sin(\ell)$.
    Consequently the curve of means
    $u\mapsto\cos(\vartheta(\epsilon,\ell_u))m+\sin(\vartheta(\epsilon,\ell_u))e(u)$
    has velocity $\sin(\vartheta_\epsilon)p'(0)/\sin(\ell)$ at zero, so that $\ol{m_2(\eta_{\epsilon,-u})}{m_2(\eta_{\epsilon,u})}=2\absof u\sin(\vartheta_\epsilon)\normof{p'(0)}/\sin(\ell)+\mathrm o(u)$.
    Moreover,
    \begin{equation*}
        W_1(\eta_{\epsilon,-u},\eta_{\epsilon,u})
        =\epsilon\,\ol{p(-u)}{p(u)}
        =2\epsilon\abs u\normOf{p'(0)}+\mathrm o(u)
        \eqcm
    \end{equation*}
    where retaining the common mass gives the upper bound in the first identity, and integrating the $1$-Lipschitz function $q\mapsto\ol q{p(-u)}$ against the two measures gives the lower bound.
    The inner limit in \eqref{eq:pos:cp:sensitivity} is thus $\sin(\vartheta_\epsilon)/(\epsilon\sin(\ell))$.
    Letting $\epsilon\searrow0$ proves the claim, since $\ell/(\sin(\ell)\,t\cot(t))=C(\ell)/(t\cot(t))$.
\end{proof}

\begin{lemma}[Radius sharpness for contraction]\label{lmm:pos:cp:sharp:radius}
    For every $0<R<\pi/2$, no constant smaller than $B_{2R}^{\mathrm{rad}}$ can hold in the contraction property for all probability measures on a closed ball of radius $R$ in the unit $3$-sphere.
    This ball has circumradius exactly $R$.
\end{lemma}
\begin{proof}
    Choose a maximizer $x\in(0,R)$ in \eqref{eq:pos:cp:radius:optimal}, and put $t:=t_{2R}(x)$.
    In \cref{lmm:pos:cp:sensitivity}, take $\ell=R+x$, an opposite pair $a_1=-a_2$, and a unit vector $v$ orthogonal to $m,e,a_1$.
    Put $o:=\cos(x)\,m+\sin(x)\,e$ and $\setC:=\overline B(o,R)$.
    The identities $\ip o{y_j}=\cos(x)\cos(t)=\cos(R)$ and $\ip o{p(0)}=\cos(R)$ place all observations on the boundary.
    With $w:=-\sin(x)\,m+\cos(x)\,e$, the curve
    \begin{equation*}
        p(u):=\cos(R)\,o+\sin(R)(\cos(u)\,w+\sin(u)\,v)
    \end{equation*}
    stays on that boundary, is orthogonal to $a_1,a_2$, and satisfies $p(0)=\cos(R+x)\,m+\sin(R+x)\,e$ and $p'(0)=\sin(R)\,v$.
    Hence \eqref{eq:pos:cp:sensitivity} approaches $B_{2R}^{\mathrm{rad}}$ within this fixed ball.
    The ball has diameter $2R$, so its circumradius is $R$.
\end{proof}

\begin{lemma}[Diameter sharpness for contraction]\label{lmm:pos:cp:sharp:diameter}
    For every $0<\delta<\pi/2$ and $L<B_\delta^{\mathrm{diam}}$, there are a compact \catof{1} space $\mc Q$ of diameter exactly $\delta$ and circumradius less than $\pi/2$, and distinct measures $\mu,\nu\in\mc P(\mc Q)$, such that
    $\ol{m_2(\mu)}{m_2(\nu)}>L\,W_1(\mu,\nu)$.
    The same conclusion holds for every finite $L$ when $\pi/2\leq\delta<\pi$.
\end{lemma}
\begin{proof}
    First suppose $0<\delta<\pi/2$, and put $\theta:=\theta_\delta$.
    In \cref{lmm:pos:cp:sensitivity}, use $n$ simplex directions and choose $t=t_n$ and $\ell=\theta$, where
    \begin{equation*}
        \cos(t_n)^2=\frac{1+(n-1)\cos(\delta)}n,
        \qquad
        p(u):=\cos(\theta)\,m+\sin(\theta)(\cos(u)\,e+\sin(u)\,v),
    \end{equation*}
    with $v$ a unit vector orthogonal to $m,e,a_1,\ldots,a_n$.
    Distinct base observations have distance $\delta$, and
    $\ip{y_j}{p(u)}=\cos(t_n)\cos(\theta)\geq\cos(\delta)$.
    For sufficiently small $u\ne0$, the spherical convex hull of the base observations and $p(-u),p(u)$ has diameter exactly $\delta$ and lies in $\overline B(m,\theta)$.
    Indeed, all vertex inner products are at least $\cos(\delta)>0$, and normalization of nonnegative combinations preserves this lower bound, as in \cref{lmm:pos:vi:sharp:diameter}.
    The choices are made in this order: first $n$, then $\epsilon$, then $u$; the space $\mc Q$ is the convex hull for this $u$, which depends on it, and $\mu,\nu$ are $\eta_{\epsilon,-u},\eta_{\epsilon,u}$.
    Since $\sum_jy_j=n\cos(t_n)m$, the hull contains $m$, so its circumradius is at most $\theta<\pi/2$; and by \cite[Theorem B and Corollary 16]{Yokota2016Barycenter-JMSJ}, the spherical means of $\eta_{\epsilon,\pm u}$ are unique and lie in $\mc Q$, so they are also the means in $\mc Q$.
    By \eqref{eq:pos:cp:sensitivity}, followed by $n\to\infty$, these choices give contraction ratios approaching
    \begin{equation*}
        \lim_{n\to\infty}\frac{C(\theta)}{t_n\cot(t_n)}
        =\frac1{\cos(\theta)}=B_\delta^{\mathrm{diam}}.
    \end{equation*}

    Now fix $\pi/2\leq\delta<\pi$ and take these examples at diameters $\delta_0\nearrow\pi/2$, with contraction ratios tending to infinity.
    Choose $\delta_0$ so that $2\theta_{\delta_0}\geq\delta$.
    The preceding convex hull contains $m$; attach an interval of length $\delta-\theta_{\delta_0}$ there.
    By the gluing theorem \cite[Theorem II.11.1]{BridsonHaefliger1999NPC-Springer}, the resulting space is compact \catof{1}.
    It has diameter exactly $\delta$, attained between the new endpoint and $p(u)$, and is contained in the ball of radius $\theta_{\delta_0}<\pi/2$ about $m$.
    At every point of the added interval, each squared-distance objective is at least its value at $m$, so the original means are unchanged.
    Distances between observations, and hence transport costs, are unchanged as well.
    Thus no finite contraction constant exists at this fixed diameter.
\end{proof}

\subsection{Obstructions at the Hemisphere Threshold}

For $N\geq2$, let
\begin{gather*}
    \setC_N:=\cb{x\in\R^{N+1}:\euclOf{x}=1,\ x_j\geq0\text{ for all }j}
    \eqcm\\
    \setC_\infty:=\cb{x\in\ell^2:\normOf{x}_{\ell^2}=1,\ x_j\geq0\text{ for all }j}
    \eqcm
\end{gather*}
with the angular distance $d(x,y)=\arccos(\ip xy)$, and put $R_N:=\arccos(1/\sqrt{N+1})$ and $\epsilon_N:=\pi/2-R_N=\arcsin(1/\sqrt{N+1})$.

\begin{proposition}[Obstructions to diameter-only bounds and at the threshold]\label{prp:pos:diameterobstruction}
    The following hold.
    \begin{enumerate}[label=(\roman*)]
        \item\label{itm:pos:obstruction:vi} For every fixed $\delta\in[\pi/2,\pi)$, no finite constant in the quadratic variance inequality holds uniformly over compact \catof{1} spaces of diameter exactly $\delta$ and circumradius less than $\pi/2$.
        \item\label{itm:pos:obstruction:risk} For such a $\delta$, there is no constant $C$ depending only on $\delta$ with $\Eof{d(\samnfm,m_2(\mu))^2}\leq C\VmofP{\mu}/n$ for all $n\in\N$, all such spaces, and all distributions $\mu$ on them, even for iid samples with uniform weights. For each fixed $n$, the best constant over this class is at least $2n$.
        \item\label{itm:pos:obstruction:order} Each $\setC_N$ is a compact \catof{1} space of diameter $\pi/2$ and circumradius $R_N$. As $N\to\infty$, the optimal radius variance constant $A^{\mathrm{rad}}_{2R_N}$ is of order $\epsilon_N^{-1}$; every coefficient $A_N$ with $\Eof{d(\samnfm,m_2(\mu))^2}\leq A_N\VmofP{\mu}/n$ for all $n\in\N$ and all $\mu\in\mc P(\setC_N)$, for iid samples with uniform weights, satisfies $A_N\geq N/4\asymp\epsilon_N^{-2}$; and the refined multiplier $C_\delta A_{\Delta,\delta}B_{\Delta,\delta}$ of \eqref{eq:pos:infty:refined} for $\setC_N$ is at most $\pi^2/(4\sin(\epsilon_N)^2)\sim\pi^2/(4\epsilon_N^2)$, whereas $C_\Delta A_\Delta^{\mathrm{rad}}B_\Delta^{\mathrm{rad}}$ is of exact order $\epsilon_N^{-3}$ and $C_\Delta^4$ of order $\epsilon_N^{-4}$.
        \item\label{itm:pos:obstruction:boundary} The space $\setC_\infty$ is a complete \catof{1} space of diameter and circumradius $\pi/2$. Every finitely supported $\mu\in\mc P(\setC_\infty)$ has a unique $2$-Fr\'echet mean, but neither a finite constant in the quadratic variance inequality nor a finite $C$ with $\Eof{d(\samnfm,m_2(\mu))^2}\leq C\VmofP{\mu}/n$ holds uniformly over all finitely supported $\mu\in\mc P(\setC_\infty)$ and $n\in\N$.
    \end{enumerate}
\end{proposition}
\begin{remark}[Reading \cref{prp:pos:diameterobstruction}]\label{rem:pos:obstruction}
    In \ref{itm:pos:obstruction:vi} and \ref{itm:pos:obstruction:risk}, the space must vary: in each individual space of the family, \cref{thm:pos:infty:independent} applies with its own finite factor $C_\Delta^4$.
    By \ref{itm:pos:obstruction:order}, the quadratic divergence of the risk coefficient as $\varrho\sqrt\kappa\nearrow\pi/2$ is order-sharp along a family of fixed diameter, and it is attained by the refined multiplier, which uses the diameter; how the best radius-only multiplier behaves is left open.
    By \ref{itm:pos:obstruction:boundary}, the hypothesis $\varrho<\pi/(2\sqrt\kappa)$ of \cref{ssec:pos:prelim} cannot be weakened to $\varrho\leq\pi/(2\sqrt\kappa)$, even in a fixed space.
    Rescaling gives the corresponding examples for every $\kappa>0$.
\end{remark}
\begin{proof}[Proof of \cref{prp:pos:diameterobstruction}]
    \emph{The orthant spaces.}
    Each $\setC_N$ is a compact geodesically convex subset of the unit sphere, hence a \catof{1} space, and $\diam(\setC_N)=\pi/2$.
    Its circumradius is $R_N<\pi/2$: the normalized all-ones vector $o_N$ gives the upper bound, while for any unit vector at least one coordinate is at most $1/\sqrt{N+1}$.
    Let $e_1,\dots,e_{N+1}$ be the coordinate vectors and
    \begin{equation*}
        \mu_N:=\frac1N\sum_{j=1}^N\delta_{e_j},\qquad
        m_N:=\frac1{\sqrt N}\sum_{j=1}^N e_j,\qquad
        q_t:=\cos(t)\,m_N+\sin(t)\,e_{N+1}.
    \end{equation*}
    The mean of $\mu_N$ is $m_N$, by Jensen's inequality for the decreasing convex function $g(s)=\arccos(s)^2$, followed by $\sum_{j=1}^N x_j\leq\sqrt N$.
    For $0\leq t\leq\pi/2$ we have $q_t\in \setC_N$, $d(q_t,m_N)=t$, and
    \begin{equation*}
        F_N(q_t)=\arccos\brOf{\cos(t)/\sqrt N}^{2},\qquad
        F_N(m_N)=\arccos\brOf{1/\sqrt N}^{2},
    \end{equation*}
    where $F_N(q):=\int d(y,q)^2\,\dl\mu_N(y)$.
    Consequently
    \begin{equation}\label{eq:pos:orthant:margin}
        \lim_{t\searrow0}\frac{F_N(q_t)-F_N(m_N)}{t^2}
        =\frac{\arccos\brOf{1/\sqrt N}}{\sqrt{N-1}}
        \longrightarrow0
        \qquad(N\to\infty).
    \end{equation}

    \emph{Proof of \ref{itm:pos:obstruction:vi}.}
    For $\delta=\pi/2$, \eqref{eq:pos:orthant:margin} rules out a positive quadratic-growth coefficient depending only on the diameter.
    For $\pi/2<\delta<\pi$, choose $N$ large enough that $2R_N\geq\delta$, and attach an interval of length $L_N:=\delta-R_N\leq R_N$ at $o_N$, with the path metric.
    Gluing \catof{1} spaces at a point preserves the \catof{1} property, by Reshetnyak's gluing theorem \cite[Theorem II.11.1]{BridsonHaefliger1999NPC-Springer}.
    The resulting compact space has circumradius $R_N$ and diameter $L_N+R_N=\delta$.
    Every observation $e_j$ has distance $R_N$ from $o_N$, so at a new point at distance $h$ along the interval,
    \begin{equation*}
        F_N(h)=(R_N+h)^2\geq F_N(o_N)\geq F_N(m_N).
    \end{equation*}
    Thus the mean and all variance ratios between original points are unchanged, and the same argument shows that the barycenter of every empirical distribution of observations from $\mu_N$ stays in $\setC_N$.

    \emph{Proof of \ref{itm:pos:obstruction:risk}.}
    Take $n$ iid observations from $\mu_N$, $1\leq n\leq N$, in $\setC_N$ or in the glued space.
    Every realization is supported on $\{e_j:j\in S\}$ for some $S$ with $\cardof S\leq n$, which lies in the open ball of radius $r$ about the normalized sum of these $e_j$ for every $r\in(\arccos(1/\sqrt{\cardof S}),\pi/2)$; so \cite[Corollary 16]{Yokota2016Barycenter-JMSJ} places the empirical barycenter in the closed convex hull of its atoms, and hence in the coordinate face $\{x\in \setC_N:x_j=0\text{ for }j\notin S\}$. Cauchy--Schwarz therefore gives $\ip{m_N}{\samnfm}\leq\sqrt{n/N}$, regardless of repeated observations. Consequently
    \begin{equation}\label{eq:pos:orthant:risk}
        d(\samnfm,m_N)^2\geq\arccos\brOf{\sqrt{n/N}}^2
        \quad\text{almost surely}.
    \end{equation}
    Taking $N=n^2$ and letting $n\to\infty$ makes the lower bound tend to $\pi^2/4$, whereas $\VmofP{\mu_N}=\frac{\pi^2}{8}(1-1/N)$ stays bounded, so no diameter-only coefficient is possible.
    Keeping $n$ fixed and letting $N\to\infty$ instead, the right-hand side of \eqref{eq:pos:orthant:risk} again tends to $\pi^2/4$, so any admissible coefficient obeys
    \begin{equation*}
      C\geq\frac{n\,\arccos\brOf{\sqrt{n/N}}^2}{\frac{\pi^2}8(1-1/N)}\longrightarrow2n
      \qquad(N\to\infty)
      \eqfs
    \end{equation*}

    \emph{Proof of \ref{itm:pos:obstruction:order}.}
    All asymptotic comparisons are as $N\to\infty$, and $\pi-2R_N=2\epsilon_N$.
    A valid variance constant for $\setC_N$ must be at least the reciprocal of \eqref{eq:pos:orthant:margin}, while $A^{\mathrm{rad}}_\Delta<1/\cos(\Delta/2)$ holds for every $\Delta$ by \eqref{eq:pos:vi:radius:upper}. Hence
    \begin{equation}\label{eq:pos:Arad:order}
        \frac{\sqrt{N-1}}{\arccos\brOf{1/\sqrt N}}
        \leq A^{\mathrm{rad}}_{2R_N}
        <\frac1{\sin(\epsilon_N)}
        \eqcm
    \end{equation}
    and both bounds are of order $\epsilon_N^{-1}$.
    For the risk coefficient, take $n=\lfloor N/2\rfloor$; then \eqref{eq:pos:orthant:risk} is at least $\pi^2/16$, and $\Eof{d(\samnfm,m_N)^2}\leq A_N\VmofP{\mu_N}/n$ forces
    \begin{equation*}
      A_N\geq\frac{n}{2(1-1/N)}\geq\frac N4
      \asymp\epsilon_N^{-2}.
    \end{equation*}
    Here $\delta=\pi/2$, so $A_\delta^{\mathrm{diam}}=+\infty$ and $A_{\Delta,\delta}=1/\cos(\Delta/2)$, and $\Delta=\pi-2\epsilon_N$.
    Since $B_{\Delta,\delta}\leq C_\delta A_{\Delta,\delta}$, the refined multiplier is therefore at most
    \begin{equation*}
      C_\delta^2 A_{\Delta,\delta}^2=\frac{C_\delta^2}{\cos(\Delta/2)^2}=\frac{\pi^2}{4\sin(\epsilon_N)^2}
      \sim\frac{\pi^2}{4\epsilon_N^2}.
    \end{equation*}
    Finally, $C_\Delta=(\pi-2\epsilon_N)/\sin(2\epsilon_N)\asymp\epsilon_N^{-1}$, so $C_\Delta^4\asymp\epsilon_N^{-4}$, and $C_\Delta A_\Delta^{\mathrm{rad}}B_\Delta^{\mathrm{rad}}$ is of exact order $\epsilon_N^{-3}$, by \eqref{eq:pos:Arad:order} and $C_\Delta<B_\Delta^{\mathrm{rad}}<2C_\Delta$ from \eqref{eq:pos:cp:radius:upper}.

    \emph{Proof of \ref{itm:pos:obstruction:boundary}.}
    Completeness of $\setC_\infty$ follows from equivalence of angular and chordal distances on the unit sphere, and the \catof{1} comparison reduces, for any triangle, to its finite-dimensional linear span.
    Its diameter is $\pi/2$, as all inner products are nonnegative, and so is its circumradius, because $x_j\to0$ as $j\to\infty$ for every $x\in\ell^2$, so that every point has distance arbitrarily close to $\pi/2$ from some $e_j$.
    The atoms of a finitely supported $\mu\in\mc P(\setC_\infty)$ have nonnegative inner products with each other, so they have positive inner product with their normalized sum and lie in an open ball of radius less than $\pi/2$; hence $\mu$ has a unique mean \cite[Theorem B]{Yokota2016Barycenter-JMSJ}.
    Each $\setC_N$ embeds isometrically into $\setC_\infty$ as a closed convex face, and the argument in the proof of \ref{itm:pos:obstruction:risk}, applied in $\setC_\infty$, shows that the means of $\mu_N$ and of its empirical distributions are the same as in $\setC_N$.
    Hence \eqref{eq:pos:orthant:margin} and \eqref{eq:pos:orthant:risk} hold in the fixed space $\setC_\infty$, which proves the claim.
\end{proof}

\section{Tools for Power Fr\'echet Means}\label{app:tools}
This appendix collects the geometric inequalities used in \cref{sec:two,sec:main}: the moment forms of the variance and contraction inequalities, refinements of the contraction property, and the comparison of the two variances. Unless explicitly stated otherwise, the space is Hadamard and $\alpha\in(1,2]$, with the notation of \cref{sec:prelim,ssec:power:prelim}.

\textbf{Notation.}
\begin{center}
\begin{tabular}{ll}
\toprule
$m_\alpha(\mu)$ & $\alpha$-Fr\'echet mean of $\mu$; \eqref{eq:power:fmdef}\\
$\phi=(2-\alpha)/(\alpha-1)$ & exponent for $\alpha<2$; \cref{sec:prelim}\\
$\sigma_\alpha^{\alpha-1}(\mu)$ & $(\alpha-1)$-moment about $m_\alpha(\mu)$; \eqref{eq:sigma:def}\\
$\Lamof{\mu}{t}$ & regularized inverse moment, \eqref{eq:power:Lambda}\\
$\Psi_\mu(t)=t\Lamof{\mu}{t}$ & contraction modulus; $\Psi_\mu(0)=0$; \cref{lmm:power:cp}\\
$W_{\alpha-1}$, $\widetilde W_{\alpha-1}$ & plain and location-aware costs; \eqref{eq:W1}, \eqref{eq:power:Wtilde}\\
$\VfpofP{s}{\mu}$, $\VmpofP{s}{\mu}$ & Fr\'echet and metric $s$-variances; \cref{ssec:prelim:moments}\\
\bottomrule
\end{tabular}
\end{center}

\subsection{Moment Forms and the Dirac Case}\label{app:power:moment}

The statistical proofs use \QIa, \VIa, and \CPa of \cref{ssec:power:prelim}, the latter two proved in \cref{lmm:power:vi,lmm:power:cp} below; recall $\alpha \in (1,2]$ and $\phi = \frac{2-\alpha}{\alpha-1}$.
A Dirac mass $\delta_q$, $q\in\mc Q$, lies in $\mc P^{\alpha-1}(\mc Q)$, has mean $m_\alpha(\delta_q) = q$ and inverse moment $\Lamof{\delta_q}{t} = t^{\alpha-2}$, and admits only the coupling $\mu\otimes\delta_q$, so that $\widetilde W_{\alpha-1}(\mu,\delta_q)\leq W_{\alpha-1}(\mu,\delta_q) = \Eof{\ol Yq^{\alpha-1}}$.
Hence \CPa with $\nu = \delta_q$, after dropping the nonnegative term $\ol mq\,\Lamof{\mu}{\ol mq}$, bounds the distance from the mean to any point by the $(\alpha-1)$-moment about that point:
\begin{equation}\label{eq:power:dirac}
	\ol mq^{\alpha-1}
	\leq
	\frac{2^{3-\alpha}}{\alpha-1}\,\EOf{\ol Yq^{\alpha-1}}
	\qquad\text{for all } q\in\mc Q
	\eqfs
\end{equation}
For $\alpha<2$, the moduli in \VIa and \CPa are inverse moments.
The statistical proofs use them through the ordinary moment $\sigma_\alpha^{\alpha-1}$, at the cost of one application of Jensen's inequality: since $t\mapsto t^{-\phi}$ is convex and $-\phi(\alpha-1)=\alpha-2$, Jensen's inequality and \cref{lmm:power:subadd:gen} give, for every $\eta\in\mc P^{\alpha-1}(\mc Q)$, every $t\in\Rp$, and $Y\sim\eta$,
\begin{equation}\label{eq:power:cp:jensen}
	\begin{aligned}
		\LamOf{\eta}{t}
		&=
		\EOf{\br{\br{\ol Y{m_\alpha(\eta)} + t}^{\alpha-1}}^{-\phi}}
		\\&\geq
		\EOf{\br{\ol Y{m_\alpha(\eta)} + t}^{\alpha-1}}^{-\phi}
		\geq
		\br{\sigma_\alpha^{\alpha-1}(\eta) + t^{\alpha-1}}^{-\phi}
		\eqfs
	\end{aligned}
\end{equation}
Inserting \eqref{eq:power:cp:jensen} into \VIa gives the \emph{moment form of the variance inequality} \cite[Lemma S7.8]{Schoetz2026Transformed-arxiv}: for all $q\in\mc Q$,
\begin{equation}\label{eq:power:vi:moment}
	\frac{\alpha(\alpha-1)}{2}\,
	\frac{\ol mq^2}{\br{\sigma_\alpha^{\alpha-1}(\mu) +\ol mq^{\alpha-1}}^{\phi}}
	\leq
	\EOf{\ol Yq^\alpha-\ol Ym^\alpha}
	\eqcm
\end{equation}
where the left-hand side is read as zero for $q=m$.
Inserting \eqref{eq:power:cp:jensen} into \CPa and using $u^{-1}+v^{-1}\geq4/(u+v)$ for $u,v\in\Rpp$ gives the \emph{moment form of the contraction property}:
\begin{equation}\label{eq:power:cp:moment}
	\begin{aligned}
	\frac{\ol{m_\alpha(\mu)}{m_\alpha(\nu)}}{\br{\sigma_\alpha^{\alpha-1}(\mu)+\ol{m_\alpha(\mu)}{m_\alpha(\nu)}^{\alpha-1}}^{\phi}+\br{\sigma_\alpha^{\alpha-1}(\nu)+\ol{m_\alpha(\mu)}{m_\alpha(\nu)}^{\alpha-1}}^{\phi}}
	&\leq
	\frac{2^{1-\alpha}}{\alpha-1}\,\widetilde W_{\alpha-1}(\mu,\nu)
	\\&\leq
	\frac{2^{1-\alpha}}{\alpha-1}\,W_{\alpha-1}(\mu,\nu)
	\eqcm
	\end{aligned}
\end{equation}
where the left-hand side is read as zero when $m_\alpha(\mu)=m_\alpha(\nu)$.
These moduli depend on $\sigma_\alpha^{\alpha-1}$ and on the distance: the one in \eqref{eq:power:vi:moment} is quadratic in $\ol mq$ below the distance scale $\sigma_\alpha^{\alpha-1}(\mu)^{1/(\alpha-1)}$ and of order $\ol mq^\alpha$ above it. They retain metric-rescaling homogeneity but are not distribution-independent quadratic moduli.
For $\alpha<2$, the layer-cake formula also gives the finite-inverse-moment criterion used in \cref{rem:power:smoothness}:
\begin{equation*}
 \LamOf{\mu}{0}<\infty
 \qquad\Longleftrightarrow\qquad
 \int_0^1 s^{\alpha-3}\,\mu\brOf{\overline B(m_\alpha(\mu),s)}\,\dl s<\infty
 \eqfs
\end{equation*}
The elementary inequalities for power functions used repeatedly are collected in \cref{sec:powfun}.
\subsection{The Contraction Property and Its Refinements}\label{app:power:cp}

\begin{lemma}[Power variance inequality]\label{lmm:power:vi}
	The property \VIa holds.
\end{lemma}
\begin{proof}
	This is the power case of the renormalized-objective theorem \cite[Theorem 5.4, p.~22]{Schoetz2025VarianceInequalities-EJP}; its hypothesis is $\Eof{\tau\pr(d(Y,o))}<\infty$ for $\tau(t)=t^\alpha$, that is, a finite $(\alpha-1)$-moment, and not finiteness of the unrenormalized objective; when $\alpha<2$, $\tau^{\prime\prime}(0)$ is not finite and is not used.
	The cited theorem evaluates the second-derivative modulus $\tau^{\prime\prime}/2$ at $\max\{\ol ym,\ol yq\}$.
	The triangle inequality bounds this argument by $\ol ym+\ol qm$, and the nonpositive exponent $\alpha-2$ gives, for $q\neq m$,
	\begin{equation*}
		\int\br{\ol yq^\alpha-\ol ym^\alpha}\,\dl\mu(y)
		\geq\frac{\alpha(\alpha-1)}2\,\ol qm^2\int\br{\ol ym+\ol qm}^{\alpha-2}\,\dl\mu(y)
		=\frac{\alpha(\alpha-1)}2\,\ol qm^2\,\LamOf{\mu}{\ol qm}
		\eqcm
	\end{equation*}
	whose left-hand side is finite by \cref{lmm:power:ana:diff}.
	This is \VIa; at $q=m$, both sides are zero, which is the continuous extension since $t^2\Lamof{\mu}{t}\leq t^\alpha$ for $t>0$.
\end{proof}
\begin{lemma}[The location-aware transport cost]\label{lmm:power:Wtilde}
    Let $\alpha\in(1,2]$ and $\mu,\nu\in\mc P^{\alpha-1}(\mc Q)$.
    \begin{enumerate}[label=(\roman*)]
        \item $\widetilde W_{\alpha-1}(\mu,\nu)=\widetilde W_{\alpha-1}(\nu,\mu)\leq W_{\alpha-1}(\mu,\nu)$, with equality $\widetilde W_1=W_1$ at $\alpha=2$; the infimum in \eqref{eq:power:Wtilde} is attained; and $\widetilde W_{\alpha-1}(\mu,\nu)=0$ if and only if $\mu=\nu$.
        \item For $\alpha<2$, $\widetilde W_{\alpha-1}$ violates the triangle inequality: on $\mc Q=\R$, the laws $\mu=\frac12(\delta_{-1}+\delta_1)$, $\eta=\frac12(\delta_0+\delta_1)$ and $\nu=\delta_0$ satisfy
        \begin{equation*}
            \widetilde W_{\alpha-1}(\mu,\eta)+\widetilde W_{\alpha-1}(\eta,\nu)
            =\frac12\br{\frac23}^{2-\alpha}+\frac12
            <1=\widetilde W_{\alpha-1}(\mu,\nu)
            \eqfs
        \end{equation*}
    \end{enumerate}
\end{lemma}
\begin{proof}
    \emph{(i).} The integrand $c(y,z):=\ol yz\,\diam\{y,z,m_\alpha(\mu),m_\alpha(\nu)\}^{\alpha-2}$, read as $0$ at $y=z$, is symmetric in $(\mu,\nu)$ and in $(y,z)$. It satisfies $0\leq c\leq\ol yz^{\alpha-1}$ because the diameter is at least $\ol yz$, with $c=\ol yz$ at $\alpha=2$. It is lower semicontinuous, being continuous off the diagonal and zero on it, and it is positive off the diagonal. The laws are concentrated on separable Borel subsets of the complete space $\mc Q$, hence tight, so $\Pi(\mu,\nu)$ is weakly compact and $\pi\mapsto\int c\,\dl\pi$ is lower semicontinuous; the infimum is attained. If it is $0$, an optimal coupling is concentrated on $\{c=0\}$, the diagonal, so $\mu=\nu$; the converse is the diagonal coupling.

    \emph{(ii).} By symmetry and uniqueness, $m_\alpha(\mu)=m_\alpha(\nu)=0$ and $m_\alpha(\eta)=\frac12$. Since $\nu$ is a Dirac measure, $\widetilde W_{\alpha-1}(\mu,\nu)=\frac12(1+1)=1$ and $\widetilde W_{\alpha-1}(\eta,\nu)=\frac12$. The couplings of $\mu$ and $\eta$ put mass $s\in[0,\frac12]$ on each of $(-1,0),(1,1)$ and $\frac12-s$ on each of $(-1,1),(1,0)$, with costs $(\frac32)^{\alpha-2}$, $0$, $2^{\alpha-1}$ and $1$. The total cost is affine in $s$ and minimal at $s=\frac12$, because $(\frac32)^{\alpha-2}\leq1<2^{\alpha-1}+1$; so $\widetilde W_{\alpha-1}(\mu,\eta)=\frac12(\frac23)^{2-\alpha}$, which is less than $\frac12$ for $\alpha<2$.
\end{proof}
The contraction property \CPa is implicit: the distance $\ol{m_\alpha(\mu)}{m_\alpha(\nu)}$ occurs inside both inverse moments on its left-hand side.
The following lemma proves \CPa and inverts it.
\begin{lemma}[Contraction property and its inversion]\label{lmm:power:cp}
	For $\eta\in\mc P^{\alpha-1}(\mc Q)$, define
	\begin{equation*}
		\Psi_\eta(t) := t\,\LamOf{\eta}{t}\quad(t>0),\qquad\Psi_\eta(0):=0
		\eqfs
	\end{equation*}
	Then $\Psi_\eta$ is a continuous, strictly increasing bijection of $\Rp$ onto itself, and
	\begin{equation}\label{eq:power:Psi:scaling}
		\Psi_\eta(ct)\geq c^{\alpha-1}\Psi_\eta(t),\qquad c\geq1,\quad t\geq0.
	\end{equation}
	For all $\mu,\nu\in\mc P^{\alpha-1}(\mc Q)$, \CPa reads
	\begin{equation}\label{eq:power:cp:sharp}
	\begin{aligned}
		\Psi_\mu\brOf{\ol{m_\alpha(\mu)}{m_\alpha(\nu)}} + \Psi_\nu\brOf{\ol{m_\alpha(\mu)}{m_\alpha(\nu)}}
		&\leq
		\frac{2^{3-\alpha}}{\alpha-1}\, \widetilde W_{\alpha-1}(\mu, \nu)
		\eqcm\\
		\ol{m_\alpha(\mu)}{m_\alpha(\nu)}
		&\leq
		\Psi_\mu^{-1}\brOf{\frac{2^{3-\alpha}}{\alpha-1}\, \widetilde W_{\alpha-1}(\mu, \nu)}
		\eqfs
	\end{aligned}
	\end{equation}
\end{lemma}
\begin{proof}
	\emph{Step 1: the map $\Psi_\eta$.}
	Let $\eta\in\mc P^{\alpha-1}(\mc Q)$.
	As $\alpha-2\leq0$, the integrand in \eqref{eq:power:Lambda} is bounded by $t^{\alpha-2}$ for $t\in\Rpp$, so $\Lamof{\eta}{t}<\infty$ and $\Psi_\eta(t)\leq t^{\alpha-1}$; in particular $\Psi_\eta(t)\to0 = \Psi_\eta(0)$ as $t\searrow0$.
	Since
	\begin{equation*}
		\frac{\dl}{\dl t}\br{t\,(r+t)^{\alpha-2}}
		=
		(r+t)^{\alpha-3}\br{r + (\alpha-1)t}
		>
		0
		\eqcm\qquad r\in\Rp,\ t\in\Rpp
		\eqcm
	\end{equation*}
	the map $\Psi_\eta$ is strictly increasing, and it is continuous on $\Rpp$ by dominated convergence.
	By \eqref{eq:power:cp:jensen} and $(\alpha-1)\phi = 2-\alpha$ we have $\Psi_\eta(t)\geq t(\sigma_\alpha^{\alpha-1}(\eta)+t^{\alpha-1})^{-\phi}\to\infty$ as $t\to\infty$, so $\Psi_\eta$ is a bijection of $\Rp$ onto itself.
	For $c\geq1$ and $t>0$, the inequalities $r+ct\leq c(r+t)$ and $\alpha-2\leq0$ give
	$c t(r+ct)^{\alpha-2}\geq c^{\alpha-1}t(r+t)^{\alpha-2}$.
	Integration with $r=\ol y{m_\alpha(\eta)}$ proves \eqref{eq:power:Psi:scaling}; at $t=0$ both sides vanish.

	\emph{Step 2: the inequality \eqref{eq:power:cp:sharp}.}
	Let $\pi\in\Pi(\mu,\nu)$; the case $m_\alpha(\mu)=m_\alpha(\nu)$ is trivial, so assume that the means differ.
	Applying \VIa for $\mu$ at $q = m_\alpha(\nu)$ and for $\nu$ at $q = m_\alpha(\mu)$ and adding, we obtain
	\begin{equation*}
	\begin{aligned}
		&\frac{\alpha(\alpha-1)}{2}\,\ol{m_\alpha(\mu)}{m_\alpha(\nu)}^2
		\br{\LamOf{\mu}{\ol{m_\alpha(\mu)}{m_\alpha(\nu)}} + \LamOf{\nu}{\ol{m_\alpha(\mu)}{m_\alpha(\nu)}}}
		\\&\qquad\leq
		\int \ol y{m_\alpha(\nu)}^\alpha - \ol y{m_\alpha(\mu)}^\alpha + \ol z{m_\alpha(\mu)}^\alpha - \ol z{m_\alpha(\nu)}^\alpha \,\dl\pi(y,z)
		\eqfs
	\end{aligned}
	\end{equation*}
	Apply \eqref{eq:power:qi} to the integrand, with $q = m_\alpha(\nu)$ and $p = m_\alpha(\mu)$:
	\begin{equation*}
	\begin{aligned}
		&\ol y{m_\alpha(\nu)}^\alpha - \ol y{m_\alpha(\mu)}^\alpha + \ol z{m_\alpha(\mu)}^\alpha - \ol z{m_\alpha(\nu)}^\alpha
		\\&\qquad\leq
		2^{2-\alpha}\alpha\,\ol{m_\alpha(\mu)}{m_\alpha(\nu)}\,\ol yz\ \diam\cb{y, z, m_\alpha(\mu), m_\alpha(\nu)}^{\alpha-2}
		\eqcm
	\end{aligned}
	\end{equation*}
	which is $2^{2-\alpha}\alpha\ol{m_\alpha(\mu)}{m_\alpha(\nu)}$ times the integrand of \eqref{eq:power:Wtilde}.
	Dividing by $\frac{\alpha(\alpha-1)}{2}\ol{m_\alpha(\mu)}{m_\alpha(\nu)}$ and taking the infimum over $\pi$ yields the first inequality of \eqref{eq:power:cp:sharp}, which is \CPa; the second follows by dropping the nonnegative $\Psi_\nu$ term and inverting the increasing map $\Psi_\mu$ of Step 1.
	Exchanging $\mu$ and $\nu$ gives the same bound with $\Psi_\nu^{-1}$, used below.
\end{proof}
\begin{remark}[Sharpness of \CPa]\label{rem:power:cp:sharp}
 On $\R$, let $\mu:=\frac12(\delta_{-1}+\delta_1)$ and $\nu_w:=\frac{1-w}2\delta_{-1}+\frac{1+w}2\delta_1$ for $w\in(0,1)$.
 Then $m_\alpha(\mu)=0$, and $D_w:=m_\alpha(\nu_w)\in(0,1)$ solves $((1+D_w)/(1-D_w))^{\alpha-1}=(1+w)/(1-w)$, so $D_w\sim w/(\alpha-1)$ as $w\searrow0$.
 Every coupling moves mass at least $w/2$ from $-1$ to $1$, and moving exactly this mass is optimal, so $\widetilde W_{\alpha-1}(\mu,\nu_w)=2^{\alpha-2}w$.
 Since $\Lamof{\mu}{D_w}$ and $\Lamof{\nu_w}{D_w}$ tend to $1$, the ratio of the left-hand side of \CPa to its right-hand side tends to $1$ as $w\searrow0$; the same holds for \eqref{eq:power:cp:moment} at each fixed $\alpha$.
\end{remark}
The following bound quantifies the reduction in transport cost discussed in \cref{rem:power:smoothness}.
\begin{remark}[Diffuse transport]\label{rem:power:diffuse}
	Let $\mu,\nu\in\mc P^{\alpha-1}(\mc Q)$ and $\epsilon\in\Rpp$, and suppose some $\pi\in\Pi(\mu,\nu)$ moves no mass further than $\epsilon$, i.e., $\ol yz\leq\epsilon$ for $\pi$-almost every $(y,z)$. Then
	\begin{equation*}
		\widetilde W_{\alpha-1}(\mu,\nu)
		\leq
		\epsilon\int \ol y{m_\alpha(\mu)}^{\,\alpha-2}\,\dl\mu(y)
		\eqcm
	\end{equation*}
	an \emph{inverse} moment of $\mu$ times $\epsilon$, whereas $W_{\alpha-1}(\mu,\nu)$ is only bounded by $\epsilon^{\alpha-1}$.
	Indeed, $\diam\{y,z,m_\alpha(\mu),m_\alpha(\nu)\}\geq\ol y{m_\alpha(\mu)}$ and $\alpha-2\leq0$ make the integrand of \eqref{eq:power:Wtilde} at most $\ol yz\,\ol y{m_\alpha(\mu)}^{\,\alpha-2}\leq\epsilon\,\ol y{m_\alpha(\mu)}^{\,\alpha-2}$ for $\pi$-almost every $(y,z)$, and the first marginal of $\pi$ is $\mu$.
	For $\alpha<2$, this bound is infinite if $\mu$ has an atom at its mean. An always finite form is
	\begin{equation*}
		\widetilde W_{\alpha-1}(\mu,\nu)\leq2^{2-\alpha}\min\cb{\Psi_\mu(\epsilon),\Psi_\nu(\epsilon)}
		\eqcm
	\end{equation*}
	with $\Psi$ as in \cref{lmm:power:cp}: for $r:=\ol y{m_\alpha(\mu)}$, the diameter is at least $\max\{\ol yz,r\}$, the map $a\mapsto a\max\{a,r\}^{\alpha-2}$ is nondecreasing, and $\max\{\epsilon,r\}\geq(\epsilon+r)/2$, so the integrand is at most $2^{2-\alpha}\epsilon(r+\epsilon)^{\alpha-2}$; the same holds with $\nu$ and the second marginal.
	Consequently, by \eqref{eq:power:Psi:scaling}, in the setting of \cref{cor:power:localfrechet}, if for all $z,z\pr\in\mc X$ some coupling of $\mu_z$ and $\mu_{z\pr}$ moves no mass further than $L\pr\rho(z,z\pr)^\beta$, then \eqref{eq:power:localfrechet:smooth} holds with $L=2^\phi L\pr$.
\end{remark}

The following one-sided form of \CPa gives continuity of means in the statistical proofs.
\begin{lemma}[One-sided location-aware contraction]\label{lmm:power:cp:onesided}
 Let $\mu,\nu\in\mc P^{\alpha-1}(\mc Q)$. Then
 \begin{equation}\label{eq:power:cp:onesided}
 \ol{m_\alpha(\mu)}{m_\alpha(\nu)}
 \leq c_\alpha\br{\sigma_\alpha^{\alpha-1}(\nu)^\phi\widetilde W_{\alpha-1}(\mu,\nu)
                         +\widetilde W_{\alpha-1}(\mu,\nu)^{1/(\alpha-1)}}.
 \end{equation}
\end{lemma}
\begin{proof}
 Write $T:=\widetilde W_{\alpha-1}(\mu,\nu)$ and $S:=\sigma_\alpha^{\alpha-1}(\nu)$.
 The case $m_\alpha(\mu)=m_\alpha(\nu)$ is immediate. Otherwise, \cref{lmm:power:cp}, keeping only the term associated with $\nu$, and \eqref{eq:power:cp:jensen} give
 \begin{equation*}
 \frac{\ol{m_\alpha(\mu)}{m_\alpha(\nu)}}{\br{S+\ol{m_\alpha(\mu)}{m_\alpha(\nu)}^{\alpha-1}}^\phi}\leq c_\alpha T.
 \end{equation*}
 If $\ol{m_\alpha(\mu)}{m_\alpha(\nu)}^{\alpha-1}\leq S$, then $\ol{m_\alpha(\mu)}{m_\alpha(\nu)}\leq c_\alpha2^\phi S^\phi T$.
 Otherwise $\ol{m_\alpha(\mu)}{m_\alpha(\nu)}^{\alpha-1}\leq c_\alpha2^\phi T$, since $(\alpha-1)\phi=2-\alpha$.
 Combining the two cases proves \eqref{eq:power:cp:onesided}.
\end{proof}

The next two lemmas connect explicit moment criteria to the inverse-moment condition \eqref{eq:power:localfrechet:smooth} and control the bias of local mixtures.
\begin{lemma}[A moment-only sufficient condition]\label{lmm:power:sufficient}
 Let $\mu,\nu\in\mc P^{\alpha-1}(\mc Q)$ and $b\in\Rpp$, and set
 \[
  \bar\sigma:=\max\{\sigma_\alpha^{\alpha-1}(\mu),\sigma_\alpha^{\alpha-1}(\nu)\}.
 \]
 \begin{enumerate}[label=(\roman*)]
  \item\label{lmm:power:sufficient:itm:jensen} If
  \begin{equation}\label{eq:power:localfrechet:jensen}
   \widetilde W_{\alpha-1}(\mu,\nu)\leq\frac b{\br{\bar\sigma+b^{\alpha-1}}^{\phi}}
   \eqcm
  \end{equation}
  then $\widetilde W_{\alpha-1}(\mu,\nu)\leq\min\{\Psi_\mu(b),\Psi_\nu(b)\}$.
  In particular, if \eqref{eq:power:localfrechet:jensen} holds for $\mu=\mu_z$, $\nu=\mu_{z\pr}$, and $b=L\rho(z,z\pr)^\beta$ for all distinct $z,z\pr\in\mc X$, then \eqref{eq:power:localfrechet:smooth} holds at every $x\in\mc X$ with the same $L$.
  \item\label{lmm:power:sufficient:itm:tworegime} Reading $\bar\sigma^{-\phi}b$ as $\infty$ if $\bar\sigma=0$ and $\alpha<2$, the two-regime condition
  \begin{equation}\label{eq:power:localfrechet:sufficient}
   \widetilde W_{\alpha-1}(\mu,\nu)\leq\min\cb{\bar\sigma^{-\phi}\,b,\ b^{\alpha-1}}
   \eqcm
  \end{equation}
  implies \eqref{eq:power:localfrechet:jensen} with $b$ replaced by $2^{\phi/(\alpha-1)}b$, and hence \eqref{eq:power:localfrechet:smooth} with $L$ replaced by $2^{\phi/(\alpha-1)}L$.
 \end{enumerate}
\end{lemma}
At the same $b$, the Jensen criterion is stronger and preserves $L$; the simpler two-regime criterion requires the stated enlargement of $L$.
\begin{proof}
 Let $\eta\in\{\mu,\nu\}$, so that $\sigma_\alpha^{\alpha-1}(\eta)\leq\bar\sigma$. Jensen's bound \eqref{eq:power:cp:jensen} gives
 \[
  \Psi_\eta(b)\geq\frac{b}{\br{\sigma_\alpha^{\alpha-1}(\eta)+b^{\alpha-1}}^\phi}\geq\frac{b}{\br{\bar\sigma+b^{\alpha-1}}^\phi}
  \eqcm
 \]
 which is \ref{lmm:power:sufficient:itm:jensen}.
 For \ref{lmm:power:sufficient:itm:tworegime}, $\bar\sigma+b^{\alpha-1}\leq2\max(\bar\sigma,b^{\alpha-1})$ and $1-(\alpha-1)\phi=\alpha-1$ give
 \[
 b\br{\bar\sigma+b^{\alpha-1}}^{-\phi}
 \geq2^{-\phi}\min\{\bar\sigma^{-\phi}b,b^{\alpha-1}\}.
 \]
 For $c\geq1$, moreover,
 \[
 \frac{cb}{\br{\bar\sigma+(cb)^{\alpha-1}}^\phi}
 \geq c^{\alpha-1}\frac{b}{\br{\bar\sigma+b^{\alpha-1}}^\phi}.
 \]
 Taking $c=2^{\phi/(\alpha-1)}$ proves the stated rescaling. Conversely, $\bar\sigma+b^{\alpha-1}\geq\max(\bar\sigma,b^{\alpha-1})$ shows that \eqref{eq:power:localfrechet:jensen} implies \eqref{eq:power:localfrechet:sufficient} without rescaling.
\end{proof}

\begin{lemma}[Stability under mixing]\label{lmm:power:cp:mixture}
 Let $\nu,\mu_1,\ldots,\mu_N\in\mc P^{\alpha-1}(\mc Q)$, and let $a_i\geq0$ with $\sum_i a_i=1$.
 Set $\eta:=\sum_i a_i\mu_i$ and
 $\epsilon:=\max_{i:a_i>0}\widetilde W_{\alpha-1}(\mu_i,\nu)$.
 With $\Psi_\nu$ as in \cref{lmm:power:cp},
 \begin{equation}\label{eq:power:cp:mixture}
 \max_{i:a_i>0}\ol{m_\alpha(\mu_i)}{m_\alpha(\nu)}
 +\ol{m_\alpha(\eta)}{m_\alpha(\nu)}
 \leq c_\alpha\Psi_\nu^{-1}(\epsilon).
 \end{equation}
\end{lemma}
\begin{proof}
 Write $m_0:=m_\alpha(\nu)$, $m_i:=m_\alpha(\mu_i)$, $m_*:=m_\alpha(\eta)$,
 and $R:=\max_{i:a_i>0}\ol{m_i}{m_0}$.
 By \cref{lmm:power:cp},
 $R\leq\Psi_\nu^{-1}(\frac{2^{3-\alpha}}{\alpha-1}\epsilon)\leq c_\alpha\Psi_\nu^{-1}(\epsilon)$,
 where the second inequality follows from \eqref{eq:power:Psi:scaling}.
 If $\ol{m_*}{m_0}\leq R$, this proves the assertion.
 Otherwise $\ol{m_*}{m_0}>R$, and, for every active $i$ and every $y,z$,
 \begin{equation*}
 \diam\cb{y,z,m_i,m_0}
 \leq\diam\cb{y,z,m_*,m_0}+R
 \leq2\diam\cb{y,z,m_*,m_0}.
 \end{equation*}
 Since $\alpha-2\leq0$, mixing approximately optimal couplings of $\mu_i$ and $\nu$ gives
 \begin{equation*}
 \widetilde W_{\alpha-1}(\eta,\nu)
 \leq2^{2-\alpha}\sum_{i=1}^{N} a_i\widetilde W_{\alpha-1}(\mu_i,\nu)
 \leq2^{2-\alpha}\epsilon.
 \end{equation*}
 Applying \cref{lmm:power:cp} to $\eta,\nu$ gives
 $\ol{m_*}{m_0}\leq\Psi_\nu^{-1}(\frac{2^{5-2\alpha}}{\alpha-1}\epsilon)\leq c_\alpha\Psi_\nu^{-1}(\epsilon)$,
 again by \eqref{eq:power:Psi:scaling}.
 Combining the bounds on $\ol{m_*}{m_0}$ and $R$ proves \eqref{eq:power:cp:mixture}.
\end{proof}
\subsection{Comparison of the Variances}\label{app:variances}
\begin{proof}[Proof of \cref{lmm:variances}]
	\emph{Upper bound.}
	For every $q\in\mc Q$ we have $\VfpofP{s}{\mu}\leq\int\ol yq^{s}\,\dl\mu(y)$.
	Integrating over $q$ with respect to $\mu$ gives $\VfpofP{s}{\mu}\leq\iint\ol yq^{s}\,\dl\mu(y)\,\dl\mu(q) = 2\VmpofP{s}{\mu}$.

	\emph{Lower bound.}
	Let $q\in\mc Q$. By the triangle inequality and \cref{lmm:power:subadd:gen},
	$\ol y{y\pr}^{s}\leq(\ol yq + \ol{y\pr}q)^{s}\leq2^{\max(0,s-1)}(\ol yq^{s} + \ol{y\pr}q^{s})$.
	Integrating over $\mu\otimes\mu$ yields $2\VmpofP{s}{\mu}\leq2^{\max(0,s-1)}\cdot2\int\ol yq^{s}\,\dl\mu(y)$, and taking the infimum over $q$ gives $\VmpofP{s}{\mu}\leq2^{\max(0,s-1)}\VfpofP{s}{\mu}$, which is the first inequality of \eqref{eq:variances:comparison} because $-\max(0,s-1) = \min(1-s,0)$.

\end{proof}
\begin{remark}[Sharpness of \cref{lmm:variances}]\label{rem:app:variances:sharp}
 The upper constant is attained on the two-point space $\{a,b\}$ with $\ol ab=1$ and $\mu$ uniform, where $\VfpofP{s}{\mu}=\frac12$ and $\VmpofP{s}{\mu}=\frac14$.
 For $s\geq1$, the lower constant is approached on a star $\R$-tree with $n\geq2$ unit-length edges and $\mu$ uniform on the leaves. For a point at distance $r\in[0,1]$ from the center, Jensen's inequality gives $\frac1n(1-r)^s+\frac{n-1}n(1+r)^s\geq(1+\frac{n-2}nr)^s\geq1$, so the center minimizes the $s$-objective, and the ratio $\VfpofP{s}{\mu}/\VmpofP{s}{\mu}=2^{1-s}(1-1/n)^{-1}$ tends to $2^{1-s}$ as $n\to\infty$.
 For $0<s<1$, take $\mu=(1-\epsilon)\delta_a+\epsilon\delta_b$ on the two-point space, with $0<\epsilon<1/2$. The ratio is then $(1-\epsilon)^{-1}\to1$ as $\epsilon\searrow0$, proving sharpness of the lower constant in this range as well.
\end{remark}
\begin{lemma}[Quadratic variances in Hadamard spaces]\label{lmm:two:variances}
 For $Y\sim\mu\in\mc P^2(\mc Q)$ in a Hadamard space,
 $\Vfof{Y}\leq\Vmof{Y}\leq2\Vfof{Y}$. If $\mc Q$ is a Hilbert space, then $\Vfof{Y}=\Vmof{Y}$.
\end{lemma}
\begin{proof}
 Put $m:=m_2(\mu)$ and let $Y\pr$ be an independent copy of $Y$.
 Integrating \eqref{eq:cat0:2fm:vi} at $q=Y\pr$ gives
 $\Vfof{Y}\leq2\Vmof{Y}-\Vfof{Y}$, proving the first inequality.
 The second is \cref{lmm:variances} at $s=2$.
 In a Hilbert space, $m=\Eof{Y}$, and independence gives
 $\Eof{\normof{Y-Y\pr}^2}=2\Eof{\normof{Y-m}^2}$ by expanding the squared norm; hence the two variances agree.
\end{proof}

\section{Proofs of the Statistical Results}\label{app:proofs}
The proofs follow the order of the main sections, after a leave-one-out inequality that is common to the three weighted-sample bounds.

\Needspace{15\baselineskip}
\textbf{Notation.} The notation from \cref{sec:prelim} is retained; the following quantities recur in these proofs.
\begin{center}
\begin{tabular}{ll}
\toprule
$\mathcal E_\mu(q)$ & excess risk, \eqref{eq:power:excess:def}\\
$Y_i\pr$, $\samdis^i$, $\samfm^i$ & copy, resampled law, and mean; \cref{ntt:power:proof}\\
$\hat\sigma_s$, $\hat\sigma_s^i$, $\hat\sigma_s^{-i}$ & empirical moments about $\popfm$, \cref{ntt:power:proof}\\
$\sigma_s$, $\Sigma_s$ & population and weighted pairwise moments; \eqref{eq:power:sigma}, \eqref{eq:power:sigmaw}\\
$\phi=(2-\alpha)/(\alpha-1)$ & exponent for $\alpha<2$; \cref{sec:prelim}\\
$C_{\mathrm{QI}}$ & quadruple-inequality constant in \cref{lmm:loo}\\
$g(t)$ & modified distance loss, \eqref{eq:power:gdef}\\
$V^{-i}$ & fluctuation omitting observation $i$; \cref{lmm:power:empmom}\\
$B_1$, $B_2$ & absorption terms, \eqref{eq:power:B}\\
$d_\alpha$ & metric $d^{\alpha/2}$; \eqref{eq:comparison:snowflake}\\
\bottomrule
\end{tabular}
\end{center}
\subsection{The Leave-One-Out Inequality}\label{app:loo}
The inequality below extends the leave-one-out step of the proof of \cite[Theorem 3]{BrunelSerres2026FiniteSample-arxiv}, after \cite[Lemmas 6 and 7]{Escande2024Minimizers-JMLR}, to weighted samples with heterogeneous laws and to $\alpha\in(1,2]$; for iid samples and transformed Fr\'echet means the corresponding step is \cite[Lemma S4.2]{Schoetz2026Transformed-arxiv}.

Let $(\mc Q,d)$ be a metric space and $\alpha\in(1,2]$.
Let $(w_i)_{i\in\N}\subset[0,1]$ with $\sum_{i=1}^\infty w_i=1$, and let $(Y_i)_{i\in\N}$ and $(Y_i\pr)_{i\in\N}$ be independent sequences of independent random variables such that $Y_i$ and $Y_i\pr$ have the same distribution $\mu_i\in\mc P(\mc Q)$ for every $i\in\N$.
Put $\popdis:=\sum_{i=1}^\infty w_i\mu_i$ and $\samdis:=\sum_{i=1}^\infty w_i\delta_{Y_i}$, and write $\popfm:=m_\alpha(\popdis)$ and $\samfm:=m_\alpha(\samdis)$ whenever these means exist and are unique.
Extending the notation of \cref{rem:pos:literature} to general $\alpha$, the \emph{excess risk} at $q\in\mc Q$ of $\mu\in\mc P^{\alpha-1}(\mc Q)$ with a unique $\alpha$-Fr\'echet mean is
\begin{equation}\label{eq:power:excess:def}
 \mathcal E_\mu(q):=\int\br{\ol yq^\alpha-\ol y{m_\alpha(\mu)}^\alpha}\,\dl\mu(y)
 \eqfs
\end{equation}
The integral is finite by \cref{lmm:power:ana:diff} and nonnegative by minimality, even when the two unrenormalized losses have infinite integrals.
\Needspace{12\baselineskip}
\begin{notation}[Leave-one-out quantities]\label{ntt:power:proof}
 We write $\sigma_s:=\sum_jw_j\Eof{\ol{Y_j}{\popfm}^s}$ for the $s$-th moment of $\popdis$ about $\popfm$, as in \eqref{eq:power:sigma}.
	For $i\in\N$, let
	\begin{equation*}
		\samdis^i := w_i\, \delta_{Y_i\pr} + \sum_{j\in\N\setminus\cb i} w_j\, \delta_{Y_j}
		\qquad\text{and}\qquad
		\samfm^i := m_\alpha(\samdis^i)
		\eqcm
	\end{equation*}
	the sample and the sample mean after resampling the $i$-th observation, and define the empirical moments
 \begin{align*}
  \hat\sigma_s&:=\sum_{j=1}^\infty w_j\,\ol{Y_j}{\popfm}^s,\\
  \hat\sigma_s^i&:=w_i\,\ol{Y_i\pr}{\popfm}^s
       +\sum_{j\in\N\setminus\cb i}w_j\,\ol{Y_j}{\popfm}^s,\\
  \hat\sigma_s^{-i}&:=\sum_{j\in\N\setminus\cb i}w_j\,\ol{Y_j}{\popfm}^s,
 \end{align*}
	the last of which is independent of $(Y_i, Y_i\pr)$.
	Note $\Eof{\hat\sigma_s} = \Eof{\hat\sigma_s^i} = \sigma_s$ and $\hat\sigma_s^{-i} \leq \hat\sigma_s$.
\end{notation}
\begin{lemma}[Leave-one-out inequality]\label{lmm:loo}
	Assume $\popdis\in\mc P^{\alpha}(\mc Q)$ and one of the following.
	\begin{enumerate}[label=(\alph*)]
		\item\label{lmm:loo:itm:hadamard} $(\mc Q,d)$ is a Hadamard space; put $C_{\mathrm{QI}}:=2^{2-\alpha}\alpha$.
		\item\label{lmm:loo:itm:catk} $\alpha=2$ and, in the setting of \cref{ssec:pos:prelim}, all $\mu_i$ lie in $\mc P(\setC)$; put $C_{\mathrm{QI}}:=2C_\Delta$, or $C_{\mathrm{QI}}:=2C_\delta$ by \cref{rem:pos:sharper}.
	\end{enumerate}
	Then $\popfm$, $\samfm$, and $\samfm^i$, $i\in\N$, exist and are unique almost surely, and
	\begin{equation}\label{eq:loo}
		\EOf{\mathcal E_{\popdis}(\samfm)}+\EOf{\mathcal E_{\samdis}(\popfm)}
		\leq
		\frac{C_{\mathrm{QI}}}{2}\sum_{i=1}^\infty w_i\, \EOf{\ol{Y_i}{Y_i\pr}^{\alpha-1}\, \ol{\samfm}{\samfm^i}}
		\eqfs
	\end{equation}
\end{lemma}
Both excess risks are nonnegative.
At $\alpha=2$, the variance inequalities \eqref{eq:cat0:2fm:vi} and \eqref{eq:2fmcatk:vi} bound each of them from below by $\ol{\popfm}{\samfm}^2$, respectively by $C_\Delta^{-1}\ol{\popfm}{\samfm}^2$; for $\alpha<2$ we drop the second one.
\begin{proof}
	\emph{Existence and integrability.}
	In case \ref{lmm:loo:itm:catk}, all distances are bounded by $\diam(\setC)$, so every sum and expectation below is finite and may be interchanged freely, and $\popdis,\samdis,\samdis^i\in\mc P(\setC)$ have unique means in $\setC$ by \cref{ssec:pos:prelim}.
	In case \ref{lmm:loo:itm:hadamard}, we have $\sigma_{\alpha-1}<\infty$, as $\popdis\in\mc P^\alpha(\mc Q)$, and $\popfm$ exists uniquely.
	As $\Eof{\hat\sigma_{\alpha-1}} = \sigma_{\alpha-1}$, we have $\hat\sigma_{\alpha-1} < \infty$ almost surely, i.e., $\samdis \in \mc P^{\alpha-1}(\mc Q)$ almost surely; the separability requirement in the definition of $\mc P(\mc Q)$ holds, as all $Y_j$ with $w_j > 0$ almost surely take values in the union of the separable Borel sets on which the corresponding $\mu_j$ are concentrated, cf.\ \cref{rem:separability}.
	Hence, $\samfm$ and, analogously, $\samfm^i$ exist and are unique almost surely.
	They are also measurable functions of the observations, as required by the independence and conditioning steps below.
	Indeed, replacing every coordinate after the $N$-th by a reference point $o$, with $o\in\setC$ in case \ref{lmm:loo:itm:catk}, gives means that depend on finitely many observations and are continuous in them, by \eqref{eq:cat0:2fm:CP} in case \ref{lmm:loo:itm:hadamard} with $\alpha=2$, by \CPp in case \ref{lmm:loo:itm:catk}, and by \cref{lmm:power:cp:onesided} in case \ref{lmm:loo:itm:hadamard} with $\alpha<2$. The truncation couplings send the tail to $o$, so the corresponding transport costs tend to zero almost surely; in case \ref{lmm:loo:itm:catk}, \CPp turns this directly into almost sure convergence of the finite-coordinate means to $\samfm$, respectively $\samfm^i$, and in case \ref{lmm:loo:itm:hadamard} so does \cref{lmm:power:cp:onesided}, applied with $\nu=\samdis$, respectively $\nu=\samdis^i$, fixed. Step 7 of the proof of \cref{prp:power:excess} carries out the same computation in detail.
	In case \ref{lmm:loo:itm:hadamard}, \eqref{eq:power:dirac} applied to $\samdis$ and $\samdis^i$ with $q = \popfm$ gives
	\begin{equation}\label{eq:power:rb}
		\ol{\samfm}{\popfm}^{\alpha-1} \leq c_\alpha\, \hat\sigma_{\alpha-1}
		\qquad\text{and}\qquad
		\ol{\samfm^i}{\popfm}^{\alpha-1} \leq c_\alpha\, \hat\sigma_{\alpha-1}^i
		\qquad\as
	\end{equation}
	As $\frac{1}{\alpha-1} \geq 1$ and $\frac{\alpha}{\alpha-1} = \phi + 2 \geq 2$, Jensen's inequality with weights $(w_j)_{j\in\N}$ yields
	$\hat\sigma_{\alpha-1}^{\frac{1}{\alpha-1}} \leq \sum_{j=1}^\infty w_j\, \ol{Y_j}{\popfm}$ and $\hat\sigma_{\alpha-1}^{\frac{\alpha}{\alpha-1}} \leq \sum_{j=1}^\infty w_j\, \ol{Y_j}{\popfm}^{\alpha}$.
	Together with \eqref{eq:power:rb} and $\hat\sigma_{\alpha-1}\, \ol{\samfm}{\popfm} \leq c_\alpha\, \hat\sigma_{\alpha-1}^{\frac{\alpha}{\alpha-1}}$, this implies
	\begin{equation}\label{eq:power:integra}
		\EOf{\ol{\samfm}{\popfm}} \leq c_\alpha\, \sigma_1
		\eqcm\qquad
		\EOf{\ol{\samfm}{\popfm}^{\alpha}} \leq c_\alpha\, \sigma_\alpha
		\eqcm\qquad
		\EOf{\hat\sigma_{\alpha-1}\, \ol{\samfm}{\popfm}} \leq c_\alpha\, \sigma_\alpha
		\eqcm
	\end{equation}
	all finite; the same bounds hold with $\samfm^i$ and $\hat\sigma_{\alpha-1}^i$ in place of $\samfm$ and $\hat\sigma_{\alpha-1}$.
	By the mean value theorem, the triangle inequality, and \cref{lmm:power:subadd:gen}, for $y \in \mc Q$ and $u \in \{\samfm, \samfm^i\}$,
	\begin{equation}\label{eq:power:mvt}
		\abs{\ol yu^{\alpha} - \ol y{\popfm}^{\alpha}}
		\leq
		\alpha\, \ol{u}{\popfm} \br{\ol y{\popfm}^{\alpha-1} + \ol u{\popfm}^{\alpha-1}}
		=
		\alpha \br{\ol u{\popfm}\, \ol y{\popfm}^{\alpha-1} + \ol u{\popfm}^{\alpha}}
		\eqfs
	\end{equation}
	Combining \eqref{eq:power:integra} and \eqref{eq:power:mvt}, and using that $(Y_j\pr)_{j\in\N}$ is independent of $\samfm$,
	\begin{align}
		\sum_{i=1}^\infty w_i\, \EOf{\abs{\ol{Y_i\pr}{\samfm}^{\alpha} - \ol{Y_i\pr}{\popfm}^{\alpha}}}
		&\leq
		\alpha \br{\EOf{\ol{\samfm}{\popfm}}\, \sigma_{\alpha-1} + \EOf{\ol{\samfm}{\popfm}^{\alpha}}}
		< \infty
		\eqcm
		\label{eq:power:dom1}
		\\
		\sum_{i=1}^\infty w_i\, \EOf{\abs{\ol{Y_i}{\samfm}^{\alpha} - \ol{Y_i}{\popfm}^{\alpha}}}
		&\leq
		\alpha \br{\EOf{\hat\sigma_{\alpha-1}\, \ol{\samfm}{\popfm}} + \EOf{\ol{\samfm}{\popfm}^{\alpha}}}
		< \infty
		\eqfs
		\label{eq:power:dom2}
	\end{align}
	Moreover, for every $i \in \N$ with $w_i > 0$ --- the only indices that contribute to any sum below --- using $\ol Zu^{\alpha} \leq 2^{\alpha-1}(\ol Z{\popfm}^{\alpha} + \ol u{\popfm}^{\alpha})$, \eqref{eq:power:integra}, and $w_i\, \Eof{\ol{Y_i}{\popfm}^{\alpha}} \leq \sigma_\alpha$,
	\begin{equation}\label{eq:power:individ}
		\EOf{\ol Zu^{\alpha}} < \infty
		\qquad\text{for all } Z \in \cb{Y_i, Y_i\pr} \text{ and } u \in \cb{\samfm, \samfm^i}
		\eqfs
	\end{equation}
	In case \ref{lmm:loo:itm:catk}, the bounds \eqref{eq:power:dom1}, \eqref{eq:power:dom2}, and \eqref{eq:power:individ} hold trivially.

	\emph{The main part.}
	By the definition of excess risk, using that $(Y_i\pr)_{i\in\N}$ is independent of $(Y_j)_{j\in\N}$,
	\begin{equation*}
		\mathcal E_{\popdis}(\samfm)
		=
		\sum_{i=1}^\infty w_i\, \EOf{\ol{Y_i\pr}{\samfm}^{\alpha} - \ol{Y_i\pr}{\popfm}^{\alpha} \,\middle|\, (Y_j)_{j\in\N}}
		\qquad\text{and}\qquad
		\mathcal E_{\samdis}(\popfm)
		=
		\sum_{i=1}^\infty w_i \br{\ol{Y_i}{\popfm}^{\alpha} - \ol{Y_i}{\samfm}^{\alpha}}
		\eqcm
	\end{equation*}
	where the second series converges absolutely almost surely by \eqref{eq:power:dom2}.
	Taking expectations --- the interchange of expectation and summation is justified by \eqref{eq:power:dom1} and \eqref{eq:power:dom2} --- and using $\Eof{\ol{Y_i\pr}{\popfm}^{\alpha}} = \Eof{\ol{Y_i}{\popfm}^{\alpha}}$, we obtain
	\begin{equation}\label{eq:power:afterVI}
		\EOf{\mathcal E_{\popdis}(\samfm)}+\EOf{\mathcal E_{\samdis}(\popfm)}
		=
		\sum_{i=1}^\infty w_i\, \EOf{\ol{Y_i\pr}{\samfm}^{\alpha} - \ol{Y_i}{\samfm}^{\alpha}}
		\eqfs
	\end{equation}
	It remains to bound the individual summands.
	We apply the quadruple inequality with $q = \samfm$, $p = \samfm^i$, $y = Y_i\pr$, and $z = Y_i$: in case \ref{lmm:loo:itm:hadamard} this is \QIa, where we bound $\diam\{Y_i, Y_i\pr, \samfm, \samfm^i\}\geq \ol{Y_i}{Y_i\pr}$ and use $\alpha-2\leq0$; in case \ref{lmm:loo:itm:catk} it is \eqref{eq:2fmcatk:qi}, all four points lying in $\setC$.
	Either way,
	\begin{equation*}
		\ol{Y_i\pr}{\samfm}^\alpha - \ol{Y_i}{\samfm}^\alpha
		\leq
		\ol{Y_i\pr}{\samfm^i}^\alpha - \ol{Y_i}{\samfm^i}^\alpha
		+ C_{\mathrm{QI}}\, \ol{\samfm}{\samfm^i}\, \ol{Y_i}{Y_i\pr}^{\alpha-1}
		\eqfs
	\end{equation*}
	Exchanging $Y_i$ and $Y_i\pr$ leaves the joint distribution of the two sequences invariant and maps $(\samfm, \samfm^i)$ to $(\samfm^i, \samfm)$; hence $(Y_i\pr, \samfm^i)$ has the same law as $(Y_i, \samfm)$, and $(Y_i, \samfm^i)$ the same law as $(Y_i\pr, \samfm)$.
	Taking expectations --- all terms are integrable by \eqref{eq:power:individ} --- therefore yields
	\begin{equation*}
		\EOf{\ol{Y_i\pr}{\samfm}^{\alpha} - \ol{Y_i}{\samfm}^{\alpha}}
		\leq
		\frac{C_{\mathrm{QI}}}{2}\, \EOf{\ol{Y_i}{Y_i\pr}^{\alpha-1}\, \ol{\samfm}{\samfm^i}}
		\eqfs
	\end{equation*}
	Inserting this into \eqref{eq:power:afterVI} proves \eqref{eq:loo}.
\end{proof}
\subsection{2-Fr\'echet Means in Hadamard Spaces}\label{app:two}
Throughout this subsection, $(\mc Q,d)$ is a Hadamard space and $\alpha=2$; we use the inequalities \eqref{eq:2fmcat0:qi}, \eqref{eq:cat0:2fm:vi}, and \eqref{eq:cat0:2fm:CP} of \cref{ssec:two:prelim}.
\begin{proof}[Proof of \cref{thm:two:infty:independent}]
    Fix $o\in\mc Q$. Tonelli's theorem and $\popdis\in\mc P^1(\mc Q)$ give
    \[
     \EOf{\sum_{i=1}^{\infty}w_i\ol{Y_i}o}=\int\ol yo\,\dl\popdis(y)<\infty.
    \]
    Hence $\samdis\in\mc P^1(\mc Q)$ almost surely and both means exist.
    The empirical mean is measurable: replace all coordinates after $N$ by $o$; the resulting finite-coordinate means are continuous in the observations by $W_1$ contraction and converge almost surely to $\samfm$ by the tail estimate below.
    If $\sum_iw_i^2\Vmof{Y_i}=\infty$, the claimed inequality is immediate.

    Assume the right-hand side is finite. Every active component then has a finite second moment, since $\Vfof{Y_i}\leq2\Vmof{Y_i}<\infty$ whenever $w_i>0$ by \cref{lmm:variances}.
    Put $r_N:=\sum_{i>N}w_i$ and
    \[
     \popdis_N:=\sum_{i=1}^Nw_i\mu_i+r_N\delta_o,\qquad
     \samdis_N:=\sum_{i=1}^Nw_i\delta_{Y_i}+r_N\delta_o,
    \]
    omitting zero-weight terms. Then $\popdis_N\in\mc P^2(\mc Q)$.
    For independent copies $Y_i\pr$, let $\samdis_N^i$ replace the $i$-th atom of $\samdis_N$ by $Y_i\pr$.
    Apply \cref{lmm:loo}\,\ref{lmm:loo:itm:hadamard} to this finite sample, including the deterministic coordinate at $o$ whose resampling displacement is zero.
    With $\alpha=2$, so that $C_{\mathrm{QI}}=2$, the two variance inequalities \eqref{eq:cat0:2fm:vi} give
    \[
     2\EOf{\ol{m_2(\popdis_N)}{m_2(\samdis_N)}^2}
     \leq\sum_{i=1}^Nw_i\EOf{\ol{Y_i}{Y_i\pr}\,
                    \ol{m_2(\samdis_N)}{m_2(\samdis_N^i)}}.
    \]
    The coupling that moves only the $i$-th atom and \eqref{eq:cat0:2fm:CP} imply
    $\ol{m_2(\samdis_N)}{m_2(\samdis_N^i)}\leq w_i\ol{Y_i}{Y_i\pr}$.
    Therefore
    \[
     \EOf{\ol{m_2(\popdis_N)}{m_2(\samdis_N)}^2}
     \leq\frac12\sum_{i=1}^Nw_i^2\EOf{\ol{Y_i}{Y_i\pr}^2}
     =\sum_{i=1}^Nw_i^2\VmOf{Y_i}.
    \]
    As $N\to\infty$, the couplings that send the tail to $o$ give
    \[
     W_1(\popdis_N,\popdis)\leq\sum_{i=N+1}^{\infty}w_i\EOf{\ol{Y_i}o}\longrightarrow0,\qquad
     W_1(\samdis_N,\samdis)\leq\sum_{i=N+1}^{\infty}w_i\ol{Y_i}o\longrightarrow0
     \quad\text{almost surely}.
    \]
    Contraction implies convergence of the two means.
    Fatou's lemma proves \eqref{eq:two:infty:independent} with constant one.

    In a Hilbert space, $\sum_iw_iY_i$ converges absolutely almost surely and in $L^1$.
    If the weighted variance sum is finite, the centered partial sums converge in $L^2$ and their second moments equal $\sum_{i\leq N}w_i^2\Vmof{Y_i}$, proving equality.
    If the variance sum is infinite, write $S=\samfm-\popfm$.
    Were $\Eof{\normof{S}^2}<\infty$, conditional Jensen applied to
    $\Eof{S\mid Y_1,\ldots,Y_N}=\sum_{i\leq N}w_i(Y_i-\Eof{Y_i})$ would bound every partial variance sum by $\Eof{\normof{S}^2}$, a contradiction.
    Thus equality holds also with extended values.
\end{proof}
\begin{proof}[Proof of \cref{cor:two:localreg}]
 Write $w_i:=w_{i,h}(x)$ and $\popfm(x):=m_2(\sum_iw_i\mu_{x_i})$.
 By \eqref{eq:localfrechet:design}, $N_n(x,h/2)>0$, so the weights are well-defined, sum to one, and vanish unless $\rho(x_i,x)\leq h$.
 The contraction property \eqref{eq:cat0:2fm:CP}, the coupling that couples $\mu_{x_i}$ with $\mu_x$ optimally in each component, and \eqref{eq:two:localreg:smooth} give
 \[
  \ol{\popfm(x)}{m_2(\mu_x)}
  \leq W_1\brOf{\sum_{i=1}^{n}w_i\mu_{x_i},\mu_x}
  \leq\sum_{i=1}^{n}w_iW_1(\mu_{x_i},\mu_x)
  \leq L\sum_{i=1}^{n}w_i\rho(x_i,x)^\beta
  \leq Lh^\beta.
 \]
 By \cref{thm:two:infty:independent},
 $\Eof{\ol{\popfm(x)}{\samfm_h\!(x)}^2}\leq\sum_iw_i^2\Vmof{Y_i}\leq v\normof{w}_2^2$, and the kernel bounds \eqref{eq:localfrechet:kernel} together with \eqref{eq:localfrechet:design} give
 \[
  \normOf{w}_2^2
  =\frac{\sum_{i=1}^{n}K(\rho(x_i,x)/h)^2}{\br{\sum_{j=1}^{n}K(\rho(x_j,x)/h)}^2}
  \leq\frac{\max_iK(\rho(x_i,x)/h)}{\sum_{j=1}^{n}K(\rho(x_j,x)/h)}
  \leq\frac{\const{Kma}}{\const{Kmi}^{-1}N_n(x,h/2)}
  \leq\frac{\const{Kmi}\const{Kma}\const{dns}}{nh^D}.
 \]
 The triangle inequality in $L^2$ proves \eqref{eq:localfrechet:designbound}.
 For $v>0$, the right-hand side of \eqref{eq:localfrechet:designbound} is $h\mapsto Lh^\beta+Ah^{-D/2}$ with $A:=\sqrt{v\const{Kmi}\const{Kma}\const{dns}/n}$; its derivative is negative for small and positive for large $h$ and vanishes only at the $h_*$ determined by $\beta Lh_*^\beta=\frac D2Ah_*^{-D/2}$, which gives \eqref{eq:localfrechet:bandwidth}.
 At $h_*$, the stochastic term $Ah_*^{-D/2}$ equals $\frac{2\beta}D$ times the bias term $Lh_*^\beta$, so the right-hand side of \eqref{eq:localfrechet:designbound} equals $\frac{2\beta+D}{D}Lh_*^\beta$; squaring and inserting \eqref{eq:localfrechet:bandwidth} proves \eqref{eq:localfrechet:optimal}.
\end{proof}
\subsection{\texorpdfstring{2-Fr\'echet Means in $\mathsf{CAT}(\kappa)$ Spaces, $\kappa>0$}{2-Frechet Means in CAT(kappa) Spaces, kappa > 0}}\label{app:pos}
Throughout this subsection we work in the setting of \cref{ssec:pos:prelim}; in particular all points, all distributions and all Fr\'echet means occurring below lie in $\setC$, and $\diam(\setC)\leq2\varrho<\infty$, so that every expectation appearing is finite.
Measurability of the empirical means in the observations is immediate here: the coupling that moves each atom to its counterpart gives $W_1(\sum_iw_i\delta_{y_i},\sum_iw_i\delta_{y_i\pr})\leq\sum_iw_id(y_i,y_i\pr)$, so \CPp makes $\samfm$ and $\samfm^i$ continuous functions of $(y_i)_{i\in\N}\in\setC^{\N}$ for the product topology, as $\setC$ is bounded and $\sum_iw_i=1$.
\begin{proof}[Proof of \cref{thm:pos:infty:independent}]
	Let $(Y_i\pr)_{i\in\N}$, $\samdis^i$ and $\samfm^i$ be as in \cref{ntt:power:proof}.
	By \cref{lmm:loo}\,\ref{lmm:loo:itm:catk} with $C_{\mathrm{QI}}=2C_\Delta$, and by \VIp applied to both excess risks in \eqref{eq:loo},
	\begin{equation}\label{eq:pos:first}
		\EOf{\ol{\popfm}{\samfm}^2}
		\leq
		\frac{C_\Delta^2}2\sum_{i=1}^\infty w_i\, \EOf{\ol{Y_i}{Y_i\pr}\, \ol{\samfm}{\samfm^i}}
		\eqfs
	\end{equation}
	The coupling that is the identity on all points except the $i$-th gives $W_1(\samdis,\samdis^i)\leq w_i\ol{Y_i}{Y_i\pr}$, so that \CPp yields $\ol{\samfm}{\samfm^i}\leq C_\Delta^2\, w_i\, \ol{Y_i}{Y_i\pr}$.
	Hence
	\begin{equation*}
		\EOf{\ol{\popfm}{\samfm}^2}
		\leq
		\frac{C_\Delta^4}2\sum_{i=1}^\infty w_i^2\, \EOf{\ol{Y_i}{Y_i\pr}^2}
		=
		C_\Delta^4\sum_{i=1}^\infty w_i^2\, \VmOf{Y_i}
		\eqfs
	\end{equation*}
    To obtain \eqref{eq:pos:infty:refined}, use $A_{\Delta,\delta}$ in both variance inequalities and $C_{\mathrm{QI}}=2C_\delta$ in \cref{lmm:loo}, so that the prefactor in \eqref{eq:pos:first} becomes $C_\delta A_{\Delta,\delta}/2$.
    The refined contraction bound gives $\ol{\samfm}{\samfm^i}\leq B_{\Delta,\delta}w_i\ol{Y_i}{Y_i\pr}$, and the same calculation yields the multiplier $C_\delta A_{\Delta,\delta}B_{\Delta,\delta}$.
\end{proof}
\begin{proof}[Proof of \cref{cor:pos:concentration}]
    Replacing observation $i$ changes $\ol{\samfm}{\popfm}$ by at most $r_i:=C_\Delta^2w_i\diam(\setC)$, by \eqref{eq:2fmcatk:CP}.
    The Doob martingale $M_N:=\Eof{\ol{\samfm}{\popfm}\mid Y_1,\ldots,Y_N}$, with $M_0:=\Eof{\ol{\samfm}{\popfm}}$, has conditional increment ranges of lengths at most $r_i$. Iterating the conditional Hoeffding bound gives
    \begin{equation*}
        \EOf{\exp\br{\lambda(M_N-\EOf{\ol{\samfm}{\popfm}})}}
        \leq\exp\br{\frac{\lambda^2}{8}\sum_{i=1}^N r_i^2},
        \qquad \lambda>0,\quad N\in\N.
    \end{equation*}
    Since $0\leq M_N\leq \diam(\setC)$ and $M_N\to \ol{\samfm}{\popfm}$ almost surely as $N\to\infty$, dominated convergence gives the same bound with $M_N$ replaced by $\ol{\samfm}{\popfm}$ and the sum extended to infinity.
    Here $\sum_{i=1}^{\infty}r_i^2<\infty$ because $\sum_{i=1}^{\infty}w_i^2\leq1$.
    Exponential Markov and optimization over $\lambda$ yield
    \begin{equation*}
        \PrOf{\ol{\samfm}{\popfm}-\EOf{\ol{\samfm}{\popfm}}\geq t}
        \leq\exp\br{-\frac{2t^2}{\sum_{i=1}^{\infty} r_i^2}},
        \qquad t>0.
    \end{equation*}
    Finally, \cref{thm:pos:infty:independent} gives
    \begin{equation*}
        \EOf{\ol{\samfm}{\popfm}}\leq\br{\EOf{\ol{\samfm}{\popfm}^2}}^{1/2}
        \leq C_\Delta^2\br{\sum_{i=1}^{\infty} w_i^2\VmOf{Y_i}}^{1/2}.
    \end{equation*}
    Since $\sum_{i=1}^{\infty} r_i^2=C_\Delta^4\diam(\setC)^2\normof{w}_2^2$, taking
    $t=C_\Delta^2\diam(\setC)\sqrt{\normof{w}_2^2\log(1/\eta)/2}$ proves \eqref{eq:pos:concentration}.
\end{proof}
\begin{proof}[Proof of \cref{cor:pos:localfrechet}]
	Write $w_i:=w_{i,h}(x)$ and $\popfm(x) := m_2(\sum_i w_i\mu_{x_i})$; by \eqref{eq:localfrechet:design}, the weights are well-defined and sum to one.
	By \CPp, the coupling that maps $\mu_{x_i}$ to $\mu_x$ optimally in each component, and \eqref{eq:two:localreg:smooth},
	\begin{equation*}
	\begin{aligned}
		\ol{\popfm(x)}{m_2(\mu_x)}
		&\leq
		C_\Delta^2\, W_1\brOf{\sum_{i=1}^{n} w_i\mu_{x_i},\ \mu_x}
		\leq
		C_\Delta^2\sum_{i=1}^{n} w_i\, W_1(\mu_{x_i},\mu_x)
		\\&\leq
		C_\Delta^2\, L\sum_{i=1}^{n} w_i\rho(x_i,x)^{\beta}
		\leq
		C_\Delta^2\, L\, h^{\beta}
		\eqcm
	\end{aligned}
	\end{equation*}
	because $w_i = 0$ unless $\rho(x_i,x)\leq h$.
	By \cref{thm:pos:infty:independent}, $\Eof{\ol{\popfm(x)}{\samfm_h\!(x)}^2}\leq C_\Delta^4\sum_i w_i^2\Vmof{Y_i}\leq C_\Delta^4 v\normof{w}_2^2$, and $v\leq\diam(\setC)^2/2<\infty$.
	The triangle inequality in $L^2$ and the bound $\normof{w}_2^2\leq\const{Kmi}\const{Kma}\const{dns}/(nh^D)$ from the proof of \cref{cor:two:localreg} give \eqref{eq:pos:localfrechet}.
\end{proof}
\subsection{Power Fr\'echet Means in Hadamard Spaces}\label{app:power}

The proof of \cref{thm:power:main} combines the leave-one-out inequality (\cref{lmm:loo}) with a bound on the fluctuations of the random modulus $\sum_{i} w_i\, \ol{Y_i}{\popfm}^{\alpha-1}$ of the moment forms \eqref{eq:power:vi:moment} and \eqref{eq:power:cp:moment} for the empirical distribution (\cref{lmm:power:empmom}).
The correction terms produced by these fluctuations are absorbed by the two leading terms, through an interpolation and an absorption inequality for the family $(\Sigma_s)_s$ (\cref{lmm:power:interpol,lmm:power:absorb}), which express moment log-convexity: on the disjoint union of the coordinate probability spaces, $\Sigma_s$ is half the integral of the $s$-th power of $w_i^{1/(\alpha-1)}d(Y_i,Y_i\pr)$.

We work in the weighted-sample setting of \cref{ssec:main:bound}, with the notation of \cref{ntt:power:proof}, so that $\Sigma_s = \frac12\sum_j w_j^{s/(\alpha-1)}\Eof{\ol{Y_j}{Y_j\pr}^s}$. Initially, the leave-one-out argument uses the additional assumption $\popdis\in\mc P^\alpha(\mc Q)$ of \cref{lmm:loo}; the last step of the proof of \cref{prp:power:excess} removes it by coordinate truncation.
The following moment bounds also hold with infinite values.
By the triangle inequality and \cref{lmm:power:subadd:gen}, $\Vmpof{s}{Y_i}=\frac12\Eof{\ol{Y_i}{Y_i\pr}^{s}}\leq2^{\max(s,1)-1}\Eof{\ol{Y_i}{\popfm}^{s}}$, so that, as $w_i \leq 1$ and $\frac{s}{\alpha-1}\geq 1$ for $s \geq \alpha-1$,
\begin{equation}\label{eq:power:sigmaw:bound}
	\Sigma_s \leq 2^{\max(s,1)-1}\, w_{\max}^{\frac{s}{\alpha-1}-1}\sum_{i=1}^\infty w_i\, \EOf{\ol{Y_i}{\popfm}^{s}} \leq 2\, w_{\max}^{\frac{s}{\alpha-1}-1}\, \sigma_s \leq 2\, \sigma_s
	\eqcm\qquad s \in [\alpha-1,\alpha]
	\eqfs
\end{equation}
In particular, $\Sigma_{\alpha-1}\leq\sigma_{\alpha-1}$.
\subsubsection{The Modified Loss}
Define the distance loss, including its endpoints, by
\begin{equation}\label{eq:power:gdef}
 g(0):=0,\qquad
 g(t):=\frac{t^2}{(\sigma_{\alpha-1}+t^{\alpha-1})^\phi},\quad t>0.
\end{equation}
For $\alpha=2$ this is $g(t)=t^2$ by the exponent-zero convention, and for $\alpha<2$ with $\sigma_{\alpha-1}=0$ it is $g(t)=t^\alpha$ on all of $\Rp$.
Every occurrence of the quotient in a risk bound carries these interpretations, here and below.
\begin{proposition}[Excess risk and modified distance losses]\label{prp:power:excess}
	In the setting of \cref{ssec:main:bound}, there is a constant $c_\alpha$ depending only on $\alpha$ such that
	\begin{equation}\label{eq:power:excess:bound}
		\EOf{\mathcal E_{\popdis}(\samfm)}
		\leq c_\alpha\br{\sigma_{\alpha-1}^{\phi}\,\Sigma_{2\alpha-2}+\Sigma_\alpha}.
	\end{equation}
	Consequently, with the inverse moment \eqref{eq:power:Lambda},
	\begin{equation}\label{eq:power:main}
		\EOf{g(\ol{\popfm}{\samfm})}
		\leq \EOf{\ol{\popfm}{\samfm}^2\LamOf{\popdis}{\ol{\popfm}{\samfm}}}
		\leq
		c_\alpha \br{\sigma_{\alpha-1}^{\phi}\, \Sigma_{2\alpha-2} + \Sigma_\alpha}
		\eqfs
	\end{equation}
	The product $\ol{\popfm}{\samfm}^2\Lamof{\popdis}{\ol{\popfm}{\samfm}}$ is zero at $\ol{\popfm}{\samfm}=0$.
\end{proposition}
The second quantity in \eqref{eq:power:main} is the loss by which \VIa bounds the excess risk from below, and the first is its moment form \eqref{eq:power:vi:moment}; the following comparison places $g$ between the quadratic and the power-$\alpha$ loss at the scale set by $\sigma_{\alpha-1}$.
\subsubsection{Conversion to Ordinary Distance Loss}
Write $\sigma:=\sigma_{\alpha-1}$.
For every $\sigma\in\Rp$ the map $t\mapsto t^2(\sigma+t^{\alpha-1})^{-\phi}$ is strictly increasing on $\Rpp$, since its derivative has the sign of $2(\sigma+t^{\alpha-1})-\phi(\alpha-1)\,t^{\alpha-1}=2\sigma+\alpha t^{\alpha-1}>0$ by $\phi(\alpha-1)=2-\alpha$; in particular $g$ is strictly increasing.
For $\sigma>0$, the scale-explicit comparison for $g$ is
\begin{equation*}
 2^{-\phi}\min\{\sigma^{-\phi}t^2,t^\alpha\}
 \leq g(t)\leq\min\{\sigma^{-\phi}t^2,t^\alpha\},\qquad t\geq0,
\end{equation*}
by $\max(\sigma,t^{\alpha-1})\leq\sigma+t^{\alpha-1}\leq2\max(\sigma,t^{\alpha-1})$: the loss is quadratic below the distance scale $\sigma^{1/(\alpha-1)}$ and of power $\alpha$ above it, with the crossover at $t=\sigma^{1/(\alpha-1)}$.
\Cref{prp:app:conversion}, proved in \cref{sec:powfun}, converts an expected $g$-loss into an $L^\alpha$ bound while retaining the scale; Jensen's inequality then covers every positive loss order $0<\gamma\leq\alpha$.
Together with an absorption step, this gives \cref{thm:power:main}; both steps are detailed in its proof below.
\subsubsection{Interpolation, Absorption, and Fluctuation Bounds}
\begin{lemma}[Interpolation]\label{lmm:power:interpol}
	In the setting of \cref{ssec:main:bound}, let $s_0,s_1>0$ and $\lambda\in[0,1]$. Then
	\begin{equation}\label{eq:power:interpol}
		\Sigma_{\lambda s_0 + (1-\lambda)s_1}
		\leq
		\Sigma_{s_0}^{\lambda}\; \Sigma_{s_1}^{1-\lambda}
		\eqcm
	\end{equation}
	i.e., the moment family is log-convex on every interval of finite positive moments. If all moments vanish, read the assertion in its displayed multiplicative form.
	In particular,
	\begin{enumerate}[label=(\roman*)]
		\item \label{lmm:power:interpol:itm:main}
		$\Sigma_{2\alpha-2}^{1+\phi} \leq \Sigma_{\alpha-1}^{\phi}\, \Sigma_{\alpha}$;
		\item \label{lmm:power:interpol:itm:chord}
		if $\phi \geq 2$, then $\Sigma_{2-\alpha}^{\phi}\, \Sigma_{2\alpha-2}^{\phi-1} \leq \Sigma_{\alpha-1}^{\phi}\, \Sigma_{\alpha}^{\phi-1}$.
	\end{enumerate}
\end{lemma}
\begin{proof}
	For all $s>0$, $\Sigma_s=\frac12\sum_{i=1}^\infty\Eof{(w_i^{1/(\alpha-1)}\ol{Y_i}{Y_i\pr})^s}$.
	Writing $s := \lambda s_0 + (1-\lambda)s_1$ and applying H\"older's inequality first within each expectation and then to the sum over $i$ with the weights $\frac12$,
 \begin{align*}
  \Sigma_s
  &=\frac12\sum_{i=1}^\infty
       w_i^{s/(\alpha-1)}\EOf{\ol{Y_i}{Y_i\pr}^{\lambda s_0+(1-\lambda)s_1}}\\
  &\leq\frac12\sum_{i=1}^\infty
       \br{w_i^{s_0/(\alpha-1)}\EOf{\ol{Y_i}{Y_i\pr}^{s_0}}}^{\lambda}
       \br{w_i^{s_1/(\alpha-1)}\EOf{\ol{Y_i}{Y_i\pr}^{s_1}}}^{1-\lambda}\\
  &\leq\Sigma_{s_0}^{\lambda}\,\Sigma_{s_1}^{1-\lambda}
  \eqcm
 \end{align*}
	which is \eqref{eq:power:interpol}; the endpoint cases $\lambda=0,1$ are identities.
	No finiteness hypothesis is needed: the H\"older chain is valid in $[0,\infty]$, and the only indeterminate product it could produce is excluded, since $\Sigma_{s_0}=0$ or $\Sigma_{s_1}=0$ forces $w_i^{1/(\alpha-1)}\ol{Y_i}{Y_i\pr}=0$ almost surely for every $i$ and hence $\Sigma_s=0$ for every $s>0$.
	In the applications below, $\Sigma_{\alpha-1}$ is always finite by \eqref{eq:power:sigmaw:bound}, whereas finiteness of $\Sigma_\alpha$ is established separately at each use.

	\ref{lmm:power:interpol:itm:main}:
	With $\lambda := 2-\alpha \in [0,1)$ we have $\lambda(\alpha-1) + (1-\lambda)\alpha = (\alpha-1)(2-\alpha) + (\alpha-1)\alpha = 2\alpha-2$, so that \eqref{eq:power:interpol} gives $\Sigma_{2\alpha-2}\leq \Sigma_{\alpha-1}^{2-\alpha}\,\Sigma_\alpha^{\alpha-1}$.
	Raising this to the power $\frac1{\alpha-1} = 1 + \phi$ yields the assertion.

	\ref{lmm:power:interpol:itm:chord}:
	Note that $\phi \geq 2$ is equivalent to $\alpha \leq \frac43$, in which case $\alpha - 1 \leq 2\alpha-2 \leq 2-\alpha \leq \alpha$ and $2\alpha-2\in[0,1]$.
	Hence, $2-\alpha$ and $2\alpha-2$ are both convex combinations of $\alpha-1$ and $\alpha$, namely with the weights $\lambda = 2\alpha-2$ and $\lambda = 2-\alpha$ on $\alpha-1$, respectively, and \eqref{eq:power:interpol} gives
	\begin{equation*}
		\Sigma_{2-\alpha} \leq \Sigma_{\alpha-1}^{2\alpha-2}\, \Sigma_\alpha^{3-2\alpha}
		\eqcm\qquad
		\Sigma_{2\alpha-2} \leq \Sigma_{\alpha-1}^{2-\alpha}\, \Sigma_\alpha^{\alpha-1}
		\eqfs
	\end{equation*}
	Multiplying the $\phi$-th power of the first with the $(\phi-1)$-th power of the second gives the assertion, because $\phi(2\alpha-2)+(\phi-1)(2-\alpha)=\phi$ and $\phi(3-2\alpha)+(\phi-1)(\alpha-1)=\phi-1$, as $\phi=\frac{2-\alpha}{\alpha-1}$ and $\phi-1=\frac{3-2\alpha}{\alpha-1}$.
\end{proof}
\begin{lemma}[Absorption]\label{lmm:power:absorb}
	In the setting of \cref{ssec:main:bound}, set
	\begin{equation}\label{eq:power:B}
		B_1 := \Sigma_{\alpha-1}^{\phi}\, \Sigma_{2\alpha-2}
		\eqcm\qquad
		B_2 := \Sigma_\alpha
		\eqfs
	\end{equation}
	Then
	\begin{enumerate}[label=(\roman*)]
		\item \label{lmm:power:absorb:itm:var}
		$\Sigma_{2\alpha-2}^{1+\frac\phi2}\leq\sqrt{B_1B_2}\leq\max(B_1,B_2)$;
		\item \label{lmm:power:absorb:itm:ros}
		if $\phi\geq2$, then $\Sigma_{2\alpha-2}\Sigma_{2-\alpha}\leq B_1^{1/\phi}B_2^{1-1/\phi}\leq\max(B_1,B_2)$.
	\end{enumerate}
\end{lemma}
\begin{proof}
 Multiply \cref{lmm:power:interpol}\,\ref{lmm:power:interpol:itm:main} by $\Sigma_{2\alpha-2}$ and take a square root to obtain the first geometric-mean bound.
 For $\phi\geq2$, multiply \cref{lmm:power:interpol}\,\ref{lmm:power:interpol:itm:chord} by $\Sigma_{2\alpha-2}$ and take the $\phi$-th root to obtain the second.
 Each geometric mean is at most the maximum of its factors. Zero cases follow directly, with the exponent-zero convention when $\phi=0$.
\end{proof}
\begin{lemma}[Fluctuations of the empirical modulus]\label{lmm:power:empmom}
	In the setting of \cref{ssec:main:bound}, set
	\begin{equation*}
		V^{-i} := \abs{\hat\sigma_{\alpha-1}^{-i} - \EOf{\hat\sigma_{\alpha-1}^{-i}}}
		\eqcm\quad i\in\N
		\eqfs
	\end{equation*}
	Let $s\in\Rpp$. Then
	\begin{enumerate}[label=(\roman*)]
		\item \label{lmm:power:empmom:itm:low}
		if $s\leq2$, then $\Eof{(V^{-i})^{s}}\leq\Sigma_{2\alpha-2}^{s/2}$;
		\item \label{lmm:power:empmom:itm:high}
		if $s\geq2$, then $\Eof{(V^{-i})^{s}} \leq c_s(\Sigma_{(\alpha-1)s} + \Sigma_{2\alpha-2}^{\frac s2})$, with a finite constant $c_s$ depending only on $s$.
	\end{enumerate}
\end{lemma}
\begin{proof}
	Write $\xi_j := \ol{Y_j}{\popfm}^{\alpha-1}$, $\xi_j\pr := \ol{Y_j\pr}{\popfm}^{\alpha-1}$, and $X_j := \xi_j - \Eof{\xi_j}$ for $j\in\N$ with $w_j>0$.
	The $X_j$ are independent and centered. For each $i$, the partial sums of $\sum_{j\ne i}w_jX_j$ converge almost surely to $\hat\sigma_{\alpha-1}^{-i}-\Eof{\hat\sigma_{\alpha-1}^{-i}}$, because $\Eof{\hat\sigma_{\alpha-1}}=\sigma_{\alpha-1}<\infty$ implies $\hat\sigma_{\alpha-1}<\infty$ almost surely.

	\emph{Individual moments.}
	For every $r\geq1$,
	\begin{equation}\label{eq:power:empmom:individual}
		\EOf{\abs{X_j}^{r}}
		\leq
		\EOf{\ol{Y_j}{Y_j\pr}^{(\alpha-1)r}}
		\eqfs
	\end{equation}
	Indeed, independence and centering give $X_j = \Eof{\xi_j - \xi_j\pr\mid Y_j}$, so that conditional Jensen, \cref{lmm:power:holder}, and the triangle inequality yield
	$\Eof{\absof{X_j}^{r}} \leq \Eof{\absof{\xi_j - \xi_j\pr}^{r}} \leq \Eof{\absof{\ol{Y_j}{\popfm} - \ol{Y_j\pr}{\popfm}}^{(\alpha-1)r}} \leq \Eof{\ol{Y_j}{Y_j\pr}^{(\alpha-1)r}}$.

	\ref{lmm:power:empmom:itm:low}:
	If $\Sigma_{2\alpha-2}=\infty$ there is nothing to prove, so assume $\Sigma_{2\alpha-2}<\infty$.
	The independent-copy identity gives
 \[
 \EOf{X_j^2}=\frac12\EOf{(\xi_j-\xi_j\pr)^2}\leq\frac12\EOf{\ol{Y_j}{Y_j\pr}^{2\alpha-2}}=\VmpOf{2\alpha-2}{Y_j}.
 \]
 Hence $\sum_{j\ne i}w_j^2\Eof{X_j^2}\leq\Sigma_{2\alpha-2}<\infty$, so the centered partial sums converge in $L^2$ to their almost sure limit. Therefore $\Eof{(V^{-i})^2}=\sum_{j\ne i}w_j^2\Eof{X_j^2}\leq\Sigma_{2\alpha-2}$. Jensen gives the stated bound for $s\leq2$.

	\ref{lmm:power:empmom:itm:high}:
	Apply the Rosenthal inequality, \cref{lmm:rosenthal} (bounding the maximum by the sum), to $(w_jX_j)_{j\leq N,\,j\ne i}$ and insert \eqref{eq:power:empmom:individual} with $r = s$ and $r = 2$; since $w_j^{s} = w_j^{(\alpha-1)s/(\alpha-1)}$, the two resulting series are bounded by $2\Sigma_{(\alpha-1)s}$ and $2\Sigma_{2\alpha-2}$.
	Fatou's lemma applied to $\absof{\sum_{j\leq N,\,j\ne i}w_jX_j}^{s}$ transfers the bound to the almost sure limit as $N\to\infty$.
\end{proof}
\subsubsection{Proof of the Weighted-Sample Bound}
\begin{proof}[Proof of \cref{prp:power:excess}]
    The means exist under $\popdis\in\mc P^{\alpha-1}(\mc Q)$ by Tonelli's theorem; the continuity argument in Step 7 also proves measurability under this assumption.
    If $\Sigma_\alpha=\infty$, the claims are immediate. Otherwise $\Sigma_{2\alpha-2}<\infty$ by interpolation and $\Sigma_{\alpha-1}\leq\sigma_{\alpha-1}<\infty$.
    In Steps 1--6, first suppose $\popdis\in\mc P^\alpha(\mc Q)$.

	Fix $i \in \N$ with $w_i > 0$; indices with $w_i = 0$ do not contribute to any of the sums below.

	\emph{Step 1 (leave-one-out bound).}
	We show
	\begin{equation}\label{eq:power:loo}
		\ol{\samfm}{\samfm^i} \leq c_\alpha\, w_i\, \ol{Y_i}{Y_i\pr}^{\alpha-1} \br{\hat\sigma_{\alpha-1}^{\,\phi} + \br{\hat\sigma_{\alpha-1}^i}^{\phi}}
		\eqfs
	\end{equation}
	By the triangle inequality, \cref{lmm:power:subadd:gen}, $\sum_j w_j = 1$, and \eqref{eq:power:rb},
	\begin{equation*}
		\sigma_\alpha^{\alpha-1}(\samdis)
		=
		\sum_{j=1}^\infty w_j\, \ol{Y_j}{\samfm}^{\alpha-1}
		\leq
		\hat\sigma_{\alpha-1} + \ol{\samfm}{\popfm}^{\alpha-1}
		\leq
		c_\alpha \hat\sigma_{\alpha-1}
		\eqcm
	\end{equation*}
	and analogously $\sigma_\alpha^{\alpha-1}(\samdis^i) \leq c_\alpha\,\hat\sigma_{\alpha-1}^i$, as well as
	\begin{equation*}
		\ol{\samfm}{\samfm^i}^{\alpha - 1}
		\leq
		\ol{\samfm}{\popfm}^{\alpha-1} + \ol{\popfm}{\samfm^i}^{\alpha-1}
		\leq
		c_\alpha\br{\hat\sigma_{\alpha-1} + \hat\sigma_{\alpha-1}^i}
		\eqfs
	\end{equation*}
	Hence, using $(x + y)^\phi \leq \max(1, 2^{\phi-1})(x^\phi + y^\phi)$,
	\begin{equation*}
		\br{\sigma_\alpha^{\alpha-1}(\samdis) + \ol{\samfm}{\samfm^i}^{\alpha-1}}^\phi + \br{\sigma_\alpha^{\alpha-1}(\samdis^i) + \ol{\samfm}{\samfm^i}^{\alpha-1}}^\phi
		\leq
		c_\alpha \br{\hat\sigma_{\alpha-1}^{\,\phi} + \br{\hat\sigma_{\alpha-1}^i}^{\phi}}
		\eqfs
	\end{equation*}
	Since the coupling of $\samdis$ and $\samdis^i$ that is the identity on all atoms except the $i$-th one yields $W_{\alpha-1}(\samdis, \samdis^i)\leq w_i\, \ol{Y_i}{Y_i\pr}^{\alpha-1}$, \eqref{eq:power:cp:moment} gives \eqref{eq:power:loo} whenever $\ol{\samfm}{\samfm^i}>0$; for $\ol{\samfm}{\samfm^i}=0$, \eqref{eq:power:loo} holds trivially.

	\emph{Step 2 (symmetrization).}
	Inserting \eqref{eq:power:loo} into \eqref{eq:loo}, after dropping the nonnegative second excess risk there, yields
	\begin{equation*}
		\EOf{\mathcal E_{\popdis}(\samfm)}
		\leq
		c_\alpha \sum_{i=1}^\infty w_i^2\, \EOf{\ol{Y_i}{Y_i\pr}^{2\alpha - 2} \br{\hat\sigma_{\alpha-1}^{\,\phi} + \br{\hat\sigma_{\alpha-1}^i}^{\phi}}}
		\eqfs
	\end{equation*}
	Exchanging $Y_i$ and $Y_i\pr$ leaves $\ol{Y_i}{Y_i\pr}$ invariant, maps $\hat\sigma_{\alpha-1}^i$ to $\hat\sigma_{\alpha-1}$, and does not change the joint distribution.
	Hence, $\Eof{\ol{Y_i}{Y_i\pr}^{2\alpha-2}(\hat\sigma_{\alpha-1}^i)^\phi} = \Eof{\ol{Y_i}{Y_i\pr}^{2\alpha-2}\hat\sigma_{\alpha-1}^{\,\phi}}$ and
	\begin{equation}\label{eq:power:sym}
		\EOf{\mathcal E_{\popdis}(\samfm)}
		\leq
		c_\alpha \sum_{i=1}^\infty w_i^2\, \EOf{\ol{Y_i}{Y_i\pr}^{2\alpha - 2}\, \hat\sigma_{\alpha-1}^{\,\phi}}
		\eqfs
	\end{equation}

	\emph{Step 3 (three-way split of the modulus).}
	For $\alpha = 2$ we have $\phi = 0$ and $\hat\sigma_{\alpha-1}^{\,\phi} = 1$, so \eqref{eq:power:sym} already proves \eqref{eq:power:excess:bound}; the variance inequality and Jensen step below give \eqref{eq:power:main}, so Steps 3--6 may be skipped; assume therefore $\alpha\in(1,2)$, so that $\phi\in\Rpp$.
	Fix $i\in\N$ with $w_i>0$ and set
	\begin{equation*}
		V^{-i} := \abs{\hat\sigma_{\alpha-1}^{-i} - \EOf{\hat\sigma_{\alpha-1}^{-i}}}
		\eqcm\qquad
		X_j := \ol{Y_j}{\popfm}^{\alpha-1} - \EOf{\ol{Y_j}{\popfm}^{\alpha-1}}
		\eqcm\quad j\in\N \text{ with } w_j>0
		\eqfs
	\end{equation*}
	From $\hat\sigma_{\alpha-1} - \sigma_{\alpha-1} = (\hat\sigma_{\alpha-1}^{-i} - \Eof{\hat\sigma_{\alpha-1}^{-i}}) + w_i\, X_i$ and the three-term case of \cref{lmm:power:subadd:gen}, $(x+y+z)^{\phi}\leq\max(1, 3^{\phi-1})(x^{\phi} + y^{\phi} + z^{\phi})$,
	\begin{equation*}
		\hat\sigma_{\alpha-1}^{\,\phi}
		\leq
		c_\alpha\br{\sigma_{\alpha-1}^{\phi} + \br{V^{-i}}^{\phi} + w_i^{\phi}\, \abs{X_i}^{\phi}}
		\eqcm
	\end{equation*}
	and inserting this into \eqref{eq:power:sym} yields
	\begin{equation}\label{eq:power:split}
		\EOf{\mathcal E_{\popdis}(\samfm)}
		\leq
		c_\alpha\sum_{i=1}^\infty w_i^{2}\br{
			\sigma_{\alpha-1}^{\phi}\, \EOf{\ol{Y_i}{Y_i\pr}^{2\alpha-2}}
			+
			\EOf{\ol{Y_i}{Y_i\pr}^{2\alpha-2}\br{V^{-i}}^{\phi}}
			+
			w_i^{\phi}\, \EOf{\ol{Y_i}{Y_i\pr}^{2\alpha-2}\, \abs{X_i}^{\phi}}
		}
		\eqfs
	\end{equation}
	The first summands add up to $2\sigma_{\alpha-1}^{\phi}\, \Sigma_{2\alpha-2}$.

	\emph{Step 4 (leave-one-out fluctuation).}
	Since $V^{-i}$ is independent of $(Y_i, Y_i\pr)$, the expectation $\Eof{\ol{Y_i}{Y_i\pr}^{2\alpha-2}(V^{-i})^{\phi}}$ factorizes into $\Eof{\ol{Y_i}{Y_i\pr}^{2\alpha-2}}\, \Eof{(V^{-i})^{\phi}}$.
	By \cref{lmm:power:empmom}\,\ref{lmm:power:empmom:itm:low} and \ref{lmm:power:empmom:itm:high}, applied to $V^{-i}$ with $s = \phi$ and using $(\alpha-1)\phi = 2-\alpha$,
	\begin{equation*}
		\EOf{\br{V^{-i}}^{\phi}}
		\leq
		\begin{cases}
			\Sigma_{2\alpha-2}^{\frac\phi2} & \text{if } \phi\leq2\eqcm\\[0.5ex]
			c_\alpha\br{\Sigma_{2-\alpha} + \Sigma_{2\alpha-2}^{\frac\phi2}} & \text{if } \phi\geq2\eqfs
		\end{cases}
	\end{equation*}
	Hence, the second summands of \eqref{eq:power:split} add up to at most $c_\alpha(\Sigma_{2\alpha-2}^{1+\frac\phi2} + \ind_{\{\phi>2\}}\, \Sigma_{2\alpha-2}\, \Sigma_{2-\alpha})$.

	\emph{Step 5 (diagonal term).}
	H\"older's inequality with the conjugate exponents $\frac{\alpha}{2\alpha-2}$ and $\frac{\alpha}{2-\alpha}$, followed by \eqref{eq:power:empmom:individual} with $r=\phi\cdot\frac{\alpha}{2-\alpha}=\frac{\alpha}{\alpha-1}\geq1$, yields
	\begin{equation*}
		\EOf{\ol{Y_i}{Y_i\pr}^{2\alpha-2}\, \abs{X_i}^{\phi}}
		\leq
		\EOf{\ol{Y_i}{Y_i\pr}^{\alpha}}^{\frac{2\alpha-2}{\alpha}}\, \EOf{\abs{X_i}^{\frac{\alpha}{\alpha-1}}}^{\frac{2-\alpha}{\alpha}}
		\leq
		\EOf{\ol{Y_i}{Y_i\pr}^{\alpha}}
		\eqcm
	\end{equation*}
	so that the third summands of \eqref{eq:power:split} add up to at most $\sum_{i} w_i^{2+\phi}\, \Eof{\ol{Y_i}{Y_i\pr}^{\alpha}} = 2\Sigma_{\alpha}$.

	\emph{Step 6 (absorption).}
	By \cref{lmm:power:absorb}\,\ref{lmm:power:absorb:itm:var} and, for $\phi\geq2$, \ref{lmm:power:absorb:itm:ros},
	\begin{equation*}
		\Sigma_{2\alpha-2}^{1+\frac\phi2}
		\leq
		\max\brOf{\Sigma_{\alpha-1}^{\phi}\, \Sigma_{2\alpha-2},\ \Sigma_{\alpha}}
		\quad\text{and}\quad
		\Sigma_{2\alpha-2}\, \Sigma_{2-\alpha}
		\leq
		\max\brOf{\Sigma_{\alpha-1}^{\phi}\, \Sigma_{2\alpha-2},\ \Sigma_{\alpha}}
		\eqcm
	\end{equation*}
	and $\Sigma_{\alpha-1}^{\phi}\leq\sigma_{\alpha-1}^{\phi}$, since $\Sigma_{\alpha-1}\leq\sigma_{\alpha-1}$.

	Combining Steps 2--6 proves \eqref{eq:power:excess:bound}.
    Finally, \VIa and Jensen's inequality \eqref{eq:power:cp:jensen} give, pointwise,
    \begin{equation}\label{eq:power:excess:conversion}
        g(\ol{\popfm}{\samfm})\leq \ol{\popfm}{\samfm}^2\LamOf{\popdis}{\ol{\popfm}{\samfm}}
        \leq\frac{2}{\alpha(\alpha-1)}\,\mathcal E_{\popdis}(\samfm).
    \end{equation}
    At $\ol{\popfm}{\samfm}=0$ all three quantities vanish by convention. Taking expectations proves \eqref{eq:power:main}. Centered-moment upper bounds may subsequently be obtained from $\Eof{\ol{Y_i}{Y_i\pr}^s}\leq2^{\max(s,1)}\Eof{\ol{Y_i}{\popfm}^s}$ for $s\in\{2\alpha-2,\alpha\}$.

    \emph{Step 7 (minimal mixture moment).}
 Fix $o\in\mc Q$ and put $r_N=\sum_{i>N}w_i$,
 \[
 \popdis_N=\sum_{i=1}^Nw_i\mu_i+r_N\delta_o,
 \qquad
 \samdis_N=\sum_{i=1}^Nw_i\delta_{Y_i}+r_N\delta_o.
 \]
 Finiteness of a pairwise moment of order $a>0$ implies finiteness of the corresponding ordinary moment: by Fubini there is a fixed $z$ with $\int d(y,z)^a\,\dl\mu_i(y)<\infty$, and the triangle inequality then permits the reference point $o$.
 Thus every active component of each truncated problem has a finite $\alpha$-moment, because $\Sigma_\alpha<\infty$.
 The added coordinate at $o$ is deterministic and has zero pairwise fluctuation.
 The case just proved applies to each truncated problem, with the corresponding partial pairwise sums: writing $\Sigma_s^{(N)}:=\sum_{i=1}^{N}w_i^{s/(\alpha-1)}\Vmpof{s}{Y_i}$ for the truncated weighted moments,
 \begin{equation}\label{eq:power:excess:truncated}
  \EOf{\mathcal E_{\popdis_N}\brOf{m_\alpha(\samdis_N)}}
  \leq c_\alpha\br{\br{\sigma_\alpha^{\alpha-1}(\popdis_N)}^{\phi}\,\Sigma_{2\alpha-2}^{(N)}+\Sigma_{\alpha}^{(N)}}
  \eqcm
 \end{equation}
 with one constant $c_\alpha$ for all $N$; the coordinate added at $o$ is deterministic, so $\Sigma_s^{(N)}\leq\Sigma_s$ for every $s$.

 Let $\eta,\eta_N\in\mc P^{\alpha-1}(\mc Q)$ with $W_{\alpha-1}(\eta_N,\eta)\to0$ as $N\to\infty$. All limits in this step are taken as $N\to\infty$.
 The one-sided contraction \eqref{eq:power:cp:onesided} with $\mu=\eta_N$ and $\nu=\eta$, together with $\sigma_\alpha^{\alpha-1}(\eta)<\infty$ and $\widetilde W_{\alpha-1}\leq W_{\alpha-1}$, gives $m_\alpha(\eta_N)\to m_\alpha(\eta)$.
 Moreover, \cref{lmm:power:holder}, the triangle inequality, and \cref{lmm:power:subadd:gen} give $|d(y,q)^{\alpha-1}-d(z,q\pr)^{\alpha-1}|\leq d(y,z)^{\alpha-1}+d(q,q\pr)^{\alpha-1}$ for all $y,z,q,q\pr\in\mc Q$; integrating against a coupling of $\eta_N$ and $\eta$ shows
 \[
 |\sigma_\alpha^{\alpha-1}(\eta_N)-\sigma_\alpha^{\alpha-1}(\eta)|
 \leq W_{\alpha-1}(\eta_N,\eta)+d(m_\alpha(\eta_N),m_\alpha(\eta))^{\alpha-1}\longrightarrow0.
 \]
 For $\alpha=2$, the usual $W_1$ contraction also gives the mean convergence directly.

 In the present truncation,
 \[
 \begin{aligned}
 W_{\alpha-1}(\popdis_N,\popdis)&\leq\sum_{i=N+1}^{\infty}w_i\EOf{\ol{Y_i}o^{\alpha-1}}\to0,\\
 W_{\alpha-1}(\samdis_N,\samdis)&\leq\sum_{i=N+1}^{\infty}w_i\ol{Y_i}o^{\alpha-1}\to0
 \quad\text{almost surely}.
 \end{aligned}
 \]
 The latter follows from Tonelli. Hence both truncated means converge and the population centered $(\alpha-1)$-moments converge. Finite-coordinate means are measurable by continuity, so the limiting empirical mean is measurable as well. This continuity argument uses only $\popdis\in\mc P^{\alpha-1}(\mc Q)$ and also applies when $\Sigma_\alpha=\infty$.
 Here $\Sigma_{\alpha-1}\leq\sigma_{\alpha-1}<\infty$ and H\"older interpolation makes $\Sigma_{2\alpha-2}<\infty$ from $\Sigma_\alpha<\infty$.
 To pass the excess risk to the limit, work on an outcome where the empirical means converge and put
 $q_N:=m_\alpha(\samdis_N)$, $m_N:=m_\alpha(\popdis_N)$,
 $h_N(y):=d(y,q_N)^\alpha-d(y,m_N)^\alpha$.
 The two sequences of centers are bounded. The mean value theorem and the power triangle inequality therefore give
 \[
  |h_N(y)|\leq C(1+d(y,o)^{\alpha-1})
 \]
 with a finite outcome-dependent $C$ independent of $N,y$.
 This envelope is $\popdis$-integrable. Dominated convergence yields
 $\int h_N\,\dl\popdis\to\mathcal E_{\popdis}(\samfm)$, and the omitted coordinates satisfy
 \[
  \left|\int h_N\,\dl\popdis_N-\int h_N\,\dl\popdis\right|
  \leq r_N|h_N(o)|+
       C\sum_{i=N+1}^{\infty}w_i\int(1+d(y,o)^{\alpha-1})\,\dl\mu_i(y)\longrightarrow0.
 \]
 Hence $\mathcal E_{\popdis_N}(q_N)\to\mathcal E_{\popdis}(\samfm)$ almost surely. These excess risks are nonnegative, so Fatou's lemma applied to \eqref{eq:power:excess:truncated}, together with $\Sigma_s^{(N)}\leq\Sigma_s$ and $\sigma_\alpha^{\alpha-1}(\popdis_N)\to\sigma_{\alpha-1}$, proves \eqref{eq:power:excess:bound} in the limit.
 The pointwise comparison \eqref{eq:power:excess:conversion}, valid under an $(\alpha-1)$-moment, then gives both bounds in \eqref{eq:power:main}, without requiring continuity of the inverse moment at zero.

\end{proof}
\begin{proof}[Proof of \cref{thm:power:main}]
    If $\Sigma_\alpha=\infty$, the bound is immediate; otherwise interpolation gives $\Sigma_{2\alpha-2}<\infty$.

	If $\sigma_{\alpha-1} = 0$, then $Y_i = \popfm$ almost surely for every $i$ with $w_i>0$, hence $\samfm = \popfm$ almost surely and there is nothing to show; so assume $\sigma_{\alpha-1}\in\Rpp$.
	Write $\mathcal G := \Eof{g(\ol{\popfm}{\samfm})}$, and set $A_1 := \sigma_{\alpha-1}^{\phi}\, \Sigma_{2\alpha-2}$ and $A_2 := \Sigma_\alpha$, so that $\mathcal G \leq c_\alpha(A_1 + A_2)$ by \cref{prp:power:excess}.

	\emph{Step 1 (conversion of the loss).}
	We apply \cref{prp:app:conversion} with $b = \sigma_{\alpha-1}$, so that $\psi = g$, and take the $\alpha$-th root of the resulting bound.
	Using $(1+x)^{\frac \phi\alpha}\leq c_\alpha(1 + x^{\frac \phi\alpha})$ for $x\in\Rp$ together with the exponent identities $\frac\phi2 - \frac\alpha2\cdot\frac \phi\alpha = 0$ and $\frac12 + \frac{\alpha-1}{2}\cdot\frac \phi\alpha = \frac1\alpha$, which hold because $(\alpha-1)\phi = 2-\alpha$, we obtain
	\begin{equation}\label{eq:power:sharp:phi}
		\br{\EOf{\ol{\popfm}{\samfm}^\alpha}}^{\frac1\alpha}
		\leq
		c_\alpha\, \Phi(\mathcal G)
		\eqcm\qquad\text{where}\qquad
		\Phi\colon\Rp\to\Rp,\ P\mapsto \sigma_{\alpha-1}^{\frac\phi2}\, P^{\frac12} + P^{\frac1\alpha}
		\eqfs
	\end{equation}
	As $t\mapsto t^{1/2}$ and $t\mapsto t^{1/\alpha}$ are nondecreasing and subadditive, so is $\Phi$, and $\Phi(cP)\leq\max(c^{\frac12}, c^{\frac1\alpha})\Phi(P)$ for $c\geq0$.
	Hence, $\mathcal G \leq c_\alpha(A_1+A_2)$ and \eqref{eq:power:sharp:phi} give
	\begin{equation*}
	\begin{aligned}
		\br{\EOf{\ol{\popfm}{\samfm}^\alpha}}^{\frac1\alpha}
		&\leq
		c_\alpha\br{\Phi(A_1) + \Phi(A_2)}
		\eqcm\\
		\Phi(A_1) &= \sigma_{\alpha-1}^{\phi}\, \Sigma_{2\alpha-2}^{\frac12} + A_1^{\frac1\alpha}
		\eqcm\\
		\Phi(A_2) &= \sigma_{\alpha-1}^{\frac\phi2}\, \Sigma_\alpha^{\frac12} + \Sigma_\alpha^{\frac1\alpha}
		\eqfs
	\end{aligned}
	\end{equation*}
	It remains to bound $A_1^{1/\alpha}$ by the right-hand side of \eqref{eq:power:rate:sharp}.

	\emph{Step 2 (absorption of $A_1^{1/\alpha}$).}
	Set $\vartheta := \sigma_{\alpha-1}^{\frac{\alpha}{\alpha-1}}$.
	By \eqref{eq:power:interpol} with $s_0=\alpha-1$, $s_1=\alpha$, and $\lambda=2-\alpha$, $\Sigma_{2\alpha-2}\leq\Sigma_{\alpha-1}^{2-\alpha}\Sigma_\alpha^{\alpha-1}$, together with $\Sigma_{\alpha-1}\leq\sigma_{\alpha-1}$ and the identity $\phi + 2 - \alpha = \phi\alpha = \frac{\alpha(2-\alpha)}{\alpha-1}$,
	\begin{equation*}
		A_1
		=
		\sigma_{\alpha-1}^{\phi}\, \Sigma_{2\alpha-2}
		\leq
		\sigma_{\alpha-1}^{\phi\alpha}\, \Sigma_\alpha^{\alpha-1}
		=
		\vartheta^{2-\alpha}\, A_2^{\alpha-1}
		\leq
		\max\brOf{\vartheta, A_2}
		\eqcm
	\end{equation*}
	the last step being the weighted arithmetic--geometric mean inequality with the weights $2-\alpha$ and $\alpha-1$, which are nonnegative and sum to $1$.
	If $A_1\leq A_2$, then $A_1^{\frac1\alpha}\leq\Sigma_\alpha^{\frac1\alpha}$.
	Otherwise $A_1\leq\vartheta$, and, using $\frac1\alpha = \frac12 + \frac{2-\alpha}{2\alpha}$ and $\vartheta^{\frac{2-\alpha}{2\alpha}} = \sigma_{\alpha-1}^{\frac\phi2}$,
	\begin{equation*}
		A_1^{\frac1\alpha}
		=
		A_1^{\frac12}\, A_1^{\frac{2-\alpha}{2\alpha}}
		\leq
		\sigma_{\alpha-1}^{\frac\phi2}\, A_1^{\frac12}
		=
		\sigma_{\alpha-1}^{\phi}\, \Sigma_{2\alpha-2}^{\frac12}
		\eqfs
	\end{equation*}
	This completes the three-term bound.
\end{proof}
\begin{proof}[Proof of \cref{cor:power:rate}]
 Extend the finite sample by zero-weight coordinates. With $w_i=n^{-1}\ind_{\{i\leq n\}}$, the identity is
 \begin{equation*}
 \Sigma_s=n^{-s/(\alpha-1)}\sum_{i=1}^n\VmpOf{s}{Y_i}
          =n^{1-s/(\alpha-1)}\bar V_s.
 \end{equation*}
 Apply \cref{thm:power:main}.
 The three terms are, up to constants depending only on $\alpha$,
 \[
 \sigma_{\alpha-1}^{\phi}\bar V_{2\alpha-2}^{1/2}n^{-1/2},\quad
 \sigma_{\alpha-1}^{\phi/2}\bar V_\alpha^{1/2}n^{-1/(2(\alpha-1))},\quad
 \bar V_\alpha^{1/\alpha}n^{-1/(\alpha(\alpha-1))}.
 \]
 This is the pairwise bound \eqref{eq:power:rate:iid}. The exponent comparisons stated in the corollary prove the rate under its uniform moment bounds. Centered-moment versions follow from $\bar V_s\leq2^{\max(s,1)-1}\sigma_s$.
\end{proof}
\subsubsection{The Roles of the Three Moments}
Denote the three summands of \eqref{eq:power:rate:sharp} by
\[
 T_1:=\sigma_{\alpha-1}^\phi\Sigma_{2\alpha-2}^{1/2},\qquad
 T_2:=\sigma_{\alpha-1}^{\phi/2}\Sigma_\alpha^{1/2},\qquad
 T_3:=\Sigma_\alpha^{1/\alpha}.
\]
For two deterministic observations at $-R$ and $R$ with equal weights, $\sigma_{\alpha-1}=R^{\alpha-1}$ while every $\Sigma_s$ is zero: separation of the component means alone can make the common-center moment arbitrarily large, and the fluctuation moments see none of it.
\begin{lemma}[Finiteness and mutual non-domination of the three moments]\label{lmm:power:moments}
 In the setting of \cref{ssec:main:bound}, the following hold.
 \begin{enumerate}[label=(\roman*)]
  \item $T_1+T_2+T_3<\infty$ if and only if $\Sigma_\alpha<\infty$. For finitely many positive weights, this holds if and only if every active component has a finite $\alpha$-moment.
  For infinitely many weights, $\popdis\in\mc P^\alpha(\mc Q)$ is not necessary: like \cref{ex:two:infinitevariance}\,(b), take $\mc Q=\R$, $w_i=2^{-i}$ and $\mu_i=\frac12(\delta_{-a_i}+\delta_{a_i})$ with $a_i=2^{i/\alpha}$; then $\sum_iw_i\Eof{\absof{Y_i}^\alpha}=\infty$, whereas $\sigma_{\alpha-1}=\sum_i2^{-i/\alpha}<\infty$ and $\Sigma_\alpha=2^{\alpha-2}\sum_i2^{-i/(\alpha-1)}<\infty$.
  \item Let $1<\alpha<2$. For every $M<\infty$ there are weights and laws with finite support on $\R$ such that $\Eof{\ol{\popfm}{\samfm}^\alpha}^{1/\alpha}>M(\Sigma_{\alpha-1}^{\phi}\Sigma_{2\alpha-2}^{1/2}+\Sigma_{\alpha-1}^{\phi/2}\Sigma_\alpha^{1/2}+\Sigma_\alpha^{1/\alpha})$; that is, $\sigma_{\alpha-1}$ cannot be replaced by $\Sigma_{\alpha-1}$ in \eqref{eq:power:rate:sharp} with a constant depending only on $\alpha$.
  \item Let $1<\alpha<2$. For each $j\in\{1,2,3\}$, the ratio $T_j/\sum_{k\neq j}T_k$ is unbounded over weights and laws with finite support on $\R$.
 \end{enumerate}
\end{lemma}
\begin{proof}
 \emph{(i).} Since $\Sigma_{\alpha-1}\leq\sigma_{\alpha-1}<\infty$, interpolation gives $\Sigma_{2\alpha-2}^{1/(\alpha-1)}\leq\Sigma_{\alpha-1}^\phi\Sigma_\alpha$ whenever $\Sigma_\alpha<\infty$, so $\Sigma_\alpha<\infty$ implies $T_1+T_2+T_3<\infty$; the converse follows from $T_3$.
 For finitely many positive weights, a finite pairwise moment gives an ordinary moment by Fubini's theorem, as in Step 7 of the proof of \cref{prp:power:excess}, and conversely.
 In the example for infinitely many weights, $\popfm=0$ by symmetry, $\Eof{\absof{Y_i}^{s}}=a_i^{s}$ and $\VmpofP{s}{\mu_i}=2^{s-2}a_i^{s}$, so $w_i\Eof{\absof{Y_i}^\alpha}=1$, $w_ia_i^{\alpha-1}=2^{-i/\alpha}$ and $w_i^{\alpha/(\alpha-1)}\VmpofP{\alpha}{\mu_i}=2^{\alpha-2}2^{-i/(\alpha-1)}$.
 In (ii) and (iii), constants in $\asymp$ may depend on $\alpha$.

 \emph{(ii), and (iii) for $j=1$.}
 Take weights $w,(1-w)/2,(1-w)/2$ and laws $\frac12(\delta_{-1}+\delta_1),\delta_{-1},\delta_1$, respectively, with $w\searrow0$.
 The population mixture is uniform on $\{-1,1\}$, so $\popfm=0$, $\sigma_{\alpha-1}=1$, and $\Sigma_s=2^{s-2}w^{s/(\alpha-1)}$ for every $s>0$.
 Thus
 \[
  T_1\asymp w,\qquad T_2\asymp w^{\alpha/(2(\alpha-1))},\qquad T_3\asymp w^{1/(\alpha-1)},
 \]
 and $T_1/(T_2+T_3)\to\infty$, which is (iii) for $j=1$.
 The empirical mean has a random sign and deterministic absolute value
 \[
  |\samfm|=
  \frac{(1+w)^{1/(\alpha-1)}-(1-w)^{1/(\alpha-1)}}{(1+w)^{1/(\alpha-1)}+(1-w)^{1/(\alpha-1)}}
  \sim\frac{w}{\alpha-1}\qquad(w\searrow0).
 \]
 Replacing $\sigma_{\alpha-1}$ by $\Sigma_{\alpha-1}=2^{\alpha-3}w$ makes all three terms of order $w^{1/(\alpha-1)}=\mathrm{o}(w)$ as $w\searrow0$, which proves (ii).

 \emph{(iii) for $j=2$.}
 Take weights $1/2,1/4,1/4$ and laws $\mu_1=(1-\epsilon)\delta_0+\frac\epsilon2(\delta_{-R}+\delta_R)$, $\delta_{-1}$, $\delta_1$, where $R\to\infty$ and $\epsilon=R^{-\alpha-1}$.
 Symmetry gives $\popfm=0$ and $\sigma_{\alpha-1}=(1+\epsilon R^{\alpha-1})/2\to1/2$.
 For each fixed $s>0$, $\Eof{|Y_1-Y_1\pr|^s}\asymp\epsilon R^s$, so
 \[
  T_1\asymp R^{(\alpha-3)/2},\qquad
  T_2\asymp R^{-1/2},\qquad T_3\asymp R^{-1/\alpha}.
 \]
 Hence $T_2/(T_1+T_3)\to\infty$.

 \emph{(iii) for $j=3$.}
 Take a single observation with law $(1-\epsilon)\delta_0+\epsilon\delta_1$, where $\epsilon\searrow0$.
 Its population mean is $m=\epsilon^{1/(\alpha-1)}/(\epsilon^{1/(\alpha-1)}+(1-\epsilon)^{1/(\alpha-1)})$.
 The first-order condition gives $\sigma_{\alpha-1}=2\epsilon(1-m)^{\alpha-1}\sim2\epsilon$, and $\Sigma_s=\epsilon(1-\epsilon)$ for every $s>0$.
 Thus
 \[
  T_1\asymp\epsilon^{\phi+1/2},\qquad
  T_2\asymp\epsilon^{1/(2(\alpha-1))},\qquad T_3\asymp\epsilon^{1/\alpha}.
 \]
 Since $\phi+1/2>1/\alpha$ and $1/(2(\alpha-1))>1/\alpha$, $T_3/(T_1+T_2)\to\infty$.
\end{proof}
The two ends are stronger than the middle: in the examples for $j=1$ and $j=3$ the risk itself is of the order of $T_j$, so those two terms are matched by a lower bound on $\Eof{\ol{\popfm}{\samfm}^\alpha}^{1/\alpha}$, whereas for $j=2$ the risk is of order $T_3$ and the example shows only that $T_2$ is not dominated by $T_1+T_3$.

\begin{remark}[The constant $c_\alpha$]\label{rem:power:calpha}
 The proof of \cref{thm:power:main} gives a constant $c_\alpha$ that is bounded on every compact subinterval of $(1,2]$, since the only $\alpha$-dependent input that is not explicit, the Rosenthal constant of \cref{lmm:rosenthal} at $s=\phi$, can be chosen locally bounded in $s$.
 Conversely, for $1<\alpha<2$, every constant $c_\alpha$ for which \eqref{eq:power:rate:sharp} holds for all weights and laws satisfies
 \begin{equation*}
  c_\alpha\geq\max\cb{2^\phi,\frac{2^{2-\alpha}}{\alpha-1}}
  \eqcm
 \end{equation*}
 so that $c_\alpha\to\infty$ as $\alpha\searrow1$.
 For the first bound, let $1<\alpha<2$ and take weights $\frac12,\frac12$ and laws $\mu_1=(1-\epsilon)\delta_0+\epsilon\delta_1$, $\mu_2=\delta_0$ on $\R$, with $\epsilon\searrow0$.
 Then $\popfm\sim(\epsilon/2)^{1/(\alpha-1)}$, the empirical mean is $\frac12$ with probability $\epsilon$ and $0$ otherwise, so the risk is $\epsilon^{1/\alpha}/2\,(1+\mathrm o(1))$, while $T_3=2^{-1/(\alpha-1)}(\epsilon(1-\epsilon))^{1/\alpha}$, $T_1=\landauO(\epsilon^{\phi+1/2})$, and $T_2=\landauO(\epsilon^{1/(2(\alpha-1))})$, both $\mathrm o(\epsilon^{1/\alpha})$; hence the ratio of the risk to $T_1+T_2+T_3$ tends to $2^{1/(\alpha-1)-1}=2^\phi$.
 For the second bound, in the example for (ii) and (iii) with $j=1$ in the proof of \cref{lmm:power:moments}, $T_1=2^{\alpha-2}w$ and $T_2+T_3=\mathrm o(w)$, while the risk is $w/(\alpha-1)\,(1+\mathrm o(1))$ as $w\searrow0$.
\end{remark}

\subsubsection{Local Constant Power Fr\'echet Regression}
\begin{proof}[Proof of \cref{cor:power:localfrechet}]
 Put $m_x:=m_\alpha(\mu_x)$, $m_i:=m_\alpha(\mu_{x_i})$,
 $\popdis:=\sum_iw_i\mu_{x_i}$, and $\popfm:=m_\alpha(\popdis)$.
 Set $\mathcal B:=Lh^\beta$, $M:=u^{1/\alpha}$, $a:=\normof{w}_2\leq1$, and
 $\epsilon:=\max_{i:w_i>0}\widetilde W_{\alpha-1}(\mu_{x_i},\mu_x)$.
 For each active index with $x_i\ne x$, put $b_i:=L\rho(x_i,x)^\beta\leq\mathcal B$.
 By \eqref{eq:power:localfrechet:smooth} and monotonicity of $\Psi_{\mu_x}$,
 \begin{equation*}
 \widetilde W_{\alpha-1}(\mu_{x_i},\mu_x)
 \leq b_i\LamOf{\mu_x}{b_i}=\Psi_{\mu_x}(b_i)\leq\Psi_{\mu_x}(\mathcal B).
 \end{equation*}
 Indices with $x_i=x$ have zero transport cost, so $\epsilon\leq\Psi_{\mu_x}(\mathcal B)$.
 Thus \cref{lmm:power:cp:mixture} gives
 \begin{equation}\label{eq:power:localfrechet:bias}
 \max_{i:w_i>0}\ol{m_i}{m_x}+\ol{\popfm}{m_x}
 \leq c_\alpha\Psi_{\mu_x}^{-1}(\epsilon)\leq c_\alpha\mathcal B.
 \end{equation}
 By \cref{lmm:variances},
 $\int\ol y{m_i}^\alpha\,\dl\mu_{x_i}(y)\leq2u=2M^\alpha$.
 The minimizing property of $\popfm$, the power triangle inequality, and \eqref{eq:power:localfrechet:bias} yield
 \begin{equation*}
 \sigma_\alpha^\alpha(\popdis)
 \leq\sum_{i=1}^{n}w_i\int\ol y{m_x}^\alpha\,\dl\mu_{x_i}(y)
 \leq c_\alpha(M+\mathcal B)^\alpha,
 \qquad
 \sigma_\alpha^{\alpha-1}(\popdis)\leq c_\alpha(M+\mathcal B)^{\alpha-1},
 \end{equation*}
 the second by Lyapunov's inequality.
 For independent copies $Y_i\pr$, the definition of $u$ gives
 $\Eof{\ol{Y_i}{Y_i\pr}^\alpha}\leq2M^\alpha$ and, since $2\alpha-2\leq\alpha$,
 $\Eof{\ol{Y_i}{Y_i\pr}^{2\alpha-2}}\leq(2M^\alpha)^{(2\alpha-2)/\alpha}$.
 Set $\tau:=\alpha/(\alpha-1)\geq2$. Then
 \begin{equation*}
 \Sigma_{2\alpha-2}^{1/2}\leq c_\alpha M^{\alpha-1}a,
 \qquad
 \Sigma_\alpha\leq M^\alpha\sum_{i=1}^{n}w_i^\tau\leq M^\alpha a^\tau.
 \end{equation*}
 Inserting these bounds into \cref{thm:power:main}, and using $\tau/2\geq1$, $1/(\alpha-1)\geq1$, and $a\leq1$, gives
 \begin{align*}
 \br{\EOf{\ol{\samfm_h\!(x)}{\popfm}^\alpha}}^{1/\alpha}
 &\leq c_\alpha\br{(M+\mathcal B)^{2-\alpha}M^{\alpha-1}a
              +(M+\mathcal B)^{(2-\alpha)/2}M^{\alpha/2}a^{\tau/2}
              +Ma^{1/(\alpha-1)}}\\
 &\leq c_\alpha(M+\mathcal B)a.
 \end{align*}
 The triangle inequality in $L^\alpha$ and $\mathcal Ba\leq\mathcal B$ give the more general bound
 \begin{equation*}
 \EOf{\ol{\samfm_h\!(x)}{m_\alpha(\mu_x)}^\alpha}^{1/\alpha}
 \leq c_\alpha\br{Lh^\beta+u^{1/\alpha}\normOf{w}_2}.
 \end{equation*}
 The bound $\normof{w}_2^2\leq\const{Kmi}\const{Kma}/N_n(x,h/2)$ and the design assumption give \eqref{eq:power:localfrechet}.
 The bandwidth calculation is the same as for \cref{cor:two:localreg}, with $v=u^{2/\alpha}$, and proves \eqref{eq:power:localfrechet:optimal}.
\end{proof}

\subsubsection{Additive Noise in Hilbert Spaces}\label{app:power:euclidean}
\begin{proof}[Proof of \cref{cor:power:additive}]
 In $\mc H$, translation equivariance gives $m_\alpha[\gamma(x)+\varepsilon]=\gamma(x)+m_\alpha[\varepsilon]=\gamma(x)$; this point lies in $\mc Q$ and minimizes $q\mapsto\Eof{\normof{\gamma(x)+\varepsilon-q}^\alpha}$ over all of $\mc H$, hence over $\mc Q$, so $m_\alpha(\mu_x)=\gamma(x)$ and
 \[
 \LamOf{\mu_x}{t}=\EOf{(\normOf\varepsilon+t)^{\alpha-2}},\qquad t>0,
 \]
 independently of $x$.
 For an independent copy $\varepsilon\pr$, the power triangle inequality gives
 $\VmpofP{\alpha}{\mu_x}=\frac12\Eof{\normof{\varepsilon-\varepsilon\pr}^\alpha}\leq2^{\alpha-1}\Eof{\normof\varepsilon^\alpha}$.
 Fix distinct $x,x\pr$.
 If $\normof{\gamma(x)-\gamma(x\pr)}=0$, the two laws agree and their transport cost vanishes.
 Otherwise, couple $Y_x=\gamma(x)+\varepsilon$ and $Y_{x\pr}=\gamma(x\pr)+\varepsilon$ using the same noise, so that $\ol{Y_x}{Y_{x\pr}}=\normof{\gamma(x)-\gamma(x\pr)}$.
 Put $a:=\gamma(x)-\gamma(x\pr)$. The four points entering the diameter in \eqref{eq:power:Wtilde} are $\gamma(x)+\varepsilon$, $\gamma(x\pr)+\varepsilon$, $\gamma(x)$ and $\gamma(x\pr)$, so that diameter is at least $\max\{\normof{\varepsilon+a},\normof{\varepsilon-a}\}$, and the parallelogram law $\normof{\varepsilon+a}^2+\normof{\varepsilon-a}^2=2(\normof\varepsilon^2+\normof{a}^2)$ gives
 \begin{equation*}
  \diam\cb{Y_x,Y_{x\pr},\gamma(x),\gamma(x\pr)}
  \geq\sqrt{\normOf\varepsilon^2+\normOf{\gamma(x)-\gamma(x\pr)}^2}
  \geq2^{-1/2}\br{\normOf\varepsilon+\normOf{\gamma(x)-\gamma(x\pr)}}
  \eqfs
 \end{equation*}
 Hence
 \begin{align*}
 \widetilde W_{\alpha-1}(\mu_x,\mu_{x\pr})
 &\leq2^{(2-\alpha)/2}\normOf{\gamma(x)-\gamma(x\pr)}\,
     \EOf{(\normOf\varepsilon+\normOf{\gamma(x)-\gamma(x\pr)})^{\alpha-2}}\\
 &=2^{(2-\alpha)/2}\Psi_{\mu_x}(\normOf{\gamma(x)-\gamma(x\pr)})\\
 &\leq\Psi_{\mu_x}(2^{\phi/2}\normOf{\gamma(x)-\gamma(x\pr)}).
 \end{align*}
 The last step uses \eqref{eq:power:Psi:scaling} and $(\alpha-1)\phi=2-\alpha$.
 Since $2^{\phi/2}\normof{\gamma(x)-\gamma(x\pr)}\leq2^{\phi/2}L_\gamma\rho(x,x\pr)^\beta=b$ and both laws have the same $\Lambda_\alpha$, monotonicity proves \eqref{eq:power:localfrechet:smooth} with $L=2^{\phi/2}L_\gamma$. \Cref{cor:power:localfrechet} then gives \eqref{eq:power:additive}, the factors $2^{\phi/2}$ and $2^{(\alpha-1)/\alpha}$ being absorbed into $c_\alpha$.
\end{proof}

\begin{proof}[Verification of the examples in \cref{rem:power:smoothness}]
 For the symmetric two-point example, take $x=0$ and $x\pr=t>0$, so that the means are $0$ and $t$.
 The coupling of corresponding atoms pairs $-1$ with $t-1$ and $1$ with $1+t$; both diameters in \eqref{eq:power:Wtilde} equal $1+t$, so it gives $\widetilde W_{\alpha-1}(\mu_0,\mu_t)\leq t(1+t)^{\alpha-2}$ for every $t>0$, and $\Lamof{\mu_0}{t}=(1+t)^{\alpha-2}$.
 The two cross costs are $(2-t)^{\alpha-1}$ and $(2+t)^{\alpha-1}$.
 For $0<t\leq1/2$, each corresponding-atom cost is smaller than either cross cost; thus the corresponding-atom coupling is optimal.
 The same comparison for ordinary transport gives $W_{\alpha-1}(\mu_0,\mu_t)=t^{\alpha-1}$.
 Its ratio to $Lt\Lamof{\mu_0}{Lt}=Lt(1+Lt)^{\alpha-2}$ tends to infinity as $t\searrow0$ for every fixed $L>0$ when $\alpha<2$.

 To see that the moment-only criterion \eqref{eq:power:localfrechet:sufficient} is strictly stronger for $\alpha<2$, let
 $\mathcal L(\varepsilon)=\frac12\delta_0+\frac14\delta_{-1}+\frac14\delta_1$ and $\gamma(x)=x$.
 Then $m_\alpha[\varepsilon]=0$ and $\sigma_\alpha^{\alpha-1}(\mu_x)=1/2$.
 For small $t>0$, each source atom is cheapest to serve by its synchronous partner: the cost is $t\cdot t^{\alpha-2}=t^{\alpha-1}$ for the atom at $0$ and $t(1+t)^{\alpha-2}$ for those at $\pm1$, whereas every other destination costs at least $1-t$.
 Summing these per-atom minima against the source masses $\frac12,\frac14,\frac14$ gives a lower bound equal to the synchronous cost, so
 \begin{equation}\label{eq:power:atomexample}
  \widetilde W_{\alpha-1}(\mu_0,\mu_t)=\frac{t^{\alpha-1}}2+\frac{t(1+t)^{\alpha-2}}2=\Psi_{\mu_0}(t)\sim\frac{t^{\alpha-1}}2
  \qquad(t\searrow0)
  \eqfs
 \end{equation}
 For $\alpha<2$, the right-hand side of \eqref{eq:power:localfrechet:sufficient} with $b=Lt$ is at most $2^{\phi}Lt=\mathrm{o}(t^{\alpha-1})$ as $t\searrow0$, so the criterion fails at $\beta=1$ for every finite $L$, although \eqref{eq:power:localfrechet:smooth} holds by \cref{cor:power:additive}, and in fact with $L=1$ and with equality, by \eqref{eq:power:atomexample}.
\end{proof}

\subsubsection{Hyperbolic Location--Scale Families}\label{app:power:hyperbolic}
\begin{proof}[Verification of \cref{ex:power:hyperbolic}]
	Let $L_\gamma,L_\zeta$ be H\"older constants of $\gamma,\zeta$.
	Spherical symmetry of $Z$ makes the law in \eqref{eq:power:hyperbolic:model} independent of the choice of $I_x$ and invariant under every isometry fixing $\gamma(x)$.
	Its unique $\alpha$-Fr\'echet mean is therefore fixed by that whole isotropy group; since the group acts transitively on each sphere of $T_{\gamma(x)}\mathbb H^k$ and $\exp_{\gamma(x)}$ is a bijection onto $\mathbb H^k$, the only common fixed point is $\gamma(x)$ itself, so the mean is $\gamma(x)$.
	The exponential map preserves distance from its base point, so
	$\ol{Y_x}{\gamma(x)}=\zeta(x)\euclof Z$.
	The radial density of $\euclof Z$ is of order $r^{-\theta-1}$ at infinity; hence $\Eof{\euclof Z^\alpha}<\infty$ and $\Eof{\euclof Z^2}=\infty$.
	Moreover
	\begin{equation*}
		\sup_x\VmpOfP{\alpha}{\mu_x}
		\leq2^{\alpha-1}\zeta_+^\alpha\EOf{\euclOf Z^\alpha}<\infty
		\eqfs
	\end{equation*}

	It remains to verify the symmetric smoothness condition. We first bound the cost $W_{\alpha-1}$.
	For any $\alpha\in(1,2]$, densities $f,g$ with respect to a common measure $\lambda$, and $o\in\mc Q$, retaining their common mass and coupling the remainders gives
	\begin{equation}\label{eq:power:localfrechet:commonmass}
		W_{\alpha-1}(f\lambda,g\lambda)
		\leq\int d(y,o)^{\alpha-1}\abs{f(y)-g(y)}\,\dl\lambda(y)
		\eqfs
	\end{equation}
	Indeed, $d(y,z)^{\alpha-1}\leq d(y,o)^{\alpha-1}+d(z,o)^{\alpha-1}$, and the two remainder marginals sum to $\absof{f-g}\lambda$.

	For $r=d(y,m)$, the density of $\exp_m(\zeta IZ)$ with respect to hyperbolic volume is
	\begin{equation*}
		f_{m,\zeta}(y)
		=c_{k,\theta}\zeta^{-k}
		\br{1+\frac{r^2}{\theta \zeta^2}}^{-(k+\theta)/2}
		\br{\frac r{\sinh(r)}}^{k-1}
		\eqcm
	\end{equation*}
	with the last factor interpreted continuously at $r=0$.
	Here $c_{k,\theta}$ is the normalizing constant of the standard $k$-dimensional Student density; the last factor is the reciprocal of the Jacobian $(\sinh(r)/r)^{k-1}$ of the hyperbolic exponential map.
	Direct differentiation gives
	\begin{equation*}
		\partial_r\log f_{m,\zeta}
		=-\frac{(k+\theta)r}{\theta \zeta^2+r^2}
		+(k-1)\br{\frac1r-\coth(r)}
		\eqcm\qquad
		\partial_\zeta\log f_{m,\zeta}
		=\frac1\zeta\br{-k+\frac{(k+\theta)r^2}{\theta \zeta^2+r^2}}
		\eqfs
	\end{equation*}
	Since $0\leq\coth(r)-1/r\leq1$ for $r>0$, these derivatives satisfy
	\begin{equation*}
		\abs{\partial_r\log f_{m,\zeta}}
		\leq A_r:=\frac{k+\theta}{2\sqrt\theta\,\zeta_-}+k-1
		\eqcm\qquad
		\abs{\partial_\zeta\log f_{m,\zeta}}
		\leq A_\zeta:=\frac{\max(k,\theta)}{\zeta_-}
		\eqfs
	\end{equation*}
 Fix $x,x\pr\in[0,1]$ and put $\ell:=\absof{x-x\pr}\leq1$.
 Join $\gamma(x)$ to $\gamma(x\pr)$ by a constant-speed geodesic $m_t$, $t\in[0,1]$, and set $\zeta_t=(1-t)\zeta(x)+t\zeta(x\pr)$ and $f_t=f_{m_t,\zeta_t}$.
 Since $t\mapsto d(y,m_t)$ is Lipschitz with constant $d(\gamma(x),\gamma(x\pr))\leq L_\gamma\ell^\beta$, $t\mapsto\zeta_t$ is Lipschitz with constant $\absof{\zeta(x)-\zeta(x\pr)}\leq L_\zeta\ell^\beta$, and $\log f_{m,\zeta}$ is Lipschitz in $(r,\zeta)$ by the bounds above, $t\mapsto f_t(y)$ is absolutely continuous for every $y$, and for almost every $t$,
 \[
 \abs{\partial_t f_t(y)}\leq B\ell^\beta f_t(y),
 \qquad B:=A_rL_\gamma+A_\zeta L_\zeta.
 \]
 Apply \eqref{eq:power:localfrechet:commonmass} with $o=\gamma(x)$, write $f_1-f_0=\int_0^1\partial_tf_t\,\dl t$, and use Tonelli's theorem:
 \begin{align*}
 W_{\alpha-1}(\mu_x,\mu_{x\pr})
 &\leq B\ell^\beta\int_0^1\int d(y,\gamma(x))^{\alpha-1}f_t(y)\,\dl\operatorname{vol}(y)\,\dl t\\
 &\leq B\ell^\beta\br{\zeta_+^{\alpha-1}\EOf{\euclOf{Z}^{\alpha-1}}+(L_\gamma\ell^\beta)^{\alpha-1}}
 \leq C\ell^\beta.
 \end{align*}
 Put $\bar\sigma_{\sup}:=\sup_x\sigma_\alpha^{\alpha-1}(\mu_x)\leq\zeta_+^{\alpha-1}\Eof{\euclof Z^{\alpha-1}}<\infty$, which bounds the pairwise maximum $\bar\sigma$ of \cref{lmm:power:sufficient} for every pair $\mu_x,\mu_{x\pr}$.
 Since $\widetilde W_{\alpha-1}\leq W_{\alpha-1}\leq C\ell^\beta$, the criterion \eqref{eq:power:localfrechet:sufficient} holds with $b=L\ell^\beta$ for every $L\geq\max\{C\bar\sigma_{\sup}^\phi,C^{1/(\alpha-1)}\}$, because $\ell\leq1$.
 Thus \cref{lmm:power:sufficient}\,\ref{lmm:power:sufficient:itm:tworegime} proves \eqref{eq:power:localfrechet:smooth} with the H\"older constant $2^{\phi/(\alpha-1)}L$.
 This verifies the claimed smoothness condition for arbitrary H\"older location and scale functions.
 Finally, the curves $\gamma=\exp_o\circ\gamma_0$ suggested in \cref{ex:power:hyperbolic} are $\beta$-H\"older whenever $\gamma_0$ is $\beta$-H\"older, and hence bounded, on $[0,1]$.
 By \eqref{eq:power:hyperbolic:exp}, $\exp_o$ is smooth as a map into $\R^{k+1}$, since $\cosh(\normof v)$ and $\sinh(\normof v)/\normof v$ are power series in $\normof v^2$; hence it is Lipschitz on bounded subsets of $T_o\mathbb H^k$ for the Euclidean distance of $\R^{k+1}$.
 Moreover, $d(p,q)\leq\euclof{p-q}$ for all $p,q\in\mathbb H^k$: from $\mink pp=\mink qq=-1$ we get $-\mink pq=1+\frac12\mink{p-q}{p-q}\leq1+\frac12\euclof{p-q}^2$, and $\operatorname{arcosh}(1+t)\leq\sqrt{2t}$ for $t\geq0$ because $\cosh(\sqrt{2t})\geq1+t$.
\end{proof}
\subsubsection{Dirichlet-Process Power Fr\'echet Means}\label{app:power:dirichlet}
\begin{proof}[Proof of \cref{cor:power:dirichlet}]
 Abbreviate $\tau:=\alpha/(\alpha-1)=2+\phi\geq2$ throughout this proof, so that $H_\tau=H_{2+\phi}$.
 For $V\sim\operatorname{Beta}(1,a)$ and $s\in\Rpp$, beta integration gives
 \[
 \EOf{V^s}=\frac{\Gamma(s+1)\Gamma(a+1)}{\Gamma(a+s+1)},
 \qquad \EOf{(1-V)^s}=\frac a{a+s}.
 \]
 Independence of the stick breaks and summation of a geometric series prove \eqref{eq:power:dirichlet:weights}.
 For any $o\in\mc Q$, Tonelli's theorem gives
 $\Eof{\int\ol yo^\alpha\,\dl P(y)}=\int\ol yo^\alpha\,\dl\mu(y)<\infty$.
 Thus $P\in\mc P^\alpha(\mc Q)$ almost surely, and its power Fr\'echet mean exists.

 Conditionally on the weights, the atoms in \eqref{eq:dirichlet:stickbreaking} are iid with law $\mu$ and the population mixture is $\mu$.
 Write $S_t:=\sum_iW_i^t$ and $v_t:=\VmpofP{t}{\mu}$.
 Then $\Sigma_{2\alpha-2}=v_{2\alpha-2}S_2$ and $\Sigma_\alpha=v_\alpha S_\tau$.
 Applying \cref{thm:power:main} conditionally gives, with $\sigma_{\alpha-1}:=\sigma_\alpha^{\alpha-1}(\mu)$ and the constant $c_\alpha$ of that theorem,
 \[
 \br{\EOf{\ol{m_\alpha(P)}{m_\alpha(\mu)}^\alpha\,\middle|\,W}}^{1/\alpha}
 \leq c_\alpha\br{
 \sigma_{\alpha-1}^\phi v_{2\alpha-2}^{1/2}S_2^{1/2}
 +\sigma_{\alpha-1}^{\phi/2}v_\alpha^{1/2}S_\tau^{1/2}
 +v_\alpha^{1/\alpha}S_\tau^{1/\alpha}}.
 \]
 Take the $L^\alpha$ norm over the weights on both sides and use Minkowski's inequality.
 Since $\alpha/2\leq1$, Jensen's inequality yields
 $\normof{S_t^{1/2}}_{L^\alpha}\leq(\Eof{S_t})^{1/2}$,
 while $\normof{S_\tau^{1/\alpha}}_{L^\alpha}=(\Eof{S_\tau})^{1/\alpha}$.
 Substitution of $\Eof{S_2}=H_2(a)$ and $\Eof{S_\tau}=H_\tau(a)$ proves \eqref{eq:power:dirichlet:prior}.
 The conjugacy statement in \cref{ssec:dirichlet} and $\tilde\mu_n\in\mc P^\alpha(\mc Q)$ give \eqref{eq:power:dirichlet:posterior} by the same argument.

 To prove the simpler bound, first observe that
 \begin{equation}\label{eq:power:dirichlet:Hbound}
 H_\tau(b)\leq\Gamma(\tau)(b+1)^{1-\tau},\qquad b>0,
 \end{equation}
 because a gamma random variable $Z$ with shape $b+1$ and unit scale satisfies
 $\Gamma(b+\tau)/\Gamma(b+1)=\Eof{Z^{\tau-1}}\geq(\Eof Z)^{\tau-1}=(b+1)^{\tau-1}$ by Jensen's inequality.
 Put $\mathcal M:=\VmpofP{\alpha}{\nu}$.
 Lyapunov's inequality and \cref{lmm:variances} imply
 $\sigma_\alpha^{\alpha-1}(\nu)\leq\sigma_\alpha^\alpha(\nu)^{(\alpha-1)/\alpha}=\VfpofP{\alpha}{\nu}^{(\alpha-1)/\alpha}\leq(2\mathcal M)^{(\alpha-1)/\alpha}$,
 and Lyapunov's inequality gives
 $\VmpofP{2\alpha-2}{\nu}\leq c_\alpha\mathcal M^{(2\alpha-2)/\alpha}$.
 Consequently each term of \eqref{eq:power:dirichlet:R} is bounded by $c_\alpha\mathcal M^{1/\alpha}$ times its weight factor.
 Since $\tau-1=1/(\alpha-1)$, $1/(2(\alpha-1))\geq1/2$, and $1/(\alpha(\alpha-1))\geq1/2$, \eqref{eq:power:dirichlet:Hbound} gives
 \begin{equation}\label{eq:power:dirichlet:coarse}
 \mathcal R_\alpha(b,\nu)
 \leq c_\alpha\frac{\VmpofP{\alpha}{\nu}^{1/\alpha}}{\sqrt{b+1}}.
 \end{equation}
 The case $\mathcal M=0$ is immediate, including at $\alpha=2$ under the exponent-zero convention.
 Apply this with $b=a+n$ and $\nu=\tilde\mu_n$.
 The computation leading to \eqref{eq:dirichlet:bayesrisk} applies verbatim with $d^\alpha$ in place of $d^2$ and gives $\Eof{\VmpofP{\alpha}{\tilde\mu_n}}=\frac{a(a+n+1)}{(a+1)(a+n)}\VmpofP{\alpha}{\mu}$.
 Raise the conditional bound to the $\alpha$-th power and take expectations to obtain \eqref{eq:power:dirichlet:bayes}.
 The prior case follows directly from \eqref{eq:power:dirichlet:prior} and \eqref{eq:power:dirichlet:coarse}.
\end{proof}

\subsection{Details Supporting the Regression Comparisons}\label{app:comparison}
The deductions below explain the specializations used in \cref{ssec:localfrechet,ssec:pos:localfrechet,ssec:power:localfrechet}.

\paragraph{From density smoothness to transport smoothness.}
The common-mass coupling estimate \eqref{eq:power:localfrechet:commonmass} at $\alpha=2$, applied to $\mu=f\lambda$, $\nu=g\lambda$ and a fixed reference point $o$, gives
\[
 W_1(f\lambda,g\lambda)\leq\int\ol yo\,\abs{f-g}\,\dl\lambda.
\]
Consequently, $\absof{f_x-f_{x\pr}}\leq L\absof{x-x\pr}^{\beta}$ with $L\in L^2(\lambda)$ implies
\begin{equation*}
 W_1(\mu_x,\mu_{x\pr})\leq\abs{x-x\pr}^{\beta}
 \br{\int\ol yo^2\,\dl\lambda}^{1/2}\normOf{L}_{L^2(\lambda)},
\end{equation*}
which is the smoothness condition of \cref{cor:two:localreg,cor:pos:localfrechet}.
Conversely, a common dominating probability measure for $\{\delta_{\gamma(x)}:x\in[0,1]\}$ with $\gamma$ injective would have to assign positive mass to uncountably many distinct points, which is impossible. The transport condition therefore does not require a common-density representation.

\paragraph{From the local constant density condition of Petersen and M\"uller to that of Sch\"otz.}
For the local constant version of the arguments of \cite{PetersenMueller2019Regression-AOS}, let the covariate $X$ have a Lipschitz density $f_X$ on $\R$, let $J$ be an interval with $f_X\geq f_->0$ on $J$, and let the conditional densities $g_y$ of $X$ given $Y=y$ be Lipschitz with a constant $L_g$ not depending on $y$.
These are the first-order density assumptions used here; the published local linear condition (L1) uses twice continuously differentiable densities and a uniform bound on $g_y^{\prime\prime}$.
As in (L1), take the conditional laws to be continuous in $x$ on open sets, so that their Bayes representation holds throughout $J$.
A Lipschitz probability density on $\R$ is bounded by the square root of its Lipschitz constant, so $g_y\leq\sqrt{L_g}$.
With $\nu$ the law of $Y$, Bayes' formula gives $f(y\mid x)=g_y(x)/f_X(x)$ for $x\in J$, and for $x,x\pr\in J$,
\begin{equation*}
 \begin{split}
 \abs{f(y\mid x)-f(y\mid x\pr)}
 &\leq\frac{\abs{g_y(x)-g_y(x\pr)}}{f_X(x)}+g_y(x\pr)\abs{\frac1{f_X(x)}-\frac1{f_X(x\pr)}}\\
 &\leq\br{\frac{L_g}{f_-}+\frac{\sqrt{L_g}\operatorname{Lip}(f_X)}{f_-^2}}\abs{x-x\pr}\eqfs
 \end{split}
\end{equation*}
Thus the density smoothness condition of \cite{Schoetz2022Regression-EJS} holds on $J$ with the constant function $L(y)=L_g/f_-+\sqrt{L_g}\operatorname{Lip}(f_X)/f_-^2$, and one may take $C_{\mathrm{SmD}}=\max\{1,\normof{L}_{L^2(\nu)}\}$.
The second moment $\int\ol y{m_2(\mu_{x_0})}^2\,\dl\nu(y)$ of the reference measure, for a fixed $x_0\in J$, is finite whenever $Y$ has a finite second moment, in particular for the bounded response spaces of \cite{PetersenMueller2019Regression-AOS}.

\begin{proof}[Verification of \cref{ex:two:phase}]
Write $\Gamma:=\gamma(\R)$ and let $\lambda$ be the image of the Lebesgue measure under $\gamma$, the arc-length measure on $\Gamma$.

\emph{(i) Transport smoothness.}
The coupling $(\gamma(x+\tau U),\gamma(x\pr+\tau U))$ of $\mu_x$ and $\mu_{x\pr}$ has cost $\Eof{\ol{\gamma(x+\tau U)}{\gamma(x\pr+\tau U)}}=\absof{x-x\pr}$, so $W_1(\mu_x,\mu_{x\pr})\leq\absof{x-x\pr}$.
For the reverse inequality and the formula for the mean, note that $\Gamma$ is closed and convex in $\mc Q$: it is complete as an isometric image of $\R$, and the unique geodesic between $\gamma(s)$ and $\gamma(t)$ is $\gamma|_{[s,t]}$.
The metric projection $P\colon\mc Q\to\Gamma$ is nonexpansive with $P|_\Gamma=\id$ \cite[Proposition II.2.4]{BridsonHaefliger1999NPC-Springer}, so $\ol{P(q)}{y}\leq\ol qy$ for all $y\in\Gamma$ and $\Eof{\ol{Y_x}{P(q)}^2}\leq\Eof{\ol{Y_x}q^2}$ for all $q\in\mc Q$.
Hence the unique minimizer $m_2(\mu_x)$ \cite[Proposition 4.3]{Sturm2003NPC-ContempMath} lies in $\Gamma$; as $\gamma$ is an isometry of $\R$ onto $\Gamma$, $\Eof{\ol{Y_x}{\gamma(t)}^2}=\Eof{(x+\tau U-t)^2}$ is minimized exactly at $t=x$, i.e., $m_2(\mu_x)=\gamma(x)$.
Now \eqref{eq:cat0:2fm:CP} gives $\absof{x-x\pr}=\ol{\gamma(x)}{\gamma(x\pr)}\leq W_1(\mu_x,\mu_{x\pr})$, which proves the first identity in \eqref{eq:two:phase:w1}.
For an independent copy $U\pr$ of $U$, $\VmofP{\mu_x}=\frac12\Eof{\ol{\gamma(x+\tau U)}{\gamma(x+\tau U\pr)}^2}=\frac{\tau^2}2\Eof{(U-U\pr)^2}=\tau^2\Eof{U^2}$.
For the design $x_i=i/n$, $x\in[0,1]$, and $4/n\leq h\leq1$, one of the half-windows $[x,x+h/2]$ and $[x-h/2,x]$ lies in $[0,1]$; it has length $h/2$ and therefore contains at least $\lfloor nh/2\rfloor\geq nh/2-1\geq nh/4$ of the points $i/n$, $i\in\intset n$.
Thus $N_n(x,h/2)\geq nh/4$, which is \eqref{eq:localfrechet:design} with $D=1$ and $\const{dns}=4$, and inserting $L=1$, $\beta=1$, $v=\tau^2\Eof{U^2}$, $D=1$, and $\const{dns}=4$ into \eqref{eq:localfrechet:designbound} yields \eqref{eq:two:phase:bound}.

\emph{(ii) Density smoothness, uniform noise.}
Let $\tau>0$ and $U\sim\dunif{[-1,1]}$, and put $A_x:=\gamma([x-\tau,x+\tau])$, so that $\mu_x=(2\tau)^{-1}\lambda|_{A_x}$.
Let $\nu$ be a probability measure with $\mu_x\ll\nu$ for all $x\in[0,1]$.
Every $s\in[-\tau,1+\tau]$ admits a rational $x\in[0,1]$ with $\absof{s-x}\leq\tau$, so $A:=\gamma([-\tau,1+\tau])$ is the union of the sets $A_x$ over $x\in T:=\Q\cap[0,1]$, and $\lambda|_A\ll\nu$; let $g$ be a density of $\lambda|_A$ with respect to $\nu$.
Then $f_x:=(2\tau)^{-1}g\ind_{A_x}$ is a version of $\dl\mu_x/\dl\nu$, and every other version agrees with $f_x$ outside a $\nu$-null set.
Suppose $\tilde f(\cdot\mid x)$, $x\in[0,1]$, were versions such that $x\mapsto\tilde f(y\mid x)$ is continuous for all $y$ outside a $\nu$-null set $N_0$.
Then $N:=N_0\cup\bigcup_{x\in T}\{\tilde f(\cdot\mid x)\neq f_x\}$ is $\nu$-null.
The set $B:=\{\gamma(s):s\in(-\tau,1-\tau),\,g(\gamma(s))>0\}$ has positive $\nu$-measure, because $1=\lambda(\gamma((-\tau,1-\tau)))=\int_{\gamma((-\tau,1-\tau))}g\,\dl\nu$.
Pick $y=\gamma(s)\in B\setminus N$.
For $x\in T$ we have $\tilde f(y\mid x)=(2\tau)^{-1}g(y)\ind\{\absof{s-x}\leq\tau\}$, which equals $(2\tau)^{-1}g(y)>0$ for rational $x\nearrow s+\tau$ and $0$ for rational $x\searrow s+\tau$; since $s+\tau\in(0,1)$, this contradicts the continuity of $x\mapsto\tilde f(y\mid x)$.
Hence no such versions exist, and the assumption of \cite{Schoetz2022Regression-EJS} fails for every $\beta>0$ and every $\nu$.
For random covariates with a positive Lipschitz density $f_X$ on the interior of the covariate interval, the conditional density of $X$ given $Y=\gamma(s)$ is proportional to $x\mapsto f_X(x)\ind\{\absof{s-x}\leq\tau\}$.
Whenever $s+\tau\in(0,1)$, it has a jump there, so the local constant Lipschitz condition on $g_y$ fails as well.

\emph{(ii) Density smoothness, triangular noise.}
Now let $U$ have the density $\psi(u):=(1-\absof u)_+$, so that $\Eof{U^2}=1/6$, and let $\nu:=(1+2\tau)^{-1}\lambda|_A$ with $A=\gamma([-\tau,1+\tau])$ as before.
For $y=\gamma(s)\in A$, the density of $\mu_x$ with respect to $\nu$ is $f(y\mid x)=(1+2\tau)\tau^{-1}\psi((s-x)/\tau)$, and since $\psi$ is $1$-Lipschitz, $\absof{f(y\mid x)-f(y\mid x\pr)}\leq(1+2\tau)\tau^{-2}\absof{x-x\pr}$.
Thus the density smoothness assumption of \cite{Schoetz2022Regression-EJS} holds with $\beta=1$, the constant function $L\equiv(1+2\tau)/\tau^2$, $C_{\mathrm{SmD}}=\max\{1,(1+2\tau)/\tau^2\}$, and the continuous map $x\mapsto m_2(\mu_x)=\gamma(x)$.
For its remaining constants, take $C_{\mathrm{Len}}=1$, which bounds $\sup_{x,x\pr\in[0,1]}\ol{m_2(\mu_x)}{m_2(\mu_{x\pr})}$, and
\[
 C_{\mathrm{Int}}=\max\left\{1,\frac{1+\tau+\tau^2}{3}\right\},
 \qquad
 \int\ol y{\gamma(0)}^2\,\dl\nu(y)
 =\frac{(1+\tau)^3+\tau^3}{3(1+2\tau)}
 =\frac{1+\tau+\tau^2}{3}.
\]

For the lower bound with $\beta=1$, let $\nu$ be any admissible reference measure, with Lipschitz constants $L(y)$.
As in the uniform case, $\lambda|_A\ll\nu$, since $\mu_x$ has a positive density with respect to $\lambda$ on $\gamma((x-\tau,x+\tau))$; let $g$ be a density of $\lambda|_A$ with respect to $\nu$.
For each $x\in T$, $f(\gamma(s)\mid x)=\tau^{-1}\psi((s-x)/\tau)\,g(\gamma(s))$ for $\nu$-almost every $\gamma(s)\in A$; as both sides are continuous in $x$ for $\nu$-almost every $y$, this holds for all $x\in[0,1]$ outside a single $\nu$-null set.
For $s\in(-\tau,1+\tau)$, the map $x\mapsto\tau^{-1}\psi((s-x)/\tau)$ has slope $\pm\tau^{-2}$ on a nondegenerate subinterval of $[0,1]$, so $L(y)\geq\tau^{-2}g(y)$ for $\nu$-almost every $y\in A$, and
\[
 C_{\mathrm{SmD}}\geq\normof{L}_{L^2(\nu)}\geq\normof{L}_{L^1(\nu)}\geq\tau^{-2}\int_Ag\,\dl\nu=\tau^{-2}\lambda(A)=\frac{1+2\tau}{\tau^2}\eqfs
\]
For $\beta\in(0,1]$ and $\tau\leq\frac12$, let $\nu$ and $L$ be admissible with exponent $\beta$ and let $g$ be as above; as for $\beta=1$, $f(\gamma(s)\mid x)=\tau^{-1}\psi((s-x)/\tau)\,g(\gamma(s))$ for all $x\in[0,1]$ and $\nu$-almost every $\gamma(s)\in A$.
For $s\in[0,1]$, one of the intervals $[s-\tau,s]$ and $[s,s+\tau]$ lies in $[0,1]$, and between its endpoints $x\mapsto\tau^{-1}\psi((s-x)/\tau)$ changes by $\tau^{-1}$; hence $L(y)\geq\tau^{-1-\beta}g(y)$ for $\nu$-almost every $y\in\gamma([0,1])$, and
\[
 C_{\mathrm{SmD}}\geq\normof{L}_{L^1(\nu)}\geq\tau^{-1-\beta}\lambda(\gamma([0,1]))=\tau^{-1-\beta}\eqfs
\]
Conversely, $\absof{\psi(u)-\psi(u\pr)}\leq\min\{1,\absof{u-u\pr}\}\leq\absof{u-u\pr}^\beta$, so the normalized arc-length measure on $A$ is admissible with $L\equiv(1+2\tau)\tau^{-1-\beta}\geq1$.
Finally, for every admissible choice, the transport constant $C_{\mathrm{Int}}^{1/2}C_{\mathrm{SmD}}$ that the implication at the beginning of this subsection extracts from the density condition is at least $(1+2\tau)/\tau^2$, since $C_{\mathrm{Int}}\geq1$, against the true value $L=1$ in \eqref{eq:two:phase:w1}.
\end{proof}

\paragraph{Transport versus the objective condition.}
If $\mc Q$ is bounded, \eqref{eq:two:localreg:smooth} implies \eqref{eq:two:localreg:objective}: for $q\in\mc Q$ and any coupling $\pi$ of $\mu_x$ and $\mu_{x\pr}$, the identity $\absof{a-b}(a+b)=\absof{a^2-b^2}$ and the reverse triangle inequality give
\begin{equation*}
 \abs{F_x(q)-F_{x\pr}(q)}
 \leq\int\abs{\ol yq-\ol{y\pr}q}\br{\ol yq+\ol{y\pr}q}\,\dl\pi(y,y\pr)
 \leq2\diam(\mc Q)\int\ol y{y\pr}\,\dl\pi(y,y\pr)
 \eqcm
\end{equation*}
and taking the infimum over $\pi$ yields \eqref{eq:two:localreg:objective} with constant $2\diam(\mc Q)L$.
The supremum over all $q\in\mc Q$ in \eqref{eq:two:localreg:objective} makes boundedness necessary: for $\mu_x:=\delta_x$ on $\R$, \eqref{eq:two:localreg:smooth} holds with $L=\beta=1$, but $\sup_q\absof{(x-q)^2-(x\pr-q)^2}=\sup_q\absof{x-x\pr}\,\absof{x+x\pr-2q}=\infty$ for $x\neq x\pr$; \cref{ex:two:phase} with a full geodesic line $\gamma$ fails in the same way.
The implication is strict: on $\mc Q=[-1,1]$, the laws $\mu_x:=\frac12(\delta_{-\sqrt x}+\delta_{\sqrt x})$, $x\in[0,1]$, have $F_x(q)=q^2+x$, so \eqref{eq:two:localreg:objective} holds with $\beta=1$ and $L=1$, whereas $W_1(\mu_x,\mu_{x\pr})=\absof{\sqrt x-\sqrt{x\pr}}$ makes \eqref{eq:two:localreg:smooth} fail at $\beta=1$ and hold only with $\beta=\frac12$.

\paragraph{A hyperbolic specialization of the earlier entropy condition.}
For $\mathbb H^k$ of curvature $-1$, comparing the volumes of disjoint balls of radius $\epsilon/2$ inside a ball of radius $\diam(B)+\epsilon/2$, with the polar volume formula $V_k(r)=\omega_k\int_0^r\sinh(t)^{k-1}\,\dl t$, gives $\log N(B,d,\epsilon)\leq c_k(\diam(B)+1+\log(\diam(B)/\epsilon))$ for every bounded subset $B$ with $\diam(B)>0$ and $0<\epsilon\leq\diam(B)$.
For Talagrand's functional $\gamma_2$ \cite[Definition 2.7.3]{Talagrand2021StochasticProcesses-Springer}, the entropy-integral bound used in \cite[(6), p.~4686]{Schoetz2022Regression-EJS} then gives $\gamma_2(B,d)\leq c_k\max(\diam(B),\diam(B)^{3/2})$, so the entropy exponent $3/2$ is admissible there, and with this estimate the moment condition of that theorem requires a moment of order greater than $4$ on full hyperbolic space, whereas \cref{cor:two:localreg} uses only the second moment.
The exponent $3/2$ cannot be lowered: for a ball $B$ of large radius $r$ in $\mathbb H^k$, $k\geq2$, the same volume comparison gives $\log N(B,d,r/4)\geq c_kr$, and $\gamma_2(B,d)\geq c\,\epsilon\sqrt{\log N(B,d,\epsilon)}$ for every $\epsilon>0$ and an absolute constant $c>0$ \cite[Exercise 2.7.8(b)]{Talagrand2021StochasticProcesses-Springer} then gives $\gamma_2(B,d)\geq c_kr^{3/2}$. Hence, on full hyperbolic space, that framework requires a moment of order greater than $4$ whatever entropy estimate is used.

\paragraph{The power target and the squared-loss frameworks.}
For $1<\alpha<2$, a comparison must concern estimators of $m_\alpha(\mu_x)$.
The target can be encoded in the earlier squared-loss frameworks through the metric
\begin{equation}\label{eq:comparison:snowflake}
 d_\alpha(q,p):=d(q,p)^{\alpha/2},
\end{equation}
for which $d_\alpha^2=d^\alpha$, so that a second moment in $d_\alpha$ is an $\alpha$-moment in $d$. Two hypotheses of those frameworks are then lost.
First, $d_\alpha$ is not geodesic on any nontrivial geodesic space: a midpoint $z$ of $q\neq p$ for $d_\alpha$ would satisfy $d(q,z)=d(z,p)=2^{-2/\alpha}d(q,p)$, contradicting the triangle inequality for $d$ since $2^{1-2/\alpha}<1$. So the Hadamard theorem of \cite{Schoetz2022Regression-EJS} does not apply to $(\mc Q,d_\alpha)$.
Second, the bounded-space Theorem 1 and the general Theorem 7 of \cite{Schoetz2022Regression-EJS} require a quadratic variance inequality in $d_\alpha$, that is, excess risk at least a positive multiple of $d(q,m)^\alpha$.
This fails already for the equally weighted law on $\{-1,1\}$, even with the response space restricted to $[-2,2]$:
\begin{equation*}
 \begin{aligned}
 \frac{\abs{1-t}^\alpha+\abs{1+t}^\alpha}{2}-1
 &=\frac{\alpha(\alpha-1)}2t^2+\landauO(t^4)\\
 &=\mathrm{o}(\abs t^\alpha),\qquad t\longrightarrow0.
 \end{aligned}
\end{equation*}
It also fails for the hyperbolic family \eqref{eq:power:hyperbolic:model}: with $m=\gamma(x)$ and a unit-speed geodesic $q_t$ through $m$, the radial and transverse Hessian eigenvalues of $r^\alpha$ are $\alpha(\alpha-1)r^{\alpha-2}$ and $\alpha r^{\alpha-1}\coth(r)$, and averaging over directions gives
\begin{equation*}
 \begin{aligned}
 \EOf{\ol{Y_x}{q_t}^\alpha-\ol{Y_x}m^\alpha}
 &=\frac{\alpha}{2k}\Bigl((\alpha-1)\EOf{\ol{Y_x}m^{\alpha-2}}\\
 &\qquad +(k-1)\EOf{\ol{Y_x}m^{\alpha-1}\coth(\ol{Y_x}m)}\Bigr)t^2
       +\mathrm{o}(t^2),\qquad t\to0.
 \end{aligned}
\end{equation*}
Both expectations are finite, because the radial density is of order $r^{k-1}$ at zero and $r^{-\theta-1}$ at infinity, with $k+\alpha-2>0$ and $\theta>\alpha$.
For the expansion, apply Taylor's theorem with integral remainder to $s\mapsto\ol y{q_s}^\alpha$, which is twice continuously differentiable on $[-\absof t,\absof t]$ for every $y$ off the geodesic segment $q([-\absof t,\absof t])$, a null set for $\mu_x$.
Its second derivative is nonnegative, so Tonelli's theorem allows taking the expectation of the remainder, and $s\mapsto\Eof{\partial_s^2\ol{Y_x}{q_s}^\alpha}$ is continuous at $s=0$, because $Y_x$ has a density that is bounded near $m$ and $r\mapsto r^{\alpha-2}$ is integrable near zero against $r^{k-1}\,\dl r$; the first-order term vanishes because $m$ is the mean, and isotropy turns the Hessian into the direction average above.
The coefficient is positive and finite, so the excess risk is $\mathrm{o}(\absof t^\alpha)$ as $t\to0$ for this full-support, infinite-variance family as well.
In addition, the unbounded-space moment requirement of \cite{Schoetz2022Regression-EJS}, read in $d_\alpha$, is a moment strictly above $\alpha$.
The local constant criterion (4.1) of \cite{PetersenMueller2019Regression-AOS} becomes the present power criterion under \eqref{eq:comparison:snowflake}, but their local rate theorems are local linear, and their entropy and growth conditions would have to be verified in $d_\alpha$.
\section{Analytic Tools}\label{sec:powfun}
This appendix collects standard scalar power inequalities (cf.\ \cite[Lemmas S1.1--S1.3]{Schoetz2026Transformed-arxiv}), the loss conversion of \cite[Lemma S3.1(i)]{Schoetz2026Transformed-arxiv}, and Rosenthal's inequality. We include proofs to make the analytic inputs self-contained.

\textbf{Notation.}
The following scalar notation is local to this appendix; the conventions of \cref{sec:prelim} apply.
\begin{center}
\begin{tabular}{ll}
\toprule
$x,y,x_i$ & nonnegative scalar arguments\\
$p$, $r$ & power exponents; \cref{lmm:power:subadd:gen,lmm:power:ana:diff,lmm:power:holder}\\
$\alpha\in(1,2]$, $\phi=(2-\alpha)/(\alpha-1)$ & mean order and exponent; \cref{prp:app:conversion}\\
$b>0$, $\psi(x)=x^2(b+x^{\alpha-1})^{-\phi}$ & scale and loss in \cref{prp:app:conversion}\\
$s$, $c_s$ & moment order and constant in \cref{lmm:rosenthal}\\
\bottomrule
\end{tabular}
\end{center}

\begin{lemma}[Powers of sums]\label{lmm:power:subadd:gen}
    Let $n\in\N$, $x_1,\dots,x_n \in\Rp$ and $p\in\Rp$, with the convention $0^0:=1$. Then
    \begin{equation*}
        \br{\sum_{i=1}^n x_i}^{\!p} \leq n^{\max(0, p-1)}\sum_{i=1}^n x_i^p
        \eqfs
    \end{equation*}
\end{lemma}
\begin{proof}
 For $0<p\leq1$ this is concavity and subadditivity; for $p\geq1$ it is Jensen's inequality with equal weights; and $p=0$ follows from the convention.
\end{proof}
\Needspace{9\baselineskip}
\begin{lemma}[Differences of powers]\label{lmm:power:ana:diff}
    Let $x_1, x_2 \in\Rp$ and $p\in(1,2]$. Then
    \begin{equation*}
        \abs{x_1^p - x_2^p}
        \leq
        \abs{x_1 - x_2}\br{x_1^{p-1} + x_2^{p-1}}
        \eqfs
    \end{equation*}
\end{lemma}
\begin{proof}
 The case $x_1=x_2$ is immediate. By symmetry assume $x_1>x_2$ and put $r=p-1\in(0,1]$. The identity
 \[
 x_1^p-x_2^p-(x_1-x_2)(x_1^r+x_2^r)
 =x_2x_1^r-x_1x_2^r
 \]
 has a nonpositive right-hand side because $r\leq1$.
\end{proof}
\begin{lemma}[Reverse triangle inequality for concave powers]\label{lmm:power:holder}
	Let $r\in(0,1]$ and $x,y\in\Rp$. Then $\absof{x^r - y^r} \leq \absof{x-y}^r$.
\end{lemma}
\begin{proof}
	Assume $x\geq y$. By \cref{lmm:power:subadd:gen} with $p = r\in(0,1]$, so that $2^{\max(0,r-1)} = 1$, applied to $x - y$ and $y$,
	$x^r = ((x-y)+y)^r \leq (x-y)^r + y^r$.
\end{proof}

The next result is the second inequality of \cite[Lemma S3.1(i)]{Schoetz2026Transformed-arxiv}. We retain a direct proof to keep the loss conversion self-contained.
It converts the loss $g$ of \cref{prp:power:excess} into the standard $L^\alpha$-loss while keeping track of the scaling; it is used in the proof of \cref{thm:power:main}.
\begin{proposition}[Conversion of the loss]\label{prp:app:conversion}
	Let $\alpha\in(1,2]$, $\phi = \frac{2-\alpha}{\alpha-1}$, and $b\in\Rpp$, and set $\psi(x) := x^2(b+x^{\alpha-1})^{-\phi}$ for $x\in\Rp$.
	Then, for every nonnegative random variable $X$,
	\begin{equation}\label{eq:app:conversion}
		\EOf{X^\alpha}
		\leq
		b^{\frac{\phi\alpha}{2}}\, \EOf{\psi(X)}^{\frac\alpha2}\br{1 + b^{-\frac\alpha2}\, \EOf{\psi(X)}^{\frac{\alpha-1}{2}}}^{\phi}
		\eqfs
	\end{equation}
\end{proposition}
\begin{proof}
 Put $s=(\alpha-1)/\alpha$ and $\omega=(2-\alpha)/\alpha$. First take $b=1$ and bounded $X$.
 Set
 \[
 f(y)=\frac{y^{2/\alpha}}{(1+y^{(\alpha-1)/\alpha})^\phi}=yh(y),
 \qquad h(y)=\br{\frac{y^s}{1+y^s}}^\phi,\qquad f(0)=0.
 \]
 For $y>0$, direct differentiation yields
 \[
 f''(y)=\frac{\omega\,h(y)\br{1+\omega+(1-s)y^s}}{y(1+y^s)^2}\geq0.
 \]
 Thus $f$ is convex and increasing on $\Rp$; at $\alpha=2$ it is the identity.
 With $r=\Eof{\psi(X)}$, Jensen gives $f(\Eof{X^\alpha})\leq r$. If $r=0$, $X=0$ almost surely.
 Otherwise define $y_*=r^{\alpha/2}(1+r^{(\alpha-1)/2})^\phi$. Writing $t=r^{(\alpha-1)/2}$, the inequality $f(y_*)\geq r$ reduces to
 \[
 (1+t)^{1+\omega}\geq1+t(1+t)^\omega,
 \]
 which follows from $(1+t)^\omega\geq1$. Monotonicity implies $\Eof{X^\alpha}\leq y_*$.
 Applying the result to $X\wedge M$ and then taking $M\to\infty$ proves it for arbitrary nonnegative $X$, with extended expectations allowed.
 Finally write $\psi_b$ for the displayed loss at scale $b$ and set $X=b^{1/(\alpha-1)}Z$. Then $\psi_b(X)=b^{\alpha/(\alpha-1)}\psi_1(Z)$, and rescaling gives exactly \eqref{eq:app:conversion}.
\end{proof}

Finally, the proof of \cref{lmm:power:empmom} uses the Rosenthal inequality.
\Needspace{10\baselineskip}
\begin{lemma}[Rosenthal's inequality; {\cite[Theorem 3]{Rosenthal1970Subspaces-IJM}}]\label{lmm:rosenthal}
	Let $X_1, \dots, X_n$ be real-valued, independent, centered random variables and let $s \geq 2$. Then
	\begin{equation*}
		\EOf{\abs{\sum_{i=1}^n X_i}^s} \leq c_s \max\brOf{ \sum_{i=1}^n\EOf{|X_i|^s},\ \br{\sum_{i=1}^n \EOf{X_i^2}}^{\frac s2}}
		\eqcm
	\end{equation*}
 where $c_s<\infty$ depends only on $s$.
 The constant can be chosen locally bounded in $s$: for symmetric $X_i$ and $s>2$, Johnson, Schechtman, and Zinn \cite[Theorem 4.1]{JohnsonEtAl1985MomentInequalities-AOP} give the inequality with $c_s\leq(Ks/\max\{1,\log(s)\})^s$ for a universal constant $K\leq7.35$, and replacing $X_i$ by $X_i-X_i\pr$ for independent copies $X_i\pr$ extends this to centered $X_i$ with $2K$ in place of $K$; for $s=2$, $c_2=1$.
\end{lemma}

\bibliographystyle{alpha-doi}
\apptocmd{\thebibliography}{\setlength{\itemsep}{2pt}\setlength{\parsep}{0pt}\setlength{\parskip}{0pt}}{}{}
\bibliography{literature}
\end{document}